\documentclass[a4paper,amscd,amssymb,epic,eepic,leqno,12pt]{amsbook}
\usepackage{hyperref,url}
\usepackage{amssymb,latexsym}
\usepackage{amsmath,amsthm}
\usepackage[latin1]{inputenc}
\usepackage{enumerate}
\usepackage[all]{xy} \CompileMatrices
\usepackage{latexsym,amsmath,amsfonts,amscd,amssymb,verbatim,epic,hyperref}

\newcommand{\secref}[1]{\S\ref{#1}}

\theoremstyle{plain}
\newtheorem{theorem}{Theorem}[section]
\newtheorem{lemma}[theorem]{Lemma}
\newtheorem{corollary}[theorem]{Corollary}
\newtheorem{proposition}[theorem]{Proposition}
\newtheorem{conjecture}[theorem]{Conjecture}
\theoremstyle{definition}
\newtheorem{definition}[theorem]{Definition}
\newtheorem{definition-theorem}[theorem]{Definition-Theorem}
\newtheorem{example}[theorem]{Example}
\theoremstyle{remark}
\newtheorem{remark}[theorem]{Remark}

\DeclareFontFamily{OT1}{rsfs}{}
\DeclareFontShape{OT1}{rsfs}{n}{it}{<->rsfs10}{}
\DeclareMathAlphabet{\curly}{OT1}{rsfs}{n}{it}

\newfont{\bfc}{cmbsy10 scaled 1200}  
\newfont{\dr}{msbm10 scaled \magstep1}  
\newfont{\sdr}{msbm8}  
\newfont{\gl}{eufm10 scaled \magstep1}  

\newcommand{\tr}{\operatorname{tr}}
\newcommand{\pr}{p}
\newcommand{\Id}{\operatorname{Id}}

\newcommand{\End}{\operatorname{End}}
\newcommand{\Hom}{\operatorname{Hom}}
\newcommand{\ad}{\operatorname{ad}}
\newcommand{\Ad}{\operatorname{Ad}}
\newcommand{\Aut}{\operatorname{Aut}}
\newcommand{\dbar}{\bar{\partial}}
\newcommand{\imag}{\mathop{{\fam0 {\textbf{i}}}}\nolimits}
\newcommand{\CC}{{\mathbb C}}
\newcommand{\PP}{{\mathbb P}}

\newcommand{\RR}{{\mathbb R}}
\newcommand{\ZZ}{{\mathbb Z}}

\renewcommand{\(}{\left(}
\renewcommand{\)}{\right)}
\newcommand{\vol}{\operatorname{vol}}
\newcommand{\Vol}{\operatorname{Vol}}
\newcommand{\defeq}{\mathrel{\mathop:}=} 

\newcommand{\wt}[1]{\widetilde{#1}}

\newcommand{\surj}{\to\kern-1.8ex\to}
\newcommand{\hra}{\hookrightarrow}
\newcommand{\lto}{\longrightarrow}
\newcommand{\lra}[1]{\stackrel{#1}{\longrightarrow}}

\renewcommand{\implies}{\Rightarrow}

\newcommand{\cA}{\mathcal{A}}

\newcommand{\cE}{\mathcal{E}}
\newcommand{\cI}{\mathcal{I}}

\newcommand{\cJ}{\mathcal{J}}
\newcommand{\cJi}{\mathcal{J}^{i}}
\newcommand{\cK}{\mathcal{K}}
\newcommand{\cKt}{\wt{\cK}}
\newcommand{\cM}{\mathcal{M}}
\newcommand{\cP}{\mathcal{P}}
\newcommand{\cF}{\mathcal{F}}

\newcommand{\cW}{\mathcal{W}}

\newcommand{\cY}{\mathcal{Y}}
\newcommand{\cYt}{\wt{\cY}}

\newcommand{\cG}{\mathcal{G}}
\newcommand{\cL}{\mathcal{L}}
\newcommand{\cO}{\mathcal{O}}

\newcommand{\cR}{\mathcal{R}}

\newcommand{\Lie}{\operatorname{Lie}}
\newcommand{\Tot}{\operatorname{Tot}}
\newcommand{\LieG}{\operatorname{Lie} \cG}
\newcommand{\cX}{\widetilde{\mathcal{G}}}
\newcommand{\LieX}{\operatorname{Lie} \cX}

\newcommand{\cH}{\mathcal{H}} 
\newcommand{\LieH}{\Lie\cH}

\newcommand{\GL}{\operatorname{GL}}

\newcommand{\Diff}{\operatorname{Diff}}

\newcommand{\Hilb}{\textnormal{Hilb}}
\newcommand{\Quot}{\textnormal{Quot}}

\DeclareFontFamily{OT1}{rsfs}{}
\DeclareFontShape{OT1}{rsfs}{n}{it}{<->rsfs10}{}
\DeclareMathAlphabet{\curly}{OT1}{rsfs}{n}{it}

\newcommand{\EEE}{{\curly E}}
\newcommand{\LLL}{{\curly L}}
\newcommand{\XXX}{{\curly X}}

\numberwithin{equation}{section}
        {\begin{list}{}{%
                \settowidth{\labelwidth}{{\emph{#1:}}}%
                \setlength{\leftmargin}{\labelwidth+\labelsep}}}%
        {\end{list}}
\newenvironment{entry}
        {\begin{list}{}%
                {%
                  \setlength{\labelsep}{8mm}%
                  \setlength{\labelwidth}{40mm}%
                  \setlength{\leftmargin}{54mm}%
                  \setlength{\rightmargin}{5pt}%
                }%
        }%
        {\end{list}}
\newlength{\Mylen}
\newcommand{\Lentrylabel}[1]{%
        \settowidth{\Mylen}{\emph{#1}}%
        \ifthenelse{\lengthtest{\Mylen > \labelwidth}}%
                {\parbox[b]{\labelwidth}
                        {\makebox[0pt][l]{\emph{#1}}\\}}%
                {\emph{#1}}
        \hfil\relax}

\newenvironment{Pentry}
        {%
                \begin{entry}}
        {\end{entry}}

\newcommand{\Mentrylabel}[1]%
        {\raisebox{0pt}[1ex][0pt]{\makebox[\labelwidth][l]%
                {\parbox[t]{\labelwidth}{\hspace{0pt}\emph{#1:}}}}}
        {\begin{entry}}%
        {\end{entry}}

  {\begin{list}{}{%
   \settowidth{\labelwidth}{{#1}}%
   \setlength{\leftmargin}{\labelwidth+\labelsep}}}
  {\end{list}}

\addtocounter{secnumdepth}{1}

\begin{document}

\pagenumbering{roman}

\newpage{\pagestyle{empty}\cleardoublepage}

   \thispagestyle{empty}
   \begin {center}
   \Large{Instituto de Ciencias Matem\'aticas}\\
   \Large{CSIC-UAM-UC3M-UCM}\vspace{0.4cm}\\
   \Large{Universidad Aut\'onoma de Madrid}\\
   \Large{Facultad de Ciencias}\\
   \Large{Departamento de Matem\'aticas}\\
   \end{center}

\vskip 5cm

   \begin{center}
      \Huge
      {\bf{
        Coupled equations \\ for K\"ahler metrics and \\
        Yang-Mills connections
                }}
   \end{center}
\vskip 1cm
\begin{center}
        {\large Mario Garcia-Fernandez}
\end{center}
\vskip 3.5cm

\hfill
\begin{center}
{\large
Thesis submitted in partial fulfillment of the requirements\\
for the degree of\\
Doctor in Mathematics at the Universidad Aut\'onoma de Madrid,\\
supervised by Dr. Luis \'Alvarez C\'onsul and Dr. \'Oscar Garc\'{\i}a Prada
\vskip 1cm
Madrid, September 2009}
\end{center}

\newpage{\pagestyle{empty}\cleardoublepage}

\vspace*{8cm}

\begin{flushright}
{\scshape\ A Andrea y a Andrea}
\end{flushright}

\chapter*{\bf Acknowledgements}

First of all, I would like to thank my advisors Luis \'Alvarez C\'onsul and \'Oscar Garc\'{i}a Prada for giving me the opportunity of doing this thesis and for suggesting this problem, a fascinating adventure throughout geometry. I am very grateful for their advice, teaching, patience and encouragement, and for many discussions and constant help during all these years. I would like to thank Luis \'Alvarez C\'onsul for being always willing to talk about maths and life, and for many blackboards and napkins full of mathematics. Many thanks also for his invaluable help with Chapters \ref{chap:Ceq} and \ref{chap:AlgStability}.

I would like to thank Professor Richard Thomas for his generosity, teaching and suggestions during my visits to the Imperial College. I am very grateful for his help with the thesis and for having explained to me the underlying geometric idea of Theorem \ref{thm:ExistenceCYMeq2}. I would like to thank Professeur Olivier Biquard for suggesting to look at the equations on $\CC^2$ and for his invaluable help with Chapter \ref{chap:C2}. I am very grateful for his advice and teaching during my visit to the Institut de Math\'ematiques de Jussieu.

I would also like to thank:

Julien Keller, for several and very useful discussions, suggestions and comments about the text, and also for his help with Proposition~\ref{propo:ANboundedcMH}. Ignasi Mundet, for his help with the text, several relevant discussions concerning Chapter \ref{chap:analytic} and for his teaching through \cite{MR1}, \cite{MR2} and \cite{MR3} and the Winter School `Geometry and Physics' in Castro Urdiales. Julius Ross, for relevant discussions about Chapter \ref{chap:AlgStability}, for his help with the text and his useful comments. Jacopo Stoppa, for several and very useful discussions during lunch time and tea breaks at the Imperial.

Professor S. K. Donaldson, for relevant discussions and suggestions and his teaching through \cite{D3}, \cite{D1}, \cite{D6}, \cite{D4}, \cite{D8}. Professor N. Hitchin, for relevant discussions and suggestions. Professeur Paul Gauduchon, for his generosity with the notes on toric varieties and extremal metrics, from which I learned many things concerning the cscK problem that I adapted for my problem in Chapter \ref{chap:analytic}.

Marina Logares, for her help with the text, relevant discussions and for giving me the opportunity to talk about this topic in the `Geometry Seminar' of the Universidade do Porto. Vicente Muñoz, for his teaching, his comments and help with the thesis. Rafael Hern\'andez and Jes\'us Gonzalo, my supervisors for the works of DEA in the UAM, for several and very useful discussions and for having made many useful comments. Ignacio Sols Lucia, for his generosity with the notes on Riemannian submersions and ruled surfaces. Tom\'{a}s Luis Gómez, Francisco Presas, Marco Castrillón, Dan Fox and Francisco Gancedo for several relevant discussions.

I am very grateful to Marta Aparicio, Roberto Rubio, Alvaro Ant\'on, Emilio Franco, Ana Pe\'on, Raul Oset, Dionisio Y\'a\~nez and Andr\'e Gama for their constant encouragement, productive comments and for their help along these years.

I thank the Consejo Superior de Investigaciones Cient\'ificas for supporting this research and the Royal Society, especially Professor Peter Newstead, for supporting a one week visit to the Imperial College. I thank the Instituto de Ciencias Matem\'{a}ticas (CSIC-UAM-UC3M-UCM) for its hospitality during these years. I thank Imperial College, Institut de Math\'ematiques de Jussieu and Universidade do Porto for the hospitality during my visits. Thanks also to Adhesivos Kefr\'en S.A. and Dr. Juan Carlos Fern\'andez Garc\'ia for the confidence on this research.

Personally, I would like to thank the people who has contributed in my first steps as a mathematician; to my advisors,
to the people in the `Seminario de F\'isica Matem\'atica' de la UCM and the `Grupo de Trabajo de Geometr\'ia Compleja'
of the ICMAT from whom I have learned a lot during these years. To the people in the `Groupe de travail de g\'eométrie k\"ahl\'erienne torique' during my visit to Jussieu. To the people in the `cscK working group' during my first visit to the Imperial. To my teachers in the Universidad Aut\'onoma de Madrid. To my teachers in the Universidad de Valencia, especially to Rafa\'el Sivera and Francisca Mascar\'o for $\pi_1$, Carmen Romero for $H^*$, Jose Mar\'ia Maz\'on and Fuensanta Andreu for $\Delta$, and Olga Gil and Antonio M. Naveira for $\nabla$. To Antonio P\'erez and my brother Ignacio, for showing me up first this wonderful world of mathematics.

\vspace{0.2cm}

A mi familia y amigos, especialmente a mis padres por su apoyo constante, ahora y siempre, y a mi hermano y a Marta por ayudarme tanto todos estos años.

\vspace{0.2cm}

Y muy en especial a Andrea, por estar a mi lado estos \'ultimos $17$ meses y ayudarme a cerrar esta etapa de mi vida, y por su valent\'ia al empezar una nueva juntos . . .

\tableofcontents

\pagenumbering{arabic}

\chapter*{\bf Introduction}

This thesis is concerned with coupled equations for a K\"ahler metric on a compact complex manifold and a connection on a holomorphic principal bundle over it. In the framework of K\"ahler geometry it is natural to consider equations with symplectic interpretation, appearing as the vanishing of a moment map. Many partial differential equations in geometry and mathematical physics appear in this way. Fundamental examples include the Hermite--Yang--Mills equation for connections and the constant scalar curvature equation for K\"ahler metrics. The moment map interpretation of these two equations provides a conceptual framework and a guideline to understand many phenomena in Yang-Mills theory and K\"ahler geometry. Our equations have a moment map interpretation, that intertwines the scalar curvature of a K\"ahler metric with the curvature of a Hermite--Yang--Mills connection. The coupling occurs as a direct consequence of the structure of the group of symmetries. We expect that the symplectic reduction process will provide moduli spaces of solutions endowed with a K\"ahler structure.

A fundamental result in complex geometry is the Hitchin--Kobayashi correspondence~\cite{D3,UY,RS}. It relates the existence of solutions to the Hermite--Yang--Mills (HYM) equation for a connection on a holomorphic principal bundle over a compact K\"ahler manifold with the Mumford stability of the bundle. As pointed out first by Atiyah and Bott for the case of Riemann surfaces \cite{AB} and generalised by S.K. Donaldson to arbitrary compact K\"ahler manifolds, solutions to the HYM equation arise as points of a symplectic reduction for the action of the gauge group $\cG$ on the space of connections. This moment map interpretation, leads one to think of the Hitchin--Kobayashi correspondence as an infinite dimensional version of the Kempf-Ness Theorem relating symplectic reduction with GIT quotients. A similar result for constant scalar curvature K\"ahler (cscK) metrics on a compact complex manifold has been conjectured by S.T. Yau, G. Tian and S.K. Donaldson. As observed by Fujiki \cite{Fj} and Donaldson~\cite{D6}, the cscK equation has a moment map interpretation where the group of symmetries is the group $\cH$ of Hamiltonian symplectomorphisms of a fixed symplectic structure. This gives a fundamental evidence of the \emph{Yau--Tian--Donaldson conjecture}, that relates the existence of solutions of the cscK equation with the stability of the complex structure on the manifold. The notion of K-stability, introduced by G. Tian in \cite{T2} and generalised by S.K. Donaldson in \cite{D4}, plays a key role in this context. The search of such a correspondence for the cscK equation is, today, one of the central problems in complex and algebraic geometry. In \secref{sec:cscK} we give a brief survey of the state of the art of the cscK problem, centering attention on those aspects relevant to the contents of this thesis. General overviews of the Kempf-Ness Theorem and the Hitchin-Kobayashi correspondence for the HYM equation are provided in \secref{sec:KNtheorem} and \secref{sec:HYM}.

The moment map interpretations of the cscK equation and the HYM equation provide a guide that we follow, in order to define our system of equations for K\"ahler metrics and connections. The moment map interpretation for the cscK equation arises when we fix a symplectic form $\omega$ on a smooth compact K\"ahlerian manifold $X$, i.e. a manifold that admits a K\"ahler structure. Let $G$ be a real compact Lie group and $E$ a smooth principal $G$-bundle over $X$. The phase space for our problem is provided by a suitable subspace of the product
\begin{equation}\label{eq:pairsinvacs0}
\cP \subset \cJ \times \cA,
\end{equation}
where $\cJ$ is the space of $\omega$-compatible complex structures on the base $X$ and $\cA$ is the space of connections on the bundle $E$. Let $G^c$ be the complexification of the group $G$. The space $\cP$ parameterizes pairs consisting of a K\"{a}hler structure on the base $X$, with K\"{a}hler form $\omega$, together with a structure of holomorphic principal bundle on the associated principal $G^c$-bundle
$$
E^c = E \times_G G^c
$$
over the corresponding complex manifold. At this point there are mainly two issues that will lead us to a `good' system of equations: the identification of the group of symmetries and the choice of a symplectic form on the space of parameters. For the first issue, let us consider the group $\Aut(E)$ of $G$-equivariant diffeomorphisms of $\Tot(E)$, the total space of the bundle $E$. Note that any element of $\Aut(E)$ covers a unique diffeomorphism on the base. The group of symmetries is taken to be the subgroup $\cX \subset \Aut(E)$ covering $\cH$, the group Hamiltonian symplectomorphisms of $(X,\omega)$. The \emph{extended gauge group} $\cX$ is an extension of $\cH$ by the gauge group $\cG$ of the bundle $E$ (see~\secref{sec:Ceqcoupling-term}), i.e. there is an extension of groups
$$
1 \to \cG \to \cX \to \cH \to 1.
$$
Recall that the groups $\cG$ and $\cH$ are, respectively, the symmetry groups for the HYM problem and the cscK problem. For the choice of symplectic form we rely on the symplectic forms used in these problems. Each factor in the product \eqref{eq:pairsinvacs0} has a natural symplectic structure provided, on $\cJ$ and on $\cA$, by the Fujiki--Donaldson symplectic structure $\omega_\cJ$ \cite{Fj,D1} and by the Atiyah--Bott symplectic structure $\omega_\cA$ \cite{AB}. Our choice of symplectic structure on $\cP$ is the simplest possible: we take a pair of non-zero coupling constants $\alpha = (\alpha_0,\alpha_1) \in \RR^2$ and consider the $2$-form
\begin{equation}\label{eq:Sympfamily0}
\omega_\alpha = \alpha_0 \cdot \omega_\cJ + \frac{4\alpha_1}{(n-1)!} \cdot \omega_\cA,
\end{equation}
where $2n$ is the real dimension of the manifold $X$. The previous normalization of $\alpha_1$ is convenient to simplify the shape of the equations. When $\frac{\alpha_1}{\alpha_0} > 0$, the symplectic form $\omega_\alpha$ is K\"{a}hler with respect to the structure of complex manifold on $\cP$ (see~\eqref{eq:complexstructureI} and Proposition~\ref{prop:CeqintegrableI}). Thus, a moduli construction using the symplectic reduction process on $(\cP,\omega_\alpha)$ will lead to a K\"ahler moduli space. The remaining ingredient is the fact that the extended gauge group $\cX$ acts on the space $\cP$ by holomorphic isometries and there exists a $\cX$-equivariant moment map for any value of the coupling constants (see~Proposition~\ref{prop:momentmap-pairs})
$$
\mu_\alpha\colon \cP \to \LieX^*,
$$
where $\LieX$ is the Lie algebra of $\cX$. Our coupled system of equations arise as the zero locus of this moment map. A priori it may seem strange that a single moment map gives rise to a system of equations. The reason for this fact is that any pair $(J,A)\in\cP$ determines a vector space splitting $\LieX^* \cong \LieG^* \oplus \LieH^*$ provided by the connection $A$, and we use this to split in two terms the condition
$$
\mu_\alpha(J,A) = 0.
$$
The details of the previous construction are provided in Chapter \ref{chap:Ceq}.

Let us now introduce the coupled equations that arise when we split the condition above. The unknown variables are a K\"ahler structure on the base $X$ and a connection on the bundle. Let $(g,\omega,J)$ be a K\"ahler structure on the smooth compact manifold $X$, where $g$, $\omega$ and $J$ are respectively the metric, the symplectic form and the
complex structure. We denote by $\mathfrak{g}$ the Lie algebra of the real compact group $G$. Recall that $E$ is a smooth principal $G$-bundle over $X$. We fix a positive definite inner product $(\cdot,\cdot)$ on $\mathfrak{g}$ invariant under the adjoint representation. Considering the space $\Omega^k(\ad E)$ of smooth $k$-forms on $X$ with values in the adjoint bundle $\ad E$, such an inner product induces a pairing
$$
\Omega^p(\ad E) \times \Omega^q(\ad E) \to \Omega^{p+q},
$$
that we write as $(a^p \wedge a^q)$ for any~$a^j \in
\Omega^j(\ad E)$, $j = p,q$. We will
say that a K\"{a}hler structure $(g,\omega,J)$ on $X$ and a connection $A$ on $E$ satisfy the \emph{coupled
equations} with coupling constants
$\alpha_0, \alpha_1 \in \RR$ if
\begin{equation}\label{eq:CYMeq00}
\left. \begin{array}{l}
\Lambda F_A = z\\
\alpha_0 S_g \; + \; \alpha_1 \Lambda^2 (F_A \wedge F_A) = c
\end{array}\right \}.
\end{equation}
Here $S_g$ is the scalar curvature of the K\"{a}hler metric, $F_A$ is the curvature of the connection, $z$
is an element in the center $\mathfrak{z}$ of the Lie algebra $\mathfrak{g}$ and $c \in \RR$ is a
topological constant. The map $\Lambda\colon \Omega^{p,q}(\ad E) \to \Omega^{p-1,q-1}(\ad E)$ is the contraction
operator acting on $(p,q)$-type valued forms determined by the K\"{a}hler structure. The link with complex geometry is provided by the additional compatibility condition
\begin{equation}\label{eq:integrabilityconnection}
F_A^{0,2} = 0,
\end{equation}
between the complex structure $J$ on the base and the connection $A$. Here $F_A^{0,2}$ denotes the $(0,2)$ part of the curvature, regarded as an $\ad E$ valued smooth form on $X$. Let $G^c$ be the complexification of the group $G$. When \eqref{eq:integrabilityconnection} holds, the pair $(J,A)$ endows the associated principal $G^c$-bundle $E^c = E \times_G G^c$ with a structure of holomorphic principal bundle over the K\"{a}hler manifold $X$. In the language of the moment map interpretation of the coupled equations \eqref{eq:CYMeq00} already explained, the compatibility condition \eqref{eq:integrabilityconnection} is encoded in the definition of the space $\cP$ \eqref{eq:pairsinvacs0}.

Similar coupled equations, for pseudo-Riemannian metrics and connections on smooth bundles, have been studied for a long time by physicists in the context of unification theory. In this context, the K\"ahler structure and the connection can be thought of, respectively, as a gravitational field and a gauge field whose interaction is governed by the system of equations. The coupling of these two variables, provided by the equations, reflects their mutual interaction. It is natural to ask if the coupled equations have a variational interpretation, i.e. if their solutions arise as critical points or absolute minima of a suitable action, since many equations in physics follow this variational principle. A fundamental example is provided by the Einstein--Yang--Mills (EYM) equations \cite{BM}, for which the action is obtained by adding the Yang--Mills functional, considered as a functional in both the metric and the connection, to the Hilbert--Einstein action. The coupled equations \eqref{eq:CYMeq00} also have a variational interpretation as the absolute minima of a suitable functional that we define in~\secref{sec:CeqscalarKal-K}. This functional is purely Riemannian, i.e. it can be considered as a functional for a Riemannian metric $g$ on the smooth compact manifold $X$ and for a connection $A$ on the $G$-bundle $E$. Given such a pair $(g,A)$, the functional is defined by
\begin{equation}\label{eq:CYMfunctional00}
CYM(g,A) = \int_X (\alpha_0 S_g - 2\alpha_1 |F_A|^2)^2 \cdot \vol_g \; + \;  2\alpha_1 \cdot \|F_A\|^2,
\end{equation}
where $\vol_g$ is the volume form of the metric $g$ and the coupling constants $\alpha_0$ and $\alpha_1$ are as in \eqref{eq:CYMeq00}. The term $\|F_A\|^2$ is the Yang--Mills functional, considered as a functional in both the metric and the connection, and $|F_A|^2 \in C^{\infty}(X)$ is the corresponding density. In Proposition~\ref{propo:Variational} we prove that, when we restrict the functional to K\"ahler metrics $g$ compatible with a fixed symplectic structure, the solutions to \eqref{eq:CYMeq00} satisfying the compatibility condition \eqref{eq:integrabilityconnection} arise as absolute minima (after suitable re-scaling of the coupling constants). In order to understand this functional better, in \secref{sec:CeqscalarKal-K} we think of the data $(g,A)$ as defining a $G$-invariant metric $\hat g$ on $\Tot(E)$. This interpretation shows that the first summand in the definition of the functional \eqref{eq:CYMeq00} can be though of, up to scaling, as the $L^2$ norm of the scalar curvature of a suitable $G$-invariant metric $\hat g$ on $E$ associated to $(g,A)$. To construct the metric $\hat g$ we have to assume the positivity of the ratio of the coupling constants
\begin{equation}\label{eq:introratio}
\frac{\alpha_1}{\alpha_0} > 0.
\end{equation}
This ratio measures, roughly speaking, the volume of the fibers of the bundle equipped with $\hat g$ with respect to the projection $E \to X$. An unexpected fact here is that \eqref{eq:introratio} is precisely the condition for the symplectic form $\omega_\alpha$ \eqref{eq:Sympfamily0} on $\cP$ \eqref{eq:pairsinvacs0} to be K\"ahler. Thus, one can expect that a moduli construction for such metrics using the symplectic reduction process leads to a K\"ahler moduli space.

This positivity of the ratio \eqref{eq:introratio} is also a necessary assumption for the main result in \secref{chap:analytic} (Theorem~\ref{thm:ExistenceCYMeq2}). In this chapter, we fix a holomorphic principal $G^c$ bundle $(E^c,I)$ over a compact complex manifold $(X,J)$, where $I$ denotes the holomorphic structure on $E^c$. We consider the space
$$
\cKt = \cK \times \cR,
$$
where $\cK$ and $\cR$ denote, respectively, the space of K\"ahler metrics in a fixed K\"ahler class and the space of smooth $G$-reductions $\cR$ on $E^c$. The \emph{coupled equations} are then equations  for a pair $(\omega,H) \in \cKt$, where the role of the connection in \eqref{eq:CYMeq00} is played by the associated Chern connection. In \secref{sec:ANgeneral-framework} we discuss the general framework provided by the Kempf-Ness theorem where our problem fits and establish the link with the moment map interpretation in~\secref{sec:Ceqcoupled-equations}. The formalism introduced provides the preliminary material necessary for the rest of this chapter in a way that is completely independent of the choice of the symplectic form in \eqref{eq:pairsinvacs0}. Using the previous formalism, in \secref{sec:ANdeformingsolutions} we prove an existence criterion for the existence of solutions to the coupled system \eqref{eq:CYMeq00} for small values of the ratio of the coupling constants $\frac{\alpha_1}{\alpha_0} > 0$. This criterion is formulated in terms of a real subgroup $\cX_I$ of the automorphism group $\Aut(E^c,I)$ of the holomorphic bundle. This lead us to the following general result (see Theorem~\ref{thm:ExistenceCYMeq2}) that we apply in \secref{sec:example1} to provide non trivial examples to the existence of solutions.
\begin{theorem}\label{thm:ExistenceCYMeq00}
Let $(X,L)$ be a compact polarised manifold, $G^c$ be a complex reductive Lie group and $E^c$ be a holomorphic
$G^c$-bundle over $X$. If there
exists a cscK metric $\omega \in c_1(L)$, $X$ has finite automorphism group and $E^c$ is stable with respect to $L$
then, given a pair of positive
real constants $\alpha_0, \alpha_1 > 0$ with small ratio $0 < \frac{\alpha_1}{\alpha_0} \ll 1$, there exists a solution
$(\omega_\alpha,H_\alpha)$ to
\eqref{eq:CYMeq00} with these coupling constants and $\omega_\alpha \in c_1(L)$.
\end{theorem}
In the following sections we consider the coupled equations \eqref{eq:CYMeq00} with fixed positive coupling constants $\alpha_0$, $\alpha_1 > 0$. In \secref{sec:ANIntegralinvariants}, we give an obstruction coming from a character
$$
\cF_\alpha \colon \Lie\Aut (E^c,I) \to \CC
$$
that generalizes the usual Futaki invariant for the cscK problem \cite{Ft0} and also other invariants defined by A. Futaki in \cite{Ft1}. This character can be considered, roughly speaking, as a $\CC$-linear extension of the moment $\mu_\alpha$ whose zero locus correspond to the solutions to the coupled equations. In \secref{sec:ANcoupledintmmap} we construct the integral of the moment map
$$
\cM_\kappa \colon \cKt \to \RR
$$
and determine a sufficient condition for its convexity and existence of lower bounds in terms of the existence of suitable smooth curves on $\cKt$ joining any two points. These curves play a role analogue of the one played by the geodesics in $\cK$ for the cscK problem. In \secref{sec:ANcomplexsurfaces} we restrict to the case in which the base is a compact complex surface. Given a pair $(\omega,H) \in \cKt$ we consider the functional $\cM_H = \cM_\kappa(\cdot,H)\colon \cK \to \RR$ and prove convexity properties of $\cM_H$ along smooth geodesics on $\cK$. We use this to prove a uniqueness result for the coupled equations \eqref{eq:CYMeq00} in the case of toric complex surfaces (see Proposition~\ref{propo:ANtoric}). We obtain lower bounds of $\cM_H$ supposing the existence of solutions of \eqref{eq:CYMeq00} (see Proposition~\ref{propo:ANboundedcMH}), when $c_1(X) \leq 0$. For the proof we adapt the argument of X. X. Chen in \cite{Ch1} concerning the Mabuchi K-energy and $\epsilon$-\emph{approximate geodesics}. We extend the previous result for K\"ahler surfaces with nonnegative bisectional curvature, making use of the weak solutions of the geodesic equation in $\cK$ found by X. X. Chen and G. Tian in \cite{ChT} and adapting their argument concerning the Mabuchi K-energy (see Proposition~\ref{propo:ANboundedcMH}).

In \secref{chap:AlgStability} we define an algebraic notion of stability for triples $(X,L,E)$ (see Definition \ref{def:extKstable}) consisting of a smooth projective variety $X$, an ample line bundle $L$ and a holomorphic vector bundle $E$ over $X$ of rank $r$ and zero degree. Our stability notion depends on a positive real parameter $\alpha$ and relies on Donaldson's definition of $K$-stability for polarized varieties (see~\secref{subsec:cscKalgKstab}). As in \cite{D4}, it is formulated in terms of flat degenerations of the triple over $\CC$, called \emph{test configurations}. The central fiber of each test configuration is endowed with a $\CC^*$-action and the stability of the triple, that we call $\alpha$-K-stability, is measured in terms of the non negativity of a suitable weight $F_\alpha$ (see \eqref{eq:extendedfutaki}) associated to this action. In \secref{sec:ASTheconjecture} we make the following conjecture (see~\ref{conjecture:alphaKpolystab}), where the constant $\alpha = \frac{r\pi^2\alpha_1}{\alpha_0}$ is a normalized ratio of the coupling constants in \eqref{eq:CYMeq00}.
\begin{conjecture}\label{conjecture:alphaKpolystab0}
If there exists a K\"{a}hler metric $\omega \in c_1(L)$ and a Hermitian metric $H$ on $E$ satisfying the coupled equations \eqref{eq:CYMeq00} with coupling constants $\alpha_0,\alpha_1$ then the triple $(X,L,E)$ is $\alpha$-K-polystable.
\end{conjecture}
The rest of this chapter is devoted to give evidence that the above conjecture holds. In \secref{sec:ASDiffeformalphainvariant}, Proposition \ref{prop:extendedalgebraicfutaki}, we prove that the weight $F_\alpha$ satisfies
$$
F_{\alpha} = \frac{-1}{4 \vol(X) \alpha_0}\cF_{\alpha}(\zeta)
$$
for degenerations with smooth cental fiber, where $\cF_\alpha$ is the character defined in \secref{sec:ANIntegralinvariants} and $\zeta$ is the generator of the $S^1 \subset \CC^*$ action on the central fiber. Thus, up to multiplicative constant factor, $F_\alpha$ equals the value of the moment map $\mu_\alpha$ \eqref{eq:thm-muX} at $\zeta$. Recall that the zero locus of $\mu_\alpha$ corresponds to the solutions of the coupled equations. Once more this fits in the general picture of the Kempf-Ness Theorem explained in \secref{sec:KNtheorem} if we view a test configuration as a $1$-parameter subgroup (see \secref{sec:ASTconfigoneparameter}) and the associated $\alpha$-invariant as the corresponding Mumford weight (see \secref{subsec:KNproof} and cf. \eqref{eq:weightmmap1}). This fact motivates the definition of $\alpha$-K-stability (see~Definition~\ref{def:extKstable}) and leads to think that this stability notion should play an important role in the existence problem for the coupled equations \eqref{eq:CYMeq2}. Moreover, this fact gives the following more direct evidence.

\emph{$1)$ If there exists a solution to the coupled equations \eqref{eq:CYMeq00} with positive coupling constants $\alpha_0$, $\alpha_1 \in \RR$ then $F_\alpha = 0$ for any product test configuration (see Corollary~\ref{cor:ASVanishingalphainvariant}).}

In \secref{sec:ASTconfigoneparameter} we interpret test configurations for a triple as $1$-parameter subgroups of suitable product group $G_k$ acting on a relative Quot scheme parameterizing triples $(X,L,E)$. This gives the following evidence.

\emph{$2)$ In the limit case $\alpha = 0$ the notion of $\alpha$-K-stability is equivalent to the notion of algebraic K-stability of the underlying polarized manifold (see Proposition~\ref{prop:stab-alpha=0}).}

In \secref{sec:ASBptestconfig} we characterize the notion of Mumford--Takemoto stability of the bundle in terms of \emph{base preserving test configurations} (see Proposition~\ref{prop:Bptestconfig}). This leads to the following evidence.

\emph{$3)$ If there exists a solution to the coupled equations \eqref{eq:CYMeq00} with positive coupling constants $\alpha_0$, $\alpha_1 \in \RR$, then $F_\alpha \geq 0$ for any base preserving test configuration, with equality only if the test configuration is a product (see Corollary~\ref{cor:ASnonnegativealphainvariant}).}

We hasten to add that Conjecture \ref{conjecture:alphaKpolystab0} is difficult to prove even in complex dimension $1$. Note that in this case the coupled equation are a system in separated variables independent of the coupling constants. In fact, they reduce to the HYM equation for the bundle due to the uniformisation theorem, that asserts that any Riemann surface has a unique cscK metric if we prescribe the total area. When $\dim_\CC X = 1$, Conjecture~\ref{conjecture:alphaKpolystab0} reduces to prove that Mumford--Takemoto stability of the bundle implies $\alpha$-K-stability of the triple for any $\alpha > 0$, due to the Hitchin--Kobayashi correspondence. The algebro-geometric approach to the proof of this fact seems to be difficult in the light of previous work by R. Thomas and J. Ross in \cite{RRT2}, where the algebraic K-stability of Riemann surfaces is analysed with purely algebro-geometric methods. A more likely approach, valid for arbitrary dimension of $X$, seems to be provided by Donaldson's methods in \cite{D8} or Mabuchi's methods in \cite{Mab2}. The main issue to adapt the arguments in \cite{D8} or \cite{Mab2} is to prove that $\alpha$-K-stability arises as a suitable asymptotic version of a Chow-type stability notion in the relative quot scheme considered in \secref{sec:ASTconfigoneparameter}. The verification of this stability issue would provide a strong evidence of the fact that \eqref{eq:CYMeq00} are gauge-theoretic equations related to the algebro-geometric moduli problem for triples $(X,L,E)$. The moduli problem for such triples has been considered by R. Pandharipande in \cite{Pd} when the base $X$ is a Riemann surface. In~\cite{Pd}, Geometric Invariant Theory is used in order to compactify the moduli space of pairs consisting of a ($10$-canonically polarised) smooth curve of genus greater than one and a semistable vector bundle over the curve. For arbitrary dimension of the base very little has been done. In \cite{ST}, G. Schumacher and M. Toma construct a moduli space of (non-uniruled) polarized K\"{a}hler manifolds equipped with isomorphism classes of stable vector bundles via the method of versal deformations. In this work the authors were able to endow the moduli space with a K\"ahler metric, under suitable cohomological restrictions (see~\secref{sec:example1}), when the base is K\"ahler--Einstein and the bundle is projectively flat. As we see in \secref{sec:example1}, the gauge-theoretical equations corresponding to the moduli construction in~\cite{Pd} and the K\"ahler moduli in~\cite{ST} provide particular solutions of the coupled equations \eqref{eq:CYMeq00}.

Finally, in \secref{sec:ASTheconjecture} we analyze the converse of the statement in Conjecture \ref{conjecture:alphaKpolystab0} in the cases $\dim_\CC X = 1$, $E$ arbitrary and $E \cong \oplus^r \cO_X$, $X$ arbitrary. We respectively check that it holds and that it reduces to the Yau--Tian--Donaldson Conjecture for the cscK problem. The latter gives evidence of the fact that, as a recent example in \cite{ACGT} suggests for the cscK problem, the notion of $\alpha$-K-stability introduced in Definition~\ref{def:extKstable} may not be sufficient to guarantee the existence of solutions of \eqref{eq:CYMeq00} and so, to formulate a converse of \ref{conjecture:alphaKpolystab0}, a stronger notion may be required. We indicate how to strengthen this condition following \cite{RRT1}.

Chapter \secref{chap:examples} contains some examples of solutions to \eqref{eq:CYMeq00}. We also discuss how the
existence of solutions in the limit case $\alpha_0 = 0$ for $G^c = SL(r,\CC)$ can be applied in certain cases, using
results of Y. J. Hong in \cite{Ho2}, to obtain a cscK metric in the projectivization of the associated vector bundle. In \secref{sec:example1} we deal with the case of vector
bundles over Riemann surfaces and projectively flat bundles over K\"{a}hler manifolds satisfying a topological
constraint. In both situations, the coupled system \eqref{eq:CYMeq00} reduces to the limit case $\alpha_1 = 0$ (cscK equation and HYM equation). In
\secref{sec:example2} we consider
homogeneous Hermitian holomorphic bundles $E$ over homogeneous K\"{a}hler manifolds and show the existence of
invariant
solutions to
\eqref{eq:CYMeq00} in any K\"{a}hler class provided that the representation inducing the homogeneous bundle $E$ is irreducible. The invariant
solutions arise as
(simultaneous) solutions to the limit systems $\alpha_0 = 0$ and $\alpha_1 = 0$ in \eqref{eq:CYMeq00}. When the
complex
dimension of the base is $2$
the invariant solutions satisfy in addition the K\"{a}hler analogue of the Einstein--Yang--Mills equation (see
equation
\eqref{eq:CEYMeq}) considered
in  \secref{sec:CeqscalarKal-K}. In \secref{sec:example3} we discuss some (well known) examples of stable bundles over
K\"{a}hler-Einstein manifolds
without infinitesimal automorphisms where Theorem \ref{thm:ExistenceCYMeq00} applies. Using a bubbling argument for Anti-Self-Dual connections, this setting provides our first genuine examples in which the metric on the base is not cscK. In \secref{sec:example4} we
discuss the relation of the limit
$\alpha_0 = 0$ in \eqref{eq:CYMeq00} with the cscK equation on ruled manifolds as mentioned above.

In \secref{chap:C2} we consider the problem of finding solutions to \eqref{eq:CYMeq00} on the non-compact (standard) complex manifold $\CC^2$ endowed with the trivial smooth principal bundle $E = \CC^2 \times SU(2)$. We prove that there is no obstruction to finding solutions for arbitrary coupling constants with fixed behaviour at infinity (see Theorem~\ref{proposition:solutionsarbitrarycoupling}). The fixed behavior at infinity is the analogue in this non-compact case of fixing the K\"{a}hler class and the topological type of the bundle. To prove Theorem~\ref{proposition:solutionsarbitrarycoupling}, we use a deformation argument and a suitable re-scaling of the solutions. As initial data we take the standard Euclidean metric and a classical solution of 't Hooft to the instanton equation. This procedure gives solutions for which the K\"ahler metric is not cscK, making explicit the `bubbling' argument in \secref{sec:example3}.

\newpage
\section*{List of symbols}
\label{sec:notation}

{\bf Lie groups, Lie algebras, manifolds and bundles}

\begin{Pentry}

\item[$G^c$] Complex reductive Lie group
\item[$\mathfrak{g}^c$] Lie algebra of $G^c$
\item[$G \subset G^c$] Maximal compact subgroup of $G^c$
\item[$\mathfrak{g}$] Lie algebra of $G$
\item[$\mathfrak{z} \subset \mathfrak{g}$] centre of $\mathfrak{g}$
\item[$X$] Compact differentiable manifold (projective scheme in Chapter \secref{chap:AlgStability})
\item[$Z$] Projective variety
\item[$E$] Smooth principal $G$-bundle over $X$ (coherent sheaf over $X$ in Chapter \secref{chap:AlgStability})
\item[$\ad E$] Adjoint bundle of $E$
\item[$E^c$] Smooth principal $G^c$-bundle over $X$
\item[$L$] Line bundle over $X$ (ample holomorphic in Chapter \secref{chap:AlgStability})
\item[$\Tot(E)$] Total space of the bundle $E$
\item[$\Tot(E^c)$] Total space of the bundle $E^c$
\item[$\Tot(L)$] Total space of the bundle $L$
\item[$(X,L,E)$] Triple consisting of a projective scheme $X$, an ample invertible sheaf $L$ and a coherent sheaf $E$ over $X$

\end{Pentry}

{\bf Geometric structures}

\begin{Pentry}

\item[$\omega$] Symplectic structure on $X$
\item[$J$] Almost complex structure on $X$
\item[$g$] Riemannian metric; also elements of the groups $\Aut \;E$ or $\Aut \;E^c$ (below)
\item[$\omega_\phi$] $\omega + dd^c \phi$, for a K\"ahler potential $\phi$ on $(X,J,\omega)$
\item[$S_g$, $S_J$] Scalar curvature of $g$, Hermitian scalar curvature of $J$
\item[$A$] Connection on $E$
\item[$\theta_A$] Vertical projection determined by $A$
\item[$\theta^{\perp}_A$] Horizontal lift determined by $A$
\item[$F_A$] Curvature of $A$
\item[$I$] $G^c$-invariant almost complex structure on $E^c$
\item[$H$] $G$-reduction on $E^c$
\item[$A_H$] Chern connection of $H$, when $E^c$ is holomorphic
\item[$F_H$] Curvature of $A_H$
\end{Pentry}

{\bf Constants and parameters}

\begin{Pentry}

\item[$2n$] Real dimension of $X$
\item[$r$] Rank of of a vector bundle $E$
\item[$\hat c \in \RR$] Normalized integral of the $2n$-form $\Lambda^2(F_A \wedge F_A) \wedge \omega^{n-2}$
\item[$\hat S \in \RR$] Normalized integral of the scalar curvature
\item[$\alpha_0, \alpha_1\in \RR$] Coupling constants
\item[$\alpha \in \RR$] Normalized ratio of the coupling constants
\item[$c \in \RR$] weighted sum of $\hat c$ and $\hat S$ with weights $\alpha_0$ and $\alpha_1$, respectively
\item[$z \in \mathfrak{z}$] Element in the center of $\mathfrak{g}$

\end{Pentry}

{\bf Spaces of geometric structures}

\begin{Pentry}

\item[$C^{\infty}(X)$] Smooth functions on $X$
\item[$C_0^{\infty}(X)$] Smooth functions with zero integral on $(X,\omega)$
\item[$\Omega^k(F)$] Space of $F$--valued smooth $k$--forms, $F$ being a vector bundle
\item[$\cJ$] Space of compatible almost complex structure on $(X,\omega)$
\item[$\Omega^{p,q}_J$] Space of smooth $(p,q)$-type forms determined by $J \in \cJ$
\item[$\cJi \subset \cJ$] Space of integrable elements of $\cJ$
\item[$\cA$] Space of connections on $E$
\item[$\cA_J^{1,1} \subset \cA$] Connections of $(1,1)$-type with respect to $J \in \cJ$
\item[$\cP$] Space of pairs $p = (J,A)$, with $J \in \cJi$ and $A \in \cA_J^{1,1}$
\item[$\cK$] Space of K\"{a}hler metrics on $X$, on a fixed K\"{a}hler class.
\item[$\cR$] Space of $G$-reductions of $E^c$
\item[$\cKt$] $= \cK \times \cR$

\end{Pentry}

{\bf More (infinite dimensional) Lie groups}

\begin{Pentry}

\item[$\cH$] Group of Hamiltonian symplectomorphisms of $(X,\omega)$
\item[$\Diff X$] Group of diffeomorphisms of $X$
\item[$\Diff_0 X$] Identity component of $\Diff X$
\item[$\Aut \;(X,J)$] Group of biholomorphisms of a complex manifold
\item[$\Aut \;(X,J,\omega)$] Group of holomorphic isometries of a K\"{a}hler manifold
\item[$\cG$] Gauge group of $E$
\item[$\Aut \;E$] Group of $G$-equivariant diffeomorphisms of $E$
\item[$\cX \subset \Aut \;E$] Extended gauge group of $E$ and $(X,\omega)$
\item[$\cG^c$] Gauge group of $E^c$
\item[$\Aut \;E^c$] Group of $G$-equivariant diffeomorphisms of $E^c$

\end{Pentry}

{\bf Vector fields and functions}

\begin{Pentry}

\item[$\phi \in C^{\infty}(X)$] Smooth function on $X$. Also K\"ahler potential on $(X,J,\omega)$
\item[$y$] Vector field on $X$ or $E$
\item[$\eta_\phi$] Hamiltonian vector field of $\phi \in C^{\infty}(X)$. Also element of $\LieH$
\item[$\eta$] Element of $\LieH$
\item[$\zeta$] Element of $\LieX$
\item[$\xi$] Element of $\LieG$
\item[$Y_\zeta$] Infinitesimal action of $\zeta$ on $\cP$

\end{Pentry}

\chapter{\bf Preliminaries}

\label{chap:preliminaries}

\section{The Kempf-Ness Theorem}
\label{sec:KNtheorem}

We start recalling some well known facts about symplectic and algebraic quotients that apply formally to our
problem. We refer the reader to \cite{McS} for the material concerning symplectic quotients and to \cite{MFK} for that
concerning Geometric Invariant Theory (GIT). For the proof of the Kempf-Ness Theorem we have mainly used the theory
developed in \cite{MR2} but also the classical references \cite{Kir} and \cite{MFK}. For further results on this topic
for the case of non-projective K\"{a}hler manifolds we refer to \cite{MR3}.

Let $V$ be a complex Hermitian vector space and let $Z \subseteq \mathbb{P}(V)$ be a compact smooth projective variety. We consider $Z$ endowed with the K\"{a}hler metric $\omega$ given by restricting the Fubini-Studi metric of
$\mathbb{P}(V)$. Suppose that a reductive complex algebraic group $G^c$ acts on $Z$ via an algebraic representation
$\rho: G^c \to SL(V)$, thus the $G^c$-action is naturally linearized with respect to the line bundle $L =
\mathcal{O}_Z(1)$. Let $G \subset G^c$ be a maximal compact subgroup that maps onto $SU(V)$. Then $G$ acts on $Z$ by
holomorphic isometries and the linearization provides a $G$-equivariant moment map $\mu:Z \to \mathfrak{g}^*$ (see
e.g.~\cite[\S 6.5.1]{DK}). In this situation there are two quotients, provided by symplectic geometry and algebraic
geometry, and a well known correspondence between them that will be explained along this section.

On the other hand, the symplectic quotient of $(Z,\omega)$ by $G$ is constructed via the moment map $\mu$. Recall that a map $\mu:Z \to
\mathfrak{g}^*$ is called a $G$-equivariant moment map if it satisfies
$$
\langle d \mu, \zeta \rangle = i_{Y_\zeta} \omega, \qquad \mu(g z) = \Ad(g^{-1})^*\mu(z),
$$
for all $z \in Z$, $g \in G$ and $\zeta \in \mathfrak{g}$. Here the vector field $Y_\zeta$ denotes the infinitesimal
action of $\zeta$ on $Z$. If $0 \in \mathfrak{g}^*$ is a regular value of $\mu$ then $\mu^{-1}(0)$ is a smooth
submanifold of $Z$ preserved by $G$ and the quotient $\mu^{-1}(0)/G$ is a compact manifold (orbifold if the $G$-action
on $\mu^{-1}(0)$ is not free) that inherits a natural symplectic structure, called the symplectic quotient of
$(Z,\omega)$ by $G$.

On the other hand, the (algebraic) GIT quotient of $Z$ by $G^c$ with respect to the linearization $L$ is constructed as follows. Let
$$
R = \bigoplus_{k \geq 0} H^0(X,L^k)
$$
be the graded ring given by tensorization, where $H^0(X,L^k)$ denotes the space of global holomorphic sections of
$L^k$. The GIT quotient of $Z$ by $G^c$ is the projective variety
$$
Z//G^c \defeq \textrm{Proj} \; R^{G^c},
$$
where $R^{G^c}$ is the (graded) ring of $G^c$-invariant elements of $R$. In more down-to-earth terms, there is an
isomorphism of graded rings $R \cong \CC[\hat{Z}]$, where $\hat{Z} \subset V$ is the affine cone of $Z$ and
$\CC[\hat{Z}]$ denotes its ring of regular functions. Since $G^c$ is a reductive complex algebraic group we can take a
finite set of generators $\{f_j\}_{j=0}^m$ of the ring $R^{G^c}$. Let $\CC[z_0, \ldots, z_m]$ be the ring of
polynomials in $m+1$ free variables and consider the (graded) ring morphism defined by
\begin{equation}\label{eq:morphismrings}
\CC[z_0, \ldots, z_m] \to R^{G^c} \colon z_j \mapsto f_j.
\end{equation}
The kernel of \eqref{eq:morphismrings} is an homogeneous ideal of $\CC[z_0, \ldots, z_m]$ cutting a projective variety
on $\mathbb{P}^m$ that turns out to be isomorphic to $Z//G^c$.

The link between the symplectic and GIT quotients is provided by the Kempf-Ness Theorem that states that there exists
an homeomorphism between them (see~\cite{KN, Kir} or~\cite[Theorem~5.4]{MR2}).
\begin{theorem}[Kempf-Ness]\label{theorem:Kempf-Ness}
\begin{equation}\label{eq:KNcorrespondence}
\mu^{-1}(0)/G \cong Z//G^c.
\end{equation}
\end{theorem}
While the points in the symplectic quotient have a clear geometric meaning, they represent $G$-orbits on
$\mu^{-1}(0)$,
it is not clear a priori which $G^c$-orbits on $Z$ are represented in the GIT quotient. In order to make the
homeomorphism explicit we shall clarify this point, for what we need to introduce the notion of stability.

\subsection{GIT Stability}
\label{subsec:KNGITstab}

Note that the inclusion
$R^{G^c} \hookrightarrow R$ induces a $G^c$-invariant rational map
\begin{equation}\label{eq:GITmorphism}
Z \dashrightarrow Z//G^c.
\end{equation}
In fact it is not always possible to make this map every where well-defined; if all the $G^c$-invariant functions of $\hat{Z}$ vanish at some point $z \in Z$ then clearly \eqref{eq:GITmorphism} is not well defined at $z$. This leads to the fundamental definition of GIT.

\begin{definition}\label{def:GITstability}
A point $z \in Z$ is said to be \emph{semistable} if the there exists a nonconstant function $f \in R^{G^c}$ such that
$f(z) \neq 0$. We say $z$ is \emph{polystable} if it is semistable and the orbit $G^c \cdot \hat{z} \subset V$ of any
lift $\hat{z} \in \hat{Z}$ of $z$ is closed in $V$. We say $z$ is \emph{stable} if it is polystable and $z$
has finite stabilizer on $G^c$.
\end{definition}

The locus of semistable, polystable and stable points will be denoted respectively by $Z^{ss}$, $Z^{ps}$ and $Z^{s}$.
It is clear that the notion of stability do not depend on the representant of any $G^c$ orbit, so it makes sense to
talk about stability of an orbit of the action. The semistable orbits are precisely those that appear in the GIT
quotient, i.e. $Z^{ss}$ is the domain of definition of the rational map \eqref{eq:GITmorphism}. Any semistable orbit
is
glued in the quotient to the unique polystable orbit on its closure and, only when restricted to $Z^s$, the map
\eqref{eq:GITmorphism} parameterizes $G^c$-orbits. This justifies the equality
$$
Z//G^c = Z^{ss}//G^c.
$$
A point will be referred to as \emph{unstable} if it is not semistable. Unstable points are characterized by the
following property: $z \in Z$ is unstable if and only if the orbit of any lift $\hat{z} \in \hat{Z}$
contains $0 \in V$ on its closure.

The main result that makes workable Definition \ref{def:GITstability} is a numerical criterion, due to David Hilbert
and David Mumford, that provides a systematic procedure to check which are the semistable points. It is stated in
terms
of the restriction of the action to \emph{$1$-parameter subgroups}. Let $\lambda: \CC^* \to G^c$ be a $1$-parameter
subgroup, i.e. a morphism of algebraic groups from $\CC^*$ to $G^c$. Since
$Z$ is projective for any $z \in Z$ there exists a limit point
$$
z_0 = \lim_{t \to 0} \lambda(t) \cdot z,
$$
that is fixed by the $\CC^*$-action. The \emph{maximal weight} $w(z,\lambda)$ associated to $\lambda$ at $z$ is the
weight of the induced $\CC^*$-action on the fiber $L_{z_0}$. Note that the morphism $\lambda$ induces a representation
of $\CC^*$ on $V$ that, in suitable coordinates $\hat{z} = (z_0, \ldots, z_n)$,
reads
$$
\lambda(t) \cdot \hat{z} = (t^{q_0}z_0, \ldots, t^{q_n}z_n),
$$
and so the maximal weight equals
$$
w(z,\lambda) = \textrm{max}\{- q_i : x_i \neq 0\}.
$$
If $w(\lambda,z) < 0$ then, taking $t \to 0$, the origin is contained in the closure of the orbit of
$\hat{z}$ in $V$ so the point $z$ is unstable. The Hilbert-Mumford theorem says that the converse is true, i.e. that
if
$z$ is unstable then $w(\lambda,z)$ is positive for some $1$-parameter subgroup $\lambda$. Moreover,

\begin{theorem}[Hilbert-Mumford Criterion]
A point $z \in Z$ is semistable if and only if $w(z,\lambda) \geq 0$ for all $1$-parameter subgroups. A point $z \in
Z$
is polystable if and only if $w(z,\lambda) \geq 0$ for all $1$-parameter subgroups with equality only if $\lambda$
fixes $z$. A point $z \in Z$ is stable if and only if $w(z,\lambda) > 0$ for all non-trivial $1$-parameter subgroups.
\end{theorem}

\subsection{The character of a complex orbit}
\label{subsec:KNcharacter}

As we have seen, the GIT quotient does not represent all the $G^c$-orbits in the K\"{a}hler manifold $(Z,J,\omega)$ so, in
order to prove Theorem
\ref{theorem:Kempf-Ness}, one has to identify which $G^c$-orbits cut the zero locus of the moment map. Before going
into the proof of the theorem we introduce a well known invariant of the $G^c$-orbit of a point $z \in Z$ that gives a necessary
condition for this fact. Our approach differs slightly from the ones in \cite{FO} and \cite{W}, since we make an
explicit use of an exact $1$-form on $G^c$ defined in~\cite{MR2}. Let $\sigma_z \in \Omega^1(G^c)$ be given by
\begin{equation}\label{eq:sigmaz}
\sigma_z(v_g) = \langle \mu(gz), - \imag \pi_1 (R_{g^{-1}})_* v_g\rangle
\end{equation}
for any $v_g \in T_g G^c$, where $R_{g^{-1}}$ denotes right translation by $g^{-1}$ and $\pi_1 \colon \mathfrak{g}^c \to \imag
\mathfrak{g}$ is the standard projection. This form is exact, invariant by the left $G$ action on $G^c$ and satisfies
\begin{equation}\label{eq:cocyclesigma}
R_g^*\sigma_z = \sigma_{gz}
\end{equation}
for any $g \in G^c$ (see~\cite[Lemma~3.1]{MR2}). Since $G^c$ retracts to $G$ and the restriction of $\sigma_z$ to $G$
is zero, this $1$-form is exact if and only if it is closed. We sketch a proof of the latter for completeness.
\begin{proposition}[Mundet i Riera \cite{MR1}]
The $1$-form $\sigma_z$ is closed.
\end{proposition}
\begin{proof}
By property \eqref{eq:cocyclesigma} it is enough to check it at the tangent space of the identity element on $G^c$.
Given $\zeta_j \in \mathfrak{g}^c$, $j = 1,2$, the simbol $\zeta_j^r$ means the right invariant vector field on $G^c$
generated by $\zeta_j$. We have to check that
\begin{equation}\label{eq:sigmazclosed1}
d \sigma_z (\zeta_1^r,\zeta_2^r) = \zeta_1^r(\sigma_z (\zeta_2^r)) - \zeta_2^r(\sigma_z (\zeta_1^r)) - \sigma_z
([\zeta_1^r,\zeta_2^r]) = 0.
\end{equation}
Since $\sigma_z$ vanishes on $G$ we can skip the case $\zeta_1, \zeta_2 \in \mathfrak{g}$. By $G$-equivariance of the
moment map it follows that \eqref{eq:sigmazclosed1} holds if $\zeta_1 \in \mathfrak{g}$ and $\zeta_2 \in \imag
\mathfrak{g}$. Finally, for the remaining case we take $\imag \zeta_j \in \imag \mathfrak{g}$ and so
\begin{equation}
\label{eq:sigmazclosed2}
\begin{split}
d \sigma_z (\imag \zeta_1^r,\imag \zeta_2^r) 
 & = \omega_z(Y_{\zeta_2},JY_{\zeta_1}) - \omega_z(Y_{\zeta_1},JY_{\zeta_2}) = 0.\qedhere
\end{split}
\end{equation}
\end{proof}
Let $(G^c)_z$ be the isotropy group of $z$ in $G^c$, with Lie algebra $(\mathfrak{g}^c)_z$ and consider the
$\CC$-linear extension of the moment map $\mu: Z \to (\mathfrak{g}^c)^*$. We define a complex valued map
$f_{z,g}\colon
(\mathfrak{g}^c)_z \to \CC$ by
\begin{equation}\label{eq:futakifinite1}
\begin{split}
f_{z,g}(\zeta) & = \langle \mu(gz), \Ad(g) \zeta \rangle \\
& = \imag \sigma_z(\zeta^l)_{|g} + \sigma_z((\imag \zeta)^l)_{|g},
\end{split}
\end{equation}
where $\pi_0\colon \mathfrak{g}^c \to \mathfrak{g}$ is the standard projection and $\zeta^l$ is the left invariant vector field on $G^c$ determined by any $\zeta \in \mathfrak{g}^c$. Then we
obtain the desired obstruction.
\begin{proposition}[\cite{W}]\label{propo:futakifinite}
The map $f_{z}\colon (\mathfrak{g}^c)_z \to \CC$ given by \eqref{eq:futakifinite1} is independent of $g \in G^c$. It defines a
character of $(\mathfrak{g}^c)_z$ that vanishes if $G^c \cdot z$ contains a zero of the moment map.
\end{proposition}
\begin{proof}
For the first part of the statement it is enough to check that $\sigma_z(\zeta^l)$ is constant on $G^c$ for any $\zeta
\in (\mathfrak{g}^c)_z$ and using \eqref{eq:cocyclesigma} we just have to check it at $1$. Thus, given any $\eta \in
\mathfrak{g}^c$,
$$
\frac{d}{dt}_{|t=0}\sigma_z(\zeta^l)_{|e^{t\eta}} = d(\sigma_z(\zeta^l)) \cdot \eta = (-i_{\zeta^l} d\sigma_z +
L_{\zeta^l} \sigma_z) \cdot \eta  = L_{\zeta^l} \sigma_z \cdot \eta = 0,
$$
where the last equality follows from \eqref{eq:cocyclesigma}.
Combining $d\sigma_z(\zeta^l) = 0$ with \eqref{eq:sigmazclosed1}, we obtain
$$
f_{z,g}([\zeta_1,\zeta_2]) = \imag \sigma_z([\zeta_1^l,\zeta_2^l]) + \sigma_z([(\imag\zeta_1)^l,\zeta_2^l]) = 0,
$$
for any pair $\zeta_1$, $\zeta_2 \in (\mathfrak{g}^c)_z$ as claimed.
\end{proof}

\begin{remark}
The complex group $G_z^c$, the complexification of $G_z$, is a complex subgroup of $(G^c)_z$ but they may not coincide
for every $z \in Z$. However, if $z \in Z$ contains a zero of the moment map in the orbit $G^c \cdot z$ then $G_z^c =
(G^c)_z$ (see~\cite[Proposition~2.4]{FO}).
\end{remark}

\subsection{The proof of Kempf-Ness Theorem}
\label{subsec:KNproof}

The key idea to prove the existence of the homeomorphism \eqref{eq:KNcorrespondence}, due to G. Kempf
and L. Ness in \cite{KN}, is to construct a convex functional on each $G^c$ orbit on $V$ whose critical points coincide with
the lifts of zeros of the moment map $\mu$. Using the Hilbert-Mumford criterion it follows that the $G^c$ orbits of
polystable points are precisely those that cut the zero locus of $\mu$. Here we adapt the proof of I. Mundet i Riera
in
\cite{MR2} to the polystable case, since it provides a general argument for non projective K\"{a}hler manifolds that
will be useful later in this thesis.

First we shall relate the tools used to construct the symplectic and GIT quotients, i.e. the moment map and the
linearization. As pointed out before, the representation $\rho: G^c \to SL(V)$ gives a lift of the $G^c$-action to the
line bundle $L$ that determines a moment map
$$
\langle \mu(z), \zeta \rangle = \frac{\imag}{2\pi} \frac{h(\hat{z},\rho_*(\zeta)\hat{z})}{|\hat{z}|^2}
$$
for the $G$ action (~\cite[\S 6.5.1]{DK}), where $\zeta \in \mathfrak{g}$, $\hat{z} \in V$ is any vector lying over $z
\in Z$ and the squared norm $|\hat{z}|^2$ is taken with respect to the Hermitian metric $h$ on $V$. The moment map and
the linearization are matched together via an infinitesimal identity on $\Tot(L)$. Indeed, if $\zeta \in
\mathfrak{g}$ and we denote respectively by $Y_\zeta$ and $\hat{Y}_\zeta$ the infinitesimal actions on $Z$ and $L$, we
have
\begin{equation}\label{eq:infinitesimalidentity}
\hat{Y}_\zeta = \theta^\perp_A(Y_\zeta) + 2\pi\imag \langle \mu , \zeta \rangle \Id.
\end{equation}
Here $A$ is the Chern connection of the Hermitian metric in $L$ (induced by $h$) and $\theta^\perp_A$ denotes the
horizontal lift of vector fields on $Z$ with respect to $A$. By $\Id$ we denote the canonical vertical vector field on
$L$ induced by the identity endomorphism. The identity \eqref{eq:infinitesimalidentity} shows that if $\zeta \in
\mathfrak{g}$ generates a $1$-parameter subgroup $\lambda_\zeta$ the corresponding maximal weight satisfies
\begin{equation}\label{eq:weightmmap1}
w(z_0,\lambda_\zeta) = 2\pi\langle\mu(z_0),\zeta\rangle.
\end{equation}
The right hand side of this expression can be rewritten using that (cf.~\cite[Definition~2.3]{MR2})
\begin{equation}\label{eq:weightmmap2}
\langle\mu(z_0),\zeta\rangle = \langle \mu(\lim_{t \to 0} \lambda(t) z), \zeta \rangle = \langle \mu(\lim_{t \to
\infty} e^{\imag t \zeta} z), \zeta \rangle = \lim_{t \to \infty} \langle \mu(e^{\imag t \zeta}z), \zeta \rangle.
\end{equation}

We now recall from \cite{MR2} the construction of a convex functional, very often referred as the \emph{integral of
the moment map}. Given $z \in Z$ the integral of the moment map $\mathcal{M}_z \colon G\backslash G^c \to \RR$ is
defined as the unique functional that satisfies the properties
\begin{equation}\label{eq:intmappropertypre}
d \mathcal{M}_z = \sigma_z \qquad \textrm{and} \qquad \mathcal{M}_z([1]) = 0,
\end{equation}
where $\sigma_z$ denotes now the $1$-form on $G\backslash G^c$ induced by the $G$-invariant form \eqref{eq:sigmaz}.
Property \eqref{eq:cocyclesigma} of $\sigma_z$ is translated into the following cocycle condition
\begin{equation}\label{eq:cocycleMz}
\mathcal{M}_z(g h) = \mathcal{M}_{h z}(g) + \mathcal{M}_z(h).
\end{equation}
We will now use $\cM_z$ to give a prove of the fundamental lemma from which Kempf-Ness Theorem follows. This lemma
says
that the polystability of the point $z$ is equivalent to the fact that $\cM_z$ attains a minimum, the latter being
equivalent to $G^c z \cap \mu^{-1}(0) \neq \emptyset$ due to \eqref{eq:intmappropertypre}. To give a clue of why this
is true recall that $G\backslash G^c$ is a symmetric space of noncompact type that can be endowed with a preferred
negatively curved Riemannian metric. The geodesics corresponding to this metric are given by maps $t \in \RR \to [e^{\imag t
\zeta}g] \in G\backslash G^c$ for any $\zeta \in \mathfrak{g}$ and $g \in G^c$. Along a geodesic, the integral of the
moment has a nice explicit expression:
\begin{equation}\label{eq:Mzalonggeod}
\mathcal{M}_z([e^{\imag t \zeta}g]) = \mathcal{M}_z([g]) + \int_0^t \langle
\mu(e^{\imag s \zeta} g z), \zeta\rangle ds.
\end{equation}
We conclude from \eqref{eq:weightmmap1} and \eqref{eq:weightmmap2} that the maximal weights of the $1$-parameter
subgroups in $G^c$ control the asymptotic behavior of $\mathcal{M}_z$ along geodesics (a priori only of those arising
from $1$-parameter subgroups, but in fact of all of them by~\cite[Lemma~2.4.8]{MR1}). We go now into the proof of the
lemma.

\begin{lemma}[Kempf-Ness]\label{lemma:KempfNess}
A point $z \in Z$ is polystable if and only if $G^c \cdot z$ contains a zero of the moment map. There is at most one
$G$-orbit inside each $G^c$-orbit on which the moment map vanishes.
\end{lemma}
\begin{proof}
Suppose first that $G^z \cdot z$ contains a zero of the moment map. We will show that $z$ is polystable using a
refinement of the Hilbert-Mumford criterion due to F. Kirwan (see \cite{Kir} pg. $107$), that shows that it is
sufficient to consider $1$-parameter subgroups $\lambda: \CC^* \to G^c$ with $\lambda(S^1) \subset G$. Let $\lambda$
be
such a $1$-parameter subgroup and denote by $\zeta \in \mathfrak{g}$ the generator of the $S^1$ action. Then, from
\eqref{eq:weightmmap1}, \eqref{eq:weightmmap2} and $\mu(z) = 0$, we have $w(z,\lambda) \geq 0$, since
$$
\frac{d}{dt} \langle \mu(e^{\imag t \zeta}z), \zeta \rangle = \omega(Y_\zeta,JY_\zeta) = \|Y_\zeta\|^2.
$$
Moreover, $w(z,\lambda) = 0$ only if $Y_{\zeta|e^{\imag t \zeta}} = 0$ for all $t > 0$. In particular $Y_{\zeta|z} =
0$
which implies that $\lambda$ fixes $z$. Suppose now that $z \in Z$ is polystable. We claim that $\mathcal{M}_z$
attains
a minimum on $G\backslash G^c$. Using the polar decomposition of $G^c$ any element $g \in G^c$ can be written as $g =
k
e^{\imag \zeta}$, with $k \in G$ and $\zeta \in \mathfrak{g}$. Combined with \eqref{eq:Mzalonggeod}, we thus obtain
the
equality
\begin{equation}
\label{eq:Integralcomputation}
\begin{split}
\mathcal{M}_z([g]) & = \int_0^1 \langle \mu(e^{\imag t \zeta}z, \zeta\rangle dt = \int_0^1 \frac{\imag}{2\pi}
\frac{h(e^{\imag t \zeta} \hat{z},\rho_*(\zeta)e^{\imag t \zeta} \hat{z})}{|e^{\imag t \zeta} \hat{z}|^2} dt\\
 & =  \frac{1}{\pi} \log (|e^{\imag t \zeta} \hat{z}|^2)|_0^1 = \frac{2}{\pi}( \log |g \hat{z}| - \log |\hat{z}|).
\end{split}
\end{equation}
Let $\{[g_j]\}_{j=0}^{\infty}$ be a minimizing sequence for $\mathcal{M}_z$ on $G\backslash G^c$. We take
representatives
$e^{\imag \zeta_j}$ of $[g_j]$, with $\zeta_j \in \mathfrak{g}$. By \eqref{eq:Integralcomputation}, the sequence $e^{\imag
\zeta_j} \hat{z}$ is bounded on $V$, so we can suppose that it converges to $\hat{z}_0$ (by taking a convergent
subsequence). Since $z$ is polystable $\hat{z}_0$ is in the $G^c$-orbit of $\hat{z}$, i.e. there exists $g \in G^c$
such that $[g]$ attains the minimum as claimed. By property \eqref{eq:intmappropertypre} and $G$-equivariance of the
moment map we conclude that $G \cdot g z \subset \mu^{-1}(0)$. Finally, let $z, z' \in G^c$ such that $\mu(z) =
\mu(z')
= 0$. Then $z' = g z$ for some $g \in G^c$ and $[1], [g] \in G\backslash G^c$ are critical points of $\cM_z$. Let $g =
k e^{\imag \zeta}$ be the polar decomposition of $g$. The second derivative of $\cM_z$ along $t \to [e^{\imag t
\zeta}]$ must vanish since
$$
\frac{d^2}{dt^2} \cM_z ([e^{\imag t \zeta}]) = \langle \mu(e^{\imag t \zeta}z), \zeta \rangle =
\omega(Y_\zeta,JY_\zeta) = \|Y_\zeta\|^2.
$$
This implies that $e^{\imag \zeta} \in G^c_z$, so $z' = g z = k e^{\imag \zeta} z = k z \in G \cdot z$ as claimed.
\end{proof}
The proof of the Kempf-Ness theorem is now straightforward. By Lemma \ref{lemma:KempfNess} we have an inclusion $\mu^{-1}(0) \subset Z^{ps} \subset Z^{ss}$ that induces a continuous map
$$
i: \mu^{-1}(0)/G \to Z^{ss}//G^c.
$$
Since any semistable orbit in the GIT quotient is glued to the unique polystable orbit on its closure, this map is
surjective (again by Lemma \ref{lemma:KempfNess}). If $i([z]) = i([z'])$ then $z' = gz$ for some $g \in G^c$, since
both
points are polystable and by Lemma \ref{lemma:KempfNess} their $G$-orbits must coincide. Therefore, $i$ is a
continuous
bijection from a compact space to a Hausdorff space and hence it is an homeomorphism.

\subsection{The non projective case}
\label{subsec:KNnonprojective}

The identification between the symplectic and algebraic quotients given by Kempf-Ness Theorem can be extended to an
arbitrary K\"{a}hler manifold (neither necessarily compact, nor projective) in a suitable sense. We follow \cite{MR3}
for this topic (see also references therein). It is in this setting where the correspondence
concerning the stable $G^c$ orbits and the zeros locus of the moment map (as in Lemma \eqref{lemma:KempfNess}) provides a
general framework that formally applies in many situations where the objects involved are infinite dimensional.

Let $(Z,\omega,J)$ be a K\"{a}hler manifold endowed with a Hamiltonian left action of a compact group $G$. Let us fix a moment map $\mu$ and suppose that $G$ acts by holomorphic isometries. By a theorem of Guillemin and Sternberg (see
\cite{GS}), the action of $G$ extends to a unique holomorphic action of $G^c$ on $Z$. The action of $G^c$, however, no
longer preserves the symplectic form of $Z$ nor does it preserve the zero level set $\mu^{-1}(0)$. The corresponce
\eqref{eq:KNcorrespondence} in the projective case motivates the search for a result analogue to Lemma
\eqref{lemma:KempfNess} in this general setup. A partial answer to this problem was given in \cite{MR2}, where a suitable notion of stability, called analytic stability, was defined using the right hand side of formula \eqref{eq:weightmmap2} as the
definition of the maximal weight. A point $z \in Z$ is then called \emph{analytically stable} if the maximal weight
\begin{equation}\label{eq:maximalweight}
w(z,\zeta) = \lim_{t \to \infty} \langle \mu(e^{\imag t \zeta}z),\zeta\rangle
\end{equation}
is strictly positive for any non-zero $\zeta \in \mathfrak{g}$. The notion of analytically stable is a priori stronger than GIT
stable in the projective case but they are in fact equivalent (a proof can be found in ~\cite[Lemma~2.4.8]{MR1}). In
general it is not straightforward that analytic stability is an invariant of the $G^c$-orbit of a point. This is
proved
in~\cite[Theorem~5.4]{MR2} as a consequence of the characterization of analytic stability in terms of the so-called
linear properness of the integral of the moment map $\cM_z$, which is constructed exactly as in the projective case. Recall that $\cM_z$ is said to be linearly proper if
\begin{equation}\label{eq:linearproper}
|\zeta| \leq C_1 \cM_z([e^{\imag \zeta}]) + C_2,
\end{equation}
for all $\zeta \in \mathfrak{g}$, suitable norm $|\cdot|$ in $\mathfrak{g}$, and positive real constants $0 < C_1, C_2
\in \RR$. The general result concerning this notion of stability can be stated as follows (see~\cite[Theorem~5.4]{MR2}).
\begin{theorem}[\cite{MR2}]\label{theorem:Kempf-Ness2}
A point $z \in Z$ is analytically stable if and only if the stabiliser $G^c_z$ of $z$ in $G^c$ is finite and $ G^c
\cdot z \cap \mu^{-1}(0) \neq \emptyset$. Furthermore, there is at most one $G$-orbit inside each $G^c$-orbit on which
the moment map vanishes.
\end{theorem}
The proof follows as the one given for Lemma \ref{lemma:KempfNess} in the previous section. The only point that
slightly differs is the ``only if'' part, but it can be easily adapted using the linear properness of $\cM_z$ (compare
\eqref{eq:Integralcomputation} with \eqref{eq:linearproper}). Generalizing previous work of A. Teleman, M. Riera \cite{MR3} has extended the notion of
analytic stability to the polystable case . He imposes some mild restrictions on the growth of $\mu$ and the $G$-action. The
maximal weights are viewed as defining a collection of maps, for each $z \in Z$, $\lambda_z: \partial(G\backslash G^c)
\to \RR$, where $\partial(G\backslash G^c)$ is the boundary at infinity of the symmetric space $G\backslash G^c$. A suitable result, similar to Lemma \ref{lemma:KempfNess}, is proved for this notion of stability (see~\cite[Theorem~1.2]{MR3}). It is very likely that this new notion of analytic polystability will be relevant for our problem and also for the cscK problem, where no notion of analytic polystability is known (analytic (semi)stability was introduced by X. X. Chen in \cite{Ch2} with the name of \emph{geodesic (semi)stability}).

\section{The Hitchin-Kobayashi correspondence for the HYM equation}
\label{sec:HYM}

In this section we recall the moment map interpretation of the HYM equation and explain how the general framework
provided by \secref{sec:KNtheorem} applies, leading to a Hitchin-Kobayashi correspondence. First we give a short
historical review of the problem concerning the HYM equation.

\subsection{Overview}
\label{subsec:HYMoverview}

In 1965 Narasimhan and Shesadri proved that holomorphic bundles over a compact Riemann surface $X$ are stable if and only if
they arise from irreducible unitary representations of $\pi_1(X)$. This connection between holomorphic and unitary structures,
that in the classical theory of line bundles is essentially equivalent to the identification between holomorphic and
harmonic 1-forms, was studied by Atiyah and Bott in \cite{AB} from a gauge-theoretic point of view. They
identified the stable bundles over $X$ with the irreducible, projectively flat unitary connections and pointed out
that the curvature of a connection can be viewed as a moment map for the action of the gauge group. This interpretation led them to consider the Yang-Mills functional as the squared norm of this moment map and the correspondence provided by the Narashiman-Seshadri Theorem as an infinite dimensional version of Kempf-Ness theorem relating symplectic reduction with GIT quotients. The proper generalization of this fact for holomorphic vector bundles over a K\"{a}hler manifold of arbitrary dimension was suggested independently by Hitchin and Kobayashi. They conjectured that the stability of a
holomorphic bundle should be related to the existence of a unitary Hermitian-Yang-Mills (HYM) connection (equivalently
Hermite-Einstein metric). The Hitchin-Kobayashi correspondence in the case of the HYM equation for
$\GL(r,\CC)$-bundles
is a  theorem of Donaldson (for complex projective varieties) and Uhlenbeck and Yau (for compact K\"ahler
manifolds)~\cite{D3, UY}, generalized by Ramanathan and Subramanian (\cite{RS}) in the algebraic case for principal
bundles with arbitrary complex reductive Lie group.

\subsection{Moment map interpretation}
\label{subsec:HYMmmap}

In this section we review the notion of Hermitian-Yang-Mills connection, explaining its moment map interpretation.
Further details on this topic are given in sections \S \ref{sec:Ceqcoupling-term} and
\S\ref{sec:Ceqcoupled-equations}.

Let $X$ be a compact symplectic manifold of dimension $2n$, with symplectic form $\omega$ and volume form $\vol_\omega
= \frac{\omega^n}{n!}$. We fix a real compact Lie group $G$ with Lie algebra $\mathfrak{g}$ and a smooth principal
$G$-bundle $E$ over $X$. We also fix a positive definite inner product
$$
(\cdot,\cdot): \mathfrak{g} \times \mathfrak{g} \to \RR
$$
on $\mathfrak{g}$ invariant under the adjoint action. Recall from the introduction that this metric gives a
pairing $\Omega^p(\ad E) \times \Omega^q(\ad E) \to \Omega^{p+q}\colon (a,b) \to (a\wedge b)$. A connection on $E$ is a $G$-equivariant splitting
$A\colon TE\to VE$ of the short exact sequence
$$
  0 \to VE\lto TE\lto \pi^*TX \to 0,
$$
where $VE=\ker d\pi$ is the vertical bundle. It follows by definition that the space $\cA$ of principal connections on
$E$ is an affine space modelled on $\Omega^1(\ad E)$. Let denote by $\cG$ the \emph{gauge group} of the bundle $E$,
i.e. the group of $G$-equivariant diffeomorphisms of $E$ that project to the identity on $X$. The action of
$g\in \cG$ on $TE$ defines a left $\cG$-action on $\cA$ given by $g \cdot A \defeq g\circ A \circ g^{-1}$ and the
2-form on $\cA$ given by
\begin{equation}\label{eq:SymfC}
\omega_{\cA}(a,b) = \int_X (a \wedge b) \wedge \omega^{n-1},
\end{equation}
for $a,b \in T_A \cA = \Omega^1(\ad E)$, with $A\in\cA$, is a symplectic structure preserved by $\cG$. It was first
observed by Atiyah and Bott~\cite{AB} when $X$ is a Riemann surface and by Donaldson~\cite{D3} in higher dimensions
that the $\cG$-action on $\cA$ is hamiltonian, with $\cG$-equivariant moment map $\mu_\cG\colon \cA\to
(\Lie\cG)^*$ given by
\begin{equation}
\label{eq:momentmap-cG}
  \langle\mu_\cG(A),\zeta\rangle = \int_X (\zeta\wedge F_A)\wedge \omega^{n-1},
\end{equation}
where $\zeta \in \LieG$ and $F_A\in\Omega^2(\ad E)$ is the curvature of $A\in\cA$. Note here that the $\cG$-invariant bilinear pairing in $\LieG = \Omega^0(\ad E)$ given by
\begin{equation}
\label{eq:LieG-dual}
\langle\zeta_1,\zeta_2\rangle_\cG \defeq \int_X (\zeta_1\wedge\zeta_2) \vol_\omega,
\end{equation}
where $\zeta_1,\zeta_2 \in \LieG$, allows us to identify (up to multiplicative constant factors) the image of the moment
map \eqref{eq:momentmap-cG} with $\Lambda_\omega F_A \in \LieG$.

We now suppose that $X$ is a K\"ahler manifold with K\"ahler form $\omega$ and consider the associated principal
$G^c$-bundle $E^c = E \times_G G^c$, where the group $G^c$ denotes the complexification of $G$. There is a
$\cG$-invariant distinguished subspace $\cA^{1,1} \subseteq \cA$ of connections $A$ with $F_A\in\Omega^{1,1}(\ad E)$.
This space is in correspondence with the space of holomorphic structures on the principal $G^c$-bundle $E^c$ over the
complex manifold $X$. Away from the singularities, it inherits a complex structure compatible with $\omega_\cA$ and
hence a K\"{a}hler structure.
Thus, the smooth locus of $\cA^{1,1}$ is a K\"{a}hler manifold endowed with a Hamiltonian $\cG$-action and we can
consider the K\"{a}hler reduction
\begin{equation}
\label{eq:SympredHYM}
  \mu_\cG^{-1}(z)/\cG,
\end{equation}
where $\mu_\cG$ is now the restriction of the moment map to $\cA^{1,1}$ and $z$ is an element of the centre
$\mathfrak{z}$ of $\mathfrak{g}$. Note
here that any element $z \in \mathfrak{z}$ determines a $\cG$-invariant element in $\LieG$ and we use the
identification $\mu_\cG(A) \in \LieG$, for
$A \in \cA$, provided by the $\cG$-invariant pairing $\langle\cdot,\cdot\rangle_\cG$ to make sense of
\eqref{eq:SympredHYM}. The elements in
$\mu^{-1}(z)$ for $z \in \mathfrak{z}$ are the \emph{Hermitian-Yang-Mills} connections that we recall in the
following definition.

\begin{definition}
A connection $A \in \cA$ is called \emph{Hermitian--Yang--Mills} if $F_A^{0,2}=0$ and it satisfies the
\emph{Hermitian--Yang--Mills equation}
\begin{equation}
\label{eq:HYM}
\Lambda F_A=z, \qquad z \in \mathfrak{z}.
\end{equation}
\end{definition}
The element $z \in \mathfrak{z}$ in the right-hand side of~\eqref{eq:HYM} is determined by
the cohomology class $\Omega\defeq [\omega]\in H^2(X)$ and the topology of the
principal bundle $E$. This follows because applying $(z_j,\cdot)$
to~\eqref{eq:HYM}, for an orthonormal basis $\{z_j\}$ of
$\mathfrak{z}\subset\mathfrak{g}$, and then
integrating over $X$, we obtain
\begin{equation}
\label{eq:z(Omega,E)}
z=\sum_j \frac{\langle z_j(E)\cup \Omega^{[n-1]},[X]\rangle}{\Vol(\Omega)} z_j.
\end{equation}
Here, $\Omega^{[k]}\defeq \Omega^k/k!$, $\Vol(\Omega)\defeq \langle
\Omega^{[n]},[X]\rangle$ and $z_j(E)\defeq [z_j \wedge F_A]\in H^2(X)$ is the
Chern--Weyl class associated to the $G$-invariant linear form $(z_j, \cdot)$
on $\mathfrak{g}$, which only depends on the topology of the bundle $E$
(see~\cite[Ch XII, \S 1]{KNII}). When the associated K\"{a}hler reduction \eqref{eq:SympredHYM} of $\cA^{1,1}$ is non-empty it turns out that this reduction has finite
dimension and inherits, in the smooth locus, a structure of K\"{a}hler manifold
(see~\cite[Corollary~3.32]{Ko}).

\subsection{The correspondence}
\label{subsec:HYMHKproof}

Consider the space of isomorphism classes of holomorphic structures on the smooth principal $G^c$-bundle $E^c$
over the complex manifold $X$, with fixed K\"{a}hler form $\omega$. Set theoretically this space can be realized as the
quotient of $\cA^{1,1}$ by the left action of the gauge group $\cG^c$ of $E^c$, often referred as the \emph{complex
gauge group}. The action of $g \in \cG^c$ on $\cA^{1,1}$ can be thought of as the push-forward of the (integrable) almost
complex structure on $\Tot(E^c)$ defined by any element $A \in \cA^{1,1}$. The notation for the complex
gauge group in not accidental, since it can be considered a complexification of $\cG$. We are thus in position to formulate an infinite dimensional version of Lemma \ref{lemma:KempfNess} characterizing $\cG^c$-orbits in $\cA^{1,1}$ that contain a HYM connection. This correspondence is often referred as \emph{the Hitchin-Kobayashi correspondence} for the HYM equation~\cite{D3, UY, RS}.

Following \cite{MR2} we define stability in terms of maximal weights. A connection $A \in \cA^{1,1}$ is \emph{stable}
if
$$
w(A,\zeta) = \lim_{t \to \infty} \langle\mu_\cG(e^{\imag t\zeta} A),\zeta \rangle > 0
$$
for any non-zero $\zeta$ in a suitable Sobolev completion of $\LieG \cong \Omega^0(\ad E)$ (see~\cite[Section~4.1.4]{MR2}). The
picture provided by \secref{sec:KNtheorem} applies formally here giving an integral of the moment map $\mathcal{M}_A:
\cG\backslash \cG^c \to \RR$, the Donaldson functional in the case of vector bundles \cite{D3}, whose critical points can be identified with zeros
of
$\mu_\cG$. The existence of the functional is due to the fact that $\cG\backslash \cG^c$ is contractible. This
functional is convex along geodesics $t \mapsto [e^{\imag t \zeta}g]$ in the infinite dimensional symmetric space
$\cG\backslash \cG^c$ and its behaviour at infinity is controlled by the corresponding maximal weights. Recall that a
holomorphic $G^c$-bundle $E^c$ is simple if the space of holomorphic global sections of $\ad E^c$ coincides with the centre of $\mathfrak{g}^c$. Identifying $\cG\backslash \cG^c$ with the space of smooth $G$-reductions of the holomorphic bundle
determined by $A$ we can state the Hitchin-Kobayashi correspondence as follows (see~\cite[Remark~4]{RS})
\begin{theorem}[{\cite{D4, RS, UY}}]\label{theorem:HKcorrespondence}
A holomorphic $G^c$-bundle is stable if and only if it is simple and there exists a unique $G$-reduction whose
associated Chern connection is Hermite-Yang-Mills.
\end{theorem}
The proof of the Hitchin-Kobayashi correspondence follows formally as in the finite dimensional case: uniqueness
follows from the existence of geodesics in $\cG\backslash \cG^c$ joining any two points and the convexity of $\cM_A$
along them. The ``if'' part is the easy one and follows as in the proof of Lemma \ref{lemma:KempfNess}. For the ``only if''
part, a theorem of Uhlenbeck and Yau \cite{UY} on weak subbundles of vector bundles allows to prove that $A$ is stable
if and only if $\cM_A$ is linearly proper (see~\cite[Lemma~6.1]{MR2}). Then, by taking a minimizing sequence and
proving regularity the latter is equivalent to attaining a minimum in $\cG\backslash \cG^c$, thus achieving the proof
of Theorem \ref{theorem:HKcorrespondence}.
\begin{remark}
Note that the statement of Theorem \ref{theorem:HKcorrespondence} for the holomorphic structure determined by $A \in
\cA^{1,1}$ is analogue to that of Theorem \ref{theorem:Kempf-Ness2}. To see this note that a $G$-reduction (a Hermitian metric when $G = U(r)$) determines $g \in \cG^c$. Then, the associated Chern connection is HYM if and only if so is $g \cdot A$.
\end{remark}

\subsection{Mumford-Takemoto slope stability}
\label{subsec:HYMslope}

In this section we restrict to the algebraic setting and explain why the analytic stability condition defined in
\secref{subsec:HYMHKproof} coincides with the slope stability for coherent sheaves due to D. Mumford and F. Takemoto.
First we recall the definition of slope stability for coherent sheaves over a polarised scheme.

Let $X$ be an $n$-dimensional projective scheme with an ample invertible sheaf $L$.  Then the \emph{slope} $\mu_L(E)=d_1(E)/d_0(E)$ of a non-zero coherent sheaf $E$ with $n$-dimensional support is defined (up to an additive factor independent of $E$) as the ratio of the first two leading coefficients of its Hilbert polynomial with respect to $L$ (cf. e.g.~\cite[\S 1.2]{HL}), given by
\[
P_L(E,k)\defeq \sum_{p=0}^n (-1)^p \dim H^p(X,E\otimes L^k) = k^n d_0(E) + k^{n-1} d_1(E) + \cdots + d_n(E).
\]
Following Simpson~\cite{Sp}, a coherent sheaf on a scheme has \emph{pure dimension $k$} if it has $k$-dimensional support and no subsheaves have lower dimensional support. Let $E$ be a coherent sheaf on $X$ which is pure of dimension $n$. Then $E$ is called \emph{Mumford semistable} (with
respect to $L$) if
\[
\mu_L(F)\leq \mu_L(E)
\]
for each proper subsheaf $F\subset E$.  Such an $E$ is \emph{Mumford stable} if the inequality is always strict and
\emph{Mumford polystable} if it is a direct sum of stable sheaves, all of them with the same slope.

Suppose now that $(X,L)$ is a smooth compact polarised variety and let $E$ be a holomorphic vector bundle (locally
free
sheaf) over $X$. As usual holomorphic $\GL(r,\CC)$-principal bundles are identified with holomorphic vector budles
over
$X$ via the standard representation on $\CC^r$. Fix some Hermitian metric on $E$ and suppose that a \emph{weak}
endomorphism $\zeta$ destabilizes the Chern connection $A$, i.e. has negative maximal weight, in the sense of the
previous section. Then it has almost everywhere constant eigenvalues and defines a filtration of $E$ by holomorphic
subbundles in the
complement of a complex codimension $2$ subvariety of $X$. Using a theorem of Uhlenbeck and Yau \cite{UY} on weak
subbundles of vector bundles this leads to a filtration of $E$ by reflexive (coherent) subsheaves (see~\cite{UY}) and
the maximal weight $w(A,\zeta)$ has then an explicit expression in terms of the slopes of the elements of the
filtration (see~\cite[Lemma~4.2]{MR2}). From this it follows that both definitions are equivalent in the stable and
semistable case and so the Theorem \ref{theorem:HKcorrespondence} can be stated in this context as follows
(see~\cite[Corollary~4]{D3}): A holomorphic vector bundle $E$ is polystable if and only if there exists a Hermitian
metric on $E$ whose associated Chern connection is Hermite-Yang-Mills.

\section{The cscK problem}
\label{sec:cscK}

In this section we briefly review the state of the art of the cscK problem, centering the attention in
those aspects relevant for the contents of this thesis. We discuss the moment map interpretation of the scalar
curvature of a K\"{a}hler metric~\cite{Fj,D1} and how the problem fits into the general framework of Section
\secref{sec:KNtheorem} \cite{D4}.

\subsection{Overview}
\label{subsec:cscKoverview}

In the $1980$'s E. Calabi \cite{Ca} settled the problem of finding K\"ahler metrics, within a fixed K\"{a}hler class on a compact complex manifold, which are critical points of the $L^2$ norm of the scalar curvature. These metrics $g$ are called \emph{extremals}, and those satisfying the cscK equation
\begin{equation}\label{eq:cscK}
S_g = const.
\end{equation}
are absolute minimizers of the action. The notion of cscK metric can be considered the natural generalization of the K\"{a}hler Einstein (KE) condition when the fixed K\"{a}hler class is not a multiple of the first Chern class $c_1(X)$. The first result concerning the cscK equation \eqref{eq:cscK} is the classical uniformisation theorem, which states that every compact Riemann surface admits a metric of constant curvature, unique modulo holomorphic automorphisms if we prescribe the total area. This result establishes a correspondence between cscK metrics and complex structures for compact surfaces. In higher dimensions this correspondence is not true, and some stability of the complex manifold is necessary in order to carry a cscK metric. The solution of the Hitchin-Kobayashi correspondence for the HYM equation gives a clue of this fact, since the existence of a KE metric on $X$ implies polystability of the holomorphic tangent bundle.

A correspondence between stable complex structures and cscK metrics was first conjectured by S. T. Yau for the KE
case. The problem of whether a compact complex manifold admits a KE metric was solved in the case $c_1(X)\leq 0$ in the $1970$'s with positive answer. The celebrated solution of Yau to the Calabi conjecture established the existence
of a Ricci flat metric, now named Calabi-Yau metric, in each K\"{a}hler class on a compact K\"{a}hler manifold $X$
with $c_1(X) = 0$. If $c_1(X) < 0$, the existence of KE metrics was proved by Aubin and Yau, independently. When $c_1(X) > 0$ it was known for a long time that there are obstructions to the existence of solutions coming from the holomorphic vector fields on the manifold (see~\cite{Ft1,Ca}). A remarkable obstruction for the existence of cscK metrics in a given K\"{a}hler class was defined by A. Futaki in \cite{Ft0} for the case of KE metrics. The \emph{Futaki invariant} is a character of the Lie algebra of the automorphism group of the complex manifold associated to any K\"{a}hler class. In \cite{T1} G. Tian proved that any compact complex surface with $c_1(X) > 0$ admits a KE metric if and only if it has vanishing Futaki invariant. In \cite{T6} the same author proved that the existence of KE metrics in arbitrary dimensions is equivalent to a suitable properness of the \emph{Mabuchi K-energy} \cite{Mab1}, the integral of the moment map.

In the analytic setting concerning cscK and extremal metrics, i.e. for arbitrary K\"{a}hler class, several
achievements have taken place during the last years. In \cite{ChT} G. Tian and X. X. Chen have proved uniqueness of solutions for the extremal case in full generality and that the Mabuchi K-energy is bounded below if there exists a cscK metric in the K\"{a}hler class. Their work is based on the study of the geodesic equation in the space of K\"{a}hler potentials introduced independently by A. Mabuchi \cite{Mab1} and S. K. Donaldson \cite{D6}. In \cite{Ch2} Chen defined a suitable notion of analytic stability, called \emph{geodesic stability}, and proved that it is an obstruction for the existence of cscK metrics.

For the case of polarised K\"{a}hler manifolds $(X,L)$, i.e. when the K\"{a}hler class is $c_1(L)$ for an ample line
bundle $L$, a new obstruction concerning GIT called $K$-stability was defined by Tian in \cite{T2} for the KE case.
Tian conjectured that a K\"{a}hler manifold with positive first Chern class admits a KE metric if and only if it is
K-stable and proved the ``only if'' part. In \cite{D4} S. K. Donaldson generalized the notion of K-stability to any
polarised manifold $(X,L)$ introducing the notion of (algebraic) test configuration. Roughly speaking, algebraic test
configurations are degenerations of the polarized manifold to general schemes with a rational number attached (the
\emph{generalized Futaki invariant}) which will be positive if the manifold is stable. We know that
$K$-stability is an obstruction for the existence of cscK metrics in polarized manifolds mainly due to the work of
Donaldson \cite{D8}, T. Mabuchi \cite{Mab2} and J. Stoppa \cite{Stp}. Donaldson has proved the converse for the case
of toric surfaces \cite{D8}. Recent work of Arezzo, Della and La Nave in \cite{ADL}, based on previous work of Ross and Thomas in \cite{RT}, has simplified the notion of $K$-stability in certain cases showing that only suitable nice degenerations need to be considered. However, the work of Apostolov, Calderbank, Gauduchon, T$\o$nnesen and Friedman in \cite{ACGT} and of G. Sz\'{e}kelyhidi in \cite{Sz1} concerning extremal metrics suggests that the algebraic notion of K-stability should be strengthened for the general case, allowing more general analytic (instead of just algebraic) test configurations.

\subsection{The scalar curvature as a moment map}
\label{subsec:cscKmmap}

We now briefly explain the moment map interpretation of the scalar curvature. This fact was first observed by
Fujiki~\cite{Fj} for the scalar curvature of K\"{a}hler manifolds and independently by Donaldson~\cite{D1} for the
Hermitian scalar curvature of almost K\"{a}hler manifolds. We hasten to add that both quantities coincide in the K\"{a}hler case, up to positive  multiplicative constant factor. In this section we follow closely Donaldson's
approach~\cite{D1}.

Throughout this section $X$ is a fixed compact symplectic manifold of dimension $2n$, with
symplectic form $\omega$. First we recall the notion of Hermitian scalar curvature of an almost K\"{a}hler manifold. An almost complex structure
$J\colon TX\to TX$ is
compatible with $\omega$ if $\omega(J\cdot,J\cdot) = \omega$ and $g_J = \omega(\cdot,J\cdot)$ is a Riemannian metric
in
$X$. Any compatible almost
complex structure $J$ defines a hermitian metric on $T^*X$ and there is a unique unitary connection on $T^*X$ whose $(0,1)$ component is the operator
$$
\dbar_J\colon \Omega^{1,0}_J\to \Omega^{1,1}_J
$$
induced by $J$. We denote by $\Omega^{p,q}_J$ the space of $(p,q)$-type forms with respect to the decomposition provided by $J$. The real $2$-form $\rho_J$ is defined as $-\imag$
times the curvature form of the
induced connection on the canonical line bundle $K_X = \Lambda^n_{\CC}T^{\ast}X$. The \emph{Hermitian scalar
curvature}
$S_J$ is the real function on
$X$ defined by
\begin{equation}
\label{eq:def-S}
  S_J \omega^{n} = n \, \rho_J \wedge \omega^{n-1}.
\end{equation}
Note that the average $\hat{S}$ of the Hermitian scalar curvature does not depend on $J$ but only on the cohomology
class of $\omega$,
\begin{equation}
\label{eq:hat-S}
  \hat{S}\defeq \frac{\int_X S_J \vol_\omega}{\int_X\vol_\omega}  = 2\pi n \cdot \frac{\langle c_1(X)\cup
  [\omega]^{n-1},[X]\rangle}{\Vol([\omega])}.
\end{equation}
The space $\cJ$ of almost complex structures on $X$ which are compatible with $\omega$ has a natural structure of
infinite dimensional K\"{a}hler
manifold. The K\"{a}hler form $\omega_{\cJ}$ is
\begin{equation}
\label{eq:SympJ} \omega_{\cJ} (\dot{J}_1,\dot{J_2}) \defeq \int_{X}\tr(J\dot{J}_1 \dot{J_2}) \vol_\omega, \text{ for } \dot{J}_1, \dot{J}_2\in T_J\cJ,
\end{equation}
whereas the complex structure $\mathbf{J}\colon T_J\cJ \to T_J\cJ$ is $\mathbf{J}\dot{J} \defeq J\dot{J}$. Here we are identifying the tangent space $T_J\cJ$ at $J \in \cJ$ with the set of endomorphisms $\dot{J}\colon TX \to TX$
such that $\dot{J} J = - J \dot{J}$ and $\dot{J}$ is symmetric with respect to the induced metric $g_J = \omega(\cdot,J\cdot)$.

The group $\cH \subset \Diff X$ of Hamiltonian symplectomorphisms consists of those diffeomorphisms $h: X \to X$ such
that there exists a Hamiltonian isotopy
$$
[0,1] \times X \to X \colon (t,x) \mapsto h_t(x)
$$
from $h_0=\Id$ to $h_1=h$ (see e.g.~\cite[\S 3.2]{McS}). Note that the Lie algebra $\LieH$ of $\cH$, consisting of the
hamiltonian vector fields on $X$, can be identified with the Lie algebra $C^{\infty}(X)/\mathbb{R}$ of smooth real
functions on $X$ modulo the constants, endowed with the Poisson bracket. There is a natural $\cH$-action on $\cJ$ by
push-forward, that is $h \cdot J \defeq h_* J = dh \circ J \circ dh^{-1}$, preserving the K\"{a}hler structure
and,
as was proved by Donaldson in \cite[Proposition~9]{D1}, it is hamiltonian, with equivariant moment map $\mu_\cH\colon
\cJ\to \LieH^*$ given by
\begin{equation}
\label{eq:scmom}
  \mu_\cH(J) = -(S_J - \hat{S}).
\end{equation}
Here we use the $L^2$ pairing provided by $\vol_\omega$ to identify the set $C_0^\infty(X)$ of functions with zero integral over $X$, with the dual of $\LieH=C^{\infty}(X)/\RR$.

The $\cH$-invariant subspace $\cJi \subseteq \cJ$ of integrable almost complex structures is (away from the
singularities) a complex submanifold so it inherits a K\"{a}hler structure. Thus, $\cJi$ is a symplectic manifold
endowed with a Hamiltonian $\cH$-action and we can construct the K\"{a}hler reduction
\begin{equation}\label{eq:modulicscK}
\mu_{\cH}^{-1} (0)/\cH
\end{equation}
where $\mu_\cH$ is now the restriction of the moment map to $\cJi$. Over $\cJi$, the Hermitian scalar curvature $S_J$
is, up to a multiplicative
constant factor, the usual Riemannian scalar curvature of the K\"{a}hler metric determined by $J$ and $\omega$. We
conclude that the K\"{a}hler
reduction \eqref{eq:modulicscK} is the moduli space of K\"{a}hler metrics with fixed K\"{a}hler form $\omega$ and
constant scalar curvature.

\subsection{Complexified orbits}
\label{subsec:cscKorbits}

We discuss now how the cscK problem fits into the general framework provided by Section \secref{sec:KNtheorem}. We
follow the argument given by Donaldson in~\cite[Section~4]{D6}.

Let $(X,J)$ be a compact complex manifold where we fix a K\"{a}hler class $[\omega]$. Let $\cK$ be the space of
K\"{a}hler metrics in $(X,J)$ with this K\"{a}hler class. We fix a K\"{a}hler form $\omega$ and consider the
corresponding group of Hamiltonian symplectomorphisms and the infinite dimensional K\"{a}hler manifold $\cJi$ of
compatible complex structures. For simplicity we will suppose that the first cohomology group $H^{1}(X)$ vanishes, so
the group of Hamiltonian symplectomorphisms $\cH$ coincides with the identity component of the group of
symplectomorphisms. As was pointed out by Donaldson in \cite{D6}, the main obstacle to apply the picture of \secref{sec:KNtheorem}
formally in this situation comes from the fact that $\cH$ does not admit a complexification (a reason for this will be
given later in this section). However, there is an integrable distribution $D \subseteq \cJi$ whose leaves play the
role of the orbits under the action of an ideal complexification. At a point $J \in \cJi$ this distribution is given
by
\begin{equation}\label{eq:distributioncscK}
D_J = \{Y_{\eta|J}; \; \eta \in \LieH^c\},
\end{equation}
where $\LieH^c = \LieH \otimes \CC$ is the complexification of $\LieH$ and
\begin{equation}\label{eq:notationcominfinitesimalcscK}
Y_\eta = Y_{\eta_0} + \mathbf{J} Y_{\eta_1} \text{ for all $\eta = \eta_0 + \imag \eta_1 \in \LieH^c$}.
\end{equation}
Recall that $\mathbf{J}$ is the complex structure on $\cJi$. The distribution $D$ is closed under the Lie bracket in
$\cJi$, since the Nijenhuis tensor of $\mathbf{J}$ vanishes, although it is not obvious that this fact implies the existence of
leaves integrating $D$ through every point $J \in \cJi$ (because $\cJi$ is infinite dimensional). To see that $D$ indeed
integrates we consider the principal $\cH$-bundle
\begin{equation}\label{eq:YsubJ}
\cY \to \cK,
\end{equation}
given by pairs $(\omega',f) \in \cK \times \Diff_0(X)$, such that $f^*\omega' = \omega$. Here we use the assumption
$H^{1}(X) = 0$ to have $\cH$ as the structure group of $\cY$, instead of the identity component of the group of
symplectomorphisms. The map
\begin{equation}\label{eq:tauJ}
\tau_J\colon \cY\to \cJi\colon (\omega',f)\mapsto f^* J,
\end{equation}
plays the role of the complexified action, in the sense that the image $\tau_J(\cY)$ of this map is a leaf of $D$
through $J$ (for this last fact it is important to work with $\cJi$ instead of the whole $\cJ$). The image of $\tau_J$
can be then thought of as the orbit of an ideal complexification of $\cH$ acting on $\cJi$. It follows by definition of
the complexified orbit $\tau_J(\cY)$ of a point $J \in \cJi$ that it contains a zero of the $\cH$-equivariant moment
map $\mu_\cH$ \eqref{eq:scmom} if and only if there exists a cscK metric in $\cK$.

The principal bundle $\cY$ is endowed with a canonical connection that provides the link between the complex
and
the symplectic points of view (see~\cite[Section~2]{D6}). To define this connection on $\cY$ we consider a smooth curve
$$
\omega_t = \omega_{\phi_t} = \omega + dd^c\phi_t \subset \cK
$$
and define its horizontal lift. The horizontal lift $f_t$ through $(\omega_0,f) \in \cY_J$
is given by the Moser curve of $\omega_t$, i.e. the curve of diffeomorphisms with initial condition $f_0 = f$,
integrating the vector $-J y_t$ defined by
$$
i_{y_t}\omega_{t} = d\dot{\phi}_t.
$$
It is straightforward that $f_t$ satisfies $f_t^* \omega_t = \omega_0$ for all $t$, as required. The curve of metrics
$\omega_t$ determines another one $J_t = \tau_J(\omega_t,f_t) \subset \cJi$ which satisfies
\begin{equation}\label{eq:derivativeJtcscK}
\dot{J_t} = \frac{d}{dt} f^{*}_t J = f^{*}_t (L_{-J y_t} J) = - J_t L_{f_t^*y_t} J_t = \mathbf{J} Y_{\eta_t},
\end{equation}
where $\mathbf{J}$ is the complex structure on $\cJi$ and $Y_{\eta_t}$ is the infinitesimal action of the Hamiltonian
vector field $\eta_t = y_{\dot{\phi}_t \circ f_t} \in \LieH$. Special curves $\phi_t$ in the space of K\"{a}hler
potentials are those for which $\dot{\phi}_t \circ f_t$ remains constant, since they determine curves $J_t$ in $\cJi$
via the horizontal lift that follow a pure imaginary direction $\mathbf{J} Y_{\eta_0}$ of the missed complexification
of $\cH$. This last condition is equivalent to
\begin{equation}\label{eq:geodesiccscK}
\ddot{\phi}_t - |d \dot{\phi}_t|^2_{\omega_t} = 0,
\end{equation}
that can be interpreted as the equation of geodesics with respect to suitable negatively curved metric in the space of
K\"{a}hler potentials, as proved independently by T. Mabuchi \cite{Mab1} and S. K. Donaldson \cite{D6}. The analogue
formulation of the geodesic equation in $\cK$ is obtained considering \eqref{eq:geodesiccscK} modulo the constant
functions. A deep analytic study of equation \eqref{eq:geodesiccscK} has been carried out by X. X. Chen and G. Tian in
\cite{ChT}, where the authors proved the existence of a unique $C^{1,1}$ geodesic (with bounded mixed second
derivatives) joining any two points in $\cK$ and that this regularity is sharp. The lack of regularity of the
solutions
gives a clue of why the $\cH$-action in $\cJi$ cannot be complexified (at least in the $C^{\infty}$-category).

\begin{example}\label{example:cscKgeodesic}
We discuss now a well known example of solution to \eqref{eq:geodesiccscK} provided by holomorphic isometries in
$(X,\omega,J)$. Let $\eta \in \LieH_J$, the Lie algebra of the isotropy group of $J$ in $\cH$ (equivalently, the Hamiltonian
Killing vector fields on $(X,J,\omega)$). Given any vector field $y$ on $X$ we denote by $f^{y}_{t}$ the flow of $y$
on
$X$. We claim that
$$
\phi_t = - \int_0^t \phi \circ (f^{J\eta}_t)^{-1},
$$
is a solution to \eqref{eq:geodesiccscK}. Note first that $\omega_t = \omega_{\phi_t} = (f_t^{J\eta})_* \omega$,
since
$$
\frac{d}{dt} (f_t^{J\eta})_* \omega = (f_t^{J\eta})_* d (i_{-J\eta}\omega) = - (f_t^{J\eta})_* (dd^c
\phi) = dd^c \dot{\phi}_t.
$$
This implies that $\eta = - y_{\dot{\phi}_t,\omega_t}$ from which it follows that
$$
\ddot{\phi}_t = (f^{J\eta}_t)_* d\phi(J\eta) = d\dot{\phi}_t(Jy_{\dot{\phi}_t,\omega_t}) = |d
\dot{\phi}_t|^2_{\omega_t},
$$
as claimed.
\end{example}

We make now more explicit the parallelism with the finite dimensional picture explained in Section
\secref{sec:KNtheorem}. Recall that $Z$, $G \subset G^c$ denote, respectively, a K\"{a}hler manifold, a compact group
acting by holomorphic isometries and its complexification. The distribution $D$ in \eqref{eq:distributioncscK} plays the
role of the tangent space $T_z(G^c\cdot z)$ to the orbit of the complexified group $G^c$, while the space $\cK$ is the
analogue of the negatively curved symmetric spaces $G \backslash G^c$. Note that the isomorphism
\begin{equation}\label{eq:GcmodGisocscK}
G \backslash G^c \cong \imag \mathfrak{g}
\end{equation}
given by the exponential map (where $\mathfrak{g}$ denotes the Lie algebra of $G$) corresponds now to an inclusion
$$
\cK \subseteq C^{\infty}(X)/\RR \cong \LieH,
$$
where $\cK$ is viewed as the space of K\"{a}hler potentials of a fixed K\"{a}hler form modulo the constant functions. The
$\cH$-principal bundle $\cY \to \cK$ and the map $\tau_J$ are respectively, the analogues of the projection $G^c \to G \backslash
G^c$
and the map $G^c \to G^c \cdot z \subset Z$, provided by the $G^c$-action on $Z$. The canonical connection on the principal $G$-bundle
$$
G^c \to G\backslash G^c
$$
is provided by the lift $[e^{\imag \zeta_t}] \to g e^{\imag \zeta_t}$ passing trough $g \in G$ and so the smooth
solutions to \eqref{eq:geodesiccscK} are the analogues of the geodesics $t \to [e^{\imag t \zeta} g]$ in
\eqref{eq:GcmodGisocscK}.

\subsection{The Futaki character}
\label{subsec:cscKFutaki}

In this section we explain an obstruction for the existence of cscK metrics in a K\"{a}hler class $[\omega_0]$ due to
A. Futaki \cite{Ft0}. This obstruction is the analogue of the character of a complex orbit introduced in \secref{subsec:KNcharacter} for the finite dimensional case. Together with the discussion in the previous sections, leads to conjecture of the
existence of a Hitchin-Kobayashi correspondence for cscK metrics. We keep the notation of the previous
sections.

To define the invariant, we proceed as in \secref{subsec:KNcharacter} and define first a closed $1$-form in the space
of
K\"{a}hler metrics. Consider the space of K\"{a}hler metrics $\cK$ of $(X,J)$ with fixed K\"{a}hker
class $[\omega]$. Given any $\dot{\phi} \in C^{\infty}(X)$, we identify it with the constant vector field on $\cK$
given
by the flow $\omega_{t\dot{\phi}} = \omega + t dd^c_J\dot{\phi}$, for any $\omega \in \cK$. Then,
the \emph{Mabuchi $1$-form} in $\cK$ is given by
\begin{equation}\label{eq:oneformMabuchi}
\sigma_J(\dot{\phi}) = - \int_X \dot{\phi}(S_{\omega} - \hat{S})\frac{\omega^n}{n!},
\end{equation}
where $\dot{\phi}$ is regarded as a tangent to $\omega \in \cK$. In order to see the parallelism with \eqref{eq:sigmaz}, recall
from \secref{subsec:cscKorbits} that given a curve $\omega_t = \omega_{\phi_t}$ in $\cK$ we can associate, via its
Moser curve of diffeomorphisms $f_t$, a curve of $\omega_0$-compatible complex structures $J_t \in \cJi$. Then, doing
a
change of variable, we can write
\begin{equation}\label{eq:sigmaJmuH}
\sigma_J(\dot{\phi}_t) = \langle \mu_\cH (J_t),\dot{\phi}_t \circ f_t \rangle
\end{equation}
along the curve $\omega_t$, where $\mu_{\cH}$ denotes the moment map in the space of $\omega_0$-compatible complex
structures $\cJi$. From this expression it is straightforward to prove that the $1$-form \eqref{eq:oneformMabuchi} is
closed. We give a proof of this fact for completeness.
\begin{proposition}[T. Mabuchi]\label{propo:oneformMabuchiclosed}
The $1$-form $\sigma_J$ is closed.
\end{proposition}
\begin{proof}
Let $\dot{\phi}_1, \dot{\phi}_2 \in C^{\infty}(X)$, regarded as vector fields in $\cK$. Then it is straightforward that
their Lie bracket vanishes: $[\dot{\phi}_1, \dot{\phi}_2] = 0$. Hence, by \eqref{eq:sigmazclosed1}, we just have to
check
that $\dot{\phi}_1(\sigma_J(\dot{\phi}_2))$ is symmetric in $j = 1,2$. Let $\omega_t$ be the curve determined by
$\dot{\phi}_1$ and consider its horizontal lift $f_t$ starting at any point over $\omega_0$. Then, using the equations
\eqref{eq:sigmaJmuH} and \eqref{eq:derivativeJtcscK}, we obtain
$$
\frac{d}{dt}_{|t = 0} (\sigma_J(\dot{\phi}_2))_{\phi + t\dot{\phi}_1} =
\omega_{\cJi}(Y_{\dot{\phi}_2},\mathbf{J}Y_{\dot{\phi}_1}) - \langle \mu_\cH(J), \omega(y_{\dot{\phi}_2},J
y_{\dot{\phi}_1})\rangle,
$$
that is symmetric, as required.
\end{proof}

Note that the $1$-form  \eqref{eq:oneformMabuchi} is invariant under the right action on $\cK$ of the group of
automorphisms $\Aut(X,J)$ given by pullback. The infinitesimal action of $\Lie \Aut(X,J)$ defines then an invariant
$\cF = \cF_{J,[\omega]}$ of the complex structure and the K\"{a}hler class $[\omega]$, due to A. Futaki \cite{Ft0},
exactly as in \secref{subsec:KNcharacter}. To give an explicit expression of the  \emph{Futaki invariant},
choose any K\"{a}hler metric $\omega \in [\omega]$ and recall that any $\eta \in \Lie \Aut(X,J)$ can be written as (see
e.g.~\cite{LS1})
$$
\eta = y_{\dot{\phi}_1} + J y_{\dot{\phi}_2} + \beta,
$$
where $y_{\dot{\phi}_j}$ is the $\omega$-Hamiltonian field associated to $\dot{\phi}_j$, $j = 1,2$, and $\beta$ is the dual
of a harmonic $1$-form in $(X,J,\omega)$. Then the infinitesimal action of $\eta$ at the K\"{a}hler metric $\omega$
is
$\eta^l = 2 \imag \partial_J\dbar_J\dot{\phi}_2$ and so the Futaki invariant given by $\imag \sigma_J(\eta^l) +
\sigma_J(J \eta)^l$ can be written as
\begin{equation}\label{eq:Futakicharacter}
\cF(\eta) = - \int_X (\dot{\phi}_1 + \imag \dot{\phi}_2)(S_{\omega} - \hat{S})\frac{\omega^n}{n!}
\end{equation}
for any $\eta \in \Lie \Aut(X,J)$. The following result is the analogue of Proposition \ref{propo:futakifinite}.
\begin{proposition}[A. Futaki]
The map $\cF\colon \Lie \Aut(X,J) \to \CC$ is independent of the K\"{a}hler metric $\omega \in [\omega]$. It defines a
character of $\Lie \Aut(X,J)$ that vanishes if $[\omega]$ contains a cscK metric.
\end{proposition}
An elegant proof, similar to the one of Proposition~\ref{propo:futakifinite}, was provided by J. P Bourguignon \cite{Bo}.

\subsection{Towards an analytic Hitchin-Kobayashi correspondence}
\label{subsec:cscKHK}

As we have seen in the previous section, not all the complexified orbits of the $\cH$-action in $\cJi$ carry a cscK
metric, due to the obstruction given by the Futaki invariant. This leads to conjecture the existence of an analytic
Hitchin-Kobayashi correspondence for cscK metrics in a given K\"{a}hler class similar to the one explained in
\secref{subsec:HYMHKproof} for the HYM equation. The key ingredient in this correspondence will be the integral of the moment map, that in this context was defined by T. Mabuchi \cite{Mab1} and it is usually called the \emph{Mabuchi K-Energy}. In this section we recall the definition of the K-energy and discuss some recent progress
in\cite{ChT} and \cite{Ch2} concerning crucial steps towards an analytic Hitchin-Kobayashi correspondence for the cscK
problem.

Let $(X,J)$ be a compact complex manifold. Consider the space $\cK$ of K\"{a}hler metrics
with
fixed K\"{a}hler class $[\omega]$. Recall from the previous section that \eqref{eq:oneformMabuchi} defines a closed
$1$-form in $\cK$. Since the space $\cK$ is contractible, the $1$-form $\sigma_J$ integrates giving the K-energy map
$$
\cM_{\omega}: \cK \to \RR,
$$
which satisfies $d \cM_{\omega} = \sigma_J$ and $\cM_{\omega}(\omega) = 0$. This implies that the critical points of
the
K-energy are the cscK metrics with fixed K\"{a}hler class $[\omega]$. Given any curve of metrics $\omega_t =
\omega_{\phi_t}$ in $\cK$, we can write explicitly
\begin{equation}\label{eq:mabuchiKenergy}
\cM_{\omega}(\omega_t) = \cM_{\omega}(\omega_0) - \int_0^t \int_X \dot{\phi}_s(S_{\omega_s} -
\hat{S})\frac{\omega^n_s}{n!} \wedge ds,
\end{equation}
that combined with \eqref{eq:sigmaJmuH} shows that $\cM_{\omega}$ is convex along geodesics in the space of K\"{a}hler
metrics. The lack of regularity of the geodesic equation makes very difficult to apply the recipe provided by the
proof of Kempf-Ness Theorem, however it has been sufficient for the proof with full generality of the following two fundamental
results, due to Chen and Tian \cite{ChT}. The first one is that there exists a unique solution of the cscK
equation in $[\omega_0]$ modulo the action $\Aut(X,J)$. The second one is that the Mabuchi $K$-energy is bounded from
below if there exists a cscK metric in $[\omega_0]$. Other remarkable progress has been made by X. X. Chen in
\cite{Ch2}, where a notion of geodesic stability has been defined and proved to be an obstruction for the existence of
cscK metrics in $[\omega_0]$. As in Section \secref{sec:KNtheorem}, Chen defines maximal weights using the
$1$-form $\sigma_J$ and suitable weak solutions to \eqref{eq:geodesiccscK} defined for all time.

\subsection{Algebraic K-stability}
\label{subsec:cscKalgKstab}

In this section we recall the definition of $K$-stability for polarized manifolds. This stability notion was introduced by G. Tian in \cite{T2} and generalized by S.K. Donaldson in \cite{D4}.

Let $(X,L)$ be a pair consisting of a projective scheme $X$ and an ample invertible sheaf $L$ over $X$. The notion of K-stability for such pairs is formulated in terms of flat degenerations of the pair over $\CC$. The central fiber of each degeneration is endowed with a $\CC^*$-action and the stability of the pair, is measured in terms of a suitable weight $- F_1$ associated to this action, the generalized Futaki invariant. When the degeneration considered is smooth it was proved by Donaldons in \cite{D4} that $- F_1$ coincides, up to multiplicative constant factor, with the evaluation of the Futaki character \eqref{eq:Futakicharacter} in the generator of the $S^1 \subset \CC^*$ action.

We start defining the degenerations of $(X,L)$ that we will consider. In what follows a \emph{$\CC^*$-action} on such a pair means a $\CC^*$-action on $X$ with a $\CC^*$-linearisation on $L$.

\begin{definition}\label{def:testconfigurationcscK}
A \emph{test configuration for $(X,L)$} is a pair $(\mathcal{X},\cL)$ consisting of a scheme $\mathcal{X}$ and an invertible sheaf $\cL$ over $\mathcal{X}$, together with a $\CC^*$-action on $(\mathcal{X},\cL)$ and a flat morphism $\pi\colon \mathcal{X}\to \CC$, such that
\begin{enumerate}
\item[(1)] %
the invertible sheaf $\cL$ is ample,
\item[(2)] %
$\pi: \mathcal{X} \to \CC$ is $\CC^*$-equivariant, where $\CC^*$ acts
on $\CC$ by multiplication in the standard way,
\item[(3)]
the fibre $(X_t,L_t)$ is isomorphic to $(X,L)$ for all
$t \in\CC\setminus\{0\}$, where $X_t = \pi^{-1}(t)$ and $L_t =
\cL_{|X_t}$.
\end{enumerate}
\end{definition}

Any $\CC^*$-action on $(X,L)$ determines a special type of test configuration, called a \emph{product configuration}, where $\mathcal{X} \cong X \times \CC$ with the induced $\CC^*$-action, $\cL$ is obtained by pulling the $\CC^*$-linearised sheave $L$, and $\pi$ is the canonical projection. This product configuration is called a \emph{trivial configuration} if the given $\CC^*$-action on $(X,L)$ is trivial, i.e. if it consists of the trivial $\CC^*$-action on $X$ and the $\CC^*$-action defined by scalar mutiplication on $L$.

Let $(X,L)$ be as in Definition~\ref{def:testconfiguration}. Let us consider a test configuration $(\mathcal{X},\cL)$ for $(X,L)$. As $0\in\CC$ is fixed, the central fibre $(X_0,L_0)$ has an induced $\CC^*$-action. For each integer $k$ we consider the integer $w(k)$ given by the weight of the $\CC^*$-action on the determinant of the cohomology of $L_0^k$, i.e. on the line (cf. e.g.~\cite[\S\S 1.2 and 2.1]{HL})
\[
  \det H^*(X_0,L_0^k)
  \defeq \bigotimes_{p=0}^n \(\det H^p (X_0,L_0^k)\)^{(-1)^{p}}.
\]
Then, by the equivariant Riemann--Roch theorem, $w(k)$ is a polynomial of degree at most $n+1$ in $k$ (see~\cite[\S 2.1]{D4}).  Thus the following quotient has an expansion
\begin{equation}\label{eq:defF(W,k)cscK}
F(k) \defeq \frac{w(k)}{kP(k)} = F_0 + k^{-1} F_1 + k^{-2} F_2 + O(k^{-3})
\end{equation}
with rational coefficients $F_i$, where $P(k)$ is the Hilbert polynomial of $\cO_{X_0}$ which has degree $n$ in $k$. Now, the Futaki invariant of the test configuration $(\mathcal{X},\cL,)$ is
\begin{equation}
\label{eq:extendedfutakicscK}
- F_1.
\end{equation}
We already have all the ingredients to give the notion of K-stability.

\begin{definition}
\label{def:extKstablecscK}
Let $(X,L)$ be a pair consisting of a projective scheme and an ample invertible sheaf. Then the pair $(X,L)$ is said to be \emph{K-semistable} if
\begin{equation}
  \label{eq:extKstablecscK}
  - F_1 \geq 0
\end{equation}
for all integers $k>0$ and all test configurations $(\mathcal{X},\cL)$ for $(X,L^k)$. This pair is \emph{K-stable} if the inequality~\eqref{eq:extKstablecscK} is strict for all non-trivial $(\mathcal{X},\cL,\cE)$. It is called \emph{$\alpha$-K-polystable} if it is K-semistable and the inequality~\eqref{eq:extKstablecscK} is always strict unless $(\mathcal{X},\cL)$ is a product configuration.
\end{definition}

As usual, any pair as in Definition~\ref{def:extKstablecscK} satisfies the following chain of implications: $(X,L)$ is K-stable $\implies$ $(X,L)$ is K-polystable $\implies$ $(X,L)$ is K-semistable. A \emph{K-unstable} triple is a triple which is not K-semistable. We know that any smooth compact polarized manifold in $K$-polystable mainly due to the work of Donaldson \cite{D8}, T. Mabuchi \cite{Mab2} and J. Stoppa \cite{Stp}. Donaldson has proved the converse for the case of toric surfaces \cite{D8}. However, as remarked by J. Ross and R. Thomas~\cite{RRT1}, a recent example in \cite{ACGT} suggests that the algebraic notion of K-stability we have just defined for may not be sufficient to guarantee the existence of a cscK metric and so a stronger notion may be required.

\chapter{\bf The coupled equations}
\label{chap:Ceq}

In this chapter we explain the moment map interpretation of the \emph{coupled equations} \eqref{eq:CYMeq00}, that
arises when we consider an extension of the infinite dimensional Lie groups involved in the moment map problems for the HYM and the cscK equation. Symplectic reductions by Lie group extensions have been studied in the literature in various degrees of generality (see~\cite{MMOPR} and references therein). Previous work includes split group extensions and more general ones, although it seems that the moment map calculations of~\secref{sec:Ceqcoupling-term}, based on Proposition~\ref{prop:ham-act-ext-grp}, have not been previously made (cf.~\cite[\S 3.2]{MMOPR}). We also describe solutions of \eqref{eq:CYMeq00} in terms of the Kaluza-Klein Theory for $G$-invariant metrics on $E$ and relate them with the Einstein--Yang--Mills equation.

Throughout this chapter, $X$ is a fixed compact symplectic manifold of dimension $2n$, with symplectic form $\omega$
and volume form $\vol_\omega =
\frac{\omega^n}{n!}$. We fix a real compact Lie group $G$ with Lie algebra $\mathfrak{g}$ and a smooth principal
$G$-bundle $E$ over $X$. We also fix
a positive definite inner product
$$
(\cdot,\cdot): \mathfrak{g} \times \mathfrak{g} \to \RR
$$
on $\mathfrak{g}$ invariant under the adjoint action, that induces a metric on the adjoint bundle $\ad E$. The
space of smooth $k$-forms on
$X$ is denoted by $\Omega^{k}$ and for any associated vector bundle $F$, $\Omega^{k}(F)$ denote the space of smooth
$F$-valued $k$-forms on $X$.
Considering the space $\Omega^k(\ad E)$, the metric on $\ad E$ extends to give a pairing
\begin{equation}\label{eq:Pairing}
\Omega^p(\ad E) \times \Omega^q(\ad E) \to \Omega^{p+q}.
\end{equation}
If $X$ has a fixed almost complex structure compatible with $\omega$ the operator
\begin{equation}
\label{eq:Lambda}
  \Lambda:\Omega^{p,q}\to \Omega^{p-1,q-1}
\end{equation}
acting  on the space of smooth $(p,q)$-forms is the adjoint of the Lefschetz operator $\Omega^{p-1,q-1}\to
\Omega^{p,q}\colon \alpha \mapsto \alpha
\wedge \omega$, with respect to the metric determined by the symplectic form $\omega$ and the given almost complex
structure. Its linear extension
$\Lambda:\Omega^{p,q}(\ad E)\to \Omega^{p-1,q-1}(\ad E)$ will be denoted in the same way (we use the same notation as,
e.g., in~\cite{D3}).

\section{Hamiltonian actions of extended Lie groups}
\label{sec:Ceqham-act-ext-grp}

The aim of this section is to describe, under certain assumptions, hamiltonian actions of a Lie group on a symplectic
manifold, in terms of a normal
Lie subgroup and the quotient Lie group. Later in this thesis we will be interested in the case in which the symplectic
manifold and the Lie groups are
infinite dimensional.

Suppose that there is an extension of Lie groups
\begin{equation}\label{eq:Ext-Lie-groups}
 1 \to \cG \lra{\iota} \cX \lra{\pr} \cH \to 1.
\end{equation}
Such an extension determines another extension of Lie algebras
\begin{equation}\label{eq:Ext-Lie-alg}
 0 \to \LieG \lra{\iota} \LieX \lra{\pr} \LieH \to 0,
\end{equation}
where the use of the same symbols $\iota$ and $\pr$ should lead
to no confusion. 
Let $\cA$ be a manifold with an action of the `extended' Lie group
$\cX$. For each $\zeta\in\LieX$, let $Y_\zeta$ be the vector field
on $\cA$ generated by the infinitesimal action of $\zeta$.

Note that the short exact sequence~\eqref{eq:Ext-Lie-alg} does not
generally split as a sequence of Lie algebras, but it always does
as a short exact sequence of vector spaces. Let $W\subset
\Hom(\LieX,\LieG)$ be the affine space of vector space splittings.
Since $\cG\subset\cX$ is a normal subgroup, there is a
well-defined $\cX$-action on $W$, given by
\[
  g\cdot \theta \defeq \Ad(g)\circ \theta \circ \Ad(g^{-1}),
  \text{ for $g\in\cX$, $\theta\in W$.}
\]
Let $\cW \subset C^{\infty}(\cA,W)$ be the space of $\cX$-equivariant smooth maps $\theta\colon \cA\to W$. The case considered in this thesis (see~\secref{sec:Ceqcoupling-term}) is an example where the space $\cW$ is non-empty. Observe that if $\cA$ is a point, then $\cW$ is the space of $\cX$-equivariant splittings of~\eqref{eq:Ext-Lie-alg} and any such equivariant splitting determines an isomorphism $\LieX \cong \LieG\rtimes\LieH$, so in this case the condition that $\cW\neq\varnothing$ is very strong.

Let $\omega_\cA$ be a symplectic form on $\cA$ such that the $\cX$-action on $\cA$ is symplectic. We want to
characterise hamiltonian $\cX$-actions on
$\cA$ in terms of $\cG$ and $\cH$ under the assumption that there exists some $\theta\in\cW$. Suppose that the
$\cX$-action is Hamiltonian with
$\cX$-equivariant moment map $\mu_{\cX}\colon \cA\to (\LieX)^*$. Using $\theta\in\cW$ we will break this map into
pieces corresponding to the Lie
algebras $\LieG$ and $\LieH$. Note first that the isomorphism
$$
W\to\Hom(\LieH,\LieX) \colon \theta\mapsto \theta^\perp,
$$
where $\theta^\perp \colon \LieH \to \LieX$ is uniquely defined by the equation
\begin{equation}
\label{eq:theta-perp} \Id - \iota\circ \theta = \theta^\perp\circ p,
\end{equation}
is $\cX$-equivariant with respect to the action in $\Hom(\LieH,\LieX)$ given by $g\cdot\theta^\perp= Ad(g)\circ
\theta^\perp \circ \Ad(\pr (g^{-1}))$
for all $g\in\cX$. Here we are using the $\cX$-action in $\LieH$ induced by the surjective homomorphism $p$ in
the exact sequence
\eqref{eq:Ext-Lie-groups}. Thus, given now $\theta\in\cW$, the map
\[
  \theta^\perp \colon \cA \to \Hom(\LieH,\LieX)
\]
is $\cX$-equivariant and we can break the moment map into
\begin{equation}
\label{eq:piecescXmmap}
\langle \mu_{\cX},\zeta\rangle = \langle \mu_{\cX}, \iota \theta\zeta\rangle + \langle \mu_{\cX},
\theta^{\perp}\pr(\zeta) \rangle, \text{ for all
$\zeta\in\LieX$,}
\end{equation}
where the summands in the right hand side define a pair of $\cX$-equivariant maps $\mu_{\cG}\colon \cA\to (\LieG)^*$
given by
$$
\langle \mu_{\cG},
\zeta \rangle \defeq \langle \mu_{\cX}, \iota\zeta\rangle
$$
for all $\zeta\in\LieG$, and $\sigma_\theta\colon \cA\to (\LieH)^*$
\begin{equation}
\label{eq:def-sigma} \langle \sigma_\theta, \eta \rangle \defeq \langle\mu_{\cX},
\theta^\perp \eta \rangle, \text{ for all $\eta\in\LieH$.}
\end{equation}
Note that since $\cG$ is a normal subgroup of $\cX$ we can require the map $\mu_{\cG}$ to be $\cX$-equivariant. It is
now straightforward from the
moment map condition for $\mu_{\cX}$ to check that $\mu_{\cG}$ is a $\cG$-equivariant moment map for the $\cG$-action
on $\cA$, i.e. $d\langle
\mu_{\cG},\zeta\rangle = i_{Y_{\zeta}}\omega_{\cA}$ for all $\zeta \in \LieG$. We will see in the next proposition
that
$\sigma_\theta$ satisfies a
similar infinitesimal condition that provides our characterization of Hamiltonian $\cX$-action, for what we first
introduce some notation. Given a
smooth map $\zeta\colon \cA\to\LieX$, let us denote by $Y_\zeta$ the vector field on $\cA$ given by
\[
  \(Y_\zeta\)_A \defeq \frac{d}{dt}\bigg|_{t=0}\exp(t\zeta_A)\cdot A,
  \text{ for all $A\in\cA$}.
\]
In particular, in the notation~\eqref{eq:theta-perp}, any smooth $\theta\colon \cA \to W$ defines a map
\[
Y_{\theta^\perp}\colon \LieH
\to\Omega^0(T\cA)\colon \eta\mapsto Y_{\theta^\perp\eta}.
\]
Note also that $d\theta$ is a $\cX$-invariant $\Hom(\Lie\cH,\Lie\cG)$-valued 1-form on $\cA$, as $\theta\in\cW$ and the affine space $\cW$ is modelled on the space of $\cX$-equivariant smooth maps $\tau \colon \cA \to \Hom(\LieH,\LieG)$.

\begin{proposition}\label{prop:ham-act-ext-grp}
Suppose that there exists some $\theta\in\cW$. Then the $\cX$-action on $\cA$ is hamiltonian if and only if so is the
action of $\cG\subset\cX$ on
$\cA$, with a $\cX$-equivariant moment map $\mu_{\cG}\colon \cA\to (\LieG)^*$, and there exists a smooth
$\cX$-equivariant map $\sigma_\theta\colon
\cA\to (\LieH)^*$ satisfying
\begin{equation}
\label{eq:muX}
 i_{Y_{\theta^\perp\eta}}\omega_\cA = \langle \mu_\cG, (d\theta) \eta\rangle + d\langle \sigma_\theta, \eta\rangle,
\text{ for all $\eta\in\LieH$}.
\end{equation}
In this case, a $\cX$-equivariant moment map $\mu_{\cX}\colon
\cA\to (\LieX)^*$ is given by
\begin{equation}
\label{eq:def-muX}
  \langle \mu_{\cX},\zeta\rangle = \langle \mu_\cG, \theta\zeta\rangle +
  \langle \sigma_\theta, \pr (\zeta) \rangle, \text{ for all $\zeta\in\LieX$.}
\end{equation}
\end{proposition}
\begin{proof}
To prove the ``only if'' part it remains to check that \eqref{eq:muX} holds. This follows by definition, by taking
differentials in
\eqref{eq:piecescXmmap} and using that
\begin{gather*}
d \langle \mu_\cG,\theta \zeta\rangle =
 \langle d\mu_\cG,\theta \zeta\rangle + \langle  \mu_\cG,(d\theta) \eta \rangle
\, \text{ and }\\
  i_{Y_\zeta}\omega =  i_{Y_{\theta\zeta}}\omega + i_{Y_{\theta^\perp\eta}}\omega,
\text{ with $\eta \defeq \pr(\zeta)$,}
\end{gather*}
where the first equation is obtained applying the chain rule, and the second one holds because $\zeta = \theta\zeta +
\theta^\perp\eta$ and $Y_\zeta$
is linear in $\zeta$. The ``if'' part is straightforward from the statement and is left to the reader.
\end{proof}

To see the dependence of $\sigma_\theta$ on the choice of $\theta\in\cW$, recall that $\cW$ is an affine space
modelled
on the vector space of
$\cX$-equivariant smooth maps $\tau \colon \cA \to \Hom(\LieH,\LieG)$, so the elements of $\cW$ are all of the form
$\theta+p^*\tau$. Now, if $\theta$
and $\sigma_\theta$ satisfy the conditions of Proposition~\ref{prop:ham-act-ext-grp}, then the same conditions is
obviously satisfied by
$\sigma_{\theta + pr^*\tau}$ with respect to $\theta+\pr^*\tau$, where
$$
\langle \sigma_{\theta+\pr^*\tau}, \eta\rangle = \langle \sigma_\theta, \eta\rangle - \langle \mu_{\cG}, \tau
\eta\rangle, \text{for all
$\eta\in\LieH$}.
$$
Condition~\eqref{eq:muX} for $\sigma_\theta$ in proposition \ref{prop:ham-act-ext-grp} generalizes the usual
infinitesimal condition
$i_{Y_\eta}\omega_\cA = d\langle\mu_\cH,\eta\rangle$ ($\eta\in\cH$) for moment maps $\mu_\cH$ for the induced
$\cH$-action on $\cA$ when the Lie group
extension~\eqref{eq:Ext-Lie-groups} splits.  To see this, suppose that there exists such a splitting. It determines an
$\cH$-action on $\cA$ and a
$\cX$-invariant $\theta_0\in W$, so for all $\eta\in\cH$, $Y_{\theta^\perp_0\eta}$ is the vector field  induced by the
infinitesimal action of $\eta$.
Then the constant map $\theta\colon \cA\to W\colon A\mapsto \theta_0$ is in $\cW$ and $d\theta=0$, so~\eqref{eq:muX}
is
precisely the condition that
$\sigma_\theta$ is a moment map for the $\cH$-action on $\cA$. Thus, for a fixed Lie group splitting of
\eqref{eq:Ext-Lie-groups},
the $\cX$-action on $\cA$ is Hamiltonian if and only if so are the induced actions of $\cG$ and $\cH$, with
\emph{$\cX$-equivariant} moment maps
$\mu_{\cG}$ and $\sigma_\theta$, respectively.

\section{The extended gauge group action on the space of connections}
\label{sec:Ceqcoupling-term}

Let $X$ be a compact symplectic manifold of dimension $2n$, with symplectic form $\omega$. We fix a compact Lie group
$G$ and a (smooth) principal $G$-bundle $E$ on $X$, with projection map $\pi\colon E\to X$. Let $\cH$ be the group of
Hamiltonian symplectomorphisms of $X$, as defined in~\secref{subsec:cscKmmap}. An \emph{automorphism} of $E$ is a
$G$-equivariant diffeomorphism $g\colon E\to E$. Any such automorphism covers a unique diffeomorphism $h\colon X\to
X$,
i.e. a unique $h$ such that $\pi\circ g=h\circ \pi$.  The (Hamiltonian) \emph{extended gauge group} $\cX$ of $E$ is
the
group  of automorphisms of $E$ which cover elements of $\cH$. Then the gauge group of $E$, already mentioned
in~\secref{subsec:HYMmmap}, is the normal subgroup $\cG\subset\cX$ of automorphisms covering the identity.
Furthermore,
there is a canonical short exact sequence of Lie groups
\begin{equation}
\label{eq:coupling-term-moment-map-1}
  1\to \cG \lra{\iota} \cX \lra{\pr} \cH \to 1.
\end{equation}
Here $\iota$ is the inclusion map and $\pr$ is the map that assigns to each automorphism the Hamiltonian
symplectomorphisms that it covers.

The only non-trivial fact that needs to be proved to show that~\eqref{eq:coupling-term-moment-map-1} is exact is that
$\pr$ is surjective. Choose a
connection $A$ on $E$. Given $h \in \cH$, by definition there exists a Hamiltonian isotopy $[0,1] \times X \to X
\colon
(t,x) \mapsto h_t(x)$ from $h_0=\Id$ to $h_1=h$, which is the flow of a smooth family of vector fields $\eta_t \in
\LieH$, i.e. with $dh_t/dt = \eta_t \circ h_t^{-1}$ (see e.g.~\cite[\S 3.2]{McS}). Let $\zeta_t\in\LieX$ be the
horizontal lift to $E$ of $\eta_t$ given by the connection $A$. The vector fields $\zeta_t$ are  $G$-invariant so its
time-dependent flow $g_t$ exists for all $t\in [0,1]$ and the $g_t\colon E\to E$ are $G$-equivariant. Since $\zeta_t$
is a lift of $\eta_t$ to $E$, its flow $g_t$ covers $h_t$, so in particular $g_t\in\cX$ for all $t$ and in particular
$g_1\in\cX$ covers $h=h_1$. Thus $\pr$ is surjective and so~\eqref{eq:coupling-term-moment-map-1} is exact, as
claimed.

\begin{remark}
Note that the existence of the short exact sequence~\eqref{eq:coupling-term-moment-map-1} holds even when the
structure
group or even the base are
non-compact. The crucial fact is that $\cH$ lies in the connected component of the identity of $\Diff X$. For
generalizations of the exact
sequence~\eqref{eq:coupling-term-moment-map-1} see~\cite{ACMM}.
\end{remark}

There is an action of the extended gauge group on the space $\cA$ of connections on $E$. To define this action,
we view the
elements of $\cA$ as $G$-equivariant splittings $A\colon TE\to VE$ of the short exact sequence
\begin{equation}
\label{eq:principal-bundle-ses}
  0 \to VE\lto TE\lto \pi^*TX \to 0,
\end{equation}
where $VE=\ker d\pi$ is the vertical bundle. Using the action of $g\in \cX$ on $TE$, the $\cX$-action on $\cA$
is $g \cdot A \defeq
g\circ A \circ g^{-1}$. Any such splitting $A$ induces a vector space splitting of the Atiyah
short exact sequence
\begin{equation}\label{eq:Ext-Lie-alg-3}
0\to \Lie\cG \lra{\iota} \Lie(\Aut E) \lra{\pr} \Lie(\Diff X) \to
0
\end{equation}
(cf.~\cite[equation~(3.4)]{AB}), where $\Lie(\Diff X)$ is the Lie algebra of vector fields on $X$ and $\Lie(\Aut E)$
is
the Lie algebra
of $G$-invariant vector fields on $E$. This splitting is given by maps
$$
\theta_A\colon \Lie(\Aut E)\to \LieG, \quad \theta_A^\perp\colon
\Lie(\Diff X) \to \Lie(\Aut E)
$$
such that $\iota\circ \theta_A + \theta_A^\perp\circ \pr =\Id$, where $\theta_A^\perp$ is the horizontal lift of
vector
fields on $X$ to vector fields
on $E$ given by $A$. This splitting restricts to another one of the exact sequence
\begin{equation}\label{eq:Ext-Lie-alg-2}
 0 \to \LieG \lra{\iota} \LieX \lra{\pr} \LieH \to 0
\end{equation}
induced by~\eqref{eq:coupling-term-moment-map-1}, because $\LieX=\pr^{-1}(\LieH)$. In the notation
of~\secref{sec:Ceqham-act-ext-grp}, it is easy
to see that the map
\begin{equation}\label{eq:theta}
\theta\colon \cA\to W\colon A\mapsto \theta_A
\end{equation}
is an element of $\cW$. It is also clear that the $\cX$-action on $\cA$ is symplectic, for the symplectic
form~\eqref{eq:SymfC}. The methods
of~\secref{sec:Ceqham-act-ext-grp} apply here to provide a moment map. Let $\hat{c}$ be the real constant
\begin{equation}\label{eq:constant-c-hat}
\hat c \defeq \frac{\int_X (F_A \wedge F_A) \wedge \omega^{n-2}}{\int_X\vol_\omega} = \frac{\langle c(E) \cup
[\omega]^{n-2},[X]\rangle}{\Vol([\omega])}
\end{equation}
where $c(E) \in H^4(X)$ is the Chern-Weyl class associated to the biinvariant bilinear form $(\cdot, \cdot)$ in
$\mathfrak{g}$, that depends only on
the topology of the bundle (see Theorem $1.1$, p. $293$ in~\cite{KNII}). Then,
\begin{proposition}\label{prop:momentmap-X}
The $\cX$-action on $\cA$ is hamiltonian, with $\cX$-equivariant moment map $\mu_{\cX}\colon \cA\to (\LieX)^*$ given
by
\begin{equation}\label{eq:thm-muX}
  \langle \mu_{\cX},\zeta\rangle = \langle \mu_\cG, \theta\zeta\rangle +
  \langle \sigma, \pr (\zeta) \rangle, \text{ for all $\zeta\in\LieX$,}
\end{equation}
where $\mu_\cG: \cA \to (\LieG)^*$ and $\sigma\colon \cA\to (\LieH)^*$~are
\begin{equation}
\label{eq:sigma}
\begin{split}
\langle \mu_\cG,\theta\zeta\rangle(A) & = \int_X(\theta_A\zeta \wedge F_A) \wedge \omega^{n-1}\\
\langle\sigma,\eta\rangle(A) & = -\frac{(n-1)}{2}\int_X \eta((F_A \wedge F_A) \wedge \omega^{n-2} - \hat c
\vol_\omega)
\end{split}
\end{equation}
for any $\eta \in \LieH \cong C^{\infty}(X)/\RR$.
\end{proposition}
\begin{proof}
By Proposition~\ref{prop:ham-act-ext-grp} this will follow from the facts that $\mu_\cG$ and $\sigma$ are
$\cX$-equivariant and that the map $\sigma$
defined by~\eqref{eq:sigma} satisfies~\eqref{eq:muX}. The $\cX$-equivariance is immediate from the formulae
\eqref{eq:sigma} and the Change of
Variable Theorem. To prove the last fact, we first compute the terms involved in~\eqref{eq:muX}.

Let $\zeta\in \Lie(\Aut E)$, $A\in\cA$. By the Leibninz rule, for all $v\in\Omega^0(TE)$,
\[
  \frac{d}{dt}\bigg|_{t=0}\(e^{t\zeta}\circ A\circ e^{-t\zeta}(v)\)
  = \theta_A[\zeta,v] - [\zeta,\theta_A(v)]
  = \theta_A[\zeta,v - \theta_A(v)],
\]
where in the second equality we have used the fact that $\zeta$ covers a vector field on $X$, so that the vector field
$[\zeta,\theta_A(v)]$ is
vertical. It is easy to see that this expression is tensorial in $v$, so at each point of $E$ it only depends on its
projection $p(v)$. Thus for any
$\zeta\in \LieX$ the infinitesimal action $Y_\zeta(A)$ on $\cA$ regarded as an element of $\Omega^1(\ad E)$ is given
by
\begin{equation}
\label{eq:infinit-action-connections}
\begin{split}
  Y_\zeta(A)(y) & = \theta_A[\zeta,\theta_A^\perp(y)] \\
  & = [\theta_A\zeta,\theta_A^\perp(y)] + \theta_A[\theta_A^\perp(\pr \zeta),\theta_A^\perp(y)]\\
  & = -d_A(\theta_A\zeta) - i_{\pr(\zeta)} F_A,
\end{split}
\end{equation}
for any $y \in \Omega^0(TX)$ (cf. with the equation before (4.2) and the equation after (3.4) in \cite{AB}. Note that
we are using a different sign
for the curvature).

Note that~\eqref{eq:infinit-action-connections} also applies for
maps $\zeta\colon \cA\to \Lie(\Aut E)$. In particular,
\begin{equation}
\label{eq:momentmap-X-1}
   Y_{\theta_A^\perp\eta}(A)   = - i_\eta F_A, \text{ for $\eta\in\LieH$.}
\end{equation}
The $\Hom(\LieH,\LieG)$-valued 1-form $d\theta$ on $\cA$ is given
by
\begin{equation}
\label{eq:momentmap-X-2}
  i_a(d\theta) = a\colon \LieH \to \LieG \colon \eta\mapsto i_\eta a,
\end{equation}
for $A\in\cA$ and $a\in T_A\cA=\Omega^1(\ad E)$. Hence, with the same notation,
\begin{equation}
\label{eq:momentmap-X-4}
  ((d\theta)\eta)(a) = a(\eta), \text{ for $\eta\in\LieH$.}
\end{equation}
For the last term of~\eqref{eq:muX}, let $A\in\cA$, $a\in
T_A\cA$, $\zeta\in\LieX$. Let $\eta=\eta_\phi\in\LieH$, where
$\phi\in C^\infty(X)/\RR$. Using the path $A_t=A+ta$, we obtain
\begin{equation}
\label{eq:momentmap-X-3}
\begin{split}
  d\langle \sigma,\eta\rangle (a) & = -\frac{n-1}{2} \frac{d}{dt}\bigg|_{t=0}
  \int_X \phi(F_{A_t} \wedge F_{A_t})\wedge \omega^{n-2} - \hat c \vol_\omega)\\
  & = -(n-1) \int_X \phi (d_Aa \wedge  F_A) \wedge \omega^{n-2}\\
  & = -(n-1) \int_X \phi \, d(a \wedge  F_A) \wedge \omega^{n-2}\\
  & = (n-1) \int_X i_{\eta}\omega \wedge (a \wedge F_A) \wedge \omega^{n-2}.\\
\end{split}
\end{equation}
Here we have used $dF_{A_t}/dt = d_A a$ for $t=0$ in the second
equality, the Bianchi identity $d_A F_A=0$ in the third equality,
and we have integrated by parts and used the formula
$d\phi=i_{\eta}\omega$ in the last equality. To compute the
integral in the last equality, note that $(a \wedge F_A)
\wedge \omega^{n-1}=0$, as $\dim X=2n$, so $i_\eta ((a \wedge
F_A) \wedge \omega^{n-1})$ equals
\[
(a(\eta) \wedge F_A)\wedge \omega^{n-1} - (a\wedge i_\eta
F_A)\wedge \omega^{n-1} - (n-1)(a\wedge F_A) \wedge
i_\eta\omega\wedge  \omega^{n-2}=0.
\]
Combined with~\eqref{eq:momentmap-X-3}, we thus obtain
equation~\eqref{eq:muX}:
\begin{equation}
\label{eq:momentmap-X-4}
\begin{split}
  d\langle \sigma, \eta \rangle (a) & =
\int_X (a\wedge i_\eta F_A)\wedge \omega^{n-1} - (a(\eta) \wedge F_A)\wedge \omega^{n-1} \\
& = \int_X (Y_{\theta_A^\perp\eta}(A) \wedge a)\wedge \omega^{n-1}
  - (((d\theta)\eta)(a)\wedge F_A)\wedge \omega^{n-1} \\
& = i_{Y_{\theta_A^\perp\eta}}\omega_\cA (a) -
\langle\mu_\cG,(d\theta)\eta\rangle(a). \qedhere
\end{split}
\end{equation}
\end{proof}

To finish this section we identify the image of the moment map $\mu_{\cX}: \cA \to (\LieX)^*$ in the above proposition
with elements in $\LieX$ as we did in \secref{subsec:HYMmmap} and \secref{subsec:cscKmmap} for the groups $\cG$ and
$\cH$. Though we do not have now a canonical $\cX$-invariant pairing in $\LieX$, we can
use the $\cX$-equivariant map defined in \eqref{eq:theta}, $\theta\colon \cA \to \Hom(\LieH,\LieG)$, to construct an
$\cA$-dependent pairing. Let us identify $\LieH \cong C^{\infty}(X)/\RR$ with $C_0^{\infty}(X)$, the space of functions with zero integral over $X$ with respect to $\vol_\omega$. Then, the $L^2$ pairing on functions on $X$ provides an
$\cH$-invariant pairing
\begin{equation}\label{eq:pairingH}
\langle\eta_1,\eta_2\rangle_{\cH} = \int_X \phi_1 \phi_2 \vol_\omega,
\end{equation}
where $\eta_j \in \LieH$ and $\phi_j \in C^{\infty}_0(X)$ is the Hamiltonian function determined by $\eta_j$, $j =
1,2$. Then, we obtain a map
$$
\langle\cdot,\cdot\rangle_{\cX}\colon \cA \to S^2\LieX^*
$$
that to each $A \in \cA$ assigns a symmetric pairing on $\LieX$ given by
\begin{equation}\label{eq:pairingscX}
\langle\zeta_1,\zeta_2\rangle_{\cX|A} = \langle
p(\zeta_1),p(\zeta_2)\rangle_{\cH} + \langle\theta_A\zeta_1,\theta_A\zeta_2\rangle_{\cG}, \text{ for
$\zeta_1, \zeta_2 \in \LieX$.}
\end{equation}
Since both, $\theta$ and $\langle\cdot,\cdot\rangle_\cG$ are $\cX$-equivariant (recall that $\cG$ is normal in $\cX$),
the map
$\langle\cdot,\cdot\rangle_{\cX}$ is $\cX$-equivariant. Given $A \in \cA$ let
$$
\phi_A = \frac{1}{4}\Lambda^2_\omega (F_A \wedge F_A) \in C^{\infty}(X),
$$
and consider the element of the Lie algebra of $\cX$ given by
\begin{equation}\label{eq:zetaA}
\zeta_A = \Lambda_\omega F_A - \theta_A^{\perp}(\eta_A) \quad \in \quad \LieX,
\end{equation}
where $\Lambda_\omega F_A \in \Omega^0(\ad E)$ is regarded as a vertical $G$-invariant vector field in $E$ and
$\eta_A \in \LieH$ is the Hamiltonian vector field associated to $\phi_A$. Then, given $\zeta \in
\LieX$ covering $\eta_\phi \in \LieH$ we have
\begin{equation}
\label{eq:LievsLiedual0}
\begin{split}
\langle\mu_{\cX}(A),\zeta\rangle & = \int_X(\theta_A\zeta\wedge F_A) \wedge \omega^{n-1} -\frac{(n-1)}{2}\int_X \phi((F_A
\wedge F_A) \wedge \omega^{n-2} - \hat c \vol_\omega)\\
& = \frac{(n-1)!}{4}\big{(}\int_X(\theta_A\zeta \wedge 4\Lambda_\omega F_A) - \phi(\Lambda^2_\omega(F_A \wedge F_A) -
\frac{2\hat c}{(n-2)!}) \cdot \vol_\omega\big{)}\\
& = \frac{(n-1)!}{4}\big{(}\langle\theta_A\zeta \wedge 4\Lambda_\omega F_A\rangle_{\cG} + \langle\phi,-(\Lambda^2_\omega(F_A
\wedge F_A) - \frac{2\hat c}{(n-2)!})\rangle_\cH\big{)}\\
& = (n-1)! \langle\zeta,\zeta_A\rangle_{\cX|A},
\end{split}
\end{equation}
and so the moment map $\mu_{\cX}$ satisfies
\begin{equation}\label{eq:LievsLiedual}
\mu_{\cX}(A) = (n-1)! \langle\zeta_A,\cdot\rangle_{\cX|A} \in \LieX^*.
\end{equation}
For the previous computation in \eqref{eq:LievsLiedual0} we have used the identities
\begin{equation}\label{eq:lambda^2-F_A}
\begin{split}
(n-1)! \Lambda_\omega F_A \vol_\omega & = F_A \wedge \omega^{n-1}\\
(n-2)! \Lambda^2_\omega (F_A \wedge F_A) \vol_\omega & = 2(F_A \wedge F_A)\wedge \omega^{n-2}.
\end{split}
\end{equation}
The $\cX$-equivariance of the r.h.s. of \eqref{eq:LievsLiedual} is justified by the equality $\zeta(gA) = g \cdot
\zeta_A$, for any $A \in \cA$, $g
\in \cX$.

\section{The coupled equations as a moment map condition}
\label{sec:Ceqcoupled-equations}

As in the previous section, let $X$ be a compact symplectic manifold of dimension $2n$, with symplectic form $\omega$,
$G$ be a compact Lie group and
$E$ be a smooth principal $G$-bundle on $X$. Let $\cJ$ be the space of almost complex structures compatible with
$\omega$ and $\cA$ be the space of
connections on $E$. Using the symplectic structures on $\cJ$ and $\cA$ induced by $\omega$ (see~\eqref{eq:SympJ}
and~\eqref{eq:SymfC}), we can define
a symplectic form on the product $\cJ \times \cA$, for each pair of non-zero real constants $\alpha = (\alpha_0,
\alpha_1)$, given by
\begin{equation}
\label{eq:Sympfamily}
\omega_\alpha = \alpha_0 \cdot \omega_\cJ + \frac{4 \alpha_1}{(n-1)!} \cdot \omega_\cA.
\end{equation}
The extended gauge group $\cX$ has a canonical action on $\cJ \times \cA$ and this action is symplectic for any $\omega_\alpha$. In the notation of~\secref{sec:Ceqcoupling-term}, this action is given by
\[
  g\cdot (J,A)=(\pr(g)\cdot J, g\cdot A), \text{ for $g\in\cX$ and
  $(J,A)\in \cJ \times \cA$},
\]
where $\pr$ is the map in the short exact sequence~\eqref{eq:coupling-term-moment-map-1}. Let $z \in \mathfrak{g}$ be an element in the centre of $\mathfrak{g}$. We define the real constant
\begin{equation}\label{eq:constant-c}
c_z = \alpha_0 \hat{S} + \alpha_1 \(\frac{2\hat c}{(n-2)!} - 4 |z|^2\),
\end{equation}
where $\hat{S}$ is as in~\eqref{eq:hat-S} and $\hat{c}$ is given by~\eqref{eq:constant-c-hat}. Combining the moment map for the $\cH$-action on $\cJ$ (see~\S\ref{subsec:cscKmmap}), Proposition~\ref{prop:momentmap-X} and equation
\eqref{eq:LievsLiedual0} we obtain the following.
\begin{proposition}
\label{prop:momentmap-pairs} The $\cX$-action on $\cJ\times \cA$ is hamiltonian with respect to \eqref{eq:Sympfamily},
with $\cX$-equivariant moment map $\mu_{\alpha}\colon \cJ \times \cA\to (\LieX)^*$ given by
\begin{equation} \label{eq:thm-muX}
\begin{split}
\langle \mu_{\alpha}(J,A),\zeta\rangle & = 4\alpha_1\int_X (\theta_A\zeta\wedge(\Lambda_\omega F_A-z)) \cdot \vol_{\omega}\\
&- \int_X  \phi(\alpha_0 S_J + \alpha_1 \Lambda^2_\omega(F_A \wedge F_A) - 4\alpha_1(\Lambda_\omega F_A\wedge z) - c_z)\vol_{\omega}
\end{split}
\end{equation}
for any $\zeta\in\LieX$ covering $\eta_\phi \in \LieH$.
\end{proposition}
\begin{proof}
Simply note that 
$$
\int_X ((\theta_A\zeta - \phi\Lambda_\omega F_A)\wedge z)\vol_{\omega}
$$
is constant over $\cA$.
\end{proof}

The space $\cJ \times \cA$ has a complex structure preserved by the $\cX$-action given by
\begin{equation}
\label{eq:complexstructureI}
\mathbf{I}_{\mid(J,A)}(\dot{J},a) = (J\dot{J},-a(J \cdot)); \quad (\dot{J},a) \in T_J\cJ\times T_A\cA.
\end{equation}
The projection $\cJ \times \cA \to \cJ$ becomes now a holomorphic submersion with respect to the complex structure in
$\cJ$ and it turns out that, for $\alpha_0, \alpha_1$ positive, this complex structure is compatible with the family of symplectic structures \eqref{eq:Sympfamily}. This endows $\cJ\times \cA$ with a structure of K\"{a}hler fibration, i.e. a holomorphic submersion with a K\"{a}hler metric $\omega_\alpha$ that is the sum of the closed form $\frac{4 \alpha}{(n-1)!} \cdot \omega_\cA$ that restricts to a K\"{a}hler metric in the fibers and the pull-back of a
K\"{a}hler metric in the base $\alpha_0 \cdot \omega_\cJ$. The formal integrability of the almost complex structure $\mathbf{I}$ is a priori not obvious
and we give a proof in the following proposition. By formal integrability we mean, as in \cite{D6}, that the associated
Nijenhuis tensor vanishes.
\begin{proposition}\label{prop:CeqintegrableI}
The almost complex structure $\mathbf{I}$ is integrable.
\end{proposition}
\begin{proof}
For this we regard $\cJ \times \cA$ as an almost complex fibration over $\cJ$ with respect to the complex structure
introduced in \S\ref{subsec:cscKmmap}. The complex structure on the base and the one on each fibre are integrable, and
hence the integrability condition of $\mathbf{I}$ reduces to evaluate the Nijenhuis tensor $N_{\mathbf{I}}$ on a pair
of vectors $\dot{J} \in T_J\cJ$, $a \in T_A\cA$ for each $(J,A)$ in the product $\cJ\times\cA$. Since $\cA$ is an
affine space we can regard $a$ as a constant vector field on $\cA$ and so on the product $\cJ\times\cA$. We can also
extend $\dot{J}$ to a vector field on $\cJ$ taking
$$
\dot{J}_{|J'} = \frac{1}{2}(J\dot{J}J' - J'J\dot{J}).
$$
Using this extension it make sense to write
\begin{align*}
N_{\mathbf{I}}(\dot{J},a) & = [\mathbf{I}\dot{J},\mathbf{I}a] - \mathbf{I}[\mathbf{I}\dot{J},a] -
\mathbf{I}[\dot{J},\mathbf{I}a] - [\dot{J},a]\\
& = [\mathbf{I}\dot{J},\mathbf{I}a] - \mathbf{I}[\dot{J},\mathbf{I}a],
\end{align*}
where the brackets denote the Lie bracket between vector fields in $\cJ\times\cA$. We have used that
$[\mathbf{I}\dot{J},a] = [\dot{J},a] = 0$, since the flow of $a$ is the identity on $\cJ$. Given any vector field
$\dot{J}$ on $\cJ$, regarded as a vector field on $\cJ \times \cA$, we denote its flow by $J_t^{\dot{J}}$. Note that
$J_t^{\dot{J}}$ induces the identity on $\cA$ and so
\begin{align*}
N_{\mathbf{I}}(\dot{J},a) & = \frac{d}{dt}_{|t = 0} \mathbf{I}a_{|J_t^{\mathbf{I}\dot{J}}} - \mathbf{I}_{|J}
\frac{d}{dt}_{|t = 0} \mathbf{I} a_{|J_t^{\dot{J}}}\\
& = - \frac{d}{dt}_{|t = 0} a(J_t^{\mathbf{I}\dot{J}} \cdot) + \mathbf{I}_{|J} \frac{d}{dt}_{|t = 0}a(J_t^{\dot{J}}
\cdot)\\
& = - a(J\dot{J} \cdot) - a(\dot{J}J \cdot)\\
& = - a(J\dot{J} + \dot{J}J \cdot) = 0,
\end{align*}
where $a$ is regarded as an element of $\Omega^1(\ad E)$. 
Note that the integrability of $\mathbf{I}$ follows without any assumption on the compatibility of $\dot{J} \in T_J\cJ$
with the symplectic structure.
\end{proof}
\begin{remark}
Note that there is another $\cX$-invariant complex structures in $\cJ \times \cA$ given by
$\mathbf{I}_{+}(\dot{J},a) = (J\dot{J}, a(J \cdot))$,
such that the projection onto $\cJ$ is pseudoholomorphic. However, it follows from the proof of the proposition above
that this almost complex
structure is not integrable and corresponds to a $\cX$-invariant almost K\"{a}hler structure for coupling constants
$\alpha_0 > 0$, $\alpha_1 < 0$ in
\eqref{eq:Sympfamily}.
\end{remark}

Suppose now that there exist K\"{a}hler structures on $X$ with K\"{a}hler form $\omega$. In the notation
of~\secref{subsec:cscKmmap}, this means that
the subspace $\cJi \subset \cJ$ of integrable almost complex structures compatible with $\omega$ is non-empty. For
each
$J\in \cJi$, let
$\cA^{1,1}_J\subset\cA$ be the subspace of connections $A$ such that $F_A \in \Omega_J^{1,1}(\ad E)$
(cf.~\S\ref{subsec:HYMmmap}). Then we have a
$\cX$-invariant (maybe singular) complex submanifold
\begin{equation}
\label{eq:cP}
  \cP\subset \cJ\times \cA,
\end{equation}
consisting of pairs $(J,A)$ with $J\in \cJi$ and $A\in \cA^{1,1}_J$. Since $\cP$ is complex, it inherits a K\"{a}hler
form by restriction and by $\cX$-invariance, $\mu_\alpha$ restricts to a moment map on $\cP$. The following system of coupled equations
\begin{equation}
\label{eq:CYMeq} \left. \begin{array}{l}
\Lambda_\omega F_A = z\\
\alpha_0 S_J \; + \; \alpha_1 \Lambda^2_\omega (F_A \wedge F_A) = c
\end{array}\right \}
\end{equation}
is the condition satisfied by the points in $\cP$ whose orbits are in the corresponding K\"{a}hler reduction
\begin{equation}\label{eq:symplecticreduccX}
\mu_\alpha^{-1}(0)/\cX.
\end{equation}
Here, $S_J$ is the scalar curvature determined by the metric $g_J = \omega(\cdot,J\cdot)$ and $c$ is the topological
constant $c_0$ given by \eqref{eq:constant-c}.

The space of pairs $\cP$ has a geometric interpretation in terms of holomorphic structures that we will use in
the rest of the thesis. Let $G^c$
be the complexification of $G$ with Lie algebra $\mathfrak{g}^c$ and consider the associated principal $G^c$-bundle
$E^c=E\times_G G^c$. Then $\cP$
parameterizes pairs consisting of a K\"{a}hler structure on $X$, with K\"{a}hler form $\omega$, together with a
structure of holomorphic principal
$G^c$-bundle on $E^c$ over the corresponding complex manifold $X$. Alternatively, $\cP$ is the space $G^c$-invariant
integrable structures on the
principal $G^c$-bundle $E^c$ inducing holomorphic structures on $X$ which are compatible with $\omega$. Fixing the
holomorphic structure on $E^c$ and
$X$, the coupled equations \eqref{eq:CYMeq} have another formulation: we can consider the system as equations
for pairs $(\omega,H)$, where $\omega$ is a K\"{a}hler metric on $X$ in a fixed K\"{a}hler class and $H$ is a smooth
$G$-reduction of the holomorphic $G^c$-bundle. The role of the connection $A$ in \eqref{eq:CYMeq} is played now by the
Chern connection of $H$.

In order to interpret $\cP$ as holomorphic structures on $E^c$, we interpret first the bigger space $\cJ \times \cA$ in terms of $G^c$-invariant almost complex structures on $\Tot(E^c)$, the total space of $E^c$. The $G^c$-invariant complex structure on $\mathfrak{g}^c$ induces an almost complex structure $\imag$ on the vertical bundle $VE^c \subset TE$, via the isomorphism $VE^c \cong \ad E^c$. We have a map from $\cJ \times \cA$, that assigns to each pair $(J,A)$ the $G^c$-invariant complex structure
\begin{equation}
\label{eq:pairsinvacs}
(J,A) \mapsto I_{J,A} \defeq \imag A + A^{\perp}J \pi,
\end{equation}
on $\Tot(E^c)$, where $\pi\colon E^c \to X$ is the projection and
$$
A\colon TE^c \to VE^c \qquad A^{\perp}\colon\pi^*TX \to
TE^c
$$
are the maps induced by the connection $A$ (cf. \eqref{eq:principal-bundle-ses}). On the other hand, any $G^c$-invariant almost complex structure $I$ on $\Tot(E^c)$ that
preserves $VE^c$ induces a unique pair of almost complex structures $I_V$ on $VE^c$ and $I_X$ on
$TX$. We claim that the image of
\eqref{eq:pairsinvacs} correspond to the space $\cI$ of $G^c$-invariant almost complex structures $I$ on $\Tot(E^c)$ that
preserve $VE^c$ and such that $I_V = \imag$ and $I_X \in \cJ$. By definition, given $I \in \cI$ it induces a
unique element $J_I \in \cJ$. Moreover, for such an $I$ there exists a unique $G$-connection $A \in \cA$ such that $I = I_{J,A}$ (see \cite{KNII}, Theorem $1.1$ p. $178$ and the remark on p. $185$), so the map
\begin{equation}\label{eq:pairsinvacs2}
\cJ \times \cA \to \cI
\end{equation}
given by \eqref{eq:pairsinvacs} is a bijection. Note that \eqref{eq:pairsinvacs2} commutes with the projections of both spaces onto $\cJ$ and it is $\cX$-equivariant, where the action of $\cX$ on $\cI$ is given by pushforward.
Moreover, \eqref{eq:pairsinvacs2} is a biholomorphisms with respect to the complex structure $\mathbf{I}$ on $\cJ \times \cA$ and the complex structure on $\cI$ given by
$$
\mathbf{I}_{\mid I}B = IB
$$
for any $I \in \cI$, where we
regard $B \in T_I\cI$ as a ($G^c$-invariant) endomorphism of $TE^c$. The integrability condition for an almost complex
structure $I \in \cI$ is
equivalent in terms of the associated pair $(J,A)$ to $J \in \cJi$ and $A \in \cA^{1,1}_J$, hence equivalently to the
condition $(J,A) \in \cP$. It
follows easily now from the naturality of the Nijenhuis tensor of any almost complex structure (see Proposition $4.11$
in \cite{McS}) that $\cP$ is a
$\cX$-invariant complex submanifold of $\cJ \times \cA$.

\begin{remark}
Alternatively, taking $z=0$ in~\eqref{eq:thm-muX}, the coupled equations \eqref{eq:CYMeq} can be defined by the condition (see~\eqref{eq:symplecticreduccX} and \eqref{eq:CYMeq})
$$
\mu_\alpha(p) \in \mathfrak{z} \subset \LieX,
$$
$p \in \cP$, provided by the identification given by the $\cX$-equivariant family of parings $\langle\cdot,\cdot\rangle_{\cX}$ given in \eqref{eq:pairingscX}. Given $p \in \cP$
let $\cX_p$ be its isotropy group
in $\cX$. This is a finite dimensional compact Lie group whose Lie algebra $\LieX_p$ contains $\mathfrak{z}$. When
$\mathfrak{z} \varsubsetneq
\LieX_p$ it makes sense to generalize the condition $\mu_\alpha(p) \in \mathfrak{z}$ (above) to
\begin{equation}\label{eq:extremalpair}
\mu_\alpha(p) \in \LieX_p
\end{equation}
obtaining new, weaker, equations. This new condition is the analogue for our situation of the well-known notion of
extremal K\"{a}hler metric
introduced by Calabi in \cite{Ca}. In fact the extremal pairs defined by \eqref{eq:extremalpair} are critical points
of the quadratic functional
\begin{equation}\label{eq:squaredmomentmap}
p = (J,A) \to \|\mu_\alpha(p)\|^2_{\cX|A},
\end{equation}
over the space $\cJ \times \cA^{HYM}$, where $\cA^{HYM}$ is the set of connections satisfying $\Lambda_\omega F_A = z$. This functional can be considered as the norm-squared of the moment map with respect to the family of metrics $\langle\cdot,\cdot\rangle_{\cX}$ indexed by $\cA$. To see this, given any $p \in \cP$ we define (cf. \eqref{eq:zetaA})
$$
\zeta_p = 4\alpha_1\Lambda_\omega F_A - \theta_A^{\perp}(\eta_p) \quad \in \quad \LieX,
$$
with coupling constants $\alpha_0, \alpha_1, c \in \RR$ as in \eqref{eq:CYMeq}, where $\eta_p$ is the Hamiltonian vector field associated to the smooth function
$$
\phi_p = \alpha_0 S_J \; + \; \alpha_1 \Lambda^2_\omega (F_A \wedge F_A) - c \in C_0^{\infty}(X).
$$
Then, a solution to \eqref{eq:extremalpair} is a critical point of \eqref{eq:squaredmomentmap} over $\cJ \times \cA^{HYM}$ since for any curve $p_t \subset \cJ \times \cA^{HYM}$ we have
$$
\frac{d}{dt} \langle \mu_\alpha(p_t),\zeta_t\rangle = 2 \omega_\alpha(Y_{\zeta_t},\dot{p}_t),
$$
where $Y_{\zeta_t}$ is the infinitesimal action of $\zeta_t = \zeta_{p_t}$.
\end{remark}

\section{Variational interpretation and canonical invariant metrics}
\label{sec:CeqscalarKal-K}

In this section we give a a variational interpretation of the coupled equations \eqref{eq:CYMeq}, as absolute minima of a suitable Calabi--Yang--Mills type functional. The unknown variables in the system \eqref{eq:CYMeq} are considered to be a K\"ahler structure on the base, with fixed symplectic form, and a connection on the bundle. We will also interpret the solutions as a choice of canonical invariant metrics on the total space of the bundle, linking with the classical Kaluza--Klein Theory for metrics and connections (see e.g.~\cite{CJ}). Throughout this section we suppose that the coupling constants $\alpha_0$ and $\alpha_1$ in \eqref{eq:CYMeq} have positive ratio
$$
\frac{\alpha_1}{\alpha_0} > 0.
$$
This will be a crucial assumption for our arguments. For simplicity, we consider that the compact Lie group $G$, with Lie algebra $\mathfrak{g}$, is semisimple. As in the previous section, we fix a metric $(\cdot,\cdot)$ on $\mathfrak{g}$ invariant by the adjoint action.

Let $(X,\omega)$ be a compact symplectic manifold and consider the space $\cJ$ of $\omega$-compatible almost complex structures on $X$. We suppose that $(X,\omega)$ is K\"ahlerian, i.e. that the subspace $\cJi \subset \cJ$ of integrable complex structures is non empty. Let $E$ be a smooth principal $G$-bundle over $X$. Given any $J \in \cJ$ we consider the associated Riemannian metric
\begin{equation}\label{eq:gJ}
J \to g_J = \omega(\cdot,J\cdot).
\end{equation}
Since $\omega$ is fixed we can recover $J$ from $g_J$, and so we will denote the latter simply by $g$. Given a pair $(g,A)$, where $g$ is given by \eqref{eq:gJ} and $A \in \cA$ is a connection on $E$, we consider the functional
\begin{equation}\label{eq:CYMfunctional}
CYM(g,A) = \int_X (\alpha_0 S_g - 2\alpha_1 |F_A|^2)^2 \cdot \vol_g \; + \;  2\alpha_1 \cdot \|F_A\|^2,
\end{equation}
where $\vol_g$ is the volume form of the metric $g$. The term $\|F_A\|^2$ is the Yang--Mills functional, considered as a functional in both the metric and the connection, and the function $|F_A|^2$ is the point-wise norm with respect to metric $g$ and the fixed inner product on $\mathfrak{g}$. Note that \eqref{eq:CYMfunctional} is well defined for an arbitrary Riemannian metric in the base. However, only when considered as a functional for the pairs $(g,A)$, as before, we will obtain that our solutions are absolute minima of the functional. Let $c \in \RR$, $\hat c$ and $\hat S$ be the real constant defined, respectively,  in \eqref{eq:constant-c}, \eqref{eq:constant-c-hat} and \eqref{eq:hat-S}. Recall that $c$, involved in the definition of the coupled system \eqref{eq:CYMeq}, equals
$$
c = \alpha_0 \cdot \hat{S} + \alpha_1 \cdot \frac{2\hat c}{(n-2)!},
$$
where $\hat S$ and $\hat c$ only depend on the topology of $X$ and $E$. Then we have the following.
\begin{proposition}\label{propo:Variational}
Suppose that $\alpha_0$ and $\alpha_1$ are positive and small enough such that $2 c < 1$. Then any pair $(g_J,A)$ satisfying the coupled equations \eqref{eq:CYMeq} and the compatibility condition $F_A^{0,2} = 0$ is an absolute minimum of $CYM$.
\end{proposition}
\begin{proof}
For any pair $(g,A)$ as in \eqref{eq:CYMfunctional}, we have (cf. e.g. proof of Lemma $7.9$ in \cite{MR2})
\begin{equation}\label{eq:identitysquaredFA}
\begin{split}
\frac{1}{2}\Lambda^2_\omega(F_A\wedge F_A) & = - |F_A|^2 + |\Lambda_\omega F_A|^2 + 4|F_A^{0,2}|^2\\
\end{split}
\end{equation}
where the point-wise norms are taken with respect to the metric $g$ and the inner product on $\mathfrak{g}$. Then,
\begin{align*}
CYM(g,A) & = \int_X (\alpha_0 S_g - 2\alpha_1 |F_A|^2 - c)^2 \cdot \vol_g + \;  2\alpha_1 \cdot \|F_A\|^2\\
& + 2 c \cdot \int_X (\alpha_0 S_g - 2\alpha_1 |F_A|^2) \cdot \vol_g \; + \; c^2 \cdot \Vol([\omega])\\
& = \|\alpha_0 S_g - 2\alpha_1 |F_A|^2 - c\|^2\\
& + 2\alpha_1(1 - 2c) \cdot (\|\Lambda_\omega F_A\|^2 + 4 \|F_A^{0,2}\|^2 )\\
& + (2 c \cdot \alpha_0 \hat S + c^2 + \alpha_1 (1 - 2c)\cdot \frac{2 \hat c}{(n-2)!}) \cdot \Vol([\omega]),
\end{align*}
where $\Vol([\omega]) = \frac{[\omega]^n}{n!}$, that only depends on the fixed symplectic form. The statement follows now from \eqref{eq:identitysquaredFA}.
\end{proof}
\begin{remark}\label{remark:Variational}
Note that the condition $2 c < 1$ imposes no restriction since we can always multiply the scalar equation in the system \eqref{eq:CYMeq} by a small positive constant. Note also that $F_A^{0,2} = 0$ is the condition for the pair $(J,A)$ to be in the space of integrable pairs $\cP$ \eqref{eq:cP}, provided that $J \in \cJi$. An analogue result, as in Proposition~\ref{propo:Variational}, is obtained if we fix a complex structure on $X$ and let $\omega$ vary in a fixed K\"ahler class.
\end{remark}

In order to understand better this functional we think in the data $(g,A)$ as defining a $G$-invariant metric $\hat g$ on $\Tot(E)$, the total space of $E$. Let $\pi \colon E \to X$ be the standard projection. The biinvariant positive definite inner product $(\cdot,\cdot)$ on $\mathfrak{g}$ induces a Riemannian metric $g_V$ on the vertical bundle $VE$, via its identification with $\ad E$. Given a connection $A$ on $E$ and a Riemannian metric $g$ on $X$ we can associate a $G$-invariant metric $\hat{g}$ on $\Tot(E)$ given by
$$
\hat g = \pi^*g + g_V(\theta_A \cdot , \theta_A \cdot),
$$
where $\theta_A \colon TE \to VE$ denotes the vertical projection determined by $A$. Endowed with this metric the
projection $\pi\colon (E,\hat{g}) \to (X,g)$ is a Riemannian submersion with totally geodesic fibres (see~\cite[Theorem~9.59]{Be}). The scalar
curvature $\hat{S}$ of $\hat{g}$ is given by (see~\cite[Proposition~9.70]{Be})
\begin{equation}
\label{eq:scalarKal-K}
S_{\hat g} = S_g - |F_A|^2 + s,
\end{equation}
where $S_g$ is the scalar curvature of $g$, $|F_A|^2$ is as in \eqref{eq:identitysquaredFA} and $s$ is a real constant that corresponds to the scalar curvature of the metric in the fibres. Note that, since $\hat g$ is $G$-invariant, $S_{\hat g}$ induces a well defined function on $X$. We can scale the metric in the fibres multiplying by a constant factor $t > 0$ and then the associated metric $\hat{g}_t$ has scalar curvature  (see~\cite[Proposition~9.70]{Be})
$$
\hat{S}_t = S_g - t|F_A|^2 + \frac{1}{t}s.
$$
We set $t = \frac{2\alpha_1}{\alpha_0} > 0$, where the coupling constants are as in \eqref{eq:CYMfunctional}. Then, the first summand in the definition of the functional \eqref{eq:CYMfunctional} can though of, up to scaling, as a Calabi functional for $G$-invariant metrics in $\Tot(E)$. This follows from the invariance of the scalar curvature of a invariant metric on $E$, which implies that the integration on the bundle reduces to an integration on the base (see~\cite{CJ}).

The previous expression for the scalar curvature \eqref{eq:scalarKal-K} of the metric $\hat g$ links with our next interpretation of the equations \eqref{eq:CYMeq}: they define a `good' choice of canonical $G$-invariant metrics on $\Tot(E)$, in the sense that a moduli construction for such metrics would lead to a K\"ahler moduli space. Let $\hat g$ be the $G$-invariant metric on $\Tot(E)$ determined by a pair $(g,A)$, where $g$ is given by \eqref{eq:gJ} and $A \in \cA$ is a connection on $E$. We fix the scale on the fibers
\begin{equation}\label{eq:scale}
t = \frac{2\alpha_1}{\alpha_0} > 0
\end{equation}
as before. We assume further that the connection $A$ is irreducible, $J \in \cJi$ (in the definition of $g$), and that
$$
F_A^{0,2} = 0.
$$
Therefore, $(g,J,\omega)$ is a K\"ahler structure on $X$ and $A$ determines a holomorphic structure on the $G^c$-bundle $E^c = E \times_G G^c$. In order to understand the conditions imposed by the system \eqref{eq:CYMeq} to the metric $\hat g$, we will first interpret the HYM condition \eqref{eq:HYM} in terms of the Einstein equation for $\hat g$. Using the irreducibility of the connection, it follows from the K\"{a}hler identities that the HYM equation is equivalent to the a priori weaker Yang-Mills equation
$$
d_A^* F_A = 0.
$$
For the case $G = U(r)$ see~\cite[Proposition~3]{D3}. In general it follows by taking a faithful representation of the complexification $G^c$ in $GL(r,\CC)$. The isomorphism
$$
TE \cong VE \oplus \pi^*TX,
$$
provided by the connection $A$, splits the Ricci tensor $R_{\hat g}$ and so the Einstein equation can be written as a system with respect to this splitting. Let $R_{hv}$ be the horizontal-vertical part of
$R_{\hat g}$. It follows now from~\cite[Proposition~9.61]{Be} that the Yang-Mills equation for $\hat g$ is equivalent to the vertical-horizontal Einstein equation $R_{hv} = 0$. On the other hand, note that the scalar equation in the coupled system \eqref{eq:CYMeq} can be written as
\[
\alpha_0 S_g \; - \; 2 \alpha_1 \cdot |F_A|^2 = c,
\]
using \eqref{eq:identitysquaredFA}. Hence, it follows from \eqref{eq:scalarKal-K} that: the tuple $(g,\omega,J,A)$ satisfies the coupled equations with coupling constants $\alpha_0$ and $\alpha_1$ if and only if $\hat g$, with scale \eqref{eq:scale} on the fibers, has constant scalar curvature and satisfies the vertical-horizontal Einstein equation $R_{hv} = 0$. Since any Einstein metric has constant scalar curvature it follows from \eqref{eq:CEYMeq} that
$$
R_{\hat g} = \lambda \hat g \quad \Rightarrow  \quad \left. \begin{array}{l}
\Lambda F_A = 0\\
\alpha_0 S_g \; + \; \alpha_1 \Lambda^2(F_A \wedge F_A) = c
\end{array}\right \}  \quad \Rightarrow \quad \quad S_{\hat g} = c,
$$
where $\lambda$ and $c$ are real constants. A crucial ingredient for this chain of implications is the compatibility condition $F_A^{0,2} = 0$, between the K\"ahler structure $(g,\omega,J)$ and the connection $A$. In the language of the moment map interpretation given in \secref{sec:Ceqcoupled-equations}, this condition implies that $(J,A) \in \cP$ \eqref{eq:pairsinvacs0}. An unexpected point here is that the positivity of the ratio $\frac{\alpha_1}{\alpha_0} > 0$, that we have assumed to build our Riemannian metrics $\hat g$ from the data $(g,A)$, is precisely the condition for the symplectic form $\omega_\alpha$ \eqref{eq:Sympfamily0} in $\cP$ to be K\"ahler. Thus, a moduli construction for such metrics using the symplectic reduction process will lead to a K\"ahler moduli space. In this sense, the coupled equations \eqref{eq:CYMeq00} provide a good choice of canonical $G$-invariant metrics in $\Tot(E)$, well suited for the framework of K\"ahler geometry.

It is natural to ask which coupled system of equations is equivalent to the Einstein condition for the metric $\hat g$. From~\cite[Proposition~9.64]{Be} it follows that the Einstein equation for such a metric (with irreducible connection) and arbitrary compact group $G$ reads
\begin{equation}
\label{eq:CEYMeq} \left. \begin{array}{l}
\Lambda_{\omega} F_A = z\\
\alpha_0 (\rho - c'\omega) = \alpha_1 (2(\Lambda_{\omega} F_A,F_A) - \Lambda_\omega (F_A \wedge F_A) - c'' \omega)
\end{array}\right \},
\end{equation}
where $z$ is an element in the centre of $\mathfrak{g}$ and $c'$, $c'' \in \RR$ are suitable constants. The $2$-form $\rho$ denotes the Ricci form of the K\"{a}hler structure on the base.
\begin{remark}
The system \eqref{eq:CEYMeq} can be considered a K\"{a}hler analogue of the Einstein-Yang-Mills equation considered in the pseudoRiemannian context (see e.g. \cite{CJ}). Note that in complex coordinates we have that
$$
2(\Lambda_{\omega} F_A,F_A) - \Lambda_\omega (F_A \wedge F_A) = - \imag \tr g^{\alpha,\overline{\beta}}(F_{\alpha,\overline{\delta}}F_{\gamma,\overline{\beta}} + F_{\gamma,\overline{\beta}}F_{\alpha,\overline{\delta}}) \cdot dz_{\gamma} \wedge d\overline{z}_{\delta},
$$
where
$$
\omega = \imag g_{\alpha,\overline{\beta}}\cdot dz_{\alpha} \wedge d\overline{z}_{\overline{\beta}} \qquad \textrm{and} \qquad F_A = F_{\alpha,\overline{\beta}}\cdot dz_{\alpha} \wedge d\overline{z}_{\overline{\beta}}.
$$
\end{remark}
\begin{remark}
The Einstein condition for the metric $\hat g$ implies strong restrictions on the tuple $(g,\omega,J,A)$. To see this, suppose that $\hat g$ is an Einstein metric. Then, since the Ricci form $\rho$ and $\omega$ in \eqref{eq:CEYMeq} are closed so $\Lambda_\omega (F_A \wedge F_A)$ is also closed. It follows then from the K\"{a}hler identities that the $4$-form $(F_A \wedge F_A)$ is harmonic. From this one has that $\Lambda^2(F_A \wedge F_A)$ is constant. Then, taking traces in the second equation of \eqref{eq:CEYMeq} it follows that the K\"{a}hler metric on the base $g$ has constant scalar curvature.
\end{remark}

\chapter{\bf Analytic obstructions and existence of solutions}
\label{chap:analytic}

Let $(X,J)$ be a compact complex manifold, $G^c$ be a complex Lie group and $(E^c,I)$ a holomorphic principal
$G^c$-bundle over $(X,J)$. We fix a maximal compact subgroup $G \subset G^c$. Then, we can consider the coupled
equations \eqref{eq:CYMeq00} as equations in the space of pairs $(\omega,H) \in \cKt$, consisting of a K\"{a}hler metric $\omega$ in a fixed K\"{a}hler class and a reduction $H$ of $E^c$ to a principal $G$-bundle. Thus, for positive real coupling constants $\alpha_0, \alpha_1$ the \emph{coupled equations} can be rewritten as
\begin{equation}
\label{eq:CYMeq2} \left. \begin{array}{l}
\Lambda_\omega F_H = z\\
\alpha_0 S_{\omega} \; + \; \alpha_1 \Lambda^2_\omega (F_H \wedge F_H) = c
\end{array}\right \}.
\end{equation}
Here, $S_{\omega}$ is the scalar curvature determined by the K\"{a}hler metric, $F_H$ is the curvature of the Chern
connection determined by the reduction $H$ and the complex structure $I$, and $z \in \mathfrak{z}$ and $c \in \RR$ are
as in \eqref{eq:CYMeq}. In this chapter we give an existence criterion for the existence of solutions to the coupled
system \eqref{eq:CYMeq2} for small values of the ratio of the coupling constants $\frac{\alpha_1}{\alpha_0} > 0$ (Theorem~\ref{thm:ExistenceCYMeq2}). This criterion is formulated in terms of a real subgroup $\cX_I$ of the automorphism group $\Aut (E^c,I)$ of the holomorphic bundle. We will also give an obstruction (Proposition~\eqref{prop:futakiobstruction}) controlled by a character
$$
\cF_\alpha \colon \Lie\Aut (E^c,I) \to \CC
$$
that generalizes the usual Futaki invariant for the cscK problem \cite{Ft0} and also other invariants defined by A. Futaki in \cite{Ft1} (see~\secref{sec:ANIntegralinvariants}). In \secref{sec:ANcoupledintmmap} we construct the integral of the moment map
$$
\cM_\kappa \colon \cKt \to \RR
$$
and determine a sufficient condition for its convexity and existence of lower bounds in terms of the existence of suitable smooth curves on $\cKt$ joining any two points. In \secref{sec:ANcomplexsurfaces} we restrict to the case in which the base is a compact complex surface. We use the integral of the moment map to prove some results concerning the uniqueness and the existence problem for the coupled equations \eqref{eq:CYMeq2}.

To obtain these results, we first discuss in \secref{sec:ANgeneral-framework} the general framework provided by Kempf-Ness Theorem where our problem fits and establish its the relation with~\secref{sec:Ceqcoupled-equations}: the pairs $(\omega,H)$ satisfying \eqref{eq:CYMeq2} correspond to zeros of the moment map $\mu_\alpha$ \eqref{eq:thm-muX} in the ``complexified orbit'' of the $G^c$-invariant complex structure $I$ of $E^c$ with respect to the action of the extended gauge group.

Throughout this chapter, we suppose that the $G$-invariant metric $(\cdot,\cdot)$ in $\mathfrak{g}$ used to define the system \eqref{eq:CYMeq2} extends to a $G^c$-invariant (maybe degenerate) bilinear pairing
\begin{equation}\label{eq:pairinggc}
(\cdot,\cdot):\mathfrak{g}^c \otimes \mathfrak{g}^c \to \CC.
\end{equation}
Note that the pullback of minus the trace with respect to any faithful representation of $G^c$ in $GL(r,\CC)$ mapping
$G$ into $U(r)$ always satisfies this property.

\section{General framework}
\label{sec:ANgeneral-framework}

In this section we show how the problem of finding solutions to the coupled equations \eqref{eq:CYMeq2} fits into the
general framework provided by the Kemp--Ness Theorem, already explained in \secref{sec:KNtheorem}. Our argument combines the one for HYM connections explained in \secref{sec:HYM} with the one given by Donaldson for the cscK problem in \cite{D6} (see~\secref{subsec:cscKorbits}). We will study the system \eqref{eq:CYMeq2} with fixed positive coupling constants $\alpha_0$, $\alpha_1 \in \RR$.

Let $(X,J)$ be a compact complex manifold, where we fix a K\"{a}hler class $[\omega]$, and denote by $\cK$ the space of K\"{a}hler metrics on $X$ with this K\"ahler class. Let $(E^c,I)$ be a holomorphic principal $G^c$-bundle over $(X,J)$ and denote by $\cR$ the space of smooth $G$-reductions on the underlying smooth bundle $E^c$. An element $H \in \cR$ will be regarded, depending on the context, as a smooth cross section of the bundle $E^c/G \to X$ or as a $G$-subbundle of $E^c$. Let us consider the product space
\begin{equation}\label{eq:Ktilda}
\cKt = \cK \times \cR.
\end{equation}
The elements in $\cKt$ will be denoted either by $(\omega,H)$ or $\kappa$, when there is no confusion. First we will establish the relation with~\secref{sec:Ceqcoupled-equations}, showing that the pairs $(\omega,H)$ correspond to points in the ``complexified orbit'' of the holomorphic $G^c$-invariant structure $I$ on $E^c$ for the action of the extended gauge group. Those pairs satisfying the coupled equations \eqref{eq:CYMeq2} will be given by the zero locus of the moment map $\mu_\alpha$ \eqref{eq:thm-muX} inside the complexified orbit. It is convenient to consider the following definition (cf.~\cite[Section~1.1]{MMOPR}).

\begin{definition}\label{def:ANmarkedHamiltonian}
A \emph{marked K\"{a}hler Hamiltonian $\cX$-space} (mK-Hamiltonian space, for short) is a  tuple $(\cP,\mathbf{I},\omega_\alpha,\cX,\mu_\alpha,p)$, where $(\cP,\mathbf{I},\omega_\alpha)$ is a K\"{a}hler manifold, $\cX$ is a Lie group that acts on $\cP$ by holomorphic isometries,
$$
\mu_\alpha: \cP \to \LieX^*
$$
is a $\cX$-equivariant moment map and $p$ is a marked point in $\cP$.
\end{definition}

Relying on the results in \secref{sec:Ceqcoupled-equations}, any $\kappa = (\omega,H) \in \cKt$ determines a mK-Hamiltonian space
\begin{equation}\label{eq:ANmarkedHamiltonian}
(\cP,\mathbf{I},\omega_\alpha,\cX,\mu_\alpha,p_I),
\end{equation}
where $\cP \subseteq \cJi \times \cA$ (see~\eqref{eq:cP}) is given by pairs $p = (J',A')$ such that $J' \in \cJi$, the space of compatible complex structures on $(X,\omega)$, and $A'$ is a compatible connection with $J'$ on the $G$-reduction $H$, i.e such that $F_A^{0,2} = 0$. Recall that $\cP$ parameterizes $G^c$-invariant complex structures on $E^c$ that cover elements in $\cJi$ (see~\eqref{eq:pairsinvacs}). Therefore, the marked point $p_I = (J,A) \in \cP$ is determined by the complex structure $I$, where $A$ denotes the Chern connection of $H$ on $(E^c,I)$. The extended gauge group $\cX$ fits into the short exact sequence
$$
1 \to \cH \to \cX \to \cG \to 1,
$$
where $\cH$ and $\cG$ denote respectively the Hamiltonian group of symplectomorphisms of $(X,\omega)$ and the gauge group of the reduction $H$. The symplectic structure, the moment map and the complex structure in \eqref{eq:ANmarkedHamiltonian} are given respectively by \eqref{eq:Sympfamily}, \eqref{eq:thm-muX} and \eqref{eq:complexstructureI}, where the two former are determined by the fixed coupling constants $\alpha_0$ and $\alpha_1$.

Let us fix $\kappa \in \cKt$ and consider the associated mK-Hamiltonian space \eqref{eq:ANmarkedHamiltonian}. As $\cH$, the group $\cX$ does not admit a complexification extending the action on $\cP$. However, there is an integrable distribution $D \subseteq T\cP$ whose leaves play the role of the orbits under the action of an ideal complexification. Recall that given $\zeta \in \LieX$, we denote by $Y_\zeta$ the infinitesimal action of $\zeta$ on $\cP$. Then, at a point $p \in \cP$ the distribution is given by
\begin{equation}\label{eq:distribution}
D_p = \{Y_{\zeta|p}; \; \zeta \in \LieX^c\},
\end{equation}
where $\LieX^c$ is the complexification of $\LieX$ and
\begin{equation}
\label{eq:notationcominfinitesimal} Y_\zeta = Y_{\zeta_0} + \mathbf{I} Y_{\zeta_1} \text{ for all $\zeta = \zeta_0 +
\imag \zeta_1 \in \LieX^c$}.
\end{equation}
The distribution $D$ is closed under the Lie bracket, since the Nijenhuis tensor of the complex structure $\mathbf{I}$ vanishes (see Proposition~\ref{prop:CeqintegrableI}), although it is not obvious from this fact that it integrates, since $\cP$ is infinite dimensional. To see that $D$ indeed integrates, we define a principal $\cX$-bundle
\begin{equation}\label{eq:Ytildap}
\cYt \to \cKt
\end{equation}
that plays the role of the missed complexification, with $\cKt$ regarded as the symmetric space in \secref{sec:KNtheorem} (see Remark~\ref{remark:Ktildasymmetric} below). For simplicity, we suppose that the first cohomology group $H^{1}(X)$ vanishes, so the group $\cH$ coincides with $\textrm{Symp}_0(X,\omega)$, the identity component of the group of symplectomorphisms. The assumption $H^1(X) = 0$ can be easily removed in the discussion replacing $\cX$ by the bigger
group of smooth $G$-invariant automorphisms of the reduction $H$ that project to $\textrm{Symp}_0(X,\omega)$. To define $\cYt$, let us consider $\Aut_0(E^c)$, the connected component of $\Id$ in the group $\Aut(E^c)$ of $G^c$-equivariant diffeomorphisms of $E^c$. Note that any element $f\in \Aut_0(E^c)$ covers a
unique diffeomorphism $\check{f}$ in $\Diff_0 X$. Then, the $\cX$-bundle $\cYt$ determined by $\kappa = (\omega,H)$ consists of triples
\begin{equation}\label{eq:ANcYtdef}
\cYt \defeq \{(\omega',H',f) \in \cKt \times \Aut_0(E^c) \colon (\check{f}^*\omega',f^*H') = (\omega,H)\}.
\end{equation}
The reduction $H'$ is viewed here as a cross section of the bundle $E^c/G \to X$ and so for any point $x \in X$ we have $f^*H'(x) = f^{-1}(H'(\check{f}(x)))$. Given any $p \in \cP$, the map
\begin{equation}\label{eq:taup}
\tau_p\colon \cYt \to \cP\colon (\omega',H',f)\mapsto f^* p
\end{equation}
plays the role of the complexified action, in the sense that the image of this map is a leaf of $D$ through $p$. Let us consider the \emph{complexified orbit} $\tau_{p_I}(\cYt)$ of the marked point $p_I \in \cP$ determined by the complex structure $I$ on $E^c$. Given a pair in $\cKt$, we can associate a point in the complexified orbit of $p_I$ via the map \eqref{eq:taup}, modulo an indeterminacy coming from the extended gauge group $\cX$. Then, it can be readily checked that $\tau_{p_I}(\cYt)$ contains a zero of the $\cX$-equivariant moment map $\mu_\alpha$ (see~\eqref{eq:thm-muX}) if and only if there exists a pair in $\cKt$ satisfying the coupled equations \eqref{eq:CYMeq2}.

\begin{remark}\label{remark:Ktildasymmetric}
When the isotropy group of $p_I$ vanishes, the previous construction shows that $\cYt$ is an infinitesimal complexification of $\cX$ in the sense of \cite{D6}. This would endow $\cKt \cong \cYt/\cX$ with a structure of Riemannian symmetric space, applying~\cite[Proposition~2]{D6}, provided that $\LieX$ has a biinvariant metric. However such a metric is not known to exist by the author (cf.~\eqref{eq:pairingscX}). A priori, i.e. removing the fixed point in $\cKt$, there is a Riemannian submersion structure on the fibration
$$
\cKt = \cK \times \cR \to \cK,
$$
where the space of K\"{a}hler metrics $\cK$ is regarded as a symmetric space with the Mabuchi metric (see~\cite{Mab1}). Note that for each $\omega \in \cK$, the fiber $\cR \cong \Omega^0(X,E^c/G)$ has a structure of Riemannian symmetric space given by the one on $G^c/G$ and integration over $X$ with respect to $\omega$. Thus, the structure of Riemannian submersion on $\cKt$ has trivial connection and varying Riemannian metric on each fiber. It is very likely that this structure is in fact Riemannian symmetric, in the sense that it admits a Levi--Civita connection whose curvature is negative, covariantly constant and admits an explicit expression in terms of the Lie bracket on the Lie algebra of the extended gauge group (see~\cite[Proposition~2]{D6}). To make sense of the latter condition, note that for any element $\kappa = (\omega,H) \in \cKt$ there is a canonical isomorphism
\begin{equation}\label{eq:symmetricspace}
T_\kappa \cKt \cong \LieX
\end{equation}
via the identifications $T_\omega \cK \cong C^{\infty}(X)/\RR \cong \LieH$ and $T_H \cR \cong \imag \LieG$ and the vector space splitting $\LieX \cong \LieH \times \LieG$ provided by the Chern connection of $H$ on $(E^c,I)$. Here the groups $\cG$, $\cX$ and $\cH$ are those corresponding to the mK-Hamiltonian space \eqref{eq:ANmarkedHamiltonian} determined by $\kappa$.
\end{remark}


We will now construct a map that will be used in \secref{sec:ANdeformingsolutions} to deform a solution $\kappa \in \cKt$ of the coupled equations \eqref{eq:CYMeq2} on the attached complexified orbit $\tau_{p_I}(\cYt)$. This will provide an existence criterion for \eqref{eq:CYMeq2} and also canonical curves in $\cKt$ along which the integral of the moment map (see~\secref{sec:ANcoupledintmmap}), whose zeros locus correspond to the solutions of \eqref{eq:CYMeq2}, will be convex. It is very likely that this curves are, in fact, the geodesics on $\cKt$ with respect to the Riemannian structure discussed in Remark~\ref{remark:Ktildasymmetric}. Let us fix a point $\kappa \in \cKt$. We will construct a map
$$
\rho \colon T_\kappa \cKt \to \tau_{p_I}(\cYt) \subset \cP,
$$
where $p_I \in \cP$, $\tau_{p_I}$ and $\cYt$ are, as before, associated to the mK-Hamiltonian space \eqref{eq:ANmarkedHamiltonian} determined by $\kappa$. Note that in the finite dimensional picture provided by \secref{sec:KNtheorem} the complex exponential defines a canonical connection on the principal bundle
$$
G^c \to G^c/G \cong \imag \mathfrak{g},
$$
that can be used to map $G^c/G$ into the K\"{a}hler manifold where the complex group acts. Hence, via the identification $T_\kappa \cKt \cong \imag \LieX$ (see~\eqref{eq:symmetricspace}), the map $\rho$ can be considered as an analogue of the composition of the (missed) complex exponential with the action on $p_I$ of the (missed) complexification of $\cX$. The strategy we will follow is: first, we will define a connection on the $\cX$-bundle $\cYt \to \cKt$. Then, choosing a suitable family of curves in $\cKt$ joining any point to $\kappa$, we will lift the curves starting at the canonical element on the fiber of $\cYt$ over $\kappa$ given by the identity. Finally, using the map $\tau_{p_I}$ \eqref{eq:taup} we will map the lifted $\cKt$ into $\cP$.

To define the connection we define the horizontal lift of a smooth curve $\kappa_t$ in $\cKt$ passing through $f \in \cYt$ in the fiber over the initial point $\kappa_0$. In the following we will denote by a $t$-subscript the elements of the mK-Hamiltonian space determined by $\kappa_t$ \eqref{eq:ANmarkedHamiltonian}. Let $\kappa_t = (\omega_t,H_t)$ be a curve in $\cKt$. As in \secref{subsec:cscKorbits}, let $\check{f}_t \subset \Diff_0(X)$ be the Moser curve of diffeomorphisms of $\omega_t$ determined by the time dependent vector field $\check{y}_t$ with initial condition $\check{f}_0 = \check{f}$, the diffeomorphism covered by $f \in \Aut_0(E^c)$. Recall that $\check{f}_t$ satisfies the condition $\check{f}_t^* \omega_t = \omega_0$ for all $t$. We now
lift $\check{f}_t$ to a curve in the group $\Aut_0(E^c)$ with the desired property $f_t^*H_t = H$ (see~\eqref{eq:ANcYtdef}). Let $f_t \subset \Aut_0(E^c)$ be the
curve integrating
\begin{equation}\label{eq:vectorfieldyft}
y_t = -\frac{1}{2}\dot{H}_t + \theta_t^\perp(\check{y}_t)
\end{equation}
with initial condition $f_0 = f$, where $\dot{H}_t \in \imag\LieG_t$ is the $H_t$-symmetric vector field determined by $\frac{d}{dt}H_t$ and $\theta_t^\perp(\check{y}_t)$ is the horizontal lift of $\check{y}_t$ with respect to the (Chern) connection of $H_t$ on $(E^c,I)$.
\begin{proposition}\label{propo:ANlift}
The curve $f_t \subset \Aut_0(E^c)$ defined by \eqref{eq:vectorfieldyft} is a lift of $(\omega_t,H_t)$ to $\cYt$ passing
trough $f$. The induced curve $p_t = \tau_p(\omega_t,H_t,f_t) \subset \cP$ determined by \eqref{eq:taup} satisfies
\begin{equation}\label{eq:derivativept1}
\dot{p_t} = \mathbf{I} Y_{\zeta_t},
\end{equation}
where $Y_{\zeta_t}$ is the infinitesimal action at $p_t \in \cP$ of
\begin{equation}\label{eq:derivativept2}
\zeta_t = f_t^*(I y_t) \in \LieX.
\end{equation}
\end{proposition}
\begin{proof}
For the first part of the statement it is enough to check that $f_t^*H_t$ is constant on $t$. Note first that the
derivative of $f_t^*H$ for any $H \in \cR$ is given by
$$
\frac{d}{dt}f_t^*H = f_t^*(2\pi_1 \theta_H(y_t))
$$
where $\theta_H$ is the (Chern) connection $1$-form of $H$ on the holomorphic bundle $(E^c,I)$ and
$$
\pi_1 \colon \LieG^c \to \imag\LieG
$$
is the standard projection from the Lie algebra of the complex gauge group of $E^c$ onto the Lie algebra of the gauge group of the $G$-reduction $H$. When the
$G$-reduction is parameterized by $t$, we denote this map by $\pi_1^t$. Then we have
$$
\frac{d}{dt}f_t^*H_t = f_t^*(2\pi_1^t \theta_t(y_t) + \dot{H}_t) = f_t^*(- \dot{H}_t + \dot{H}_t) = 0.
$$
To prove \eqref{eq:derivativept1}, recall from \secref{sec:Ceqcoupled-equations} that the space $\cP$ is complex isomorphic to a suitable space of integrable $G^c$-invariant complex structures on $E^c$ and that this isomorphism is $\cX$-equivariant (see~\eqref{eq:pairsinvacs2}). Then, any element of the curve $p_t \in \cP$ determines a unique $G^c$-invariant complex structure $I_t = f_t^*I$ on $E^c$ and so
$$
\frac{d}{dt} I_t = L_{f_t^*y_t}I_t = - I_t L_{f_t^*(I y_t)}I_t = \mathbf{I} Y_{f_t^*(I y_t)}.
$$
Thus, $\zeta_t = f_t^*(I y_t)$ as required. Note that $Iy_t \in \cX_t$ implies that $\zeta_t \in \LieX$.
\end{proof}

The previous proposition shows that the connection we have just defined satisfies reasonable properties that bring us close to the map $\rho$ of our interest. Recall that we have fixed a pair $\kappa = (\omega,H)$. Then, given any function $\dot{\phi} \in C^{\infty}_0(X)$, with zero integral with respect to $\vol_\omega$, and any $\dot{H} \in \imag \LieG$, we identify $(\dot{\phi},\dot{H})$ with the vector field in the tangent space $T_\kappa \cKt$ given by the curve
\begin{equation}\label{eq:ANkappat}
\kappa_t = (\omega_t,H_t) \defeq (\omega_{t \dot{\phi}},(e^{t\dot{H}})^*H) \subset \cKt,
\end{equation}
where
$$
\omega_{t \dot{\phi}} = \omega + t dd^c \dot{\phi}
$$
and $e$ denotes the exponential map in the complex gauge group $\cG^c$ of the smooth $G^c$-bundle $E^c$. Given $\dot{\kappa} = (\dot{\phi},\dot{H}) \in T_\kappa \cKt$ we define the map $\rho$ as
\begin{equation}\label{eq:imaginaryexponential}
\rho \colon T_\kappa \cKt\to \tau_{p_I}(\cYt) : \dot{\kappa} \mapsto \rho(\dot{\kappa}) \defeq f_1^*p_I,
\end{equation}
where $f_1$ is the value at $t = 1$ of the horizontal lift $f_t \subset \Aut_0(E^c)$ (see~\eqref{eq:vectorfieldyft}) of the curve $\kappa_t$ \eqref{eq:ANkappat} starting at $f_0 = \Id$. Equivalently we can write (see~\eqref{eq:taup})
$$
\rho(\dot{\kappa}) = \tau_{p_I}(\kappa_1,f_1).
$$
Note that this map is only well defined on the open subset given by those pairs $(\dot{\phi},\dot{H})$ such that $\dot{\phi}$ is a K\"{a}hler potential on $(X,J,\omega)$. Given $\dot{\kappa} = (\dot{\phi},\dot{H}) \in T_\kappa \cKt$ as before, we will denote
\begin{equation}\label{eq:zetaphixi}
\zeta_{\dot{\kappa}} = - \imag \dot{H} + \theta_A^\perp(\eta_{\dot{\phi}}) \in \LieX,
\end{equation}
where $\eta_{\dot{\phi}} \in \LieH$ is the Hamiltonian vector field on $(X,\omega)$ corresponding to $\dot{\phi}$ and $\theta_A^\perp(\eta_{\dot{\phi}})$ denotes its horizontal lift with respect to the connection $A$ determined by the marked point $p_I \in \cP$. Then, by Proposition~\ref{propo:ANlift}, we have the following.
\begin{corollary}
The differential of $\rho$ at the origin $0 \in T_\kappa \cKt$ is
\begin{equation}\label{eq:imaginaryexponential2}
d \rho_{|0}(\dot{\kappa}) = \mathbf{I} Y_{\zeta_{\dot{\kappa}}},
\end{equation}
where $Y_{\zeta_{\dot{\kappa}}}$ is the infinitesimal action of $\zeta_{\dot{\kappa}}$ given by~\eqref{eq:zetaphixi} at $p_I \in \cP$.
\end{corollary}
\begin{proof}
We claim that the horizontal lift $f_t$ of the curve $\kappa_t$ \eqref{eq:ANkappat} is defined by the time dependent vector field
\begin{equation}\label{eq:vectorfieldyft3}
y_t = -I(- \imag \dot{H} + \theta_t^\perp(J \check{y}_t)),
\end{equation}
where $\check{y}_t$ is the vector field on $X$ covered by $y_t$. To prove this, note that the $H_t$-symmetric tangent vector $\dot{H}_t$ is given by
\begin{equation}\label{eq:ANHdot}
\dot{H}_t  = s_t^*(s_t^{-1}\dot{s}_t + \dot{s}_ts_t^{-1}) = 2 s_t^*(\pi_1 y_{s_t}) \quad \in \imag \LieG_t,
\end{equation}
where $s_t = e^{t\dot{H}}$, and so $y_{s_t} = \dot{s}_ts_t^{-1} = \dot{H}$ implies $\dot{H}_t = 2\dot{H}$ for all $t$. Then, it is straightforward from \eqref{eq:vectorfieldyft} that
\begin{align*}
y_t & = - \dot{H} + \theta_t^\perp(\check{y}_t)\\
& = - I(- \imag \dot{H} + \theta_t^\perp(J\check{y}_t)),
\end{align*}
as claimed. Recall from \secref{subsec:cscKorbits} that $J \check{y}_t = y_{\dot{\phi},\omega_{t\dot{\phi}}}$, where the latter is the Hamiltonian vector field of $\dot{\phi}$ with respect to $\omega_{t\dot{\phi}}$. Then, given a curve $s \mapsto (s\dot{\phi},s\dot{H}) \in T_\kappa \cKt$, we have $\rho(s\dot{\phi},s\dot{\xi}) = (f_1(s))^*p_I$ by definition, where $f_1(s)$ is the value at $t = 1$ of the curve $f_t(s) \subset \Aut(E^c)$ determined by
\begin{align*}
y_t(s) & = - I(- \imag s\dot{H} + \theta_{ts}^\perp(y_{s\dot{\phi},\omega_{ts\dot{\phi}}}))\\
& = - s I(- \imag \dot{H} + \theta_{ts}^\perp(y_{\dot{\phi},\omega_{ts\dot{\phi}}}))\\
& = y_{ts}
\end{align*}
and $f_0(s) = \Id$. Then, $\rho(s\dot{\phi},s\dot{\xi}) = (f_{ts})^*_{|t=1}p_I = f_s^*p_I$ and \eqref{eq:imaginaryexponential2} follows from \eqref{eq:derivativept1} in Proposition~\ref{propo:ANlift}.
\end{proof}

The map \eqref{eq:imaginaryexponential} behaves at the origin, due to the identity \eqref{eq:imaginaryexponential2}, as the composition of the missed complex exponential with the action of the complexified group on the marked point $p_I$. In the next section we will compose $\rho$ with the moment map $\mu_\alpha$ given by \eqref{eq:thm-muX} obtaining a new map that will allow us to deform a given solution to the coupled equations \eqref{eq:CYMeq2}, thereby obtaining a criterion of existence. Note that
the choice of curve $\kappa_t$ made in \eqref{eq:ANkappat} for the definition of $\rho$ is not canonical. In fact, due to the identity \eqref{eq:derivativept1}, it would have been more natural to consider curves $\kappa_t$ such that the associated vector field $f_t^*(Iy_t) \in \LieX$ \eqref{eq:derivativept2} satisfies
\begin{equation}\label{eq:virtualoneparameter0}
\frac{d}{dt} f_t^*(Iy_t) = 0.
\end{equation}
If $\kappa_t = (\omega_t,H_t)$ solves \eqref{eq:virtualoneparameter0}, the K\"{a}hler metric $\omega_t$ is a geodesic in $\cK$ (see Proposition~\eqref{propo:ANvirtualoneparameter}). This provides an evidence of the fact that \eqref{eq:virtualoneparameter0} is the geodesic equation with respect to the Riemannian structure on $\cKt$ described in Remark~\ref{remark:Ktildasymmetric}. Recall that that the Riemannian submersion structure on $\cKt$ has trivial connection and varying Riemannian metric on each fiber that, together with the previous fact, fits with~\cite[Theorem~3.3]{Vi}.

The existence of smooth solutions to \eqref{eq:virtualoneparameter0} has important consequences for the
existence and uniqueness problem for the coupled equations, as we will see in \secref{sec:ANcoupledintmmap}. In this sense, the solutions to \eqref{eq:virtualoneparameter0} are the analogue in our context of the geodesics in the space of K\"{a}hler metrics in the cscK problem. We finish this section with an example of smooth solution to \eqref{eq:virtualoneparameter0} defined for all $t \in [0,\infty[$, coming from the Lie algebra of the isotropy group $\cX_I$ of the marked point $p_I \in \cP$ on the extended gauge group $\cX$ (cf.~\secref{example:cscKgeodesic}). It is very likely that a general existence result for short time, as the one for the geodesics in $\cK$ (see~\cite[Remark~3.3]{Mab1}), should hold for \eqref{eq:virtualoneparameter0}. Note however that, the lack of a uniform $0 < t_0 \in \RR$ over $\cKt$ for which the solutions to \eqref{eq:virtualoneparameter0} are defined does not allow one to construct a suitable map $\rho$ as in \eqref{eq:imaginaryexponential}. This can be considered an indirect proof of the fact that the extended gauge group $\cG$ does not admit a complexification extending the $\cX$-action on the complex manifold $\cP$.

\begin{example}\label{example:Ceqgeodesic}
Let $\zeta \in \LieX_I \subset \LieX$, where we suppose fixed $(\omega,H) \in \cKt$. Then we can write
$$
\zeta = \xi + \theta_A^\perp(\eta),
$$
where $\xi \in \LieG$ and $\eta \in \LieH$ is the Hamiltonian vector field of $\phi \in C^{\infty}(X)$. Given any vector field $y$ on $X$ (resp. on $E^c$), we denote by $f^{y}_{t}$ the flow of $y$ on $X$ (resp. on $E^c$). Let us define the curve
$$
(\omega_t,H_t) = ((f^{J\eta}_t)_*\omega,(f^{I\zeta}_t)_*H),
$$
where $I\zeta = \imag \xi + \theta_A^\perp(J\eta)$. We claim that $(\omega_t,H_t)$ is a solution to $\frac{d}{dt}\zeta_t = 0$ in $\cKt$, where $\zeta_t$ is as in
\eqref{eq:derivativept2}. Recall from Example \ref{example:cscKgeodesic} that $J\eta = \check{y}_t$ in
\eqref{eq:vectorfieldyft}. Combined with
$$
\dot{H}_t = (f_t^{I\zeta})_*(-2\pi_1 \theta_H(I\zeta)) = -2 (f_t^{I\zeta})_*(\imag \xi),
$$
it implies that
\begin{align*}
y_t & = (f_t^{I\zeta})_*(\imag \xi) + \theta_t^\perp(J\eta))\\
& = (f_t^{I\zeta})_*(\imag \xi + \theta_{f_t^{-I\zeta}\cdot A}^\perp(J\eta))\\
&= (f_t^{I\zeta})_* I\zeta = I\zeta,
\end{align*}
and so $\zeta_t = f_t^*(Iy_t) = -\zeta$, since $I\zeta$ is holomorphic.
\end{example}

\begin{remark}\label{remark:ANcomputations}
For further computations it is convenient to express the time dependent vector fields $y_t$ and $\zeta_t$, defined respectively in \eqref{eq:vectorfieldyft} and \eqref{eq:derivativept2}, in a different way. Given a curve
$(\omega_t,H_t) \in \cKt$ starting at the fixed point $\kappa = (\omega,H) \in \cKt$, we can suppose that $H_t = s_t^*H$ for a smooth curve $s_t = e^{\imag \xi_t}$, where $\xi_t \subset \imag \LieG$. Recall from \eqref{eq:ANHdot} that the $H_t$-symmetric tangent vector $\dot{H}_t$ is given by
$$
\dot{H}_t  = 2 s_t^*(\pi_1 y_{s_t}) \quad \in \imag \LieG_t,
$$
where $y_{s_t} = \dot{s}_ts_t^{-1}$. Setting $\omega_t = \omega + dd^c \phi_t$, it follows from \eqref{eq:vectorfieldyft} that
\begin{equation}
\label{eq:vectorfieldyft2}
\begin{split}
y_t & = -\pi^t_1 y_{s_t} + \theta_t^\perp(\check{y}_t)\\
& = s^*_t\(-\pi_1 y_{s_t} + \theta_{s_t \cdot A}^\perp(\check{y}_t)\)\\
& = - I s^*_t\(- \imag \pi_1 y_{s_t} + \theta_{s_t A}^\perp(J \check{y}_t)\),
\end{split}
\end{equation}
where $\pi_1^t\colon \LieG^c \to \imag\LieG_t$ is the standard projection of $H_t$ and $p_I = (J,A)$ is the marked point associated to $\kappa$ \eqref{eq:ANmarkedHamiltonian}. Let us define
\begin{equation}\label{eq:vectorfieldygt}
y_{g_t} = \pi_0(y_{s_t}) + \theta_{s_t \cdot A}^\perp(\check{y}_t).
\end{equation}
Note that $y_{g_t}$ preserves $H_0$. From \eqref{eq:vectorfieldyft2} it follows that $y_t = s_t^*(- y_{s_t} +
y_{g_t})$ and so the lift equals $f_t = s^{-1}_t g_t$. Finally, this gives
\begin{equation}\label{eq:vectorfieldzetat}
\zeta_t = g_t^*\(- \imag \pi_1 y_{s_t} + \theta_{s_t \cdot A}^\perp(y_{\dot{\phi}_t,\omega_t})\) = g_t^*(- \imag \pi_1
y_{s_t}) + \theta_{f_t^{-1} \cdot A}^\perp(y_{\dot{\phi}_t \circ
\check{f}_t}).
\end{equation}
\end{remark}

\section{Deforming solutions in a complexified orbit}
\label{sec:ANdeformingsolutions}

In this section we perturb a solution to the coupled equations \eqref{eq:CYMeq2} with coupling constants $\alpha_0$, $\alpha_1$ to new solutions for close values of the coupling constants. In order to do this we need to impose conditions on the spaces of holomorphic vector fields of the manifold and of the total space of the bundle.

Let us first explain the underlying geometric idea of our deformation argument. Recall that $(Z,J,\omega)$, $G \subset G^c$ are
a K\"{a}hler manifold, a compact group acting by holomorphic isometries on $Z$ and its complexification. Pick $z \in
Z$, a zero of the moment map $\mu: Z \to \mathfrak{g}^*$, and consider the map
\begin{equation}\label{eq:imaginaryexponentialfinite}
L: \mathfrak{g} \to \mathfrak{g}^* \colon \zeta \to L(\zeta) \defeq \mu(e^{\imag\zeta}\cdot z).
\end{equation}
The kernel of $dL_{|0}$ is the Lie algebra of the isotropy group $G_z$ of $z$ in $G$, due to
$$
\langle dL_{|0}(\zeta),\zeta'\rangle = \omega(Y_{\zeta'},JY_{\zeta})_{|z}
$$
for any $\zeta$, $\zeta ' \in \mathfrak{g}$, where $Y_\zeta$ denotes the infinitesimal action of $\zeta$ on $Z$.
Therefore, when $G_z$ is finite, $dL_{|0}$ is an isomorphism. In the situation of our interest the moment map $\mu_\alpha$ is parameterized by a pair of coupling constants $\alpha_0$, $\alpha_1 \in \RR$. When $G_z$ is finite the initial zero $z \in Z$ of the moment map $\mu_{\alpha}$ can be perturbed, via the Implicit Function Theorem, to zeros of $\mu_{\alpha'}$ for closed values $\alpha'$ of the coupling constants.

We now apply this picture to our infinite dimensional setup. First we will construct an analogue of \eqref{eq:imaginaryexponentialfinite} using the map $\rho$ \eqref{eq:imaginaryexponential} defined in \secref{sec:ANgeneral-framework}. We keep the notation of the previous section for a holomorphic principal $G^c$-bundle $(E^c,I)$ over a compact complex manifold $(X,J)$. We fix a pair
$$
\kappa = (\omega,H) \in \cK \times \cR = \cKt
$$
consisting of a K\"{a}hler metric $\omega$ on the fixed K\"{a}hler class $[\omega]$ and a $G$-reduction $H$ of the underlying smooth bundle $E^c$. For any pair $\alpha = (\alpha_0,\alpha_1) \in \RR^2$ of positive coupling constants we have an associated mK-Hamiltonian space (see Definition~\ref{def:ANmarkedHamiltonian})
$$
\kappa \to (p_I,\cP,\mathbf{I},\omega_\alpha,\cX,\mu_\alpha)
$$
as in \eqref{eq:ANmarkedHamiltonian}, where only the symplectic structure $\omega_\alpha$ \eqref{eq:Sympfamily} and the $\cX$-equivariant moment map $\mu_\alpha$ \eqref{eq:thm-muX} depend on the coupling constants. Recall that (see~\eqref{eq:cP})
$$
\cP \subset \cJi \times \cA
$$
is given by pairs $p = (J',A')$ such that $J' \in \cJi$, the space of compatible complex structures on $(X,\omega)$, and $A'$ is a $(1,1)$-type connection with respect to $J'$ on the $G$-reduction $H$. We identify $\cK$ with the space of K\"ahler potentials $\phi \in C^{\infty}_0(X)$, with zero integral with respect to $\vol_\omega$. The construction of \secref{sec:ANgeneral-framework} gives a well defined map (see~\eqref{eq:imaginaryexponential})
$$
\rho \colon \cK \times \imag \LieG \to \tau_{p_I}(\cYt) \subset \cP,
$$
where $\tau_{p_I}(\cYt)$ is the complexified $\cX$-orbit of the marked point $p_I$ in the K\"{a}hler manifold given by $(\cP,\mathbf{I},\omega_\alpha)$. We will identify $\LieG \cong \imag \LieG$ via the change of variable $\dot{H} = \imag \xi$ in \eqref{eq:imaginaryexponential} to simplify the notation. Then, composing with the moment map $\mu_\alpha$, \eqref{eq:imaginaryexponential} provides a well defined map
$$
L\colon \RR^2 \times \cK\times \LieG \to C_0^\infty(X) \times \LieG
$$
given by
\begin{equation}\label{eq:implicitmap}
L(\alpha_0,\alpha_1,\phi,\xi) \defeq (\alpha_0 S_{J_{\phi,\xi}} + \alpha_1 \Lambda^2(F_{A_{\phi,\xi}}\wedge
F_{A_{\phi,\xi}}) -
c, \Lambda F_{A_{\phi,\xi}}),
\end{equation}
where $\rho(\phi,\xi) = (J_{\phi,\xi},A_{\phi,\xi}) \in \cP$ and $c \in \RR$ is as in \eqref{eq:CYMeq2}. Note that we have normalized the $\LieG$ factor in \eqref{eq:implicitmap}. This allows us to perturb solutions also in the limit $\alpha_1 = 0$.
\begin{remark}
Here $\mu_\alpha$ is regarded as a map
$$
\mu_\alpha \colon \cP \to \LieH \times \LieG,
$$
provided by the vector space splitting $\LieX \cong \LieH \times \LieG$ and the pairing \eqref{eq:pairingscX} given by any $G$-connection $A$ on the reduction $H$.
\end{remark}

Recall that $\cJi$ denotes the space $\omega$-compatible complex structures in $X$. Let
$$
P: C^{\infty}_0(X) \to T_J \cJi
$$
be the infinitesimal action of $C^{\infty}_0(X)$ on $\cJi$, where the former is regarded as the Lie algebra of the group of Hamiltonian sysmplectomorphisms on $(X,\omega)$. Then it follows from
\eqref{eq:imaginaryexponential2}, from the definition of the complex structure $\mathbf{I}$ \eqref{eq:complexstructureI} and from the moment map property of the scalar curvature (see \cite{D1}) that the
differential at the origin of $L$ is given by
\begin{equation}
\label{eq:impL}
\begin{split}
dL_{0|0}(\phi,\xi) & = \alpha_0 P^*P \phi + 2\alpha_1 \Lambda^2 (F_A \wedge d_A J(d_A \xi + i_{\eta_\phi}F_A)),\\
dL_{1|0}(\phi,\xi) & = \Lambda d_A J (d_A \xi + i_{\eta_\phi}F_A) = d_A^*(d_A \xi + i_{\eta_\phi}F_A),
\end{split}
\end{equation}
where $L =  (L_0,L_1)$, $p_I = (J,A)$ is the marked point in $\cP$ and $P^*$ denotes the adjoint operator with
respect to the metric on $\Omega^0(\End TX)$ determined by $J$ and $\omega$. The kernel of the system \eqref{eq:impL} has a geometric interpretation in terms of holomorphic $G^c$-invariant vector fields in $E^c$ when $\alpha_0$ and $\alpha_1$ have the same sign (cf. Proposition $1$ in \cite{LS1}). Given $(\phi,\xi) \in C^{\infty}_0(X) \times \LieG$ we will denote
\begin{equation}\label{eq:zetaphixi2}
\zeta_{\phi,\xi} = \xi + \theta_A^\perp(\eta_{\phi}) \in \LieX
\end{equation}
as in \eqref{eq:zetaphixi}, where $\eta_\phi$ is the Hamiltonian vector field on $(X,\omega)$ corresponding to $\phi$ and $\theta_A^\perp(\eta_\phi)$ denotes its horizontal lift with respect to the connection $A$ determined by the marked point $p_I = (J,A) \in \cP$.

\begin{lemma}\label{lemma:impfunc}
If $\alpha_0 \cdot \alpha_1 > 0$ and $\Lambda_\omega F_A = z$, then a pair $(\phi,\xi) \in C_0^{\infty}(X) \times \LieG$ is in the kernel of \eqref{eq:impL} if and only if $\zeta_{\phi,\xi}$ \eqref{eq:zetaphixi2} is a $G^c$-invariant real holomorphic vector field on $(E^c,I)$.
\end{lemma}
\begin{proof}
Note that the real $G^c$-invariant vector field \eqref{eq:zetaphixi2} is holomorphic if and only if
$$
P\phi = 0, \quad d_A \xi + i_{\eta_\phi}F_A = 0,
$$
or, equivalently, if its infinitesimal action $(Y_{\zeta_{\phi,\xi}})_{|p_I}$ at the point $p_I \in \cP$ vanishes (see~\eqref{eq:infinit-action-connections}). Hence, if \eqref{eq:zetaphixi2} is holomorphic, $(\phi,\xi) \in \textrm{Ker} \; dL_{|0}$, by \eqref{eq:impL}. To prove the converse we fix $(\alpha_0,\alpha_1) \in \RR^2$ and consider the pairings
\begin{align*}
\langle\cdot,\cdot\rangle_\cH & \colon \LieH \times \LieH \to \RR,\\
\langle\cdot,\cdot\rangle_\cG & \colon \LieG \times \LieG \to \RR,
\end{align*}
defined in \eqref{eq:pairingH} and \eqref{eq:LieG-dual}. Then, given $(\phi_j,\xi_j) \in C^{\infty}_0(X) \times \LieG$ with $j = 0,1$, the differential $dL = (dL_0,dL_1)$ satisfies
\begin{align*}
\langle dL_{0|0}(\phi_0,\xi_0),\phi_1\rangle_\cH + \alpha_1 \cdot \langle dL_{1|0}(\phi_0,\xi_0),\xi_1\rangle_\cG & = \langle d\mu_\alpha(\mathbf{I}Y_{\zeta_0}),\zeta_1\rangle_{|p}\\
& = \omega_\alpha(Y_{\zeta_1},\mathbf{I}Y_{\zeta_0})_{|p},
\end{align*}
where $\zeta_j = \zeta_{\phi_j,\xi_j} \in \LieX$ is as in \eqref{eq:zetaphixi2}. Recall that the symplectic structure $\omega_\alpha$ of \eqref{eq:Sympfamily} and the moment map $\mu_\alpha$ of \eqref{eq:thm-muX} (where we set $z=0$) depend on the coupling constants $\alpha_0$ and $\alpha_1$. If $\alpha_0\cdot\alpha_1 > 0$ then $\omega_\alpha$ is compatible with either $\mathbf{I}$  or $-\mathbf{I}$ and so the previous equation implies that if $(\phi_0,\xi_0) \in \textrm{Ker} \; dL_{|0}$ then $Y_{\zeta_0} = 0$ (setting $\zeta_0 = \zeta_1$). We conclude that the vector field $\zeta_0$ is holomorphic as claimed.
\end{proof}

Note that when $\alpha_0\cdot\alpha_1 >0$, the kernel of $dL_{|0}$ can be identified with the Lie algebra of the isotropy
group $\cX_I$ of $p_I \in \cP$ inside the extended gauge group $\cX$. Let $H^0(TX)$ be the space of holomorphic sections
of $TX$ with respect to complex structure $J$. Taking suitable Sobolev completions of
$C^{\infty}_0(X) \times \LieG$ and applying the Implicit Function Theorem we obtain the following result.
\begin{proposition}\label{thm:ExistenceCYMeq}
Suppose that $\kappa = (\omega,H) \in \cKt$ is a solution to the coupled equations \eqref{eq:CYMeq2} with coupling constants $(\alpha_0,\alpha_1) \in \mathbb{R}^2$. Let $(p_I,\cP,\mathbf{I},\omega_\alpha,\cX,\mu_\alpha)$ be the associated mK-Hamiltonian space and $\mathfrak{z} \subset \LieX_I$ the centre of the Lie algebra of the compact group $G \subset G^c$. Then,
\begin{enumerate}[i)]
    \item If $\LieX_I = \mathfrak{z}$ and $\alpha_0 \cdot \alpha_1 > 0$, then there exists an open neighbourhood $U$ of
        $(\alpha_0,\alpha_1)$ in $\mathbb{R}^2$ and a $C^1$ map $f$ from $U$ to the complexified orbit of the marked point $p_I$ such that $f(t,s)$ is locally the unique solution of the coupled equations
        \eqref{eq:ceqc2} with coupling constants $(t,s)$ over this complexified orbit modulo the action of $\cX$.

    \item If $\LieX_I = \mathfrak{z}$, $H^0(TX) = 0$, $\alpha_0 \neq 0$ and $\alpha_1 = 0$, then there exists an open neighbourhood $U$ of $(\alpha_0,0)$ with the same properties as in $i)$.
\end{enumerate}
\end{proposition}
\begin{proof}
Given a positive integer $k > 0$, let $C^k(X)$ and $\Omega^{0,k}(X,\ad E)$ be, respectively, the spaces of $C^k$
functions and sections of $X$ and $\ad E$ and denote by $H^k$ and $W^k$ their $L^2$ completions. Let
\begin{equation}\label{eq:implicitmap2}
L_{k}\colon U_{k} \subseteq H^k \times W^{k-1} \to H^{k-4} \times W^{k-3}
\end{equation}
be the map induced by $L$ \eqref{eq:implicitmap} in a neighbourhood $U_{k}$ of the origin. Since the map $L$ of \eqref{eq:implicitmap} is not explicit, we check that it is well defined. We take $k \gg 0$ such that $H^k$ is an algebra and all the equations below hold in a strong sense. Given $(\phi,\xi) \in U_k$ the automorphism $f_1 \in \Aut(E^c)$ involved in the definition of the map (see~\eqref{eq:imaginaryexponential})
$$
\rho(\phi,\xi) = f_1^*p
$$
has $k-2$ derivatives in $L^2$, due to the formula $\eqref{eq:vectorfieldyft3}$ and the Picard's Theorem on ordinary differential equations for Banach spaces (see~\cite[Theorem~3.1]{MB}). Recall from Remark~\ref{remark:ANcomputations} that we have a decomposition $f_1 = s^{-1} \cdot g$ where $s$ and $g$ have, respectively, $k-1$ and $k-2$ derivatives in $L^2$. Therefore, the equality
$$
F_{A_{\phi,\xi}} = F_{f_1^*A} = g^*F_{s \cdot A}
$$
implies that the terms $\Lambda F_{A_{\phi,\xi}}$ and $\Lambda^2(F_{A_{\phi,\xi}} \wedge F_{A_{\phi,\xi}})$ in \eqref{eq:implicitmap} are respectively in $W^{k-3}$ and $H^{k-3}$ (note that the symplectic form is fixed and so $\Lambda$ is a $C^{\infty}$ operator). On the other hand, if $\check{f}_1$ is the diffeomorphism on $X$ that $f_1$ covers, then
$$
S_{J_{\phi,\xi}} = S_{\check{f}_1^*J} = S_{\omega_\phi} \circ \check{f}_1,
$$
in \eqref{eq:implicitmap}, where $S_{\omega_\phi}$ denotes the scalar curvature of the $H^{k-2}$ metric provided by the complex structure $J$ and the $H^{k-2}$ symplectic form $\omega_\phi = \omega + dd^c \phi$. Therefore, the function $S_{J_{\phi,\xi}}$ is of class $H^{k-4}$. Hence, the previous computation shows that
$$
L_k(\phi,\xi) \subset H^{k-4} \times W^{k-3}
$$
for $k$ large enough, as required, provided that $(\phi,\xi) \in H^k \times W^{k-1}$.

We now proof the statement using a deformation argument. For simplicity we suppose $\mathfrak{z} = \{0\}$, i.e. $\mathfrak{g}$ is semi-simple, and we explain later how to adapt the argument. The linearization of $L_k$ at the origin in the two cases $i)$ and $ii)$ is given by \eqref{eq:impL} and it is invertible by hypothesis, provided that the induced operators $P^*P\colon H^k \to H^{k-4}$ and
$$
\Delta_A = d_A^* d_A\colon W^{k-1} \to W^{k-3}
$$
are self-adjoint elliptic operators. The result follows applying the Implicit Function Theorem in Banach spaces (see e.g.~\cite[p.~72,Theorem~3.10]{Au}). The regularity of the solutions follows, by local uniqueness, taking arbitrary large values
of $k$. If $\mathfrak{g}$ is not semi-simple the argument can be adapted by changing $W^k$ and $W^{k-2}$ in
\eqref{eq:implicitmap2} by the orthogonal complement $\mathfrak{z}^\perp$ of $\mathfrak{z}$ on each space and considering the map induced by \eqref{eq:implicitmap} and the orthogonal projection. The same argument holds since this restricted map has the same linearization as \eqref{eq:implicitmap} at the origin.
\end{proof}

As a straightforward consequence of the previous result we obtain,

\begin{theorem}\label{thm:ExistenceCYMeq2}
Let $(X,L)$ be a polarised complex projective manifold, $G^c$ be a complex reductive Lie group and $E^c$ be a holomorphic
$G^c$-bundle over $X$. If there
exists a cscK metric $\omega \in c_1(L)$, $X$ has finite automorphism group and $E^c$ is stable with respect to $L$
then, given a pair of positive
real constants $\alpha_0, \alpha_1 > 0$ with small ratio $0 < \frac{\alpha_1}{\alpha_0} \ll 1$, there exists a solution
$(\omega_\alpha,H_\alpha)$ to
\eqref{eq:CYMeq2} with these coupling constants and $\omega_\alpha \in c_1(L)$.
\end{theorem}
\begin{proof}
Let $\omega \in c_1(L)$ be a cscK metric. By the Hitchin-Kobayashi correspondence for the HYM
equation \cite{D3,UY,RS}, $E^c$ admits a reduction $H$ to a maximal compact subgroup $G \subset G^c$ with
HYM Chern connection $A$. The pair $(\omega,H)$ is then a solution to the coupled equation \eqref{eq:CYMeq2} with $\alpha_0 \neq 0$, $\alpha_1 = 0$. Let $\cX$ be the extended gauge group \eqref{eq:Ext-Lie-groups} associated to the
symplectic form $\omega$ and the
reduction $H$. We claim that the
Lie algebra of the subgroup $\cX_I \subset \cX$ preserving the holomorphic structure in $E^c$ reduces to the centre
$\mathfrak{z}$ of the Lie algebra
of $G$. To see this, consider any $G$-invariant (real) holomorphic vector field $\zeta \in \LieX_I$ and let $\eta \in
\Omega^0(TX)$ be its projection
onto $X$. By definition of $\cX_I$ (see also \eqref{lemma:impfunc}), $\eta$ is a Hamiltonian Killing vector field that
defines a $1$-parameter group of isometries of $(X,\omega)$ and hence, by hypothesis, identically vanishes. The vector
field $\zeta$ is then an $A$-parallel vertical vector field and can be identified with a holomorphic $G^c$-invariant
vector field on $E^c$. Since $E^c$ is stable, the Lie algebra of holomorphic $G^c$-invariant
vector fields, canonically identified with $H^0(X,\ad E^c)$, is isomorphic to the centre $\mathfrak{z}^c$ of $G^c$.
Finally, since $\zeta$ preserves the reduction $H$, we conclude that $\zeta \in \mathfrak{z} \subset \mathfrak{z}^c$, as
claimed. In the situation considered, we are under the hypothesis of Proposition \ref{thm:ExistenceCYMeq} and the
statement holds.
\end{proof}

\begin{remark}
The hypothesis of Theorem \ref{thm:ExistenceCYMeq2} are satisfied in a large class of examples. In \secref{sec:example3} we will discuss the special case in which the polarized manifold $(X,L)$ satisfies $c_1(L) = \lambda \cdot c_1(X)$, for $\lambda \in \ZZ$.
\end{remark}

\begin{remark}
The statement of Theorem \ref{thm:ExistenceCYMeq2} is also true for arbitrary K\"ahler manifolds, not necessarily
algebraic, if we impose the additional condition of simplicity on $(E^c,I)$ (see~\cite[p.~261]{UY}).
\end{remark}

\section{Integral invariants}
\label{sec:ANIntegralinvariants}

Let $(E^c,I)$ be a holomorphic $G^c$-bundle over a compact complex manifold $(X,J)$. We fix a K\"{a}hler class
$[\omega]$ on $X$ and denote by $\cKt = \cK \times \cR$ the space of pairs $(\omega,H)$ consisting of a K\"{a}hler
metric $\omega \in [\omega]$ and a smooth $G$-reduction $H \in \cR$. We consider the coupled equations
\eqref{eq:CYMeq2} for elements in $\cKt$, with fixed positive coupling constants
$$
0 < \alpha_0,
\alpha_1 \in \RR.
$$
In this section we give an obstruction to the existence of solutions to \eqref{eq:CYMeq2} controlled by a character of the Lie algebra of $\Aut (E^c,I)$, the group of automorphisms of the holomorphic bundle over $X$. This generalizes previous constructions of A. Futaki in \cite{Ft0,Ft1}. To define the character we adapt the argument of J. P. Bourguignon in \cite{Bo}. To simplify some formulae we suppose $z = 0$ in the system \eqref{eq:CYMeq2}, as there is no significant difference with the general case.

To define the character we proceed as in \secref{subsec:KNcharacter}, so we define first a closed $1$-form $\sigma_I$ in $\cKt$. To do this we need a technical lemma that relates $\sigma_I$ with the moment map $\mu_\alpha$ whose zero locus correspond to the solutions to the coupled equations \eqref{eq:CYMeq2}. We will use this lemma to prove the analogue of Proposition~\ref{propo:oneformMabuchiclosed}, about suitable properties of the $1$-form necessary for the definition of the invariant. For a first reading one can skip Lemma~\ref{lemma:sigmaImugamma} and the proof of Proposition~\ref{propo:sigmaIclosed}.

Let us first define the $1$-form. Fix $\kappa = (\omega,H) \in \cKt$. As in \secref{sec:ANgeneral-framework}, we identify the tangent space
$$
T_\kappa \cKt \cong C^{\infty}_0(X) \times \imag \LieG,
$$
where $C^{\infty}_0(X)$ are smooth functions with zero integral with respect to $\vol_\omega$ and $\LieG$ is the gauge group of the $G$-reduction $H$. Recall from~\eqref{eq:ANkappat} that given $(\dot{\phi},\dot{H}) \in T_\kappa \cKt$, with the previous identification, we consider it as the vector given by the curve
\begin{equation}\label{eq:ANkappat2}
\kappa_t = (\omega_t,H_t) \defeq (\omega_{t \dot{\phi}},(e^{t\dot{H}})^*H) \subset \cKt,
\end{equation}
where
$$
\omega_{t \dot{\phi}} = \omega + t dd^c \dot{\phi}
$$
and $e$ denotes the exponential map in the complex gauge group $\cG^c$ of the smooth $G^c$-bundle $E^c$. Given $\dot{\kappa} = (\dot{\phi},\dot{H}) \in T_\kappa \cKt$, we define the $1$-form $\sigma_I$ at $\kappa \in \cKt$ as
\begin{equation}\label{eq:oneformsigmaI}
\sigma_I(\dot{\phi},\dot{H}) = - \int_X \big{(}\dot{\phi}(\alpha_0 S_{\omega} + \alpha_1 \Lambda_{\omega}^2(F_H \wedge
F_H) - c) + 2\alpha_1(\imag \dot{H},\Lambda_{\omega}F_H)\big{)} \frac{\omega^n}{n!},
\end{equation}
where $c \in \RR$ is as in \eqref{eq:CYMeq2}. The relation of the $1$-form \eqref{eq:oneformsigmaI} with the moment map of Section \secref{sec:Ceqcoupled-equations} is given by the following result, that we will use systematically in the following.
\begin{lemma}\label{lemma:sigmaImugamma}
Fix $\kappa \in \cKt$ and consider $(p_I,\cP,\mathbf{I},\omega_\alpha,\cX,\mu_\alpha)$, the associated mK-Hamiltonian space, where the $\cX$-equivariant moment map $\mu_\alpha$ is determined by the coupling constants $\alpha_0$, $\alpha_1$. Then, given any curve $\kappa_t$ in $\cKt$ with initial condition $\kappa$, we have
$$
\sigma_I(\dot{\kappa}_t) = \langle \mu_{\alpha}(f_t^{-1}\cdot p), f_t^*(Iy_t)\rangle,
$$
where $f_t \subset \Aut(E^c)$ is the horizontal lift of $\kappa_t$ defined in \eqref{eq:vectorfieldyft}.
\end{lemma}
\begin{proof}
Let $\kappa = (\omega,H)$ and consider a curve $\kappa_t = (\omega_t,H_t) = (\omega_{\phi_t},s_t^*H)$ in $\cKt$ with
initial condition $\kappa$, $s_t = e^{\xi_t}$ and $\xi_t \in \imag \LieG$. Consider the horizontal lift
$f_t$ of this curve through $f_0 = \Id$ defined by \eqref{eq:vectorfieldyft}. Recall that
$$
(\check{f}_t^*\omega_t,f^*_tH_t) = (\omega,H)
$$
and also that $f_t = s_t^{-1}g_t$, where $g_t$ preserves the fixed $G$-reduction $H$ on $E^c$, for all $t$. Define the following curves:
$$
\omega_t = \omega_{\phi_t}, \quad A_t = f_t^{-1} \cdot A, \quad J_t = \check{f}_t^*J, \quad \textrm{and} \quad \eta_t,
$$
where the latter is the Hamiltonian vector field of the function $\dot{\phi}_t \circ \check{f}_t$ on $X$ with respect to the fixed symplectic form $\omega$. Recall from Remark \ref{remark:ANcomputations} that
$$
\dot{H}_t  = 2 s_t^*(\pi_1 y_{s_t}) \quad \in \imag \LieG_t,
$$
where $y_{s_t} = \dot{s}_ts_t^{-1}$ and $\pi_1$ is, as in \secref{sec:ANgeneral-framework}, the projection onto $\imag
\LieG$. Combined with the formula $F_{H_t} = s_t^*(F_{s_t A})$ (see~\cite[Section~1.1]{D3}), we have
\begin{align*}
\sigma_I(\dot{\phi}_t,\dot{H}_t) & = - \int_X \dot{\phi}_t\(\alpha_0 S_{\omega_t} + \alpha_1
\Lambda_{\omega_t}^2(F^2_{s_tA}) - c\) + 4\alpha_1\(\imag \pi_1 y_{s_t},\Lambda_{\omega_t}F_{s_tA}\)
\frac{\omega_t^n}{n!}\\
& = - \int_X \dot{\phi}_t \circ \check{f}_t\(\alpha_0 S_{J_t} + \alpha_1 \Lambda_{\omega}^2F^2_{A_t} - c\) +
4\alpha_1 \(\imag \pi_1g_t^*y_{s_t},\Lambda_{\omega}F_{A_t}\) \frac{\omega^n}{n!}\\
& = \langle \mu_{\alpha}(J_t,A_t), - \imag \pi_1 g_t^*(y_{s_t}) + \theta^{\perp}_{A_t}(\eta_{\dot{\phi}_t
\circ \check{f}_t}) \rangle\\
& = \langle \mu_\alpha(f_t^{-1} \cdot p), f_t^*(Iy_t) \rangle,
\end{align*}
where $\hat{c} \in \RR$ is defined by \eqref{eq:constant-c-hat}. To do the changes of variable we have used the fact that the pairing $\Omega^p(\ad E^c) \times \Omega^q(\ad E^c) \to \Omega^{p+q} \otimes \CC$ induced by \eqref{eq:pairinggc} is $\cG^c$-invariant. The last equality follows from the identity \eqref{eq:vectorfieldzetat}.
\end{proof}

Note that $\cKt$ has a right action by pull-back of the group of holomorphic automorphisms $\Aut(E^c,I)$. Then, the $1$-form defined in \eqref{eq:oneformsigmaI} satisfies the following.
\begin{proposition}\label{propo:sigmaIclosed}
$\sigma_I$ is $\Aut(E^c,I)$-invariant and closed.
\end{proposition}
\begin{proof}
To prove the invariance note that for any pair $(\omega,H) \in \cKt$, any K\"{a}hler potential $\phi$ and any $f
\in \Aut(E^c,I)$ we have $F_{f^*H} = f^*F_H$ and
$$
\check{f}^*\omega_{\phi} =
(\check{f}^*\omega)_{\check{f}^*\phi}.
$$
Then, the $\Aut(E^c,I)$-invariance follows, after a change of variable in \eqref{eq:oneformsigmaI}, from the
$\cG^c$-invariance of the pairing $\Omega^p(\ad E^c) \times \Omega^q(\ad E^c) \to \Omega^{p+q} \otimes \CC$ induced by
\eqref{eq:pairinggc}. To prove the second part of the statement, fix $\kappa \in \cKt$ and consider
$$
(p_I,\cP,\mathbf{I},\omega_\alpha,\cX,\mu_\alpha),
$$
the associated mK-Hamiltonian space, where the $\cX$-equivariant moment map $\mu_\alpha$ is determined
by the coupling constants $\alpha_0$, $\alpha_1$. We will prove that $\sigma_I$ is closed at the arbitrary fixed point $\kappa$. Given
$$
\dot{\kappa}_j = (\dot{\phi}_j,\dot{H}_j) \in C_0^{\infty}(X) \times \imag \LieG, \qquad j = 1,2,
$$
regarded as a vector field on $\cKt$, the Lie bracket $[\dot{\kappa}_1,\dot{\kappa}_2]$ vanishes at $\kappa$ and so,
combined with \eqref{eq:sigmazclosed1}, it is enough to prove that $\dot{\kappa}_1(\sigma_I (\dot{\kappa}_2))$ is
symmetric in $j = 1,2$ at $\kappa$. Let
$$
\kappa_t = (\omega_{t\dot{\phi_1}},s_t^*H),
$$
where $s_t = e^{t \dot{H}_1}$. Then, using the same notation as in the previous lemma, we obtain
\begin{align*}
\dot{\kappa}_1(\sigma_I (\dot{\kappa}_2)) & = \frac{d}{dt}_{|t = 0} \sigma_I (\dot{\kappa}_2)_{|\kappa_t}\\
& = \frac{d}{dt}_{|t = 0} \langle \mu_\alpha(g_t^{-1}s_t \cdot p_I), g_t^*(-\imag\dot{H}_2) +
\theta^{\perp}_{g_t^{-1}s_t A}(\eta_{\dot{\phi}_2 \circ \check{f}_t}) \rangle =\\
& = \omega_\alpha(Y_2,\mathbf{I}Y_1) + \langle \mu_\alpha(p_I), \dot{\kappa}_{2,1}\rangle,
\end{align*}
where $\omega_\alpha$ is a symplectic form on $\cP$ compatible with $\mathbf{I}$, $Y_j$ is the infinitesimal action of
the vector field $-\imag \dot{H}_j + \theta^{\perp}_A(\eta_{\dot{\phi}_j}) \in \LieX$ at $p_I \in \cP$ and
\begin{align*}
\dot{\kappa}_{2,1} & = \frac{d}{dt}_{|t = 0} g_t^*(-\imag\dot{H}_2) + \theta^{\perp}_{g_t^{-1}s_t
A}(\eta_{\dot{\phi}_2 \circ \check{f}_t})\\
& = \theta_A[\theta^{\perp}_A(-J\eta_{\dot{\phi}_1}),-\imag\dot{H}_2 + \theta^{\perp}_A(\eta_{\dot{\phi}_2})] -
[\imag\dot{H}_1,\theta^{\perp}_A(J\eta_{\dot{\phi}_2})] - \theta^{\perp}_A(\eta_{\omega(\eta_{\dot{\phi}_2},J
\eta_{\dot{\phi}_1})})\\
& = F_A(J\eta_{\dot{\phi}_1},\eta_{\dot{\phi}_2}) + d_A(\imag\dot{H}_1)(J\eta_{\dot{\phi}_2}) +
d_A(\imag\dot{H}_2)(J\eta_{\dot{\phi}_1})
- \theta^{\perp}_A(\eta_{\omega(\eta_{\dot{\phi}_2},J \eta_{\dot{\phi}_1})}).
\end{align*}
Hence, $\dot{\kappa}_1(\sigma_p (\dot{\kappa}_2)) = \dot{\kappa}_2(\sigma_p (\dot{\kappa}_1))$ and so $d \sigma_I = 0$
holds.
\end{proof}

Due to Proposition \ref{propo:sigmaIclosed}, the holomorphic structure $I$ and the K\"{a}hler class $[\omega]$ determine, as in \secref{subsec:KNcharacter}, a complex character
$$
\cF_{\alpha,I,[\omega]} \colon \Lie \Aut(E^c,I) \to \CC \colon \zeta \to \imag \sigma_I(\zeta^l) + \sigma_I(I \zeta)^l,
$$
where $\zeta^l$ denotes the infinitesimal action of $\zeta \in \Lie \Aut(E^c,I)$ on $\cKt$. When there is no possible confusion we write simply $\cF_\alpha$ for the previous character. To give an explicit expression for $\cF_\alpha$, recall from \secref{subsec:cscKFutaki} that, given a K\"{a}hler form $\omega \in
\cK$, any holomorphic vector field $y \in \Lie \Aut(X,J)$ can be written as
$$
\eta = \eta_{\phi_1} + J \eta_{\phi_2} + \beta,
$$
where $\eta_{\phi_j}$ is the $\omega$-Hamiltonian associated to $\phi_j$, $j = 1,2$, and $\beta$ is the dual of an
harmonic $1$-form with respect to the K\"{a}hler structure $(X,J,\omega)$. Then the infinitesimal action of $\zeta \in
\Lie\Aut(E^c,I)$ at $(\omega,H) \in \cKt$ is
$$
\zeta^l = (\phi_2,2\pi_1\theta_H(\zeta))
$$
if $\zeta$ covers $y$ and the value of the character $\cF_\alpha$ at $\zeta$ can be written as
\begin{equation}\label{eq:alphafutakismooth}
\cF_{\alpha}(\zeta) = - \int_X \big{(}\phi(\alpha_0 S_{\omega} + \alpha_1 \Lambda_{\omega}^2(F_H \wedge F_H) - c) -
4\alpha_1(\theta_H\zeta,\Lambda_{\omega}F_H\big{)} \frac{\omega^n}{n!},
\end{equation}
where $\phi = \phi_1 + \imag \phi_2$. The following result is the analogue of Proposition~\ref{propo:futakifinite}.
\begin{proposition}\label{prop:futakiobstruction}
The map $\cF_{\alpha}\colon \Lie \Aut(E^c,I) \to \CC$ is independent of the $(\omega,H) \in \cKt$. It defines a
character of $\Lie \Aut(E^c,I)$ that vanishes if $\cKt$ contains a solution to the coupled equations
\eqref{eq:CYMeq2}.
\end{proposition}
\begin{proof}
Formally the proof is the same as the proof of Proposition~\ref{propo:futakifinite}.
\end{proof}
When $\alpha_1 = 0$ the expression \eqref{eq:alphafutakismooth} is, up to multiplicative constant factor, the \emph{Futaki invariant} of the
complex manifold $(X,J)$ and the K\"{a}hler class $[\omega]$. When $\alpha_0 = 0$ the character $\cF_\alpha$
generalizes an invariant due also to A. Futaki (see Theorem $1.1$ in \cite{Ft1}), who considered the smaller Lie algebra of $G^c$-invariant holomorphic vector fields on $(E^c,I)$ that project onto a vector field that vanishes somewhere on $X$ (i.e. onto a holomorphic complex Hamiltonian vector field).

We finish this section with an example of how the invariant $\cF_\alpha$ controls the existence of solutions to the \emph{coupled equations} for small ratio of the coupling constants.

\begin{example}
Let $(X,L)$ be a polarized manifold that admits an extremal, non cscK, metric $\omega$ in $c_1(L)$ (e.g.
$\CC\mathbb{P}^2$ blown up in one point with
the anticanonical polarization \cite{Ca}). Recall that this means that the scalar curvature $S_\omega$ is the
Hamiltonian function of a real
holomorphic Killing vector field $\eta$. The classical Futaki invariant of the K\"{a}hler class $c_1(L)$ (take
$\alpha_0 = 1$ and
$\alpha_1 = 0$ in \eqref{eq:alphafutakismooth}) evaluated on $\eta$ is
$$
\cF(\eta) = - \int_X(S_\omega - \hat{S})^2\vol_\omega < 0.
$$
Let $E^c$ be the holomorphic $\CC^*$-principal bundle of frames of $L$. The holomorphic vector field $\eta$ can be
trivially lifted to a holomorphic
vector field on $E^c$ and it follows by definition that the associated character $\cF_\alpha$ given by \eqref{eq:alphafutakismooth} evaluated at $\eta$ is negative for sufficiently small values of $\alpha_1/\alpha_0 > 0$. Hence, the triple $(X,L,E^c)$ does not admit a solution to the coupled system \eqref{eq:CYMeq2} for this values of the coupling constants. Given an arbitrary holomorphic principal bundle $E^c$ over
$X$, the obstruction to lift the holomorphic vector field $\eta$ on $X$ to a $G^c$-invariant one in $E^c$ lies in $H^1(X,\ad E^c)$ (cf. equation \eqref{eq:infinit-action-connections}). When this obstruction vanishes the same argument can be applied.
\end{example}

\section{The integral of the moment map}
\label{sec:ANcoupledintmmap}

In this section we show that there exists a suitable integral of the moment map for the coupled equations that extends
the Donaldson functional for the HYM equation and the Mabuchi K-Energy for the cscK equation. We determine a sufficient condition for its convexity and the existence of lower bounds in terms of the existence of smooth solutions to certain
partial differential equation on $\cKt$ that generalizes the geodesic equation on $\cK$. When this condition holds, we
generalize the obstruction defined in \secref{sec:ANIntegralinvariants}. Throughout this section we will consider the coupled equations \eqref{eq:CYMeq2} with fixed positive coupling constants $0 < \alpha_0, \alpha_1 \in \RR$.

As in the previous section, let $(E^c,I)$ be a holomorphic $G^c$-bundle over a compact complex manifold $(X,J)$. We fix
a K\"{a}hler class $[\omega]$ on $X$ and denote by $\cKt = \cK \times \cR$ the space of pairs $(\omega,H)$ consisting
of a K\"{a}hler metric $\omega \in [\omega]$ and a smooth $G$-reduction $H \in \cR$. Since the space $\cKt$ is
contractible, the $1$-form $\sigma_I$ given by \eqref{eq:oneformsigmaI} integrates, i.e there exists an integral of the moment map
\begin{equation}\label{eq:ANintegralmmap}
\cM_{\kappa}: \cKt \to \RR,
\end{equation}
that satisfies $d \cM_{\kappa} = \sigma_I$ and $\cM_{\kappa}(\kappa) = 0$ for $\kappa = (\omega,H) \in \cKt$. This implies that the
critical points of $\cM_{\kappa}$ are the solutions to the coupled equations in $\cKt$. Given any curve
$\kappa_t$ in $\cKt$, we can write
\begin{equation}\label{eq:ANintegralmmap}
\cM_{\kappa}(\kappa_t) = \cM_{\kappa}(\kappa_0) + \int_0^t \sigma_I(\dot{\kappa}_s) \wedge ds.
\end{equation}
Combined with Lemma~\ref{lemma:sigmaImugamma} this implies the following.
\begin{proposition}\label{propo:ANvirtualoneparameter}
The functional $\cM_{\kappa}$ given by \eqref{eq:ANintegralmmap} is convex along smooth solutions $\kappa_t = (\omega_{\phi_t},H_t)$ of the system
\begin{equation}
\label{eq:virtualoneparameter}
\left. \begin{array}{l}
\frac{1}{2}\ddot{H}_t - d_t\dot{H}_t(Jy_{\dot{\phi}_t,\omega_t}) - \imag
F_t(y_{\dot{\phi}_t,\omega_t},Jy_{\dot{\phi}_t,\omega_t}) = 0\\
\\
\ddot{\phi}_t \; - \; |d\dot{\phi}_t|^2_{\omega_t}= 0
\end{array}\right \},
\end{equation}
where $d_t$ and $F_t$ denote, respectively, the covariant derivative and the curvature of the Chern connection determined by $H_t$.
\end{proposition}
\begin{proof}
We fix $\kappa = (\omega,H) \in \cKt$ and consider the mK-Hamiltonian space \eqref{eq:ANmarkedHamiltonian} and the integral of the moment map $\cM_\kappa$ determined by $\kappa$ . Given a smooth curve $\kappa_t = (\omega_t,H_t) \subset \cKt$ with $\omega_t = \omega_{\phi_t}$ let $f_t \subset \Aut(E^c)$ be the associated lift defined by \eqref{eq:vectorfieldyft}. Denote $\zeta_t = f_t^*(Iy_t) \in \LieX$ as in Proposition~\ref{propo:ANlift}. Then, applying \eqref{eq:ANintegralmmap} and Lemma \ref{lemma:sigmaImugamma} we have
\begin{equation}\label{eq:convexityintmmap}
\begin{split}
\frac{d^2}{dt^2} \cM_{\kappa}(\kappa_t) & = \frac{d}{dt} \sigma_I(\dot{\kappa}_t)\\
& = \frac{d}{dt} \langle \mu_\alpha(f_t^{-1}\cdot p), \zeta_t\rangle\\
& = \omega_\alpha(Y_{\zeta_t},\mathbf{I}Y_{\zeta_t}) + \langle \mu_\alpha(f_t^{-1}\cdot
p),\frac{d}{dt}\zeta_t\rangle,
\end{split}
\end{equation}
where $\omega_\alpha(Y_{\zeta_t},\mathbf{I}Y_{\zeta_t}) = \|Y_{\zeta_t}\|^2_\alpha$ is the squared norm of the
infinitesimal action of $\zeta_t$ in the K\"{a}hler manifold $(\cP,\mathbf{I},\omega_\alpha)$ associated to
$(\omega,H)$. Therefore, $\cM_\kappa$ is convex along the solutions to $\frac{d}{dt}\zeta_t = 0$ and strictly convex if
$Y_{\zeta_t}$ do not vanishes. Note that if $Y_{\zeta_t} = 0$, then $f_t$ preserves the complex structure $I$ for all
$t$ and so the curve $\kappa_t = (f_t)_*\kappa_0$ is linked by a smooth curve of holomorphic automorphisms on
$\Aut(E^c,I)$. We now check that the system \eqref{eq:virtualoneparameter} is equivalent to the vanishing of
$\frac{d}{dt}\zeta_t$. Note first that the latter condition is equivalent to the vanishing of $\zeta_t' =
(f_t)_*\frac{d}{dt}\zeta_t$, that we can write as
\begin{align*}
\zeta_t' & = [y_t,Iy_t] + I\dot{y}_t\\
& = [-\frac{1}{2}\dot{H}_t + \theta_t^{\perp}(\check{y}_t),-\frac{\imag}{2}\dot{H}_t +
\theta_t^{\perp}(J\check{y}_t)] \\
& -\frac{\imag}{2}\ddot{H}_t + (\frac{d}{dt}\theta_t^{\perp})(J\check{y}_t) +
\theta_t^{\perp}(\frac{d}{dt}J\check{y}_t)\\
& = [\theta_t^{\perp}(\check{y}_t),\theta_t^{\perp}(J\check{y}_t)] -
\frac{\imag}{2} d_t\dot{H}_t(\check{y}_t) + \frac{1}{2}
d_t\dot{H}_t(J\check{y}_t)\\
& -\frac{\imag}{2}\ddot{H}_t + (\frac{d}{dt}\theta_t^{\perp})(J\check{y}_t) +
\theta_t^{\perp}(\frac{d}{dt}J\check{y}_t),
\end{align*}
where $\theta_t^{\perp}$ denotes the horizontal lift of vector fields with respect to the Chern connection determined by $H_t$. Denote respectively by $\eta_t$ and $y'_t$ the Hamiltonian vector field associated to the smooth function $\dot{\phi}_t \circ \check{f}_t$ with respect to $\omega$, and $\phi'_t = \ddot{\phi}_t - |d\dot{\phi}_t|^2_{\omega_t}$ with respect to $\omega_t$. Then we have that
$$
\frac{d}{dt}J\check{y}_t = \frac{d}{dt}(\check{f}_t)_*(\eta_t) =
- [\check{y}_t,J\check{y}_t] + y'_t,
$$
due to the equality $J \check{y}_t = y_{\dot{\phi}_t,\omega_t}$, that combined with
$$
(\frac{d}{dt}\theta_t^{\perp})(J\check{y}_t) = -
\partial_t(\dot{H}_t)(J\check{y}_t),
$$
gives
\begin{align*}
\zeta'_t & = \theta_t[\theta_t^{\perp}(\check{y}_t),\theta_t^{\perp}(J\check{y}_t)] +
(\dbar_t - \partial_t)\dot{H}_t(J\check{y}_t) - \frac{\imag}{2}\ddot{H}_t +
\theta_t^{\perp}(y'_t)\\
& = F_t(Jy_{\dot{\phi}_t,\omega_t},y_{\dot{\phi}_t,\omega_t}) + \imag d_t\dot{H}_t(Jy_{\dot{\phi}_t,\omega_t})
- \frac{\imag}{2}\ddot{H}_t + \theta_t^{\perp}(y'_t).
\end{align*}
The statement follows by splitting $\zeta'_t$ into its $A_t$ vertical and horizontal part for each value of $t$.
\end{proof}

The existence of smooth solutions to \eqref{eq:virtualoneparameter} has important consequences for the
existence and uniqueness problem for the coupled equations that we describe in the following proposition. Recall from
Example \ref{example:Ceqgeodesic} that any real holomorphic vector field in $\Lie \cX_I$ gives a smooth solution to
\eqref{eq:virtualoneparameter} defined for all time, that we call a \emph{trivial solution}. As mentioned in \secref{sec:ANgeneral-framework}, it is very likely that a short time existence result as the one for the geodesics in $\cK$ (see~\cite[Remark~3.3]{Mab1} holds for \eqref{eq:virtualoneparameter}. This would imply that the first of the following obstructions extends
the one coming from the character $\cF_\alpha$ defined in the previous section. Long time existence of smooth solutions is not
even true in general for the geodesic equation in $\cK$ (see~\cite{ChT}).

\begin{proposition}\label{propo:ANuniqueness}
\hspace{0.5cm}
\begin{enumerate}
  \item If there exists a solution $\kappa \in \cKt$ to the coupled equations  then for any smooth solution
      $\kappa_t \subset \cKt$ to \eqref{eq:virtualoneparameter}, with $t \in [0,s[$, $s \in \RR \cup \{\infty\}$
and initial condition $\kappa_0 = \kappa$, the following inequality holds
      $$
      \lim_{t \to s} \sigma_I(\dot{\kappa}_t) \geq 0,
      $$
      with equality only if $\kappa_t$ is a trivial solution.
  \item If any two points in $\cKt$ are joined by a smooth solution to \eqref{eq:virtualoneparameter}, then there
      exists at most one solution to the coupled equations modulo the action of $\Aut(E^c,I)$. Under the same
      hypothesis, if there exists one solution, then the integral of the moment map is bounded from below.
\end{enumerate}
\end{proposition}
\begin{proof}
We use the same notation as in Proposition \ref{propo:ANvirtualoneparameter}. To prove $1)$ let $\kappa_t$, $t \in
[0,s[$ with $s \in \RR \cup \infty$ be a smooth solution to \eqref{eq:virtualoneparameter}. Then, it follows from
\eqref{eq:convexityintmmap} that
$$
\frac{d}{dt} \sigma_I(\dot{\kappa}_t) = \|Y_{\zeta_0}\|^2_\gamma \geq 0,
$$
where $\zeta_0 = \zeta_t$ for all $t$, since \ref{propo:ANvirtualoneparameter} is satisfied. If $\kappa_0 = \kappa$, a
solution to the coupled equations, then $\sigma_I(\dot{\kappa}_0) = 0$ and so $\sigma_I(\dot{\kappa}_t) \geq 0$. If
$\sigma_I(\dot{\kappa}_t) = 0$ for all $t$, then $\zeta_0$ is holomorphic and $\kappa_t$ is therefore a trivial
solution. Thus, we have an equality
$$
\lim_{t \to s} \sigma_I(\dot{\kappa}_t) = \sigma_I(\dot{\kappa}_0) = \cF_\alpha(\zeta_0).
$$
To prove $2)$, fix a solution $\kappa_0 \in \cKt$ to the coupled equations. Then, if $\kappa_1$ is another solution, it
is linked to $\kappa_0$, by hypothesis, by a smooth solution $\kappa_t$ to \ref{propo:ANvirtualoneparameter}, with $t
\in [0,1]$. It follows then from \eqref{eq:convexityintmmap} and $1)$ that $\kappa_t$ is a trivial solution since
$\sigma_I(\dot{\kappa}_1) = \sigma_I(\dot{\kappa}_0)$. The lower bound of $\cM_\kappa$ for any $\kappa \in \cKt$
follows from $1)$ and \eqref{eq:ANintegralmmap}.
\end{proof}

\section{Complex surfaces}
\label{sec:ANcomplexsurfaces}

In this section we restrict to the case in which the base $X$ is a compact complex surface. We use the integral of
the moment map constructed in the previous section to prove some results about the uniqueness and existence problem for the coupled equations \eqref{eq:CYMeq2} with fixed positive coupling constants $0 < \alpha_0, \alpha_1 \in \RR$. In this section we suppose that $\mathfrak{z} \ni z = 0$ in \eqref{eq:CYMeq2}.

Let $(E^c,I)$ be a holomorphic $G^c$-bundle over a compact complex surface $(X,J)$. We fix a K\"{a}hler class
$[\omega]$ on $X$ and suppose that $(\omega,H) \in \cKt$ is a solution to \eqref{eq:CYMeq2}. Consider the functional on the space $\cK$, of K\"{a}hler metrics with fixed K\"{a}hler class $[\omega]$, given by
$$
\cM_\kappa(\cdot,H)\colon \cK \to \RR,
$$
where $\cM_\kappa$ is the integral of the moment  map defined in \eqref{eq:ANintegralmmap} for a choice of $\kappa \in
\cKt$. Without loss of generality we suppose that $\kappa = (\omega,H)$ and denote by $\cM_H$ the functional above.
Then, along a curve $\omega_t = \omega_{\phi_t}$ in $\cK$, we can write
\begin{equation}
\label{eq:ANcMH}
\begin{split}
\cM_H(\omega_t) & = - \int_0^t \int_X \big{(}\dot{\phi}_s(\alpha_0 S_{\omega_s} + \alpha_1 \Lambda_{\omega_s}^2(F_H
\wedge F_H) - c)\big{)}\frac{\omega_s^2}{2!} \wedge ds\\
& = - \int_0^t \int_X \dot{\phi}_s(\alpha_0 S_{\omega_s} - c)\frac{\omega_s^2}{2!} - 2 \alpha_1 \int_0^t \int_X
\dot{\phi}_s (F_H \wedge F_H) \wedge ds.
\end{split}
\end{equation}
The crucial fact for our discussion concerning complex surfaces is the following: since $(\omega,H)$ is a solution to
\eqref{eq:CYMeq2} with $\mathfrak{z} \ni z = 0$, the Chern connection $A_H$ is Anti-Self-Dual and so
$$
- 2(F_H \wedge F_H) = |F_H|^2 \cdot \omega^2,
$$
that is, a non-negative real $4$-form, where the point-wise norm $|F_H|^2$ is taken with respect to $\omega$.
Straightforward computations as in Lemma~\ref{lemma:sigmaImugamma} show then that
\begin{equation}\label{eq:ANconvexitycM}
\frac{d^2}{dt^2} \cM_H(\omega_t) = 2\alpha_0\|Y_{\eta_t}\|^2 - \int_X (\ddot{\phi}_t -
|d\dot{\phi}_t|^2_{\omega_t})(\alpha_0 S_{\omega_t} - c)\frac{\omega_t^2}{2!} + \alpha_1 \int_X \ddot{\phi}_t |F_H|^2
\cdot \omega^2,
\end{equation}
where $\|Y_{\eta_t}\|^2$ denotes the squared norm of the infinitesimal action of the time dependent Hamiltonian vector field $\eta_t = \eta_{\dot{\phi}_t \circ f_t} \in \LieH$ in the space of complex structures $\cJi$ compatible with $\omega$, where $f_t$ is the Moser curve of $\omega_t$ as in \eqref{eq:sigmaJmuH}. As an immediate consequence we obtain
\begin{lemma}
If any point in $\cK$ is linked to $\omega$ by a smooth geodesic, then $\cM_H$ is bounded from below. For any other
metric $\omega' \in \cK$ such that $(\omega',H)$ is a solution to \eqref{eq:CYMeq2}, there exists a holomorphic
transformation $f \in \Aut (X,J)$ such that $f^*\omega' = \omega$. Moreover, if the connection $A_H$ is not flat, then
$\omega$ is the unique metric in $\cK$ with this property.
\end{lemma}
\begin{proof}
The first part of the statement follows from $d \cM_{H|\omega} = 0$ and \eqref{eq:ANconvexitycM}, due to the
inequality
\begin{equation}\label{eq:ANconvexitycM2}
\frac{d^2}{dt^2} \cM_H(\omega_t) = 2\alpha_0\|Y_{\eta_t}\|^2 + \alpha_1 \int_X |d\dot{\phi}_t|^2_{\omega_t} |F_H|^2
\cdot \omega^2 \geq 0
\end{equation}
along smooth geodesics. For the second part, suppose that $\omega_{\phi_t}$ joins $\omega$ to another solution $\omega'$
of \eqref{eq:CYMeq2}. Then, $\eta_t = \eta_0$ since $\phi_t$ is a geodesic in the space of K\"{a}hler potentials and $d
\cM_{H|\omega'} = 0$. This implies that \eqref{eq:ANconvexitycM2} is an equality and that $\eta_0$ is a holomorphic
vector field with $f_t$ being its flow. We conclude that $f_1^*\omega' = \omega$ with $f_1$ holomorphic and also that
$$
|d\dot{\phi}_0|^2_{\omega_t} \cdot |F_H|^2 = 0
$$
everywhere on $X$ for all values of $t$. If the connection $A_H$ is not flat, then $\dot{\phi}_0$ is constant in a non
empty open set of $X$, but since $\eta_0 = \eta_{\dot{\phi}_0}$ is holomorphic, then $\dot{\phi}_0$ is constant on $X$.
Therefore $f_1$ is the identity transformation and $\omega' = \omega$ as claimed.
\end{proof}
As a corollary of the previous lemma we obtain a result for toric complex surfaces, where smooth geodesics joining any
two invariant metrics in $\cK$ are known to exist. Recall that a complex manifold is said to be toric if it is endowed with a holomorphic action of a complex torus with a dense orbit. We fix a a maximal compact subgroup $T$ of the complex torus and consider $\cK^T$, the subspace of $T$-invariant K\"{a}hler metrics in $\cK$.
\begin{proposition}\label{propo:ANtoric}
Let $(E^c,I)$ be a holomorphic $G^c$-bundle over a smooth toric compact complex surface $(X,J)$. If there exists a
solution $(\omega,H)$ to \eqref{eq:CYMeq2} with $\omega \in \cK^T$, then $\cM_H\colon \cK^T \to \RR$ is bounded from below. For any other K\"{a}hler metric $\omega' \in \cK^T$ such that $(\omega',H)$ is a solution to \eqref{eq:CYMeq2}, there exists a holomorphic transformation $f \in \Aut (X,J)$ such that $f^*\omega' = \omega$. Moreover, if the bundle is not topologically flat, then $\omega$ is the unique metric in $\cK^T$ with this property.
\end{proposition}

We obtain a similar result for a complex surface with $c_1(X) \leq 0$. For the proof we adapt the argument of X. X. Chen in \cite{Ch1} concerning the Mabuchi K-energy.

\begin{proposition}\label{propo:ANboundedcMH}
Let $(E^c,I)$ be a holomorphic $G^c$-bundle over a smooth compact complex surface $(X,J)$ with $c_1(X) \leq 0$.
\begin{enumerate}
  \item If there exists a solution $(\omega,H) \in \cKt$ to \eqref{eq:CYMeq2}, then $\cM_H$ is bounded from below.
  \item If the topological constant $c$ vanishes, then $\omega$ is the unique metric in $\cK$ such that $(\omega,H)$ is a solution to \eqref{eq:CYMeq2}.
\end{enumerate}
\end{proposition}
\begin{proof}
Let $\omega'$ be the unique K\"{a}hler--Einstein metric in $[\omega]$, that it is known to exists due to the work of Aubin and Yau when $c_1(X) < 0$ and Yau's solution to the Calabi Conjecture when $c_1(X) = 0$ (see~\cite{Au}). Choose an arbitrary K\"{a}hler metric $\omega_\phi \in \cK$. By \cite[Lemma~7]{Ch1} for every $0 < \epsilon \ll 1$ there exists an $\epsilon$-approximate smooth geodesic $\phi_t \colon [0,1] \times X \to \RR$ in the space of K\"{a}hler potentials joining $\omega$ with $\omega_\phi$, i.e. such that
\begin{equation}\label{eq:epsilongeodesic}
(\ddot{\phi}_t - |d\dot{\phi}_t|^2_{\omega_t}) \cdot \omega_t^2 = \epsilon \cdot (\omega')^2.
\end{equation}
Set $\varphi_t = (\ddot{\phi}_t - |d\dot{\phi}_t|^2_{\omega_t}) > 0$, $\omega_t = \omega_{\phi_t}$. Then, the functional $\cM_H$ along the curve $\omega_t$ satisfies
\begin{align*}
\frac{d^2}{dt^2} \cM_H(\omega_t) & = 2\alpha_0\|Y_{\eta_t}\|^2 - \int_X \varphi_t(\alpha_0 S_{\omega_t} - c)\frac{\omega_t^2}{2!} + \alpha_1 \int_X \ddot{\phi}_t |F_H|^2
\cdot \omega^2\\
& = 2\alpha_0\|Y_{\eta_t}\|^2 - \alpha_0 \int_X \varphi_t S_{\omega_t}\frac{\omega_t^2}{2!} + \epsilon \cdot c \int_X \frac{(\omega')^2}{2!} + \alpha_1 \int_X \ddot{\phi}_t |F_H|^2
\cdot \omega^2\\
& \geq - \alpha_0 \int_X \varphi_t S_{\omega_t}\frac{\omega_t^2}{2!} + \epsilon \cdot c \cdot \Vol(X),
\end{align*}
where we have used that $\ddot{\phi}_t \geq 0$ due to \eqref{eq:epsilongeodesic}. Using now the identity (see~\cite[p.~223]{Ch1})
$$
- \int_X \varphi_t S_{\omega_t}\frac{\omega_t^2}{2!} = \int_X |d \varphi_t|^2_{\omega_t} \frac{1}{\varphi_t}\frac{\omega_t^2}{2!} - \int_X \varphi_t\Lambda_{\omega_t}\rho_{\omega'} \frac{\omega_t^2}{2!}
$$
we conclude that
\begin{equation}\label{eq:cMHconvex}
\frac{d^2}{dt^2} \cM_H(\omega_t) \geq \epsilon \cdot c \cdot \Vol(X),
\end{equation}
since $\Lambda_{\omega_t}\rho_{\omega'} \leq 0$ due to the fact that $\omega'$ is a K\"{a}hler--Einstein metric and $c_1(X) \leq 0$. Therefore, from
$$
\frac{d}{dt}_{|t = 0} \cM_H(\omega_t) = 0,
$$
integrating on $[0,1]$ it follows that
$$
\cM_H(\omega_{\phi}) - \cM_H(\omega) \geq \epsilon \cdot c \cdot \Vol(X) \cdot \frac{t^2}{2}|_0^1 =  \epsilon \cdot \frac{c \cdot \Vol(X)}{2}.
$$
Since $c \in \RR$ is a topological constant, making $\epsilon \to 0$, we conclude that $\cM_H$ is bounded in $\cK$ as claimed in $1)$. To prove $2)$ suppose that $\omega_\phi$ is also a solution to \eqref{eq:CYMeq2} in $\cK$. Note that if $c = 0$ then \eqref{eq:cMHconvex} together with $\frac{d}{dt}_{|t = 1} \cM_H(\omega_t) = 0$ gives
$$
0 = 2\alpha_0\|Y_{\eta_t}\|^2 + \alpha_0 \int_X |d \varphi_t|^2_{\omega_t} \frac{1}{\varphi_t}\frac{\omega_t^2}{2!} + \alpha_1 \int_X \ddot{\phi}_t |F_H|^2 \cdot \omega^2,
$$
from which $\eta_t$ is a holomorphic vector field on $(X,J)$ for all $t \in [0,1]$. Since $\eta_t$ is a complex Hamiltonian vector field, it vanishes somewhere on $X$ but since $c_1(X) \leq 0$ then $\eta_t$ is parallel and so it identically vanishes. Therefore $\omega = \omega_\phi$ as claimed.
\end{proof}
\begin{remark}
Note that as $\mathfrak{z} \ni z = 0$ in \eqref{eq:CYMeq2}, $c = 0$ implies that $c_1(X) = 0$ and $\hat c  = 0$ (defined in~\eqref{eq:constant-c-hat}), and so the bundle in topologically trivial.
\end{remark}

We finish this chapter proving that the lower bound on $\cM_H$ is an obstruction to the existence of solutions to \eqref{eq:CYMeq2} in any compact complex surfaces which admits a K\"ahler metric with nonnegative bisectional curvature. We use the result about weak solutions of the geodesic equation in $\cK$ found by X. X. Chen and G. Tian in \cite{ChT} and adapt their argument about the Mabuchi K-energy. The condition on the bisectional curvature is imposed in order to have weak solutions to the geodesic equation in the space of K\"ahler potentials with bounded second derivatives (see~\cite{Bl}).

\begin{proposition}\label{propo:ANboundedcMH}
Let $(E^c,I)$ be a holomorphic $G^c$-bundle over a smooth compact complex surface $(X,J)$ which admits a K\"ahler metric $\omega_0$ with nonnegative bisectional curvature. If there exists a solution
$(\omega,H) \in \cKt$ to \eqref{eq:CYMeq2} with $[\omega] = [\omega_0]$, then $\cM_H$ is bounded from below.
\end{proposition}
\begin{proof}
Let $(\omega,H) \in \cKt$ be a solution to \eqref{eq:CYMeq2}. Let $\omega_\phi = \omega + dd^c \phi$ be any other K\"{a}hler metric in $\cK$
for a given K\"{a}hler potential $\phi$. We will prove that
\begin{equation}\label{eq:inequalitycMH}
\cM_H(\omega_\phi) - \cM_H(\omega) \geq 0.
\end{equation}
To do this, we will use the following decomposition formula (cf.~\cite[formula~6.1]{ChT})
\begin{equation}\label{eq:cMHdecomposed}
\begin{split}
\cM_H(\omega_{\phi}) & = \alpha_0 \cdot \cM_\omega(\omega_{\phi}) + c'\cdot I(\omega_\phi) + \alpha_1 \int_X \phi |F_H|^2 \cdot \omega^2, \quad \textrm{with}\\
I(\omega_\phi) & = \frac{1}{3!}\sum_{j = 0}^2\int_X \phi \cdot \omega_\phi^j \wedge \omega^{2-j},
\end{split}
\end{equation}
where $c' = c - \alpha_0 \hat S \in \RR$ and $\cM_\omega$ is the Mabuchi K-energy on $\cK$. Given any positive integer $l$, let $\Sigma_l = [-l,l] \times [0,1]$ be the rectangle domain, regarded as Riemann surface with boundary, and consider the homogeneous complex Monge--Amp\`ere equation
\begin{equation}\label{eq:ANcomplexMongeAmp}
(\pi^*\omega + 2\imag\partial\dbar\phi)^{2+1} = 0,
\end{equation}
where $\phi: \Sigma_l \times X \to \RR$ is a real function and $\pi:\Sigma_l \times X \to X$ denotes the projection. Recall from~\cite{D6} that a curve of K\"{a}hler potentials $\phi_t$, $t \in [0,1]$, satisfies the geodesic equation
\eqref{eq:geodesiccscK} if and only the function $\phi(s,t,x) = \phi_t(x)$ on $\Sigma_l \times X$ satisfies
\eqref{eq:ANcomplexMongeAmp}, and also that an arbitrary solution $\phi(s,t)$ of \eqref{eq:ANcomplexMongeAmp} on the square $\Sigma_l$ satisfies
\begin{equation}\label{eq:geodesicSigmal}
\Delta_{s,t} \phi^l \defeq \(\frac{d^2}{ds^2} + \frac{d^2}{dt^2}\) \phi^l =  |d(\frac{d}{dt}\phi^l) - d^c(\frac{d}{ds}\phi^l)|_{\omega_{\phi^l}}^2.
\end{equation}
By \cite[Theorem~1.1]{ChT}, there exists an almost smooth solution $\phi^l$ to \eqref{eq:ANcomplexMongeAmp} (see~\cite[Definition~1.3.3]{ChT}) with boundary values
$$
\phi^l(s,0) = 0, \; \phi^l(s,1) = \phi, \;  \phi^l(\pm l,t) = (1-t)\phi \qquad (s,t) \in \Sigma_l.
$$
To be more precise, one may modify slightly the boundary of $\Sigma_l$ and the boundary data so that the domain is smooth without corners (see~\cite[Section~6.2]{ChT}). Let us fix a K\"{a}hler metric on $\Sigma_l \times X$ that does not depend on $l$. According to \cite{ChT}, there exists a positive constant $C > 0$ independent of $l$ such that
\begin{equation}\label{eq:boundphil}
0 \leq |\partial \dbar \phi^l| \leq C,
\end{equation}
where the norm is taken with respect to the fixed K\"{a}hler metric on $\Sigma_l \times X$ and
$$
0 \leq (\omega_{\phi^l(s,t)})^2
$$
for any $(s,t) \in \Sigma_l$. Note here that $\omega_{\phi^l(s,t)}$ can be less regular than $C^{\infty}$ (a priori it is only a bounded two form on $X$). Moreover, the additional assumption on the existence of a K\"ahler metric on $\cK$ with nonnegative bisectional curvature, implies that all the partial derivatives of second order of $\phi^l$ are bounded. Due to the decomposition formula for the K-energy (see~\cite[formula~6.1]{ChT}), it induces a well defined map
$$
\cM_\omega\colon \Sigma_l \to \RR \colon (s,t) \mapsto \cM_\omega(\omega_{\phi^l(s,t)}).
$$
This map is $C^1$ continuous up to the boundary of $\Sigma_l$ and weakly subharmonic in the interior (see~\cite[Theorem~6.1.1]{ChT}). The previous bounds on $\phi^l$ and the decomposition formula \eqref{eq:cMHdecomposed} implies then that we have a well defined map
$$
\cM_H\colon \Sigma_l \to \RR \colon (s,t) \mapsto \cM_H(\omega_{\phi^l(s,t)}).
$$
Let $\nu: \RR \to \RR$ be a smooth non-negative function such that $\nu \equiv 1$ on $[-\frac{1}{2},\frac{1}{2}]$ and
vanishes outside $[-\frac{3}{4},\frac{3}{4}]$. Set
$$
\nu^l(s) = \frac{\nu(\frac{s}{l})}{\int_\RR\nu(s)ds}
$$
and define
\begin{align*}
f^l(t) & = \int_\RR \nu^l(s)\cM_H(s,t) ds\\
& = \alpha_0 \cdot \int_\RR \nu^l(s) \cM_\omega(s,t) ds + c'\cdot \int_\RR \nu^l(s) I(s,t) ds\\
& + \alpha_1 \cdot \int_\RR \nu^l(s) \int_X \phi^l(s,t) |F_H|^2 \omega^2 \wedge ds,
\end{align*}
where $I(s,t) = I(\omega_{\phi^l(s,t)})$. Then,
$$
f^l(0) = \int_\RR \nu^l(s)\cM_H(\omega) ds = \cM_H(\omega)
$$
$$
f^l(1) = \int_\RR \nu^l(s)\cM_H(\omega_\phi) ds = \cM_H(\omega_\phi)
$$
and
$$
\frac{d f^l}{dt}_{|t=0} = 0,
$$
since $(\omega,H)$ is a solution to \eqref{eq:CYMeq2}. Then, we have that
\begin{align*}
f^l(1) - f^l(0) & = \int_0^1 \int_0^r\frac{d^2f^l}{dt^2} dt \wedge dr\\
& = \int_0^1 \int_0^r \int_\RR \nu^l(s)\frac{d^2 \cM_H(s,t)}{dt^2} ds \wedge dt \wedge dr\\
& = \int_0^1 \int_0^r \int_\RR \nu^l(s) \Delta_{s,t}\cM_H(s,t) ds \wedge dt \wedge dr \\
& - \int_0^1 \int_0^r \int_\RR \nu^l(s)\frac{d^2 \cM_H(s,t)}{ds^2} ds \wedge dt \wedge dr\\
& = \alpha_0 \int_0^1 \int_0^r \int_\RR \nu^l(s) \Delta_{s,t}\cM_\omega(s,t) ds \wedge dt \wedge dr \\
& + \alpha_1 \cdot \int_0^1 \int_0^r \int_\RR \nu^l(s) \int_X \Delta_{s,t} \phi^l |F_H|^2 \cdot \omega^2 \wedge ds \wedge dt \wedge dr\\
& - \int_0^1 \int_0^r \int_\RR \nu^l(s)\frac{d^2 \cM_H(s,t)}{ds^2} ds \wedge dt \wedge dr,
\end{align*}
where we have used that $\phi^l$ has bounded second derivatives and that the $I$ term makes no contribution to the second derivatives along solutions to \eqref{eq:ANcomplexMongeAmp} (see the comments before the proof of~\cite[Theorem~6.1.1]{ChT}).
Note also that
$$
\Delta_{s,t} \phi^l = \frac{\partial^2\phi^l}{\partial z \partial \overline{z}},
$$
which is bounded by \eqref{eq:boundphil} and is non negative due to \eqref{eq:geodesicSigmal}, where $z = t + \imag s$ is the complex coordinate on $\Sigma_l$. The previous facts imply that we have a lower bound
\begin{align*}
f^l(1) - f^l(0) & \geq - \int_0^1 \int_0^r \int_\RR \nu^l(s)\frac{d^2 \cM_H(s,t)}{ds^2} ds \wedge dt \wedge dr\\
& = - \int_0^1 \int_0^r \int_\RR \frac{d^2\nu^l(s)}{ds^2}\cM_H(s,t) ds \wedge dt \wedge dr\\
& = - \frac{1}{l^2 v} \int_0^r \int_\RR \frac{d^2\nu^l}{ds^2}_{|\frac{s}{l}}\cM_H(s,t) ds \wedge dt \wedge dr,
\end{align*}
where $v = \int_\RR \nu(s) ds$. The bound \eqref{eq:boundphil} and the decomposition formulae \eqref{eq:cMHdecomposed} for $\cM_H$ and $\cM_\omega$ provide a uniform bound $|\cM_H(s,t)| \leq C'$ that, as in the proof of~\cite[Theorem~6.2.1]{ChT}, gives
$$
\cM_H(\omega_\phi) - \cM_H(\omega) = f^l(1) - f^l(0) \geq - \frac{C'}{2l}.
$$
This shows that \eqref{eq:inequalitycMH} holds as $l \to \infty$.
\end{proof}

\chapter{\bf Algebraic stability}
\label{chap:AlgStability}

Let $(X,L,E)$ be a triple consisting of a compact complex manifold $X$, an ample line bundle $L$ and a holomorphic vector bundle $E$ over $X$. We know from Section \secref{chap:analytic} that the pairs $(\omega,H)$ consisting of a K\"{a}hler metric $\omega \in c_1(L)$ and a Hermitian metric $H$ on $E$ correspond to points in the ``complexified orbit'' of the holomorphic structure of the bundle of frames of $E$ with respect to the action of the extended gauge group. This fact, together with the moment map interpretation of the coupled equations \eqref{eq:CYMeq2} leads to formulate a conjecture in the vain of the usual identification of algebraic and symplectic quotients of projective manifolds provided by the Kempf-Ness
Theorem, namely: \emph{the existence of solutions $(\omega,H)$ to the coupled equations \eqref{eq:CYMeq2}, with $\omega \in c_1(L)$, is equivalent to a suitable polystability of the triple $(X,L,E)$ in the sense of GIT.} This chapter is devoted to the formulation of a stability notion adapted to this
problem and to give evidence that the above conjecture holds.

\section{K-stability}
\label{sec:ASAlgebraicstb}

In this section we propose a notion of stability for a triple $(X,L,E)$ consisting of a projective scheme $X$, an ample invertible sheaf $L$ and a coherent sheaf $E$ over $X$. Our stability notion depends on a positive real parameter $\alpha$ and is modelled on Donaldson's definition of $K$-stability for polarized varieties (see~\secref{subsec:cscKalgKstab}). As in \cite{D4}, it is formulated in terms of flat degenerations of the triple over $\CC$. The central fiber of each degeneration is endowed with a $\CC^*$-action and the stability of the triple, that we call $\alpha$-K-stability, is controlled by a suitable numerical invariant $F_\alpha$ (see \eqref{eq:extendedfutaki}) associated to this action. We are mainly interested in the case in which $E$ is a holomorphic vector bundle over a smooth projective variety $X$, but even in this case we need to consider singular degenerations and hence our geometric objects in this section will be arbitrary algebraic schemes and coherent sheaves, rather than vector bundles and manifolds. When the degeneration considered is smooth, $F_\alpha$ will coincide, up to a multiplicative constant factor, with the value of the character $\cF_\alpha$ given by \eqref{eq:alphafutakismooth} at the generator of the $S^1 \subset \CC^*$ action (see~\eqref{sec:ASDiffeformalphainvariant}).

Let $(X,L,E)$ be a triple consisting of a projective scheme $X$, an ample invertible sheaf $L$ and a coherent sheaf $E$ over $X$. We start defining the degenerations of $(X,L,E)$ that we will consider. In what follows, a \emph{$\CC^*$-action} on such a triple means a $\CC^*$-action on $X$ with $\CC^*$-linearisations on $L$ and $E$.

\begin{definition}\label{def:testconfiguration}
A \emph{test configuration for $(X,L,E)$} is a triple $(\mathcal{X},\cL,\cE)$ consisting of a scheme $\mathcal{X}$, an invertible sheaf $\cL$ and a coherent sheaf $\cE$ over $\mathcal{X}$, together with a $\CC^*$-action on $(\mathcal{X},\cL,\cE)$ and a flat morphism $\pi\colon \mathcal{X}\to \CC$, such that
\begin{enumerate}
\item[(1)] %
the invertible sheaf $\cL$ is ample and the sheaf $\cE$ is flat over $\CC$,
\item[(2)] %
$\pi: \mathcal{X} \to \CC$ is $\CC^*$-equivariant, where $\CC^*$ acts
on $\CC$ by multiplication in the standard way,
\item[(3)]
the fibre $(X_t,L_t,E_t)$ is isomorphic to $(X,L,E)$ for all
$t \in\CC\setminus\{0\}$, where $X_t = \pi^{-1}(t)$, $L_t =
\cL_{|X_t}$ and $E_t=\cE_{|X_t}$.
\end{enumerate}
\end{definition}

Any $\CC^*$-action on $(X,L,E)$ determines a special type of test configuration, called a \emph{product configuration}, where $\mathcal{X} \cong X \times \CC$ with the induced $\CC^*$-action, $\cL$ and $\cE$ are obtained by pulling the $\CC^*$-linearised sheaves $L$ and $E$, and $\pi$ is the canonical projection. This product configuration is called a \emph{trivial configuration} if the given $\CC^*$-action on $(X,L,E)$ is trivial, i.e. if it consists of the trivial $\CC^*$-action on $X$ and the $\CC^*$-action defined by scalar mutiplication on $L$ and $E$.

Let $(X,L,E)$ be as in Definition~\ref{def:testconfiguration} and such that $E$ has $n$-dimensional support, where $n=\dim X$. Let $(\mathcal{X},\cL,\cE)$ be a test configuration for $(X,L,E)$. Then, the fibres $(X_t,L_t,E_t)$ are all necessarily isomorphic for $t\neq 0$, due to the $\CC^*$-action. Since $\pi$ and $\cE$ are flat the Hilbert polynomial of $E_t$ is independent of $t \in \CC$~\cite[Ch. III, Thm 9.9]{Ha}, so the central fibre $(X_0,L_0,E_0)$ satisfies that $X_0$ is an $n$-dimensional projective scheme, $L_0$ is ample and $E_0$ has $n$-dimensional support. Recall that the degree of the Hilbert polynomial of $E_t$ equals the dimension of its support (see~\cite[Ch. III, Ex. 5.7(a)]{Ha} for the fact that $L_t$ is ample). As $0\in\CC$ is fixed, the central fibre $(X_0,L_0,E_0)$ has an induced $\CC^*$-action. For each integer $k$, we define the integer $w(k)$ as the weight of the induced $\CC^*$-action on the determinant of the cohomology of $E_0\otimes L_0^k$, i.e. on the line (cf. e.g.~\cite[\S\S 1.2 and 2.1]{HL})
\[
  \det H^*(X_0,E_0\otimes L_0^k)
  \defeq \bigotimes_{p=0}^n \(\det H^p (X_0,E_0\otimes L_0^k)\)^{(-1)^{p}}.
\]
Then, by the equivariant Riemann--Roch theorem, $w(k)$ is a polynomial of degree at most $n+1$ in $k$ (cf.~\cite[\S 2.1]{D4}).  Thus the following quotient has an expansion
\begin{equation}\label{eq:defF(W,k)}
F(E_0,k) \defeq \frac{w(k)}{kP_{L_0}(E_0,k)} = F_0(E_0) + k^{-1} F_1(E_0) + k^{-2} F_2(E_0) + O(k^{-3})
\end{equation}
with rational coefficients $F_i(E_0)$, where $P_{L_0}(E_0,k)$ is the Hilbert polynomial of $E_0$, which has degree $n$ in $k$. Now, given a real number $\alpha>0$, the \emph{$\alpha$-invariant} of the test configuration $(\mathcal{X},\cL,\cE)$ is
\begin{equation}
\label{eq:extendedfutaki}
F_{\alpha}(\mathcal{X},\cL,\cE) \defeq -F_1(\mathcal{O}_{X_0}) - \alpha \(F_2(E_0) - F_2(\mathcal{O}_{X_0}) \),
\end{equation}
where $F_j(\mathcal{O}_{X_0})$, $0 \leq j \in \ZZ$, are obtained from formula \eqref{eq:defF(W,k)} setting $E_0 = \mathcal{O}_{X_0}$. When there is no confusion we will denote it just by $F_{\alpha}$.

We already have all the ingredients to define the notion of $\alpha$-K-stability. As usual, we define the \emph{degree} of a coherent sheaf over a projective scheme with respect to a given ample invertible sheaf in terms of the first two leading coefficients of its Hilbert polynomial (see~\cite[Definition~1.2.11]{HL} for the precise definition).

\begin{definition}
\label{def:extKstable}
Let $(X,L,E)$ be a triple consisting of a projective scheme of dimension $n$, an ample invertible sheaf and a coherent sheaf of degree zero which is pure of dimension $n$. Let $\alpha$ be a positive real number. Then the triple
$(X,L,E)$ is said to be \emph{$\alpha$-K-semistable} if
\begin{equation}
  \label{eq:extKstable}
  F_{k\alpha}(\mathcal{X},\cL,\cE) \geq 0
\end{equation}
for all integers $k>0$ and all test configurations $(\mathcal{X},\cL,\cE)$
for $(X,L^k,E)$. This triple is \emph{$\alpha$-K-stable} if the inequality~\eqref{eq:extKstable} is strict for all non-trivial $(\mathcal{X},\cL,\cE)$. It is called \emph{$\alpha$-K-polystable} if it is $\alpha$-K-semistable and the inequality~\eqref{eq:extKstable} is always strict unless $(\mathcal{X},\cL,\cE)$ is a product configuration.
\end{definition}

The assumption in Definition~\ref{def:extKstable} that $E$ has degree zero is used to make the formulation of $\alpha$-K-stability simpler. When the sheaf $E$ has non-zero degree, the notion of $\alpha$-K-stability involves a formula for the definition of the invariant $F_\alpha$ which is more complicated than~\eqref{eq:extendedfutaki}. This is an interesting issue to which we wish to return in future work.

As usual, any triple as in Definition~\ref{def:extKstable} satisfies the following chain of implications: $(X,L,E)$ is $\alpha$-K-stable
$\implies$ $(X,L,E)$ is $\alpha$-K-polystable $\implies$ $(X,L,E)$ is $\alpha$-K-semistable.

An \emph{$\alpha$-K-unstable} triple is a triple which is not $\alpha$-K-semistable, and an \emph{$\alpha$-destabilizing test configuration} for $(X,L,E)$ is a test configuration with $F_{\alpha} < 0$. Note that $\alpha$-K-unstability is an open condition in $\alpha$, i.e. if a triple is not $\alpha$-K-semistable then it admits a $\alpha$-destabilizing test configuration $(\mathcal{X},\cL,\cE)$ and, by definition, it also destabilizes the triple for close values of the constant $\alpha$.

\begin{remark}\label{rem:Kstable-L^k}
Note that if $(X,L,E)$, as in Definition~\ref{def:extKstable}, is $\alpha$-K-(semi/poly)stable then for any integer $k>0$ the triple $(X,L^k,E)$ is $(k\alpha)$-K-(semi/poly)stable. Note also that
$$
F_{k\alpha}(\mathcal{X},\cL^k,\cE) = F_\alpha(\mathcal{X},\cL,\cE),
$$
for any test configuration $(\mathcal{X},\cL,\cE)$. This follows simply because the invariants $F_i$ defined in~\eqref{eq:defF(W,k)} satisfy $F_i(L^k_0,E_0) = k^{1-i} F_i(L_0,E_0)$, as $F(L_0^k,E_0,h) = k F(L_0,E_0,hk)$.

\end{remark}

\section{A differential-geometric formula for the $\alpha$-invariant}
\label{sec:ASDiffeformalphainvariant}

In this section we relate the character $\cF_\alpha$ defined in \eqref{eq:alphafutakismooth} and the invariant $F_\alpha$ \eqref{eq:extendedfutaki} of a test configuration when it has smooth central fiber. As a consequence we prove that the existence of solutions to the coupled equations \eqref{eq:CYMeq2} implies the vanishing of the $\alpha$-invariant on product test configurations.

Let $(X,L,E)$ be a triple consisting of a smooth compact polarized variety $(X,L)$ and a holomorphic vector bundle $E$ over $X$ with rank $r$ and degree zero. Given a positive integer $k > 0$, let $(\mathcal{X},\cL,\cE)$ be a test configuration for the triple $(X,L^k,E)$ with smooth central fiber $(X_0,L_0,E_0)$. Associated to $E_0$ we have the holomorphic $\GL(r,\CC)$-principal bundle of frames $(E^c_0,I)$ and the character
$$
\cF_\alpha: \Lie \Aut(E^c_0,I) \to \CC
$$
defined in \eqref{eq:alphafutakismooth}, corresponding to the holomorphic structure $I$, the K\"{a}hler class $c_1(L_0)$, the biinvariant pairing $(\cdot,\cdot) = - \tr$ in the Lie algebra of $\GL(r,\CC)$ and a pair of real constants $\alpha_0$, $k \cdot \alpha_1 \in \RR$. Let $\zeta \in \Lie \Aut(E^c_0,I)$ be the holomorphic vector field generating the induced $S^1$-action on $E_0$ and define the real constant $\alpha = \frac{\alpha_1\pi^2 r}{\alpha_0}$.

\begin{proposition}
\label{prop:extendedalgebraicfutaki}
\begin{equation}
\label{eq:alphafutakirelation}
F_{k\alpha}(\mathcal{X},\cL,\cE) = \frac{-1}{4 \vol(X) \alpha_0}\cF_{\alpha}(\zeta).
\end{equation}
\end{proposition}
\begin{proof}
To simplify the notation, we will drop the zero subscript in the smooth triple corresponding to the central fiber of the test configuration. We have to compute the coefficients of the normalized weight
\[
F(E,m) = \frac{w_m}{mP_L(E,m)} = F_0(E) + m^{-1} F_1(E) + m^{-2} F_2(E) + O(m^{-3})
\]
of the $\CC^*$-action on the sheaf cohomology of $E(m)$ though since the action is algebraic it is enough to compute
the weight of the induced $S^1$-action. We will use the Riemann-Roch formula and its equivariant version
for what we choose metrics $h$ on $L$ and $H$ on $E$ such that $\omega = \frac{\imag}{2\pi}F_h$ is a K\"{a}hler metric on $X$ and $\tr F_H = 0$ (we are allowed to suppose this since the degree of $E$ is zero). Averaging over $S^1$ if necessary we can suppose in addition that $h$ and $H$ are fixed by the $S^1$-action. Recall that we have expansions
\begin{align*}
\frac{P_L(E,m)}{m^n} & = d_0 + m^{-1} d_1 + m^{-2} d_2  + O(m^{-3})\\
\frac{w_m}{m^{n+1}} & = c_0 +  m^{-1} c_1 + m^{-2} c_2  + O(m^{-3}).
\end{align*}
A direct computation shows then that
\begin{equation}
\label{eq:F(E)1}
\begin{split}
F_0(L,E) & = \frac{c_0}{d_0}\\
F_1(L,E) & = \frac{c_1 d_0 - c_0 d_1}{d_0^2}\\
F_2(L,E) & = \frac{- d_1 F_1(E)}{d_0} + \frac{c_2 d_0 - c_0 d_2}{d_0^2}.
\end{split}
\end{equation}
By Riemann-Roch formula we have
\[
P_L(E,m) = \int_X ch(E \otimes L^m) Td(X) = \int_X \tr e^{m \omega
\Id + \frac{i}{2\pi}F_H} Td(X).
\]
Taking representatives $\rho_i \in \Omega^i(X)$ of the Todd
class of $X$, i.e.
\[
Td(X) = [1 + \rho_1 + \rho_2 + \ldots] \in H^*(X,\mathbb{Z}),
\]
we obtain coefficients
\begin{align*}
d_0 & = r \cdot \vol(X)\\
d_1 & = r \cdot \int_X \tr \Lambda \rho_1 \frac{\omega^n}{n!}\\
d_2 & = \frac{1}{2} \int_X \Lambda^2_\omega [\frac{-1}{8\pi^2} \tr
F_H \wedge F_H  + r \cdot \rho_2] \cdot \frac{\omega^n}{n!}.
\end{align*}

To compute the weight we use, as in \cite{D1}, the equivariant Riemann-Roch formula.
By this formula, $w_m$ is given by the coefficient of $t$ in
\[
\int_X ch(E \otimes L^m)(t) \cdot Td(X)(t),
\]
where now the Chern and Todd classes are regarded as equivariant
cohomology classes. Let $\zeta$ be holomorphic vector field generating the $S^1$-action on $E$ and $\eta$ the holomorphic vector field on $X$ that it covers. To make the calculation we use the de Rham model
of equivariant cohomology with the complex $(\Omega_X^*)^{S^1}[t]$ and differential $d + ti_\eta$. Now, writing
\[
\zeta = \theta^\perp_{A_H}(\eta) + a
\]
where $a \in \LieG_H$, we have relations (see~\eqref{eq:infinit-action-connections} and cf.~\cite[Section~2.2]{D1})
\[
d_{A_H} a + i_y F_H = 0; \qquad df = i_\eta \omega
\]
for some $f \in C^{\infty}(X,\RR)$, provided that $\zeta$ and $\eta$
are real holomorphic vector fields, $\eta$ lifts to a holomorphic
vector field on $L$ and $A_H$ is the Chern connection of $H$. Then
the equivariant Chern caracter $ch(E \otimes L^m)(t) \in H^2_{S^1}(X)$ is represented by
\[
ch(t) = \tr e^{m[\omega - t f] + \frac{i}{2\pi}[F_H + t a]}
\]
in this model. Let $Td(t) \defeq 1 + (\rho_1 + t R_1) + (\rho_2 + t
R_2) + \ldots$ be the representative for $Td(X)$, where $R_1$ is a
function on $X$ and $R_2$ is a $2$-form. Then
\begin{align*}
\int_X  & ch(t) \cdot Td(X)(t) = \int_X \tr e^{m (\omega \Id - tf) + \frac{i}{2\pi}(F_H + t a)}
\cdot Td(t) \\
& = \int_X \tr \sum_{l = 0}^{n+1} \frac{1}{l!}(m (\omega \Id - tf) +
\frac{i}{2\pi}(F_H + t a))^l \cdot \\
&\qquad \cdot [1 + (\rho_1 + t R_1) + (\rho_2
+ t R_2)] + O(t^2)\\
& = r \cdot \int_X (m^n \frac{\omega^n}{n!} + m^{n-1}\frac{\omega^{n-1}}{(n-1)!}) \wedge (1 + \rho_1 + \rho_2) \\
& + t \cdot r \cdot \int_X (m^n \frac{\omega^n}{n!} + m^{n-1}\frac{\omega^{n-1}}{(n-1)!}) \wedge (R_1 +
R_2) \\
& + t \cdot \int_X \tr [-f m^{n+1} \frac{\omega^n}{n!}\Id +
m^n(\frac{\imag}{2\pi}a \frac{\omega^n}{n!} - f \frac{\omega^{n-1}}{(n-1)!}\Id) \\
& + m^{n-1}(\frac{- a \Lambda F_H}{4\pi^2} + \frac{\imag a
\omega^{n-1}}{2 \pi (n-1)!} + \frac{f \Lambda^2 F_H^2}{16\pi^2}
\frac{\omega^n}{n!} \\
& - \frac{f \omega^{n-2}}{(n-2)!}\Id)] \wedge (1 + \rho_1 + \rho_2) + m^{n+1}O(t^2,m^{-3}).
\end{align*}
Hence
\begin{align*}
w_m & = \frac{d}{dt}_{|t=0}\int_X ch(L^m \otimes E)(t) \cdot
Td(X)(t) \\
& = m^{n+1} \big{(}-r \cdot \int_X f \frac{\omega^n}{n!}\big{)} + m^{n}
\big{(} \int_X (\tr \frac{\imag}{2\pi} a + r \cdot (R_1 - f \Lambda \rho_1))
\frac{\omega^n}{n!}\big{)} + \\
& + m^{n-1}\big{(} \int_X \tr \frac{\imag a}{2\pi} \cdot
\Lambda(\frac{\imag F_H}{2\pi} + \rho_1 \Id) \frac{\omega^n}{n!} +
r \cdot \int_X \Lambda R_2 \frac{\omega^n}{n!} + \\
& - \frac{1}{2} \int_X f \tr \Lambda^2 (\frac{-1}{8\pi^2} F_H^2 + \rho_2 \Id)\frac{\omega^n}{n!}
\big{)} + m^{n+1}O(m^{-3})
\end{align*}
and the coefficients of the weight are
\begin{equation}
\label{eq:Cs}
\begin{split}
c_0 & = - r \cdot \int_X f \frac{\omega^n}{n!},\\
c_1 & =  \int_X (\tr \frac{\imag}{2\pi} a + r \cdot(R_1 - f \Lambda \rho_1))
\frac{\omega^n}{n!},\\
c_2 & = \int_X \tr \frac{\imag a}{2\pi} \cdot
\Lambda(\frac{\imag F_H}{2\pi} + \rho_1 \Id) \frac{\omega^n}{n!} +
r \cdot \int_X \Lambda R_2 \frac{\omega^n}{n!},\\
& - \frac{1}{2} \int_X f \tr \Lambda^2 (\frac{-1}{8\pi^2} F_H^2 + \rho_2 \Id)\frac{\omega^n}{n!}.
\end{split}
\end{equation}
Note that taking traces in $d_{A_H} a + i_y F_H = 0$ we obtain that
$\tr a$ is constant. Now, if for any function $f \in C^{\infty}(X)$
we define its average as
$$
\widehat{f} = \frac{\int_X f \frac{\omega^n}{n!}}{\vol(X)}
$$
it follows from \eqref{eq:F(E)1} that
\begin{equation}
\label{eq:F(E)2}
\begin{split}
F_0(E) & = - \widehat{f}\\
F_1(E) & = \frac{\imag \tr a}{2 \pi r} + \widehat{R}_1 - \frac{1}{\vol(X)} \int_X f(\Lambda \rho_1 -
\widehat{\Lambda \rho_1})\frac{\omega^n}{n!},\\
F_2(E) & = \frac{-1}{4\pi^2\vol(X) r} \int_X \tr \big{(}a \Lambda F_H  - \frac{1}{4}f(\Lambda^2F_H^2 -
\widehat{\Lambda^2F_H^2})\big{)}\frac{\omega^n}{n!},\\
& - \widehat{\Lambda \rho_1} \cdot F_1(\mathcal{O}_X) + \widehat{\Lambda R_2} - \frac{1}{2\Vol(X)}\int_X
f(\Lambda^2\rho_2 - \widehat{\Lambda^2\rho_2})\frac{\omega^n}{n!},
\end{split}
\end{equation}
where $F_1(\mathcal{O}_X)$ has been obtained from the second formula setting $E = \mathcal{O}_X$. A similar manipulation with the $F_2$ term shows that
\[
F_2(E) - F_2(\mathcal{O}_X) = \frac{-\int_X \tr \big{(}a \Lambda F_H  - \frac{1}{4}f(\Lambda^2F_H^2 -
\widehat{\Lambda^2F_H^2})\big{)}\frac{\omega^n}{n!}}{4\pi^2\vol(X) r}.
\]
Note that $\rho_\omega/2$ also represents the cohomology class of $\rho_1$. Since the previous computation does not depend on the chosen representantive we can set $\rho_\omega/2 = \rho_1$. Then, we now from Lemma $2.2.3$ in \cite{D6} that $\widehat{R}_1 = 0$ so
\[
F_1(\mathcal{O}_X) = \frac{-1}{4\vol(X)} \int_X f(S_\omega - \widehat{S})\frac{\omega^n}{n!}
\]
where we have used the identity $\Lambda_\omega \rho_\omega = S_\omega/2$. Taking coupling constants $\alpha_0$, $k \alpha_1 \in \RR$ for the definition of $\cF_\alpha$ and setting $\alpha = \frac{r\pi^2\alpha_1}{\alpha_0}$ we get
\[
F_{k \alpha}(\mathcal{X},\cL,\cE) = \frac{-1}{4\vol(X)\alpha_0}\cF_\alpha(\zeta).
\qedhere
\]
\end{proof}

As a corollary of the previous result we obtain the following relation between the existence of solutions to the coupled equations \eqref{eq:CYMeq2} and product test configurations for the triple.

\begin{corollary}\label{cor:ASVanishingalphainvariant}
Let $(X,L,E)$ be a triple consisting of a smooth compact polarized variety $(X,L)$ and a holomorphic vector bundle $E$ over $X$ with rank $r$ and degree zero. Suppose that there exists a solution $(\omega,H)$ to the coupled equations \eqref{eq:CYMeq2} with $\omega \in c_1(L)$ and positive coupling constants $\alpha_0$, $\alpha_1 \in \RR$. Let $\alpha = \frac{r\pi^2\alpha_1}{\alpha_0}$. Then, given a positive integer $k > 0$, and a product test configuration $(\mathcal{X},\cL,\cE)$ for the triple $(X,L^k,E)$, we have
$$
F_{k\alpha}(\mathcal{X},\cL,\cE) = 0.
$$
\end{corollary}
\begin{proof}
Note that given $k > 0$, the initial solution $(\omega,H)$ to \eqref{eq:CYMeq2} provides another solution $(k \cdot \omega,H)$ on $(X,L^k,E)$ with coupling constants $(\alpha_0,k\alpha_1)$. Note also that since $(\mathcal{X},\cL,\cE)$ is a product test configuration, the central fiber is isomorphic to $(X,L^k,E)$. The result follows now from Proposition \ref{prop:extendedalgebraicfutaki} and Proposition \ref{prop:futakiobstruction}.
\end{proof}

Proposition \ref{prop:extendedalgebraicfutaki} shows that $F_{k\alpha}$ is, up to multiplicative constant factor, equal to $\cF_\alpha(\zeta)$ and hence to the moment map \eqref{eq:thm-muX} whose zero locus corresponds to the solutions of
the coupled equations \eqref{eq:CYMeq2}. Once more this fits in the general picture of the Kempf-Ness Theorem explained in \secref{sec:KNtheorem} if we think of $F_{k\alpha}$ as the Mumford weight (see \secref{subsec:KNproof} and cf. \eqref{eq:weightmmap1}) of a suitable $1$-parameter subgroup associated to each test configuration. This interpretation of a test configuration is the topic of the next section.

\section{Test configurations as 1-parameter subgroups}
\label{sec:ASTconfigoneparameter}

Test configurations for triples $(X,L,E)$ can be interpreted as $1$-parameter
subgroups of a certain group acting on a relative Quot-scheme over a Hilbert
scheme. This generalizes a similar description of test configurations for
polarized varieties $(X,L)$ in terms of $1$-parameter subgroups of a certain
group acting on a Hilbert scheme, first given by Donaldson~\cite{D4} and
studied further by Ross \& Thomas~\cite{RRT2}. When $\dim_\CC X=1$ and $L$ is
a power of the canonical bundle, Pandharipande~\cite{Pd} used Geometric
Invariant Theory, for the same group action that we consider on a relative Quot-scheme to
compactify the moduli space of pairs consisting of a smooth curve of genus
greater than one and a semistable vector bundle over the curve.

Let $H$ and $P$ be polynomials of the same degree $n$ with positive leading
coefficients.
Let $V_k$ and $W_k$ be complex vector spaces of dimensions $H(k)$ and $P(k)$,
respectively, for each $k$ large enough such that $H(k)$ and $P(k)$ are positive.
Let
\[
  \Hilb=\Hilb^{H'}_{\PP}
\]
be the Hilbert scheme parametrising closed subschemes of the projective space
$\PP=\PP V_k$ with Hilbert polynomial $H'(h)\defeq H(kh)$. Let
$\XXX\subset\PP\times\Hilb$ be the universal family of subschemes of
$\PP$ parametrised by $\Hilb$ and $\pi_H\colon \XXX\to\Hilb$ the composition
of the closed embedding $\XXX\hra\PP\times\Hilb$ and the canonical
projection $\PP\times\Hilb\to\Hilb$.
Then $\XXX$ is a relative projective scheme over $\Hilb$ via the map $\pi_H$,
and $\LLL = \pr^*\cO_{\PP}(1)$ is relatively very ample, where
$\pr\colon\XXX\to\PP$ is the composition of $\XXX\hra\PP\times\Hilb$ and
the projection $\PP\times\Hilb\to\PP$. Furthermore, the fibre
$X_x=\pi_H^{-1}(x)$ over a point $x\in\Hilb$ is the closed subscheme of
$\PP$ parametrised by $x$, and the restriction $L_x = \LLL_{|X_x}$ is very
ample. Let
\[
  \Quot=\Quot^{P'}_{\XXX/\Hilb}(W_k \otimes \LLL^{-1})
\]
be the relative Quot-scheme which parametrises quotient sheaves of the fibres
of $W_k \otimes \LLL^{-1}$ over the fibres of $\pi_H$ with Hilbert
polynomial $P'(h) \defeq P(kh)$ with respect to $\LLL$. Let
\[
  \varphi\colon W_k \otimes q^*\LLL \surj \EEE
\]
be the universal family of quotient sheaves parametrised by $\Quot$, where
$q$ is the projection $\XXX\times\Quot\to\XXX$.
Note that there is a natural projective map $\kappa \colon \Quot \to
\Hilb$.
Then $\EEE$ is a coherent sheaf over
$\XXX\times\Quot$, flat over $\Quot$, and the fibre of $\varphi$ over a
point $e\in\Quot$ is the quotient sheaf $\varphi_e\colon W_k\otimes L_x
\surj E_e$ parametrised by $e$, where $x=\kappa(e)$.

The group $\GL(V_k)$ acts canonically on $\PP$, inducing an action on
$\Hilb$ and hence on $\PP\times\Hilb$. As the latter preserves $\XXX$,
$\GL(V_k)$ also acts on $\Quot$ and $\pi_H$ is clearly
$\GL(V_k)$-equivariant. As $\Quot$ also has a canonical $\GL(W_k)$-action
and both actions commute, the product $G_k = \GL(V_k) \times \GL(W_k)$ also
acts on $\Quot$.

Let $(X,L,E)$ be a triple consisting of a projective scheme, an ample invertible
sheaf and a coherent sheaf, such that the Hilbert polynomials of $(X,L)$ and
$E$ with respect to $L$ are $H$ and $P$, respectively, that is,
\[
  H(k)=\chi(X, L^k) \text{ and } P(k)=\chi(X, E\otimes L^k) \text{ for all $k$.}
\]
Let $k$ be large enough that $L^k$ is very ample,
\[
  H^i(X,L^k)=0 \text{ and } H^i(X,E\otimes L^k)=0 \text{ for all $i>0$.}
\]
and $E\otimes L^k$ is globally generated. Then there are isomorphisms
\begin{equation}
\label{eq:Tconfigoneparameter-1}
  H^0(X,L^k) \cong V_k^*
  \text{ and }
  H^0(X,E\otimes L^k) \cong W_k^*.
\end{equation}

\begin{proposition}
Under these assumptions, there is a bijection between the set of 1-parameter subgroups of $G_k$, up to the choice of a pair of isomorphisms~\eqref{eq:Tconfigoneparameter-1}, and the set of test  configurations $(\mathcal{X}, \cL, \cE)$ for $(X,L^k,E)$ with $\cL$ very ample, $\cE\otimes\cL$ globally generated,
$$
R^i\pi_*\cL = R^i\pi_*(\cE\otimes\cL)=0
$$
for $i>0$ and surjective evaluation map $\pi_*(\cE\otimes\cL)\otimes\cL^{-1}\to\cE$, up to the choice of isomorphisms $H^0(X_0,L_0)\cong V_k^*$ and $H^0(X_0,E_0\otimes L_0)\cong W_k^*$
(with $\pi$ as in Definition~\textup{\ref{def:testconfiguration}}).
\end{proposition}
\begin{proof}
To see this, fix the isomorphisms~\eqref{eq:Tconfigoneparameter-1}. Then the
left-hand isomorphism in~\eqref{eq:Tconfigoneparameter-1} and the linear
system $\lvert L^k\rvert$ determine a closed embedding
\begin{equation}
\label{eq:Tconfigoneparameter-2}
  i_x\colon X\hra \PP.
\end{equation}
Since $E\otimes L^k$ is globally generated, the evaluation map $H^0(E\otimes
L^k)\otimes L^{-k} \surj E$ is an epimorphism which, together with the
right-hand isomorphism in~\eqref{eq:Tconfigoneparameter-1}, determine another
one
\begin{equation}
\label{eq:Tconfigoneparameter-3}
  \varphi_e \colon W_k \otimes L^{-k} \surj E.
\end{equation}
Since $\cO_{\PP}(1)_{|X}=L^k$, the embedding $i_x$
of~\eqref{eq:Tconfigoneparameter-2} and the epimorphism $\varphi_e$
of~\eqref{eq:Tconfigoneparameter-3} define points
\begin{equation}
\label{eq:Tconfigoneparameter-4}
  x\in\Hilb, \quad e\in\Quot, \quad \text{with } \kappa(e)=x,
\end{equation}
and so we have identifications $X = X_x, L^k = L_x$ and $E=E_e$.
Note that each pair of isomorphisms~\eqref{eq:Tconfigoneparameter-1} defines a
point $e$ as in~\eqref{eq:Tconfigoneparameter-4} and the set of all points
defined by the set of all the isomorphisms~\eqref{eq:Tconfigoneparameter-1} is
an orbit under the $G_k$-action on $\Quot$.

Let $\lambda$ be a 1-parameter subgroup of $G_k$, i.e. a morphism of groups
\[
  \lambda=(\lambda_0,\lambda_1) \colon \CC^* \to G_k,
\]
where $\lambda_0\colon \CC^*\to\GL(V_k)$, $\lambda_1 \colon \CC^* \to
\GL(W_k)$. Let $f_t=\lambda_0(t)\cdot x$, $g_t=\lambda(t)\cdot e$, for
$t\in\CC^*$. As $\Hilb$ and $\Quot$ are proper over $\CC$,
\begin{equation}
\label{eq:Tconfigoneparameter-6}
  f_0= \lim_{t\to 0} f_t \in \Hilb, \quad
  g_0= \lim_{t\to 0} g_t \in \Quot
\end{equation}
exist, and $f_t$ and $g_t$ extend to morphisms $f\colon \CC\to \Hilb \colon
t\mapsto f_t$ and $g\colon \CC\to \Quot \colon t\mapsto g_t$, which satisfy
$f=\kappa\circ g$. 
Using now $f$ to pull back the universal family $\XXX \subset \PP
\times \Hilb$ and the invertible sheaf $\LLL$, and $g$ to pull back the
universal quotient sheaf $\EEE$, we obtain a test configuration
\[
  (\mathcal{X}, \cL, \cE) = (\XXX \times_{\Hilb} \CC, f^*\LLL, g^*\EEE)
\]
for $(X,L^k,E)$. By construction, $\cL$ is very ample over the canonical
projection $\pi\colon \mathcal{X} \to\CC$, $R^i\pi_*\cL =
R^i\pi_*(\cE\otimes\cL)=0$ for $i>0$ and the natural evaluation map
$\pi_*(\cE\otimes\cL)\otimes\cL^{-1} \to \cE$ is surjective.

For the converse, let $(\mathcal{X}, \cL, \cE)$ be a test configuration for
$(X,L^k,E)$, with $\cL$ very ample over $\pi\colon \mathcal{X} \to\CC$,
$$
R^i\pi_*\cL = R^i\pi_*(\cE\otimes\cL)=0
$$
for $i>0$ and surjective evaluation
map $\pi_*(\cE\otimes\cL)\otimes\cL^{-1} \to \cE$. These conditions clearly
imply that $\pi_*\cL$ and $\pi_*(\cE \otimes \cL)$ are vector bundles of ranks
$H(k)$ and $P(k)$, respectively. Since the central fibre $(X_0,L_0,E_0)$ has
an induced $\CC^*$-action, there are also induced $\CC^*$-actions on
$H^0(X_0,L_0)$ and $H^0(X_0,E_0\otimes L_0)$, so each choice of isomorphisms
\begin{equation}
\label{eq:Tconfigoneparameter-5}
H^0(X_0,L_0)\cong V_k^* \text{ and } H^0(X_0,E_0\otimes L_0)\cong W_k^*
\end{equation}
defines a 1-parameter subgroup $\lambda = (\lambda_0, \lambda_1) \colon \CC^*
\to G_k$.

To see that $(\mathcal{X}, \cL, \cE)$ is recovered from $\lambda$, note that
the assumptions on $\cL$ provide a $\CC^*$-equivariant closed embedding
$\mathcal{X} \subset \PP_\CC(\pi_*\cL)$ of schemes over $\CC$ with
$\cL\cong\cO(1)_{|\mathcal{X}}$, where $\cO_{\PP_\CC}(1)$ is the standard
invertible sheaf on the projective space bundle $\PP_\CC(\pi_*\cL)$ over $\CC$
associated to $\pi_*\cL$.
Now, $\pi_*\cL$ and $\pi_*(\cE \otimes \cL)$ are trivial as vector bundles (as
they are defined over $\CC$), with possibly non-trivial $G_k$-actions and
central fibres
\[
  (\pi_*\cL)_0=H^0(X_0,L_0) \text{ and }
  (\pi_*(\cE\otimes\cL))_0=H^0(X_0,E_0\otimes L_0),
\]
since $\cL$ and $\cE$ are flat over $\CC$. Hence two trivialisations of these
vector bundles together with the isomorphisms~\eqref{eq:Tconfigoneparameter-5}
define isomorphisms
\begin{equation}
\label{eq:Tconfigoneparameter-7}
  \pi_*\cL\cong V_k\times\CC, \quad \pi_*(\cE\otimes\cL) \cong W_k\times\CC,
\end{equation}
of $\CC^*$-linearised vector bundles, where $\CC^*$ acts on $V_k$ and $W_k$
via $\lambda_0$ and $\lambda_1$, and $V_k\times\CC$ and $W_k\times\CC$ are
vector bundles over $\CC$ with their diagonal $\CC^*$-actions. Then the
left-hand isomorphism~\eqref{eq:Tconfigoneparameter-7} induces an isomorphism
of $\CC^*$-linearised varieties $(\PP_\CC(\pi_*\cL),\cO_{\PP_\CC}(1)) \cong
(\PP\times\CC,\pr^*\cO_\PP(1))$ over $\CC$, where $\pr\colon
\PP\times\CC\to\PP$ is the canonical projection. This isomorphism together
with the closed embedding $\mathcal{X} \subset \PP_\CC(\pi_*\cL)$ provide a
$\CC^*$-equivariant map $f\colon \CC\to\Hilb\colon t\mapsto f_t$, such that
\[
  (\mathcal{X},\cL) \cong f^*(\XXX,\LLL) \text{ as $\CC^*$-linearised schemes.}
\]
Furthermore, the right-hand isomorphism~\eqref{eq:Tconfigoneparameter-7}
together with the surjective evaluation map $\pi_*(\cE\otimes\cL) \otimes
\cL^{-1} \surj \cE$ provide a $\CC^*$-equivariant surjective morphism
$W_k\otimes\cL^{-1} \surj \cE$, which induces a $\CC^*$-equivariant map
$g\colon \CC\to \Quot\colon t\mapsto g_t$ such that $f=\kappa \circ g$ and
\[
  \cE \cong g^*\EEE \text{ as $\CC^*$-linearised sheaves.}
\]
In particular, $(X,L^k,E)\cong (X_x,L_x,E_e)$ where $x= f_1$ and $e=g_1$, so
\[
  f_t = \lambda_0(t) \cdot x \text{ and } g_t=\lambda_1(t) \cdot e
  \text{ for all $t\in\CC^*$,}
\]
because $f$ and $g$ are $\CC^*$-equivariant, and therefore $(\mathcal{X}, \cL,
\cE)$ is recovered from $\lambda$ because $\Hilb$ and $\Quot$ are projective
and hence separated over $\CC$, so the limits~\eqref{eq:Tconfigoneparameter-6}
are unique. This concludes the construction of the bijection between
1-parameter subgroups of $G_k$ and test configurations.
\end{proof}

\begin{remark}\label{rem:lift-test-config}
As already mentioned, the above correspondence between test configurations for a triple $(X,L^k,E)$ and 1-parameter subgroups of $G_k$ generalizes a similar correspondence between test configurations for a pair $(X,L^k)$ and 1-parameter subgroups of $\GL(V_k)$ (see~\cite{D4,RRT2}). Let $(X,L,E)$ be as in Definition~\ref{def:testconfiguration} and $k$ be large enough that $L^k$ is very ample, $H^i(X,L^k)=H^i(X,E\otimes L^k)=0$ for all $i>0$ and $E\otimes L^k$ is globally generated. Let $(\mathcal{X},\cL)$ be a test configuration for $(X,L^k)$. Then we can lift $(\mathcal{X},\cL)$ to a test configuration $(\mathcal{X},\cL,\cE)$ for $(X,L^k,E)$. Furthermore, if $(\mathcal{X},\cL)$ is a product or a trivial configuration, then it can be respectively lifted to a product or a trivial configuration for $(X,L^k,E)$. This follows simply because a test configuration for $(X,L^k)$ corresponds to a 1-parameter subgroup $\lambda_0 \colon \CC^* \to \GL(V_k)$ and so a possible lift is the one associated to the 1-parameter subgroup $\lambda=(\lambda_0,\Id_{W_k}) \colon \CC^* \to G_k$. When $(\mathcal{X},\cL)$ is a product test configuration then $\lambda_0$ fixes the point $x \in \Hilb$, defined by $(X,L^k)$ as before. Then $\lambda$ fixes $\epsilon \in \Quot$, defined by $(X,L^k,E)$, and so the associated test configuration is a product. The statement about trivial configurations can now be readily checked.
\end{remark}

We will now see that $\alpha$-K-stability reduces to K-stability of the underlying polarised variety in the limit case $\alpha=0$ (note that Definition~\ref{def:extKstable} also makes sense for $\alpha=0$). We will also prove that $\alpha$-K-semistability for all $0<\alpha\ll 1$ implies $K$-semistability of the polarised variety $(X,L)$. In the following proposition, K-(semi/poly)stability for a polarised variety (see~\secref{subsec:cscKalgKstab}) is defined as first introduced by Donaldson~\cite{D4} (after Tian~\cite{T2}), but in the slightly different formulation of Ross\&Thomas~\cite[Definition~2.4]{RRT1}.

\begin{proposition}\label{prop:stab-alpha=0}
A triple $(X,L,E)$ is $\alpha$-K-(semi/poly)stable, with $\alpha=0$, if and only if the polarised variety $(X,L)$ is algebraically K-(semi/poly)stable. If $(X,L,E)$ is $\alpha$-K-semistable for all $0<\alpha\ll 1$, then $(X,L)$ is algebraically K-semistable.
\end{proposition}

\begin{proof}
Let $(\mathcal{X},\cL,\cE)$ be a test configuration for $(X,L^k,E)$. Note that, in the limit case $\alpha=0$, $F_{\alpha}(\mathcal{X},\cL,\cE)$ given by~\eqref{eq:extendedfutaki} is the Futaki invariant $-F_1(\mathcal{X},\cL)$, as defined by Donaldson~\cite[\S 2.1]{D4} but with opposite sign. The ``if'' part follows immediately from this observation. To prove the ``only if'' part for K-semistability let $(\mathcal{X},\cL)$ be a test configuration for $(X,L^k)$. Taking $l \gg 0$ large enough that Remark~\ref{rem:lift-test-config} applies $(\mathcal{X},\cL^l)$ lifts to a test configuration $(\mathcal{X},\cL^l,\cE)$ for $(X,L^{lk},E)$. Then, in the limit case $\alpha=0$,
$$
F_{lk\alpha}(\mathcal{X},\cL^l,\cE) = - F_1(\mathcal{X},\cL^l) = - F_1(\mathcal{X},\cL),
$$
and so the ``only if'' part holds for K-semistability. The ``only if'' parts for K-stability and K-polystability now follow from the results in Remark~\ref{rem:lift-test-config} about lifts of product and trivial configurations.

To prove the second part of the statement suppose that $(X,L)$ is not K-semistable. Then there exists a destabilizing test configuration $(\mathcal{X},\cL)$ for $(X,L^k)$ and taking $l\gg 0$ we can lift $(\mathcal{X},\cL^l)$ to a test configuration for $(X,L^{lk},E)$. Therefore, by definition, the corresponding $\alpha$-invariant is negative for $0 <\alpha\ll 1$ and so $(X,L,E)$ is $\alpha$-K-unstable for these values of $\alpha$.
\end{proof}

\section{Base-preserving configurations and stable vector bundles}
\label{sec:ASBptestconfig}

In this section, we show that $\alpha$-K-stability of a triple implies Mumford--Takemoto stability of the coherent sheaf. First, we use a special type of test configurations to reformulate the usual definition of Mumford stability for a coherent sheaf (see~\secref{subsec:HYMslope}) in terms of the invariant $F_2$ defined in~\eqref{eq:defF(W,k)}.

Let $X$ be an $n$-dimensional projective scheme with an ample invertible sheaf $L$. Let $E$ be a coherent sheaf on $X$
which is pure of dimension $n$. We will use the following special type of test configurations. A \emph{base-preserving
test configuration} for $(X,L,E)$ is a test configuration $(\mathcal{X},\cL,\cE)$ such that $(\mathcal{X},\cL)$ is a
trivial configuration for $(X,L)$, i.e. the product of $(X,L)$ and $\CC$, where $\CC^*$ acts trivially on $(X,L)$ and
by multiplication on $\CC$ (see Remark~\ref{rem:lift-test-config} for details).

\begin{proposition}\label{prop:Bptestconfig}
  Let $E$ be a coherent sheaf over an $n$-dimensional projective scheme $X$
  with an ample invertible sheaf $L$. Suppose that $E$ is pure of dimension
  $n$. Then $E$ is Mumford semistable if and only if
\begin{equation}
  \label{eq:Bptestconfig}
  F_2(E_0) \leq 0,
\end{equation}
for all base-preserving test configurations $(\mathcal{X},\cL,\cE)$ for
$(X,L^k,E)$, for all integers $k>0$. The sheaf $E$ is Mumford stable if and only if the
inequality~\eqref{eq:Bptestconfig} is strict for all
non-trivial base-preserving configurations $(\mathcal{X},\cL,\cE)$, and
Mumford polystable if and only if the inequality~\eqref{eq:Bptestconfig} is
strict unless $(\mathcal{X},\cL,\cE)$ is a product base-preserving
configuration.
\end{proposition}

\begin{proof}
It is helpful to observe that a base-preserving test configuration is
equivalent to the data of a $\CC^*$-linearised coherent sheaf $\cE$ over
$X\times\CC$ (where $\CC^*$ acts trivially on $X$ and by multiplication on
$\CC$), which is flat over $\CC$ and has fibres $E_t\cong E$ for
$t\in\CC\setminus\{0\}$. This is a product configuration precisely when
$\cE=\pr^*E$, where $\pr\colon X\times\CC\to X$ is the canonical projection,
for a $\CC^*$-action on $E$ covering the trivial $\CC^*$-action on $X$, and it
is a trivial configuration if $\CC^*$ acts on $E$ by scalar multiplication.

For the ``only if'' part, let $F \subset E$ be a subsheaf and $P_L(F,k) = \chi(F
\otimes L^k)$ its Hilbert polynomial. Then there exists a base-preserving configurations for $(X,L,E)$ whose central
fibre is $E_0 = F \oplus E/F$ where $t \in \CC^*$ acts as $\textrm{diag}(t,1)$, i.e. with weight $1$ on $F$ and weight $0$ on $E/F$ (see e.g \cite[p~43]{RT}). Recall that, since $E$ is pure of dimension $n = \dim
X$, we have
\[
P_L(F,k) = k^n d_0(F) + k^{n-1} d_1(F) + \ldots,
\]
for all integer $k$ with $d_0(F) > 0$, and that the slope $\mu_L(F)$ of $F$ is defined as the quotient $\frac{d_1(F)}{d_0(F)}$. Computing now the normalized weight $F(E_0,k)$ for this degeneration, we obtain
\[
F(E_0,k) = \frac{P_L(F,k)}{kP_L(E,k)} = \frac{d_0(F)}{d_0(E)}\(k^{-1} + k^{-2}\(\mu_L(F) -
\mu_L(E)\) + O(k^{-3})\)\!.
\]
In other words, $F_2(E_0)$ equals the difference of slopes $\mu_L(F) - \mu_L(E)$, up to a multiplicative positive
constant, so the condition that the weak/strict inequality holds in~\eqref{eq:Bptestconfig} for all non-trivial base-preserving test configurations implies that $E$ is Mumford semistable/stable.

For the converse, let $\cE$ be a $\CC$-flat $\CC^*$-linearised coherent sheaf over $X\times\CC$. Choosing $k \gg 0$, we obtain a surjective map
\begin{equation}
\label{eq:sujectsheaf}
\epsilon: W_{k} \otimes L^{-k} \twoheadrightarrow E
\end{equation}
as in~\secref{sec:ASTconfigoneparameter}, where $W_{k}^* \cong H^0(X,E\otimes L^{k})$. We can choose $k$ such that \eqref{eq:sujectsheaf} is surjective for any fiber $\cE_t$ of $\cE$, and so the $1$-parameter subgroup $\lambda_k : \CC^* \to GL(W_k)$ determined by the action on the central fiber $E_0$ splits $W_k^* \cong H^0(X,E_0\otimes L^{k})$ into eigenspaces
$$
W_k = \bigoplus_{n \in \ZZ} V_n,
$$
where $\CC^*$ acts with weight $n$ on $V_n$ and only a finite number of the $V_n$ are non trivial. The eigenspaces $V_n$ give a weight filtration $V_{\leq n} = \oplus_{j \leq n} V_j$ of $W_k$. Their images $F_{\leq n}$ under \eqref{eq:sujectsheaf} give a filtration of $E$ such that $E_0$ is the graded sheaf (see~\cite[Lemma~4.4.3]{HL})
\[
E_0 \cong \bigoplus_{n \in \ZZ} F_{\leq n}/F_{\leq n-1}.
\]
The weight $w_l$ of the $\CC^*$-action on $\det H^0(X,E_0\otimes L^l)$ for $l \gg k$ is then given by
\begin{align*}
w_l & = \sum_{n \in \ZZ} (-n) P_L(F_{\leq n}/F_{\leq n-1},l)\\
& = \sum_{n \in \ZZ} (-n) \(P_L(F_{\leq n},l) - P_L(F_{\leq n-1},l)\)\\
& = \sum_{n \in \ZZ} P_L(F_{\leq n},l)
\end{align*}
(see~\cite[Lemma~4.4.4]{HL}). Thus, the normalized weight $F(E_0,l)$ equals
\begin{align*}
F(E_0,l) & = \sum_{n \in \ZZ} \frac{P_L(F_{\leq n},l)}{l \cdot P_{L}(E,l)}\\
& = \sum_{n \in \ZZ} \frac{d_0(F_{\leq n})}{d_0(E)}\(l^{-1} + l^{-2}\(\mu_L(F_{\leq n}) - \mu_L(E)\) + O(l^{-3})\)\!.
\end{align*}
This equality together with equation \eqref{eq:defF(W,k)} imply that if $E$ is Mumford--Takemoto semistable/stable, then the weak/strict inequality holds in~\eqref{eq:Bptestconfig} for all non-trivial base-preserving test configurations $(\mathcal{X},\cL,\cE)$.

Finally, we will prove the result about polystability. If the equality in~\eqref{eq:Bptestconfig} holds only for product base-preserving configurations, then we have seen that $E$ must be Mumford--Takemoto semistable. If $E$ is strictly semistable, there exists a proper subsheaf $E'\subset E$ with $\mu_L(F) = \mu_L(E)$. Since the degeneration associated to $F$ has zero weight, $E \cong F \oplus E/F$ and therefore we conclude by iteration that $E$
must be Mumford polystable. For the converse, suppose that $E$ is Mumford polystable, i.e.
\[
  E \cong \bigoplus_i F_i
\]
where $F_i$ are Mumford stable sheaves with $\mu(F_i) = \mu(E)$. Given a base-preserving configuration for which we have an equality in~\eqref{eq:Bptestconfig}, the corresponding central fibre must be equal to
\[
  E_0 \cong \bigoplus_j F_j
\]
i.e. $E_0\cong E$, so this configuration must be a product configuration.
\end{proof}

\begin{corollary}\label{cor:Bptestconfig}
Let $(X,L,E)$ be a triple consisting of a projective scheme of dimension $n$, an ample invertible sheaf and a coherent sheaf, pure of dimension $n$. Then we have the following implications for all $\alpha>0$: if $(X,L,E)$ is $\alpha$-K-(semi/poly)stable, then $E$ is Mumford (semi/poly)stable with respect to $L$.
\end{corollary}

\begin{proof}
Immediate from Proposition~\ref{prop:Bptestconfig} and the fact that for any base-preserving configuration $(\mathcal{X},\cL,\cE)$ for $(X,L^k,E)$, the central fibre of $(\mathcal{X},\cL)$ is $(X_0,L_0) = (X,L^k)$ with the trivial $\CC^*$-action, so $F_i (\cO_{X_0})=0$ for all $i>0$ and hence~\eqref{eq:extendedfutaki} reduces to
\[
  F_{\alpha}(\mathcal{X},\cL,\cE) = - \alpha F_2(E_0),
\]
whose sign does not depend on the actual value of $\alpha>0$.
\end{proof}

Another consequence of Proposition~\ref{prop:Bptestconfig} is the following relation between the existence of solutions to the coupled equations \eqref{eq:CYMeq2} and base-preserving configurations for the triple.

\begin{corollary}\label{cor:ASnonnegativealphainvariant}
Let $(X,L,E)$ be a triple consisting of a smooth compact polarized variety $(X,L)$ and a holomorphic vector bundle $E$ over $X$ with rank $r$ and degree zero. Suppose that there exists a solution $(\omega,H)$ to the coupled equations \eqref{eq:CYMeq2} with $\omega \in c_1(L)$. Then, given a positive integer $k > 0$, any $\alpha > 0$ and a base-preserving test configuration $(\mathcal{X},\cL,\cE)$ for the triple $(X,L^k,E)$, we have
$$
F_{k\alpha}(\mathcal{X},\cL,\cE) \geq 0,
$$
with equality only if the test configuration is a product.
\end{corollary}
\begin{proof}
If there exists a solution to \eqref{eq:CYMeq2} then $E$ is Mumford--Takemoto polystable by the Hitchin--Kobayashi correspondence for vector bundles and so the statement holds from Proposition~\ref{prop:Bptestconfig} as in Corollary \ref{cor:Bptestconfig}.
\end{proof}

\section{A conjecture for degree zero holomorphic vector bundles}
\label{sec:ASTheconjecture}

In this section we formulate a conjecture relating the notion of $\alpha$-K-stability for triples $(X,L,E)$ defined in section \secref{sec:ASAlgebraicstb} with the existence of solutions to the coupled equations \eqref{eq:CYMeq2}. We will suppose that $(X,L)$ is a smooth compact polarized variety and $E$ is a holomorphic vector bundle over $X$ with rank $r$ and degree zero. Recall from Proposition \ref{prop:extendedalgebraicfutaki} that the numerical invariant $F_\alpha$ associated to any test configuration with smooth central fiber is, up to multiplicative constant factor, equal to the moment map \eqref{eq:thm-muX} at the generator of the $S^1 \subset \CC^*$ action and  that the zero locus of this moment map corresponds to the solutions of the coupled equations. Once more this fits in the general picture of the Kempf-Ness Theorem explained in \secref{sec:KNtheorem} if we view a test configuration as a $1$-parameter subgroup (see \secref{sec:ASTconfigoneparameter}) and the associated $\alpha$-invariant as the corresponding Mumford weight (see \secref{subsec:KNproof} and cf. \eqref{eq:weightmmap1}). This fact motivates Definition~\ref{def:extKstable} and leads to think that the notion of $\alpha$-K-stability should play an important role in the existence problem for the coupled equations \eqref{eq:CYMeq2}.

Concerning direct relations between the equations \eqref{eq:CYMeq2} and the notion of $\alpha$-K-stability, we have proved in Corollary~\ref{cor:ASVanishingalphainvariant} that if there exists a solution to the coupled equations \eqref{eq:CYMeq2} with positive coupling constants $\alpha_0$, $\alpha_1 \in \RR$, then $F_\alpha = 0$ for any product test configuration with $\alpha = \frac{r\pi^2\alpha_1}{\alpha_0}$. Under the same hypothesis, we have proved in Corollary~\ref{cor:ASnonnegativealphainvariant} that $F_\alpha \geq 0$ for any base preserving test configuration, with equality only if the test configuration is a product. We have also proved in Proposition \ref{prop:stab-alpha=0} that in the limit case $\alpha \to 0$ the notion of $\alpha$-K-stability is equivalent to the notion of algebraic K-stability of the underlying polarized manifold.

These facts provide supporting evidence for the following conjecture.

\begin{conjecture}\label{conjecture:alphaKpolystab}
If there exists a K\"{a}hler metric $\omega \in c_1(L)$ and a Hermitian metric $H$ on $E$ satisfying the coupled equations \eqref{eq:CYMeq2} with coupling constants $\alpha_0,\alpha_1$ then the triple $(X,L,E)$ is $\alpha$-K-polystable.
\end{conjecture}

\begin{remark}
The previous conjecture for the limit case $\alpha \to 0$, i.e. for the cscK problem, has been verified mainly thanks to the work of Tian \cite{T2}, Donaldson \cite{D8}, Mabuchi \cite{Mab2} and Stoppa \cite{Stp}.
\end{remark}

%

This conjecture is a difficult problem even in complex dimension $1$. Note that in this case the coupled equation are a system in separated variables independent of the coupling constants. In fact, they reduce to the HYM equation for the bundle due to existence of cscK metrics in any compact Riemann surface. When $\dim_\CC X = 1$, Conjecture~\ref{conjecture:alphaKpolystab} reduces to prove that Mumford--Takemoto stability of the bundle implies $\alpha$-K-stability of the triple for any $\alpha > 0$. The algebro-geometric approach to the proof of this fact seems to be very difficult in the light of previous work of R. Thomas and J. Ross in \cite{RRT2}. A more likely approach, valid for arbitrary dimension of $X$, seems to be provided by Donaldson's methods in \cite{D8} or Mabuchi's methods in \cite{Mab2}. The main issue to adapt the arguments in \cite{D8} or \cite{Mab2} will be to prove that $\alpha$-K-stability arises as a suitable assymtotic version of a Chow-type stability notion in the relative quot scheme considered in Section \secref{sec:ASTconfigoneparameter}. The verification of this stability issue will provide a strong evidence of the fact that \eqref{eq:CYMeq2} are gauge theoretic equations related to the algebro-geometric moduli problem for triples $(X,L,E)$. As an immediate future project we will try to prove that the existence of solutions to \eqref{eq:CYMeq2} implies the $\alpha$-K-semistability of the triple adapting Donaldson's methods in \cite{D8}. This conjecture was first verified for the limit case $\alpha = 0$, i.e. for the cscK problem, after the work of S. Zhang, H. Luo and S.K. Donaldson.

Concerning the converse of the statement in Conjecture~\ref{conjecture:alphaKpolystab}, as remarked by J. Ross and R. Thomas~\cite{RRT1}, a recent example in \cite{ACGT} suggests that the algebraic notion of K-stability defined in~\cite{D4} for a polarised variety may not be sufficient to guarantee the existence of a cscK metric and so a stronger notion may be required. Since the $\alpha$-K-stability of a triple $(X,L,E)$ for $\alpha = 0$ coincides with the $K$-stability of $(X,L)$, it seems equally necessary to strengthen the $\alpha$-K-stability condition introduced in Definition~\ref{def:extKstable}, so that the converse of the statement in Conjecture~\ref{conjecture:alphaKpolystab} may actually hold. This can be done exactly as in \cite{RRT1} for polarised varieties, by requiring the inequality~\eqref{eq:extKstable} in Definition~\ref{def:extKstable} for configurations $(\mathcal{X},\cL,\cE)$, where $\cL$ may now be an ample $\RR$-divisor.

To finish this chapter, let us analyze two particulary simple situations, when $X$ is a Riemann surface and when $E$ is the trivial bundle of rank $r$ over a toric complex surface, where the $\alpha$-K-polystability of the triple implies the existence of solutions to the equations \eqref{eq:CYMeq2}. Recall from \secref{sec:ASBptestconfig}, Corollary \ref{cor:Bptestconfig}, that $\alpha$-K-(semi/poly)stability implies Mumford (semi/poly)stability for any $\alpha > 0$. Due to the Theorem of Donaldson-Uhlembeck-Yau, the latter stability condition is equivalent to the existence of solutions to the Hermitian-Yang-Mills equation~\eqref{eq:HYM} for each K\"{a}hler metric in $c_1(L)$.

Suppose first that $X$ is a Riemann surface. Given a positive real constant $\alpha > 0$, if the triple $(X,L,E)$ is $\alpha$-K-polystable in the sense of Definition \ref{def:extKstable} then there exists a solution to the HYM equation for any K\"{a}hler metric $\omega \in c_1(L)$. Chose $\omega$ to be the unique cscK metric in this K\"{a}hler class and $H$ the associated Einstein--Hermitian metric on $E$. Then $(\omega,H)$ is a solution to \eqref{eq:CYMeq2} since the term $\Lambda_\omega^2(F_H \wedge F_H)$ in \eqref{eq:CYMeq2} vanishes. We conclude that if $\dim_{\CC} X = 1$ and the triple $(X,L,E)$ is $\alpha$-K-polystable with $\alpha > 0$, then there exists a solution to \eqref{eq:CYMeq2}, and so the converse of the statement in Conjecture~\ref{conjecture:alphaKpolystab} is verified.

Suppose now that $E \cong \oplus^r \cO_X$ is the trivial bundle and $(X,L,E)$ is $\alpha$-K-polystable for some $\alpha \geq 0$. We claim that $(X,L)$ is algebraically K-polystable. Let $(\mathcal{X},\cL)$ be a test configuration for $(X,L^k)$. Then, by its very definition, this test configuration defines one for the triple $(X,L^k,\cO_X)$ and so one for $(X,L,\oplus^r \cO_X)$ that we denote by $(\mathcal{X},\cL,\cE)$. Then, if $(X_0,L_0,E_0)$ is its central fiber we have that $E_0 \cong \oplus^r \cO_{X_0}$ as $\CC^*$-linearized sheaves over $X_0$. Then it is immediate to check that the normalized weight $F(E_0,k)$ satisfies
$$
F(E_0,k) = \frac{w_k(E_0)}{k P_{L_0}(E_0,k)} = \frac{r w_k(\cO_{X_0}) }{k r P_{L_0}(\cO_{X_0},k)} = F(\cO_{X_0},k)
$$
and so we have an equality
$$
0 \leq F_{lk\alpha}(\mathcal{X},\cL^l,\cE) = - F_1(\mathcal{X},\cL^l) = - F_1(\mathcal{X},\cL).
$$
Thus, $(X,L)$ is algebraically K-semistable. If the $\alpha$-invariant vanishes then $(\mathcal{X},\cL,\cE)$ is a product and so is $(\mathcal{X},\cL)$. We conclude that $(X,L)$ is K-polystable as claimed. Suppose in addition that the Yau--Tian--Donaldson conjecture is satisfied for $(X,L)$ (e.g. $(X,L)$ is a toric polarized surface or $L^k \cong K_X$, $0 \geq k \in \ZZ$ \cite{Au,D9}). Then, by the previous reasoning, $\alpha$-K-polystability of $(X,L,\oplus^r \cO_X)$ implies the existence of a cscK metric $\omega \in c_1(L)$. Thus, the pair $(\omega,A)$, with $A$ any flat connection on $\oplus^r \cO_X$, satisfies the coupled equations \eqref{eq:CYMeq2} for any $\alpha$ and we conclude the following. If the triple $(X,L,\oplus^r\cO_X)$ is $\alpha$-K-polystable and $(X,L)$ satisfies the Yau--Tian--Donaldson conjecture, then there exists a solution to the coupled equations \eqref{eq:CYMeq2}. If $X$ is a toric complex surface, Donaldson has proved the Yau--Tian--Donaldson conjecture \cite{D9} and so $\alpha$-K-polystability of $(X,L,\oplus^r\cO_X)$ implies the existence of solutions to \eqref{eq:CYMeq2}.

\chapter{\bf Examples and cscK metrics on ruled manifolds}
\label{chap:examples}

This chapter contains some examples of solutions to \eqref{eq:CYMeq00}. In \secref{sec:example4} we also discuss how the existence of solutions to the limit case $\alpha_0 = 0$ can be applied, using results of Y. J. Hong in \cite{Ho2}, to obtain cscK metrics on ruled manifolds.

In  \secref{sec:example1} we discuss the case of vector bundles over Riemann surfaces and projectively flat bundles over K\"{a}hler manifolds satisfying a topological constraint. These are the situations considered in the moduli constructions for pairs in \cite{Pd} and \cite{ST}. In both cases, the gauge theoretical equations equivalent to the moduli problem provide particular solutions to the coupled equations \eqref{eq:CYMeq00}. In this examples the equations \eqref{eq:CYMeq00} reduce to the limit case $\alpha_1 = 0$ (cscK equation and HYM equation) and we give a geometric explanation of this fact.

In \secref{sec:example2} we consider homogeneous Hermitian holomorphic bundles $E$ over homogeneous K\"{a}hler manifolds and show the existence of
invariant
solutions to
\eqref{eq:CYMeq00} in any K\"{a}hler class provided that the representation defining $E$ is irreducible. The invariant
solutions arise as
(simultaneous) solutions to the limit systems $\alpha_0 = 0$ and $\alpha_1 = 0$ in \eqref{eq:CYMeq00}. When the
complex
dimension of the base is $2$
the invariant solutions satisfy in addition the K\"{a}hler analogue of the Einstein--Yang--Mills equation (see equation~\eqref{eq:CEYMeq}) considered
in  \secref{sec:CeqscalarKal-K}.

In \secref{sec:example3} we discuss some (well known) examples of stable bundles over K\"{a}hler-Einstein manifolds without infinitesimal automorphisms where Theorem \ref{thm:ExistenceCYMeq00} applies. This provide our first examples of solutions in which the K\"ahler metric is not cscK, showing that the coupled equations \eqref{eq:CYMeq00} are more than the cscK equation and the HYM equation together.

In \secref{sec:example4} we discuss the relation of the limit $\alpha_0 = 0$ in \eqref{eq:CYMeq00} with the cscK equation on ruled manifolds as mentioned above.

\section{Riemann surfaces and projectively flat bundles}
\label{sec:example1}

Let $E^c$ be a holomorphic principal $G^c$ bundle over a polarized compact Riemann surface $(X,L)$. We fix a maximal compact subgroup $G \subset G^c$. In this case the coupled equations \eqref{eq:CYMeq2}, for a $G$-reduction $H$ and a K\"ahler metric $\omega$, split into the system in separated variables
\begin{equation}\label{eq:CYMeqdecoupled}
\left. \begin{array}{l}
F_H = z \cdot \omega\\
S_\omega \; = \;\; \hat S
\end{array}\right \},
\end{equation}
since $\dim_\CC X = 1$ and the term $(F_H \wedge F_H)$ vanishes. Recall that $F_H$ denotes the curvature of the Chern connection $A$ of $H$ and that $\hat S \in \RR$ and $z \in \mathfrak{z}$ are constant over $X$. Then, the solutions to the coupled equations \eqref{eq:CYMeq00} are given by pairs $(\omega,H)$, where $\omega$ is the unique cscK metric in $c_1(L)$ modulo automorphisms and $H$ is a reduction such that $A$ is a HYM connection \eqref{eq:HYM}. Due to the Hitchin-Kobayashi correspondence, examples of solutions satisfying \eqref{eq:projectivelyflatomega} are given by polystable $G^c$-bundles over $X$. When $L$ is a multiple of the canonical bundle of $X$, Pandharipande~\cite{Pd} used Geometric Invariant Theory to compactify the moduli space of pairs $(X,E)$ consisting of a smooth curve of genus greater than one and a semistable vector bundle over the curve. In \cite[Proposition~8.2.1]{Pd} the author proves that such a pair is GIT stable if and only if $E$ is Mumford stable. Therefore, the gauge theoretical equation corresponding to this moduli construction is simply the HYM equation. The latter is equivalent to \eqref{eq:CYMeqdecoupled} since there is always a cscK metric in $c_1(L)$.

Let now $(X,\omega)$ be a compact K\"{a}hler manifold of arbitrary dimension. Suppose that $E$ is a projectively flat vector bundle over $X$ satisfying the topological constraint
\begin{equation}\label{eq:topconstraint}
c_1(E) = - \frac{r\lambda}{2\pi}[\omega],
\end{equation}
where $\lambda \in \RR$ is determined by the first Chern class of the bundle and the K\"{a}hler class $[\omega]$. Then, doing a conformal change if necessary, there exists a Hermitian metric $H$ on $E$ which is a solution to (see~\cite[Corollary~2.7]{Ko})
\begin{equation}\label{eq:projectivelyflatomega}
F_H = z \cdot \omega,
\end{equation}
where $z = \imag \lambda \Id \in \mathfrak{z}$. Therefore, as can be readily checked from the equations, the pair $(\omega,H)$ is a solution to \eqref{eq:CYMeq2} if and only if $\omega$ is a cscK metric. We conclude that, when $E$ is projectively flat and \eqref{eq:topconstraint} is satisfied, there exists a solution to \eqref{eq:CYMeq2} if $[\omega]$ admits a cscK metric. Suppose that
$$
X \cong \CC^n/\Gamma
$$
is a complex torus given by a lattice $\Gamma$ in $\ZZ^{2n}$. In this case, examples of holomorphic vector bundles $E$ admitting a projectively flat Hermitian structure $H$ are given by representations of a central extension of $\Gamma$ into $U(r) \subset GL(r,\CC)$. It follows from~\cite[Theorem~7.54]{Ko} that $\pm c_1(E)$ is a K\"{a}hler class admitting a cscK metric $\omega$ so $E$ satisfies the constraint \eqref{eq:topconstraint}.

In \cite{ST}, G. Schumacher and M. Toma build a moduli space of (non-uniruled) polarized K\"{a}hler manifolds equipped with isomorphism classes of stable vector bundles via the method of versal deformations. In this work the authors were able to endow the moduli space with a K\"ahler metric, under the cohomological constraint \eqref{eq:topconstraint}, when the base is K\"ahler Einstein and the bundle is projectively flat. The gauge theoretical equations corresponding to this moduli construction are therefore equivalent to \eqref{eq:CYMeqdecoupled} and so, by the discussion in the previous paragraph, they provide particular solutions to the coupled equations \eqref{eq:CYMeq2}. Recall that the cscK equation and the KE equation are equivalent, using Hodge theory, if the class of the polarization is a multiple of the first Chern class of $X$.

In the two situations discussed the coupled equations \eqref{eq:CYMeq2} admit solutions coming from the system in separated variables \eqref{eq:CYMeqdecoupled}. There is an underlying geometric reason for this fact provided by the extended gauge group $\cX$ \eqref{eq:Ext-Lie-groups} associated to a solution $(\omega,H)$ to \eqref{eq:projectivelyflatomega} (in the case of curves we assume stability of the bundle in order to have such solution). The point is that the Chern connection $A$ of $H$ determines a Lie algebra splitting of the short exact sequence
$$
0 \to \LieG \to \LieX \to \LieH \to 0,
$$
where $\LieG$ and $\LieH$ denote, respectively, the Lie algebras of the gauge of the reduction $H$ and the group of Hamiltonian symplectomorphisms in $(X,\omega)$. The splitting is given by the Lie algebra homomorphism
\begin{equation}\label{eq:splittingLiecX}
\Gamma\colon \LieH \cong C^{\infty}(X)/\RR \to \LieX: \phi \to \theta_A^{\perp}(\eta_\phi) - z \phi,
\end{equation}
where $i_{\eta_\phi}\omega = d\phi$ and $\theta_A^\perp$ denotes the horizontal lift with respect to the connection $A$. To see this note that
\begin{align*}
[\Gamma(\phi_1),\Gamma(\phi_2)] & = [\theta_A^{\perp}(\eta_{\phi_1}) - z \phi_1,\theta_A^{\perp}(\eta_{\phi_2}) - z \phi_2]\\
& = \theta_A^{\perp}([\eta_{\phi_1},\eta_{\phi_2}]) - z \{\phi_1,\phi_2\} + (F_A - z\omega)(\eta_{\phi_1},\eta_{\phi_2})\\
& = \Gamma(\{\phi_1,\phi_2\}) + (F_A - z\omega)(\eta_{\phi_1},\eta_{\phi_2}),
\end{align*}
where $\{\phi_1,\phi_2\}$ is the Poisson bracket in $C^{\infty}(X)/\RR$ given by $\omega$. Note that this homomorphism does not extend in
general to the Lie algebra of the group of diffeomorphisms of $X$. Therefore, when $\dim_\CC X = 1$ or $E$ is projectively flat, the coupled system \eqref{eq:CYMeq2} is able to admit `decoupled' solutions due to the fact that $\LieX$ is a semidirect product of $\LieG$ and $\LieH$.

\section{Homogeneous bundles over homogeneous K\"{a}hler manifolds}
\label{sec:example2}

For the basic material on this topic we refer to \cite{Be} and \cite{Ko}. Let $X$ be a compact homogeneous
K\"{a}hlerian manifold (i.e. admitting a K\"{a}hler metric), of a compact group $G$. In other words $X = G/G_o$, for a closed subgroup $G_o \subset G$,
equipped with the canonical $G$-invariant complex structure (see Remark $8.99$ in \cite{Be}). In this situation, homogeneous holomorphic vector bundles $E$ of rank $r$ over $X$ are in one to one correspondence with representations
of $G_o$ on $GL(r,\CC)$. For any invariant K\"{a}hler metric $\omega$ on $X$ there exists a unique $G$-invariant
Hermitian--Yang--Mills unitary connection $A$ provided that the representation inducing E is irreducible (see~\cite[Proposition~6.1]{Ko}). Moreover, for any such choice of invariant metric and connection, the function $\Lambda_\omega^2\tr F_A \wedge F_A$ on $X$ is $G$-invariant and hence constant. Therefore, it turns that $A$ satisfies the stronger equation
\begin{equation}\label{eq:ext-HYM}
\left. \begin{array}{l}
\Lambda_{\omega} F_A = \imag \lambda \Id\\
\Lambda^2_{\omega} \tr F_A \wedge F_A = - \frac{4\hat{c}}{(n-1)!}
\end{array}\right \},
\end{equation}
where $\hat c \in \RR$ is as in \eqref{eq:constant-c-hat} and $\lambda \in \RR$ is determined by the first Chern class of the bundle and the K\"{a}hler class $[\omega]$. The system \eqref{eq:ext-HYM} corresponds to the limit
$$
\alpha_0 \to 0
$$
in \eqref{eq:CYMeq2}. Let us fix a pair of arbitrary coupling constants $\alpha_0, \alpha_1 >0$ and a homogeneous holomorphic vector bundle $E$ over $X$ associated to an irreducible representation. Then, any K\"{a}hler class on $X$ determines a unique $G$-invariant solution $(\omega,A)$ to the coupled
equations with coupling constants $\alpha_0$ and $\alpha_1$. To see this note that every de Rham class on $X$ (in particular, every
K\"{a}hler class) contains a unique $G$-invariant representative, obtained from an arbitrary representative by averaging. Trivially, the scalar
curvature of any $G$-invariant K\"{a}hler metric is constant. Therefore, the unique $G$-invariant solution of \eqref{eq:CYMeq2} arises as a simultaneous solution of the cscK equation and \eqref{eq:ext-HYM}, corresponding to the limit cases $\alpha_0 = 0$, and $\alpha_1 = 0$.

We will now give some examples of solutions to the Einstein--Yang--Mills type equations~\eqref{eq:CEYMeq} in the case of complex surfaces. Let $X$ be a compact complex surface and $E$ a homogeneous holomorphic vector bundle over $X$ of rank $r$ and zero degree. Then, the condition to be Hermitian--Yang--Mills for a unitary connections $A$ in $E$ is equivalent to the Anti-Self-Dual (ASD) equation and so \eqref{eq:CEYMeq} can be restated as
\begin{equation}
\label{eq:CEYMeqASD} \left. \begin{array}{l}
F_A^{+} = 0\\
\alpha_0 (\rho_\omega - c'\omega) = \alpha_1 (2(\Lambda_{\omega} F_A,F_A) - \Lambda_\omega (F_A \wedge F_A) - c'' \omega)
\end{array}\right \},
\end{equation}
for real constants $c'$ and $c'' \in \RR$. Since $\dim_\CC X = 2$, if $A$ is ASD with respect to $\omega$ we have that (see~\eqref{eq:identitysquaredFA})
\begin{align*}
2(\Lambda_{\omega} F_A,F_A) - \Lambda_\omega (F_A \wedge F_A) & = - \Lambda_\omega (F_A \wedge F_A)\\
& = \Lambda_\omega (|F_A|^2 \cdot \omega^2)\\
& = |F_A|^2 \cdot \omega,
\end{align*}
where the pointwise norm of the curvature is taken with respect to $\omega$. Suppose now that $(X,\omega)$ is a compact homogeneous K\"{a}hler-Einstein surface, i.e. it is a complex torus or is simply-connected (see Corollary $8.98$ in \cite{Be}). We suppose also that the vector bundle $E$ is homogeneous and associated to an irreducible representations of $G_o$ on $GL(r,\CC)$. Consider the Hermitian metric $H$ on $E$ whose Chern connection $A$ is HYM. Since the function $|F_A|^2$ is constant over $X$ by invariance, the pair $(\omega,A)$ satisfies the equation \eqref{eq:CEYMeqASD}. Relying in the discussion of \secref{sec:CeqscalarKal-K} the invariant Riemannian metric associated to $(\omega,A)$ in the total space of the $U(r)$-bundle of frames of $(E,H)$ over $X$ is Einstein.

\section{Stable bundles over K\"{a}hler-Einstein manifolds}
\label{sec:example3}

We see now some examples where Theorem~\ref{thm:ExistenceCYMeq2} can be applied. The following will provide our first examples of solutions to the coupled equations \eqref{eq:CYMeq2} in which the K\"ahler metric is not cscK.

Let $X$ be a high degree hypersurface of $\mathbb{P}^3$. By theorems of T. Aubin and S. T. Yau (see e.g. Theorem $11.7$ in \cite{Be}), there exists a unique K\"{a}hler-Einstein metric $\omega \in c_1(X)$ with negative (constant) scalar curvature. Moreover, $c_1(X) < 0$ implies that the group of automorphisms of the complex manifold $X$ is discrete (see~\cite[Proposition~2.138]{Be}). Let $E$ be a smooth $SU(2)$-principal bundle over $X$ with second Chern number
$$
k = \frac{1}{8\pi^2}\int_X \tr F_A \wedge F_A \in \ZZ,
$$
where $A$ is a connection on $E$. When $k$ is sufficiently large, the moduli space $M_k$ of Anti-Self-Dual (ASD) connections $A$ on $E$ with respect to $\omega$ is non-empty (see~\cite[Section~10.1.14]{DK}). Moreover, taking $k$ large enough, $M_k$ is non-compact but admits a compactification. Let $A$ be a connection that determines a point in $M_k$. Then, $A$ is irreducible and so we can apply Theorem \ref{thm:ExistenceCYMeq2} obtaining solutions $(\omega_\alpha,A_\alpha)$ to \eqref{eq:CYMeq00} for nonzero values of the coupling constants $\alpha_0$, $\alpha_1$ and small ratio
$$
\alpha = \frac{\alpha_1}{\alpha_0}.
$$
We claim that, if the pointwise norm
\begin{equation}\label{eq:examplesFAnorm}
|F_{A_0}|_{\omega_0}^2\colon X \to \RR
\end{equation}
of the initial HYM connection $A_0 = A$ with respect to the KE metric $\omega_0 = \omega$ is not constant, then $\omega_\alpha$ is not cscK for $0 < \alpha \ll 1$. To see this note that $(\omega_\alpha,A_\alpha)$ approaches uniformly to $(\omega_0,A_0)$ (see Lemma~\ref{lemma:impfunc}) and so
$$
\lim_{\alpha \to 0} \||F_{A_\alpha}|_{\omega_\alpha}^2 - |F_{A_0}|_{\omega_0}^2\|_{L^{\infty}} \to 0.
$$
Therefore, if $\eqref{eq:examplesFAnorm}$ is not constant then $|F_{A_\alpha}|_{\omega_\alpha}^2$ is not constant for small $\alpha$ and our claim follows from
$$
S_{\omega_\alpha} = \frac{c}{\alpha_0} - \alpha \cdot \Lambda^2_{\omega_\alpha}(F_{A_\alpha} \wedge F_{A_\alpha})  = \frac{c}{\alpha_0} + \alpha \cdot |F_{A_\alpha}|_{\omega_\alpha}^2,
$$
where $c \in \RR$. This last equation is satisfied since $(\omega_\alpha,A_\alpha)$ is a solution to \eqref{eq:CYMeq00}. To choose an ASD connection for which \eqref{eq:examplesFAnorm} is not a constant, we consider a sequence of ASD connections $\{A^l\}_{l=0}^\infty$ on $M_k$ approaching to a point in the boundary of the compactification. When $l \gg 0$ the connections $A_l$ start bubbling. This bubbling is reflected in the fact that the function \eqref{eq:examplesFAnorm} becomes more and more concentrated in a finite number of points of the manifold. Therefore, eventually, we obtain an ASD irreducible connection for which \eqref{eq:examplesFAnorm} is not a constant. To be more precise, recall that any point in the the boundary of the compactification of $M_k$ is given by an ideal connection (see~\cite[Definition~4.4.1]{DK}), i.e. an unordered $d$-tuple $(p_1, \ldots, p_d)$ of points on $X$ and a connection $A_\infty$ in $M_{k-d}$, the moduli space of ASD connections on suitable smooth $SU(2)$-bundle $E_{k-d}$ with second Chern number $k-d$. If $[A_l] \to [A_\infty]$ then, for any continuous function $f$ on $X$ (see~\cite[Theorem~4.4.4]{DK})
\begin{equation}\label{eq:convergencemeasure}
\lim_{l \to \infty} \int_X f \tr F_{A_l} \wedge F_{A_l} = \int_X f \tr F_{A_\infty} \wedge F_{A_\infty} + 8\pi^2 \sum_{m=1}^{d}f(p_m)
\end{equation}
We take $A_\infty$ in $M_{k -d}$ with $d > 0$. If $|F_{A_l}|_{\omega}^2$ is constant for every $l$, using the previous equation and the equality
$$
|F_{A_l}|_{\omega}^2 \cdot \omega^2 = \tr F_{A_l} \wedge F_{A_l}
$$
we obtain that $d = 0$, and so a contradiction (e.g. take in \eqref{eq:convergencemeasure} a sequence of test functions $f_n$ approaching the characteristic function of $X^0$ on $X$).

The hypothesis of Theorem~\ref{thm:ExistenceCYMeq2} hold in much more generality. A family of examples generalizing the previous one is provided by stable bundles over K\"ahler--Einstein manifolds. Let $(X,L)$ be a compact polarised manifold whose first Chern class $c_1(X)$ is positive negative or zero and
$$
c_1(L) = \lambda c_1(X), \qquad \lambda \in \ZZ.
$$
When $c_1(X) < 0$ (e.g. if $X$ is a high degree hypersurface of $\mathbb{P}^m$), $X$ has finite group of automorphisms and there exists, by the already mentioned results of Aubin and Yau, a unique K\"{a}hler-Einstein metric $\omega \in c_1(L)$. If $c_1(X) = 0$ there exists, by S. T. Yau solution to Calabi's Conjecture (see e.g.~\cite[Theorem~11.7]{Be}), a unique Ricci-flat metric on
$c_1(L)$. The dimension of the group of automorphisms of such manifolds is equal to its first Betti number (see \cite[Remark~11.22]{Be}), hence
simply connected ones (e.g. $K3$ surfaces) are examples of complex Ricci-flat manifolds with finite group of automorphisms. If $c_1(X) > 0$ it is still unknown  in general if $X$ carries or not a K\"{a}hler-Einstein metric. Let us restrict to the case
$$
X = \mathbb{P}^2 \; \sharp \; m \overline{\mathbb{P}}^2,
$$
the complex surface obtained by blowing up $\mathbb{P}^2$ at $m$ generic points (see \cite{TY}). If we take $m$ such that $3 < m < 8$ then $c_1(X) > 0$, $X$ has finite automorphism group (see Remark $3.12$ in \cite{T4}) and it was proved in \cite{TY} that $X$ admits a K\"{a}hler-Einstein metric.

On the other hand, given a compact polarised manifold $(X,L)$ (without any assumption on $c_1(X)$), an asymptotic
result of Maruyama \cite{Ma} states
that $L$-stable vector bundles $E$ over $X$  of rank $r$ exist for $r > \dim X > 2$ and
\begin{equation}\label{eq:stablebundlesufcondition}
c_2(E) \cdot c_1(L)^{n-2} >> 0.
\end{equation}
If $X$ has finite group of automorphisms and it is endowed with a K\"{a}hler-Einstein metric $\omega \in c_1(L)$ as before, we can apply Theorem~\ref{thm:ExistenceCYMeq2}.

\section{CscK metrics on ruled manifolds}
\label{sec:example4}

We now briefly discuss the relation between equation \eqref{eq:ext-HYM}, given by the limit
$$
\alpha_0 \to 0
$$
in \eqref{eq:CYMeq2}, and the existence of solutions to the cscK equation on ruled manifolds. We use existence results of Y. J. Hong~\cite{Ho1,Ho2}.

Let $E$ be a holomorphic stable vector bundle with degree zero over a compact K\"{a}hler manifold $(X,\omega)$ with constant scalar curvature. Let $H$ be
the (unique) Hermitian metric on $E$ whose Chern connection $A$ is HYM, given by Donaldson--Uhlenbeck--Yau Theorem (\cite{D3}\cite{UY}). Let $\mathbb{P}(E)$ be the projectivized bundle of $E$ and $L$ the tautological bundle over $\mathbb{P}(E)$. We consider the induced connection in $L^{\ast}$ by $A$ and denote its curvature by $F_{\textrm{A}_{L^{\ast}}}$. The $2$-form
$$
\frac{i}{2\pi} F_{\textrm{A}_{L^{\ast}}}
$$
is nondegenerate on the fibres since it induces the Fubini-Study metric. Then, for $k$ large enough
$$
\widehat{\omega}_k = \frac{i}{2\pi} F_{\textrm{A}_{L^{\ast}}} + k
\pi^{\ast}\omega
$$
is a Kähler metric in $\mathbb{P}(E)$. When $X$ has finite automorphisms group Y.J. Hong proved in~\cite{Ho1} that $[\widehat{\omega}_k]$ has a cscK metric for $k \gg 0$, using a deformation argument. Let denote by $\cX$ the extended gauge group of the $PU(r)$ bundle of frames associated to the reduction $H$ and the symplectic form $\omega$. We denote by $\cX_I \subset \cX$ the subgroup preserving the connection $A$. In~\cite{Ho2} he was able to remove the assumption on $\Aut X$ under the additional assumptions (see~\cite[Definition~I.A]{Ho2}):
$$
\cX_I \subset \Aut \mathbb{P}(E)
$$
finite and
\begin{equation}\label{eq:Hongcondition}
\Lambda^2_{\omega} (\tr F_A \wedge \tr F_A + \tr F_A \wedge
\rho_{\omega} + F_A \wedge F_A) = const.,
\end{equation}
that, by our assumption on the degree of $E$, reduces to
$$
\Lambda^2_{\omega} \tr F_A \wedge F_A = - \frac{4\hat{c}}{(n-1)!} \in \RR.
$$
The condition \eqref{eq:Hongcondition} appears when one splits the linearization of the cscK equation on $\mathbb{P}(E)$ into vertical and horizontal parts with respect to the connection $A$.

Then, we conclude that the existence of a solution to \eqref{eq:ext-HYM} when $c_1(E) = 0$ and $\cG_I$ is finite is a sufficient condition for the existence of a cscK metric in $[\widehat{\omega}_k]$ for $k \gg 0$ (see~\cite[Theorem~III.A]{Ho2}). We would like to study further this relation trying to prove that the existence of solutions to the coupled equations for small $\frac{\alpha_1}{\alpha_0} > 0$ implies the existence of constant scalar curvature K\"{a}hler metrics on $\mathbb{P}(E)$ with K\"{a}hler class $c_1(L_k)$ for large $k$. This would provide a generalization of Hong's results in \cite{Ho2}.

\chapter{\bf Solutions on $\CC^2$}
\label{chap:C2}

The idea of this chapter was suggested to the author by Professor Olivier Biquard and developed during a visit to the Institute de Mathematiques de Jussie under his supervision, for which the author is very grateful. Let us consider $\CC^2$ with the standard complex structure endowed with the trivial smooth principal bundle $E = \CC^2 \times SU(2)$. Given a real coupling constant $\alpha \in \RR$ and any integer $k \in \ZZ$ we will look for pairs $(\omega,A)$, where $\omega$ is a K\"{a}hler metric on $\CC^2$ and $A$ is an $SU(2)$-connection on $E$, satisfying the coupled equations \eqref{eq:CYMeq00}, that we can rewrite in this context using the Hodge star operator $*_{\omega}: \Omega^k \to \Omega^{4-k}$ as
\begin{equation}\label{eq:ceqc2}
\left. \begin{array}{l}
F_A^{+_{\omega}} = 0\\
S_\omega \; + \; \alpha *_{\omega} \tr F_A \wedge F_A = \alpha 8 \pi^2 k
\end{array}\right \}.
\end{equation}
The connections and metrics that we will consider have special asymptotic properties at infinity. In fact, any connection
$A$ will be required to satisfy
the topological constraint
\begin{equation}\label{eq:topologyfixed}
\int_{\CC^2} \tr F_A \wedge F_A = 8 \pi^2 k, \qquad \qquad k \in
\ZZ.
\end{equation}
The fixed asymptotic behavior at infinity is analogue in this non-compact case of fixing the K\"{a}hler class and the
topological type of the bundle. We will show that in $\CC^2$ there is no obstruction for finding solutions for
arbitrary coupling constant $\alpha$.
Solutions with $\alpha = 0$ are provided by pairs $(\omega_0,A)$, where $\omega_0$ is the euclidean metric
$\frac{\imag}{2}\sum_j dz_j \wedge
d\overline{z}_j$ and $A$ is any $k$-instanton, i.e. any solution to the Anti-Self-Duality equation
\begin{equation}\label{eq:ASDeq}
F_A^{+_{\omega_0}} = 0
\end{equation}
with fixed topological invariant \eqref{eq:topologyfixed}. In this section we will show that for each $\alpha \in \RR$
and $k \in \ZZ$ there exists a
solution $(\omega_\alpha,A_\alpha)$ of the coupled equations \eqref{eq:ceqc2} with coupling constant $\alpha$ and
fixed
topological invariant $8\pi^2
k$. The pair $(\omega_\alpha,A_\alpha)$ converges uniformly at infinity to a pair $(\omega_0,A_{0,\alpha})$, where
$\omega_0$ is standard Euclidean
metric and $A_{0,\alpha}$ is a $k$-instanton. Moreover, the will give solutions for which the K\"ahler metric is not cscK, making explicit the `bubbling' argument in \secref{sec:example3}. The method we will follow to find solutions to \eqref{eq:ceqc2} has two
steps: first, we deform a solution for zero coupling constant $\gamma$ via the application of the implicit function theorem in weighted
H\"older spaces. For this it will be
crucial the previous work of C. Arezzo and F. Pacard in \cite{AP} concerning the mapping properties of the scalar
curvature in weighted spaces.
Second, we apply a rescaling procedure via real homotheties. The uniform convergence at infinity will be
provided by the fact that
$A_{\alpha} - A_{0,\alpha}$ has coefficients in a weighted H\"older space. For simplicity in the first part we develop
the arguments starting with the
basic $1$-instanton (whose definition we recall below) and explain how the argument generalizes for a suitable choice
of $k$-instanton.

\section{Preliminaries: Instantons on $\CC^2$.}
\label{sec:C2instantons}

We recall now some well known facts about instantons on $\CC^2 \cong \RR^4$ (see \cite{At}, \cite{ADHM} and
\cite{DK}). Following \cite{At} we identify $\mathbb{R}^4$ with the algebra of quaternions $\mathbb{H}$ via the map
$$
x = (x_1,x_2,x_3,x_4) \to x_1 + x_2 \cdot \imag + x_3 \cdot \mathbf{j} + x_4 \cdot \mathbf{k} \in \mathbb{H}
$$
where $\imag,\mathbf{j},\mathbf{k}$ satisfy the usual relations. The group $Sp(1)$ of quaternions with unit norm
$$
x \overline{x} = |x|^2 = 1
$$
gets identified with $SU(2)$ and its lie algebra can be viewed as the pure imaginary quaternions with basis $i,j,k$. The basic $1$-instanton is defined as the $SU(2)$ connection
$$
A = \textrm{Im} \frac{\overline{x}dx}{1 + |x|^2} = \frac{1}{2}
\cdot \frac{\overline{x}dx - d\overline{x}x}{1 + |x|^2}.
$$
The curvature of $A$ can be written as
$$
F_A = \frac{d\overline{x} \wedge dx}{(1 + |x|^2)^2},
$$
where
\begin{align*}
d\overline{x} \wedge dx & = 2(dx_1 \wedge dx_2 - dx_3 \wedge dx_4)\cdot \imag + 2(dx_1 \wedge dx_3 +
dx_2 \wedge dx_4) \cdot \mathbf{j}\\
& + 2(dx_1 \wedge dx_4 - dx_2 \wedge dx_3) \cdot \mathbf{k}.
\end{align*}
The components in $\imag$, $\mathbf{j}$ and $\mathbf{k}$ form a basis of ASD $2$-forms and with respect to the flat metric $\omega_0$ so $F_A^+ = 0$.

Any other $1$-instanton $A'$ with respect to $\omega_0$ is produced by a conformal transformation
$$
f_{\beta,b}: x \to \beta(x - b),
$$
where $\beta$ is a positive real scalar and $b \in \mathbb{H}$. Multi-instantons or $k$-instantons, are given by
pulling back the direct sum of the basic instanton $A$ on $\mathbb{H}^k$ via maps $u: \mathbb{H} \to \mathbb{H}^k$ of the form
\begin{equation}\label{eq:umultiinstanton}
u(x) = (x\Id - B)\lambda,
\end{equation}
where $B$ is a symmetric $k \times k$ matrix of quaternions and $\lambda$ is column vector of quaternions in
$\mathbb{H}^m$ satisfying some compatibility conditions (see \cite[p.~25]{At}). Hence, any $k$-instanton can be written as
$$
A_u = \textrm{Im} \frac {u^* du}{1 + |u|^2},
$$
where the norm in $\mathbb{H}^k$ is taken with respect to the direct sum of the euclidean norm.

\section{Mapping properties of the connection-laplacian}
\label{sec:C2Mapconnectlap}

Let $A$ be the basic $1$-instanton in $\CC^2$ and consider its covariant derivative $d_A$ acting on sections $a \in
C^{\infty}(\RR^4,\mathfrak{su}(2))$. In this section we study the mapping properties of the Hodge laplacian $\Delta_A$
as a map between weighted
H\"{o}lder spaces. This is a self-adjoint elliptic operator of order $2$ with respect to the Hermitian product
provided by the trace $\tr$ and
the Hermitian product in forms given by $\omega_0$. We shall discuss for which values of the weight $\Delta_A$ is an
isomorphism. First we define
weighted H\"{o}lder spaces of our interest. The general theory about weighted spaces that we will use can be found in
\cite{LM}, \cite{B}, \cite{PR},
\cite{LP} and \cite{CC}.

Let $k \in \mathbb{N}$, $\beta \in (0,1)\subseteq\RR$, $\delta, r \in \RR$, $r>0$ be numbers and consider
$$
B_r = \{z \in \CC^2: |z|<r\}
$$
and $\overline{B}_r$ its clousure. For some fixed $r_0$ the space $C^{k,\beta}_\delta(\CC^2 \backslash B_{r_0})$ is
defined to be the space
$C^{k,\beta}_{loc}(\CC^2 \backslash B_{r_0})$ of locally H\"{o}lder functions $\phi$ for which the following norm
$$
\|\phi\|_{C^{k,\beta}_\delta(\CC^2 \backslash B_{r_0})} \defeq
\textrm{sup}_{r \geq r_0} r^{-\delta}
\|\phi(r\cdot)\|_{C^{k,\beta}(\overline{B}_2 \backslash B_1)}
$$
is finite. From now on we fix $r_0 > 0$ and consider the following.
\begin{definition}\label{def:Ckalphadelta}
Given $k$, $\beta$ and $\delta$ as before, the space $C^{k,\beta}_\delta$ is defined to be the space of functions
$\phi
\in C^{k,\beta}_{loc}(\CC^2)$
for which the following norm is finite
$$
\|\phi\|_{C^{k,\beta}_\delta} \defeq \|\phi\|_{C^{k,\beta}(\overline{B}_{2r_0})} + \|\phi\|_{C^{k,\beta}_\delta(\CC^2
\backslash B_{r_0})}.
$$
\end{definition}
In a similar way we define the spaces $D^{k,\beta}_\delta$ of sections of $\CC^2 \times \mathfrak{su}(2)$.
\begin{definition}\label{def:Dkalphadelta}
Given $k$, $\beta$ and $\delta$ as before, the space $D^{k,\beta}_\delta$ is defined to be the space of functions $a:
\CC^2 \to \mathfrak{su}(2)$, $a = a_1 \cdot \imag + a_2 \cdot \mathbf{j} + a_3 \cdot \mathbf{k}$, such that $a_l \in C^{k,\beta}_\delta$ for $l = 1,2,3$.
\end{definition}
\begin{remark}
The sign notation we use for the weight $\delta$ agrees with that in \cite{AP}, \cite{LP} and \cite{B} but has
opposite sign with respect to the one in \cite{LM}. With our notation the weight reflects order of growth: functions in $C^{k,\beta}_\delta$ grow at most like $|x|^\delta$ when $|x| \to
\infty$.
\end{remark}

In the next proposition we summarize the properties of the spaces $C^{k,\beta}_{\delta}$ that we will need for our discussion. Analogue properties are satisfied by the spaces $D^{k,\beta}_\delta$.
\begin{proposition}[\cite{CC},\cite{PR},\cite{LP}]\label{propo:propertiesweighted}
\hspace{0.5cm}
\begin{enumerate}
    \item There is a continuous embedding $C^{k,\beta}_{\delta} \subseteq C^{k',\beta'}_{\delta'}$ if $\alpha > \alpha'$,
    $k > l$ and $\delta' > \delta$.
    \item If $\phi \in C^{k,\beta}_{\delta}$ then $\nabla \phi \in
    C^{k-1,\beta}_{\delta-1}$.
    \item The product is a continuous map from $C^{k,\beta}_{\delta} \times C^{k,\beta}_{\delta'}$ to $C^{k,\beta}_{\delta + \delta'}$.
\end{enumerate}
\end{proposition}

Let $\Delta_A$ be the Hodge Laplacian of the basic $1$-instanton given by
$$
\Delta_A = d_A^*d_A = - *_e d_A *_e d_A,
$$
where $*_e$ denotes the Hodge operator of the euclidean metric $\omega_0$ acting on forms. To apply the theory of elliptic operators on weighted spaces it is convenient to express the previous Hodge Laplacian with respect to the radial coordinate $0 < r < \infty$ in $\RR^4$. Let  $*_{S^3}$ and $d_{S^3}$ be, respectively, the Hodge operator on forms and the De Rham differential in the Riemannian three sphere $S^3 \subseteq \RR^4$.
\begin{lemma}
There exists a connection $B$ in the trivial $SU(2)$-bundle over the sphere $S^3$ such that
\begin{equation}\label{eq:Laplacianpolarcoordinates}
\Delta_A = - \partial_{rr}^2 - \frac{3}{r} \partial_{r} -
\frac{1}{r^2} *_{S^3} (d_{S^3} + \frac{r^2}{1 + r^2} B) *_{S^3}
(d_{S^3} + \frac{r^2}{1 + r^2} B).
\end{equation}
\end{lemma}
\begin{proof}
To compute formula \eqref{eq:Laplacianpolarcoordinates} let us take polar coordinates in $\RR^4$
\begin{align*}
x_1 & = r \cos \phi_1 = r w_1,\\
x_2 & = r \cos \phi_2 \sin \phi_1 = r w_2,\\
x_3 & = r \sin \phi_3 \sin \phi_2 \sin \phi_1= r w_3,\\
x_4 & = r \cos \phi_3 \sin \phi_2 \sin \phi_1 = r w_4,
\end{align*}
with $0 < r < \infty$ and $\sum_{l=1}^4 w_l^2 = 1$. Note that the basic instanton has the following expression in polar coordinates $x = r w$
\begin{align*}
A & = \frac{1}{2(1 + r^2)}(\overline{x} \cdot dx - d\overline{x} \cdot x)\\
& = \frac{1}{2(1 + r^2)}(r\overline{w} \cdot d(rw) - d(r\overline{w}) \cdot r w)\\
& = \frac{r^2}{2(1 + r^2)}(\overline{w} \cdot dw - d\overline{w}
\cdot w) = \frac{r^2}{1 + r^2} B,
\end{align*}
where $B$ is the $\mathfrak{su}(2)$-valued $1$-form
\begin{equation}\label{eq:connectionB}
B = \frac{1}{2}(\overline{w} \cdot dw - d\overline{w} \cdot w).
\end{equation}
Given $w \in S^3$ and any tangent vector $v_{w} \in T_{w}S^3$ we have
$$
B_w(v_{w}) = \lim_{r \to \infty} A_{rw}(r \cdot v_{w}),
$$
where $r:\RR^4 \to \RR^4$ is the homothety given by multiplication by $r$. Thus, $B$ has a coordinate free expression and so it extends to a true connection on $S^3 \times SU(2)$. It can be readily checked from \eqref{eq:connectionB} that $B$ is $SU(2)$-invariant and flat, although we will not need this properties for our discussion. The covariant derivative of $A$ can be now written as
$$
d_A = \partial_r + d_{S^3} + \frac{r^2}{1 + r^2} B.
$$
Let $a \in C^{\infty}(\RR^4,\mathfrak{su}(2))$. If we denote by $*_e$ the euclidean Hodge operator in $\RR^4$ and by $\vol_{S^3}$ the standard volume form in $S^3$ we have
\begin{align*}
\Delta_A a & = d_A^* d_A = - *_e d_A *_e d_A a \\
& = - *_e (\partial_r + d_{S^3} + \frac{r^2}{1 + r^2} B) *_e
(\partial_r a \cdot dr + d_{S^3}a + \frac{r^2}{1 + r^2} [B,a])\\
& = - *_e (\partial_r + d_{S^3} + \frac{r^2}{1 + r^2} B)(\partial_r
a \cdot r^3 \vol_{S^3}  - r dr \wedge (d_{S^3}a + \frac{r^2}{1 +
r^2} [B,a]))\\
& = - *_e (\partial_r(r^3 \partial_r a) dr \wedge \vol_{S^3} + r dr
\wedge (d_{s^3} + \frac{r^2}{1 + r^2} B)*_{S^3}(d_{S^3}a +
\frac{r^2}{1 + r^2} [B,a]))\\
& = - \partial^2_{rr}a - \frac{3}{r}\partial_r a - \frac{1}{r^2}
*_{S^3}(d_{s^3} + \frac{r^2}{1 + r^2} B)*_{S^3}(d_{S^3}a +
\frac{r^2}{1 + r^2}B)a.
\end{align*}
\end{proof}

Let $k \geq 2$ be a fixed positive integer. It follows from the properties of the weighted H\"{o}lder
spaces summarized in Proposition \ref{propo:propertiesweighted}, that, for all $\beta \in [0,1)$ and $\delta \in \RR$, the differential operator $\Delta_A$ induces a well defined bounded linear operator
$$
\Delta_{A,\delta}: D^{k,\beta}_\delta \to D^{k-2,\beta}_{\delta-2}.
$$
Following the notation in \cite{LM} we consider the operator
\begin{equation}\label{eq:Laplacianinfinity}
\Delta_{\infty} = - r^2 \partial_{rr}^2 - 3 r \partial_{r} + \Delta_B,
\end{equation}
where $\Delta_B$ denotes the Hodge Laplacian of the connection \eqref{eq:connectionB}. The operator \eqref{eq:Laplacianinfinity} induces well defined bounded linear operator $\Delta_{\infty,\delta}: D^{k,\beta}_\delta \to D^{k-2,\beta}_{\delta}$ for each value of the weight $\delta \in \RR$. Making the change of variable $t = \ln r$ we obtain a translation invariant operator (see \cite[p.~414]{LM})
\begin{equation}\label{eq:Laplacianinfinityt}
\Delta_{\infty,t} = - \partial_{tt}^2 - 2 \partial_{t} + \Delta_B
\end{equation}
with associated operator
\begin{equation}\label{eq:Laplacianinfinitytdelta}
\Delta_{\infty,t}(\delta) = - \delta^2 - 2 \delta + \Delta_B, \qquad \qquad \delta \in \RR,
\end{equation}
acting on sections $b \in C^{\infty}(S^3,\mathfrak{su}(2))$. Let $0 \leq \lambda_1 \leq \lambda_2 \leq \ldots$ be the
eigenvalues, counted with multiplicity, of the operator $\Delta_B$, the Hodge Laplacian of the connection \eqref{eq:connectionB}. Then, the homogeneous problem $\Delta_{\infty,t}(\delta) b = 0$ has nontrivial solutions if and only if
\begin{equation}\label{eq:indicialroots}
\delta_l^{\pm} = - 1 \pm \sqrt{1 + \lambda_l},
\end{equation}
where $\delta_l^{\pm} \in \RR$ are the indicial roots of $\Delta_{A}$ (see~\cite[p.~416]{LM}). Then, since the $\lambda_l$ are all greater or equal than $0$ we obtain the following result.
\begin{lemma}\label{lemma:indicialinterval}
There is no indicial root of $\Delta_{A}$ in $]-2,0[ \subseteq
\RR$.
\end{lemma}
It follows from \cite[Theorem~6.1]{LM} that the operator $\Delta_{A,\delta}$ is Fredholm if and only if $\delta$
is not an indicial root of $\Delta_A$. From \cite[Lemma~7.1]{LM} it also follows that if $\delta_1, \delta_2 \in \RR$ are weights such that there is no indicial root in $[\delta_1,\delta_2]$ then the dimension of Kernel and coKernel do not change. The heuristic reason for this fact is that the solutions to the homogeneous problem corresponding to $\Delta_\infty$ (equivalently to $\Delta_{\infty,t}$) do determine at $\infty$ all the possible asymptotic
behaviors of the solutions of the homogeneous problem corresponding to $\Delta_A$. Since $\Delta_A$ is self adjoint, using duality arguments in weighted Sobolev spaces (see the comments after the proof of Proposition~\ref{propo:rangeinjectsurjec}) we also obtain that
\begin{equation}\label{eq:dualitysobolev}
\textrm{CoKer} \Delta_{A,\delta} \cong \textrm{Ker} \Delta_{A,- \delta - 2},
\end{equation}
provided that $\delta$ is not an indicial root. We now apply this facts to prove the main result of this section.
\begin{proposition}\label{propo:rangeinjectsurjec}
The operator $\Delta_{A,\delta}$ is injective if $\delta < 0$. It is surjective if $\delta$ is not an indicial root
and
$\delta > -2$.
\end{proposition}
\begin{proof}
Suppose first that there exists $a \in D^{k,\beta}_{\delta}$ such that $\Delta_A a = 0$. By elliptic regularity it follows that $a$ is smooth. Let $B_r \subseteq \RR^4$ be
the open ball of radius $r$. If $\delta < -1$ then we can integrate by parts obtaining
\begin{align*}
0 & = \int_{\RR^4} \tr a \Delta_A a \vol_e = \lim_{r \to \infty} \int_{B_r} \tr a \Delta_A a \vol_e\\
& = \lim_{r \to \infty} \int_{B_r} \tr d_A a \wedge *_e d_A a + \lim_{r \to \infty} \int_{\partial B_r} \tr a \wedge *_e d_A a\\
& = \lim_{r \to \infty} \int_{B_r} \tr d_A a \wedge *_e d_A a = \int_{\RR^4} \tr d_A a \wedge *_e d_A a = \|d_A a\|^2_{\RR^4},
\end{align*}
where $\vol_e$ denotes the usual volume form in $\RR^4$. For the previous computation we have used that $|d_A a|^2_{\RR^4} \in D^{k-2,\beta}_{2\delta - 2} \subseteq L^1(\RR^4)$ due to $2\delta - 2 + 3 = 2\delta + 1 < -1$ and that the term $|\int_{\partial B_r} \tr a \wedge *_e d_A a|$ grows at most like $r^{2\delta + 2}$ as $r \to \infty$. Since the connection $A$ is irreducible $d_A a = 0$ implies $a = 0$ thus $\Delta_{A,\delta}$ is injective if $\delta < -1$. Now, since there is no indicial root of $\Delta_A$ in $]-2,0[$, we conclude that $\Delta_{A,\delta}$ is injective if $\delta < 0$. To prove $(2)$ we recall from \eqref{eq:dualitysobolev} that $\Delta_{A,\delta}$ is surjective if and only if $\Delta_{A,-2 - \delta}$ is injective, provided that $\delta$ is not an indicial root. This holds if $\delta > -2$ due to $(1)$.
\end{proof}

Note that in \cite{LM} the authors work in weighted Sobolev spaces, though it follows from standard inclusion relations between weighted H\"{o}lder and Sobolev spaces (see e.g. \cite{LP}) that their results hold also in weighted H\"{o}lder spaces. The rest of this section is devoted to clarify this point and to give a proof of~\eqref{eq:dualitysobolev}. The reader familiarized with the theory of weighted spaces can skip this part going directly to \secref{sec:C2Firstsol}. Let us first define weighted Sobolev spaces: we denote by $W^{0,p}_\delta$ be the space of locally integrable functions $\phi: \RR^4 \to \RR$ such that the following norm is finite
$$
\|\phi\|_{p,0,\delta} = (\int_{\RR^4} |<x>^{-\delta} \phi|^p
<x>^{-4} \vol_e)^{1/p}
$$
where $<x> = (1 + |x|^2)^{1/2}$ and $\vol_e = \frac{\omega_0^2}{2!}$ denotes the usual volume form in $\RR^4$.
For a non-negative integer $k$ the weighted Sobolev space $W^{k,p}_\delta$ is the set of locally integrable functions $\phi$ such that $|\nabla^l \phi| \in W^{0,p}_{\delta - l}$ for $0 \leq l \leq k$ with norm
$$
\|\phi\|_{p,k,\delta} = \sum_{l=0}^k \|\nabla^l \phi\|_{p,0,\delta - l}.
$$
In a similar way we define weighted sobolev spaces for sections in $\CC^2 \times \mathfrak{su}(2)$ that we denote by
$U^{k,p}_\delta$. Then, the spaces $W^{k,p}_\delta$ satisfy the following properties (see~\cite{LP})
\begin{enumerate}
    \item There is a continuous embedding $W^{k,p}_{\delta} \subseteq W^{k',p}_{\delta'}$ if $k > l$ and $\delta' > \delta$.
    \item If $\phi \in W^{k,p}_{\delta}$ then $\nabla \phi \in W^{k-1,p}_{\delta-1}$.
    \item There is a continuous embedding $C^{k,\beta}_{\delta-\epsilon} \subseteq W^{k,p}_{\delta}$ for every $p \geq 1$ and $\epsilon > 0$.
    \item Suppose $p > 1$ and $l - k - \beta > 4/p$, then there is a continuous embedding $W^{l,p}_{\delta} \subseteq C^{k,\beta}_{\delta}$.
    \item The dual space of $W^{k,p}_{\delta}$ is $W^{k,p'}_{-4-\delta}$, with $\frac{1}{p'} + \frac{1}{p} = 1$.
\end{enumerate}
Analogue properties are satisfied by the spaces $U^{k,p}_\delta$. Property $2)$ implies that $\Delta_A$ induces a well defined operator
$$
\Delta^U_{\delta}: U^{k,2}_\delta \to U^{k-2,2}_{\delta -2}
$$
for each $\delta \in \RR$. We will prove from the previous inclusion properties that there exists isomorphisms $\textrm{Ker} \Delta^U_{\delta} \cong \textrm{Ker} \Delta_{A, \delta}$ and $\textrm{coKer} \Delta^U_{\delta} \cong \textrm{coKer} \Delta_{A, \delta}$. This proves that our previous assertions concerning H\"older spaces follow from the results in \cite{LM}. Note first that the dimension of the Kernel and coKernel in weighted Sobolev spaces of a self adjoint elliptic operator do not depend on $k$. For the Kernel it follows from elliptic regularity and for the coKernel it follows from the fact that, since our operator $\Delta_A$ is self-adjoint, any Kernel is isomorphic to some coKernel with dual weight. Choosing any $\epsilon >0$ and large $l >> 0$ from proposition \ref{propo:propertiesweighted} we have the inclusions
\begin{equation}\label{eq:inclusionSobHolder}
U^{l,2}_{\delta} \subseteq D^{k,\beta}_{\delta} \subseteq
U^{k,2}_{\delta + \epsilon}
\end{equation}
inducing others $\textrm{Ker} \Delta^U_{\delta} \subseteq \textrm{Ker} \Delta_{A, \delta}^{k,\alpha} \subseteq \textrm{Ker} \Delta^U_{\delta + \epsilon}$. From this we conclude, choosing $\epsilon$ small, $\textrm{Ker} \Delta^U_{\delta} \cong \textrm{Ker} \Delta_{A, \delta}$ since the dimension of the Kernel
in weighted Sobolev spaces do not change if we do not cross an indicial root (see~\cite[Lemma~7.1]{LM}). On the other hand, using \eqref{eq:inclusionSobHolder} and elliptic regularity we obtain
$$
\textrm{Im}(\Delta^{U}_{\delta}: U^{l,2}_{\delta} \to
U^{l-2,2}_{\delta - 2}) = D^{k-2,\beta}_{\delta-2} \cap
\textrm{Im}(\Delta_A: D^{k,\beta}_{\delta} \to
D^{k-2,\beta}_{\delta - 2})
$$
By the second theorem of isomorphism $\textrm{coKer} \Delta^{U}_{\delta} \leq \textrm{coKer} \Delta_{A,\delta}$. A
similar argument using the right hand-side inclusion in \eqref{eq:inclusionSobHolder} shows $\textrm{coKer}
\Delta_{A,\delta} \leq \textrm{coKer} \Delta^{U}_{\delta + \epsilon}$, from where we obtain $\textrm{coKer} \Delta^U_{\delta} \cong \textrm{coKer} \Delta_{A, \delta}$ choosing $\epsilon$ small enough.

Finally, we give a proof of \eqref{eq:dualitysobolev} using the previous arguments. Since $\Delta_A$ is self-adjoint and the dual space $(U^{k-2,2}_{\delta - 2})'$ is isomorphic to $U^{k-2,2}_{- \delta - 2}$, if $\delta$ is not an indicial root we have an isomorphism $\textrm{CoKer} \Delta^U_{\delta} \cong \textrm{Ker} \Delta^U_{- \delta - 2}$ (see~\cite[Proposition~4.1]{LM}). Hence, the operator $\Delta_{A,\delta}$ is surjective if and only if the operator $\Delta_{A,-2 - \delta}$ is injective, provided that $\delta$ is not an indicial root.

\section{Solutions for arbitrary coupling constant}
\label{sec:C2Firstsol}

Now we prove the existence of solutions to \eqref{eq:ceqc2} for arbitrary coupling constant. First we deform the
solution with $\alpha = 0$ provided
by the basic $1$-instanton $A$ in $\CC^2$ and the standard flat K\"{a}hler metric $\omega_0$. Let $E^c = \CC^2 \times
SL(2,\CC)$ be the $SL(2,\CC)$
principal bundle endowed with the holomorphic structure provided by $A$. We fix the standard Hermitian metric $H_0$ in
$\CC^2$. Given a smooth map $a
\in C^{\infty}(\CC^2,\mathfrak{su}(2))$ we
can consider the induced hermitian metric $H_0 \cdot e^{\imag a}$ and its Chern connection associated to the
holomorphic structure in $E^c$. We denote
by $F_{e^{\imag a}}$ the curvature of this connection. Note that
$$
\frac{d}{dt}_{t=0}F_{e^{\imag ta}} = \imag \dbar_A\partial_Aa.
$$
Let $\cK \subseteq C^{\infty}(\CC^2,\RR)$ be the space of K\"{a}hler potentials of the euclidean metric $\omega_0$.
Let
$$
L: \mathbb{R} \times \cK \times C^{\infty}(\CC^2,\mathfrak{su}(2)) \to C^{\infty}(\CC^2,\RR) \times
C^{\infty}(\CC^2,\mathfrak{su}(2))
$$
be the differential operator
\begin{equation}\label{eq:operatorL}
L(\alpha,\phi,a) = (S_{\omega_{\phi}} + \alpha *_{\omega_{\phi}}\tr F_{e^{\imag a}}^2 - \alpha 8 \pi^2,
\Lambda_{\omega_{\phi}}F_{e^{\imag a}})
\end{equation}
where $\omega_\phi = \omega_0 + dd^c\phi$. Due to the fact that $A$ is an Hermitian--Yang--Mills
connection the derivative of $L$
with respect to the last two variables at $(0,0,0)$ is
\begin{equation}\label{eq:differentialphia}
d_{\phi,a}L_{|(0,0,0)} = (\Delta_{\omega_0}^2
\phi,\frac{1}{2}\Delta_{A_0}a + <F_{A_0},\partial \dbar \phi>)
\end{equation}
where $\Delta_{\omega_0}$ is the Hodge laplacian of the euclidean
metric.
\begin{lemma}\label{lemma:C2impfunc}
Let $\delta_1, \delta_2 \in \RR$ with $0 < \delta_1 < 1$ and $-2 < \delta_2 < 0$. We fix $k \geq 4$ a non-negative
integer and $0 < \beta < 1$, $\beta
\in \RR$. Then the differential operator $L$ defined in \eqref{eq:operatorL} induces a well defined $C^1$ map
\begin{equation}\label{eq:operatorLinduced}
L_{\delta_1,\delta_2}: V_{\delta_1,\delta_2} \to C^{k-4,\beta}_{\delta_1-4} \times D^{k-4,\beta}_{\delta_2-2}
\end{equation}
where $V_{\delta_1,\delta_2}$ is a neighbourhood of $(0,0,0)$ in $\RR \times C^{k,\beta}_{\delta_1} \times
D^{k-2,\beta}_{\delta_2}$. For such a pair
$(\delta_1,\delta_2)$ there exists $\epsilon = \epsilon(\delta_1,\delta_2)
> 0$ and a $C^1$ curve
$$
\Phi:]-\epsilon,\epsilon[ \to C^{\infty}_{\delta_1} \times D^{\infty}_{\delta_2}
$$
of solutions to the system $L(\alpha,\Phi(\alpha)) = 0$.
\end{lemma}
\begin{proof}
Let $\delta_1, \delta_2 \in \RR$ with $0 < \delta_1 < 1$ and $-2 < \delta_2 < 0$. We fix the basic $1$-instanton $A$
with curvature $F_A$. Note that the components of $F_A$ are in $C^{\infty}_{-4}$ and the components of the euclidean metric $\omega_0$ are in $C^{\infty}_{0}$. We can take a neighbourhood
$$
(0,0,0) \in V_{\delta_1,\delta_2} \subseteq \RR \times C^{k,\beta}_{\delta_1} \times D^{k-2,\beta}_{\delta_2}
$$
such that if $(0,\phi,a) \in V_{\delta_1,\delta_2}$ then $\det(\omega_\phi) \neq 0$. We denote by $\propto$ equality modulo multiplicative constants. Let $\phi$ such that $(0,\phi,0) \in V_{\delta_1,\delta_2}$. If $s_{\phi}$ denotes the scalar curvature of $\omega_\phi$, then $s_{\phi} \in C^{k-4,\beta}_{\delta-4}$ by \cite[Proposition~4.1]{AP}, provided that $0 < \delta_1 < 1$. On the other hand, if $a \in D^{k-2,\beta}_{\delta_2}$ then
$$
h = e^{\imag a} = \sum^{\infty}_{l=0} \frac{(\imag a)^l}{l!} \in D^{k-2,\beta}_{0},
$$
by properties of the weighted H\"older spaces summarized in Proposition~\ref{propo:propertiesweighted}. Recall that $F_{e^{\imag a}} = F_A + \dbar_A h^{-1} \partial_A h$. then, since the components of $A$ are in $C^\infty_{-1}$, we have
\begin{align*}
\partial_A h \in D^{k-3,\beta}_{\delta_2-1} &  \Rightarrow \dbar_A h^{-1} \partial_A h \in D^{k-4,\beta}_{\delta_2-2}\\
&  \Rightarrow F_{e^{\imag a}} \in D^{k-4,\beta}_{\textrm{max}\{\delta_2 - 2,-4\}} = D^{k-4,\beta}_{\delta_2 - 2}\\
& \Rightarrow F_{e^{\imag a}}^2 \in D^{k-4,\beta}_{2\delta_2 - 4}.
\end{align*}
Finally, denoting by $\propto$ equality modulo multiplicative constant, we have
\begin{align*}
*_{\omega_{\phi}} \tr F_{e^{\imag a}}^2 & \propto (\det \omega_{\phi})^{-1} \Lambda^2 F_{e^{\imag a}}^2 \in C^{k-4,\beta}_{2\delta_2 - 4} \subset C^{k-4,\beta}_{\delta_1 - 4},\\
\Lambda_{\omega_\phi} F_{e^{\imag a}} & \propto (\det \omega_{\phi})^{-1} \sum_{j = 0}^1 \Lambda^j F_{e^{\imag a}}\wedge (dd^c\phi)^j \in C^{k-2,\beta}_{\delta_2 -2},
\end{align*}
where $\Lambda$ denotes the adjoint of the Lefschetz operator $\alpha \to \alpha \wedge \omega_0$ acting on forms and multiplicative constants on each summand has been omitted. For the last computation we have used that (omitting multiplicative constant factors)
$$
\det \omega_\phi \propto 1 + \sum_{l=1}^{2} \Lambda^l (dd^c\phi)^{l} \in C^{k-2,\beta}_0.
$$
We thus conclude that the differential operator $L$ defined in \eqref{eq:operatorL} induces well defined $C^1$
operators $L_{\delta_1,\delta_2}:
V_{\delta_1,\delta_2} \to C^{k-4,\beta}_{\delta_1-4} \times D^{k-4,\beta}_{\delta_2-2}$ for the chosen values of the
weights. From \cite[Proposition~5.5]{AP} we know that the double Laplacian
$$
\Delta_{\omega_0}^2:C^{k,\beta}_{\delta_1} \to C^{k-4,\beta}_{\delta_1-4}
$$
is surjective and has Kernel equal to the constant functions if $0 < \delta_1 < 1$. We also know that the connection-laplacian $\Delta_{A,\delta_2}$ is an isomorphism for
$-2 < \delta_2 < 0$. From \eqref{eq:differentialphia} we conclude that $d_{\phi,a}L_{\delta_1,\delta_2|(0,0,0)}$ is an
isomorphism modulo constant functions in the weighted space $C^{k,\beta}_{\delta_1}$. Applying the implicit function theorem we obtain a curve of solutions
$\Phi:]-\epsilon,\epsilon[ \to
V_{\delta_1,\delta_2}$ to $L(\alpha,\Phi(\alpha)) = 0$ for small $\epsilon > 0$. The regularity of the solutions
follows from local uniqueness choosing arbitrary large values of the integer $k$.
\end{proof}

\begin{remark}
A crucial fact for the existence result is that the operator $L$ defined in \eqref{eq:operatorL} induces well defined
operators
$L_{\delta_1,\delta_2}$ for that values of the weights for which its linearization (see equation
\eqref{eq:differentialphia}) induces an isomorphism
in weighted H\"{o}lder spaces. This is not obvious a priori and strongly depends on the shape of the equations
\eqref{eq:ceqc2}.
\end{remark}

The curve $\Phi$ in the statement of Lemma~\ref{lemma:C2impfunc} provides a curve of pairs $(\omega_\alpha,A'_{\alpha})$
where $\omega_\alpha = \omega_0
+ dd^c \phi_\alpha$ is a K\"{a}hler metric in $\CC^2$ and $A'_\alpha$ is the Chern connection of the
Hermitian metric $e^{\imag
a_\alpha}H_0$. This curve of pairs is a curve of solutions to \eqref{eq:ceqc2} when we consider the system for
$SL(2,\CC)$ connections. To obtain
$SU(2)$-connections solving \eqref{eq:ceqc2} we must do a complex a gauge transformation $h_\alpha = e^{\imag
a_\alpha}$.
The new curve of
$SU(2)$-connections is
\begin{equation}\label{eq:familyinstantons}
A_\alpha = h_\alpha A^{1,0} h_\alpha^{-1} + h_\alpha \partial
h_\alpha^{-1} + h_\alpha^{-1} A^{1,0} h_\alpha + h_\alpha^{-1}
\dbar h_\alpha
\end{equation}
that tends uniformly to $A$ at infinity since $a_\alpha \in C^{\infty}_{\delta_2}$ with weight $-2 < \delta_2 < 0$.
Note also that the $2$-form
$\omega_\alpha - \omega_0 = dd^c \phi_\alpha$ has coefficients in $C^{\infty}_{\delta_1 - 2}$ with $0
< \delta_1 < 1$ so
$\omega_\alpha$ is asymptotically euclidean at infinity. Finally, a similar discussion as in the proof of
Proposition~\ref{propo:rangeinjectsurjec} shows that we can write
$$
\int_{B_r} \tr F_A \wedge F_A - \int_{B_r} \tr F_{A_\gamma} \wedge
F_{A_\gamma} = \int_{B_r} d\sigma = \int_{\partial B_r} \sigma
$$
where $B_r \subseteq \RR^4$ is the open ball of radius $r$ and $\sigma$ is a three form such that
$$
\lim_{r \to \infty} \int_{\partial B_r} \sigma = 0,
$$
due to the fact that its coefficients are in suitable $C^{\infty}_{\delta}$ spaces. This proves that the topological
constraint \eqref{eq:topologyfixed} is satisfied for every $\alpha$. An important remark is that the point-wise squared norm $|F_A|^2$ of the curvature of the basic instanton is non-constant with respect to the euclidean metric. In fact,
$$
|F_A|^2 = \frac{24}{(1 + |x|^2)^4},
$$
which is a bell-shaped function on $\CC^2$ centered at the origin. Similar arguments as in \secref{sec:example3}, show that when we apply Lemma~\ref{lemma:C2impfunc} to this initial data it produces solutions of the coupled system \eqref{eq:ASDeq} in which the K\"ahler metric has non constant scalar curvature.

A similar discussion can be done for $k$-instantons, with arbitrary $k \in \ZZ$ choosing suitable functions $u$ in
\eqref{eq:umultiinstanton}. For example, one can take
$$
u(x) = (x\Id - B)\lambda
$$
with $B$ a diagonal matrix with different quaternionic entries and $\lambda$ equals $(\lambda_1, \ldots,\lambda_k)$
with $\lambda_j$ real positive
scalars. The resulting $k$-instanton therefore looks like a superposition of $k$ instantons with scales $\lambda_j$
and centres $b_j$ (see \cite{At}).
This correspond to the classical solutions discovered by t' Hooft. For this particular choice of $u$'s the coefficients
of the corresponding instanton $A_u$ are again in $C^{\infty}_{-1}$, the coefficients of the curvature $F_{A_u}$ are
in $C^{\infty}_{-4}$, the connection-Laplacian admits an expression similar to \eqref{eq:Laplacianpolarcoordinates}
and, therefore, similar results as that of Lemma~\ref{propo:rangeinjectsurjec} and Lemma~\ref{lemma:solutionsmallcoupling} can be proved. We sum up this discussion in the following Lemma.

\begin{lemma}\label{lemma:solutionsmallcoupling}
Let $k \in \ZZ$, $\delta_1 \in ]0,1[ \subseteq \RR$ and $\delta_2 \in ]-2,0[ \subseteq \RR$. For each triple
$(k,\delta_1,\delta_2)$ and suitable
choice of a $k$-instanton $A$ there exists a small interval $I \subseteq \RR$ containing the origin such that for each
$\alpha \in I$ there exists a
solution $(\omega_\alpha,A_\alpha)$ of the coupled equations with coupling constant $\alpha$ and fixed topological
invariant \eqref{eq:topologyfixed}.
The metric $\omega_\alpha$ is an asymptotically euclidean K\"{a}hler metric with $\omega_\alpha - \omega_0 \in
C^{\infty}_{\delta_1 - 2}$ and
$A_\alpha$ converges uniformly to $A$ with $A_\alpha - A \in D^\infty_{\delta_2 - 1}$.
\end{lemma}

\begin{remark}
Note that the interval $I$ in the previous proposition depends a priori on the concrete $k$-instanton we chose to
apply
the implicit function theorem
(thus on the concrete function $u$ in formula \eqref{eq:umultiinstanton}). Note also that the connection $A_\alpha$ has
bounded $L^2$ norm in $\RR^4$
with respect to the Hermitian metric provided by
$\omega_\alpha$ due to the formula
$$
\|F_{A_\alpha}\|^2_{\omega_\alpha} =
\|F_{A_\alpha}^{-}\|^2_{\omega_\alpha} = \int_{\RR^4} \tr
F_{A_\alpha} \wedge F_{A_\alpha} = 8 \pi^2 k.
$$
\end{remark}

We finish with the main result of this section that holds by the previous Lemma together with a suitable re-scaling of
the solutions to
\eqref{eq:ceqc2} obtained for small coupling constant $\alpha$.

\begin{theorem}\label{proposition:solutionsarbitrarycoupling}
Let $k \in \ZZ$, $\delta_1 \in ]0,1[ \subseteq \RR$ and $\delta_2\in ]-2,0[ \subseteq \RR$. For each triple
$(k,\delta_1,\delta_2)$ and each $\alpha
\in \RR$ there exists a solution $(\omega_\alpha,A_\alpha)$ of the coupled equations with coupling constant $\alpha$
and fixed topological invariant
\eqref{eq:topologyfixed}. The metric $\omega_\alpha$ is an asymptotically euclidean K\"{a}hler metric with
$\omega_\alpha - \omega_0 \in
C^{\infty}_{\delta_1 - 2}$. For each $\alpha$ there exists a $k$-instanton $A_\alpha'$, such that $A_\alpha -
A_\alpha'
\in D^\infty_{\delta_2 - 1}$.
\end{theorem}
\begin{proof}
Let $(\omega,A)$ be a solution to \eqref{eq:ceqc2} with non-vanishing coupling constant $\alpha$ and topological
invariant \eqref{eq:topologyfixed}
for some $k \in \ZZ$. We claim that for each $\beta \in \RR$ such that $\alpha \cdot \beta > 0$ there exists a
solution
$(\omega,A_\beta)$ with
coupling constant $\beta$ and the same topology. Let $c = \beta/\alpha > 0$ and consider the K\"{a}hler form
$c\omega$.
Then the pair $(c\omega,A)$
satisfies
\begin{align*}
S_{c\omega} \; + \; \beta *_{c\omega} \tr F_A \wedge F_A & =
\frac{1}{c}S_\omega \; + \; \frac{\beta}{c^2} *_{\omega} \tr F_A
\wedge F_A\\
& = \frac{1}{c}(S_\omega \; + \; \alpha *_{\omega} \tr F_A \wedge
F_A)\\
& = \beta 8 \pi^2 k
\end{align*}
so it is a solution with coupling constant $\beta$, due to the fact that the ASD equation \eqref{eq:ASDeq} is
conformally invariant. Consider now the
homothety $h_{1/\sqrt{c}}(z) = z/\sqrt{c}$, that satisfies $h_{1/\sqrt{c}}^* c\omega = \omega$. Then it clear that
$(\omega,A_\beta)$ is a solution
with coupling constant $\beta$. The connection $A_\beta$ has fixed topological invariant $8 \pi^2 k$ in
\eqref{eq:topologyfixed} since the instanton
number $k$ is preserved by conformal transformations (see \cite{At} and \cite{DK}). It is clear from the proof that if
the initial connection
converges at infinity to some $k$-instanton $A'$, then the re-scaled connection converges to $h_{1/\sqrt{c}}^*A'$.
\end{proof}

\chapter{\bf Conclusions and future directions}

In this thesis we have defined coupled equations for K\"ahler metrics and Hermite--Yang--Mills connections with symplectic and variational interpretations. We have seen that these equations are well suited for the moduli problem with K\"ahler structure for metrics and connections. We have constructed families of examples of solutions and we have related the coupled equations with the cscK problem on ruled manifolds. We have defined a new notion of stability for triples $(X,L,E)$, consisting of a projetive polarised manifold $(X,L)$  and a holomorphic vector bundle $E$. A conjecture relating this stability notion with the existence of solutions of the coupled system has been stated and supporting evidence has been given.

This thesis starts a programme, the one of studying coupled equations for pairs consisting of a K\"ahler metric and a connection, and sets many open questions that will be considered in future work. Possible future directions are: $1)$ construct moduli spaces of metrics and connections with K\"ahler structure, $2)$ proving a Hitchin--Kobayashi correspondence for the coupled equations, starting with the case of toric manifolds and toric vector bundles, $3)$ development of gluing methods for the coupled equations adapting those in \cite{AP} for asymptotically Euclidean K\"ahler metrics and the well known methods for instantons of H.C. Taubes and S.K. Donaldson (see~\cite{DK}), $4)$ study the relation of the coupled equations with the cscK equation for a K\"ahler metric on the total space of an associated bundle.

As an immediate future project we will try to prove that the existence of solutions to \eqref{eq:CYMeq00} implies the $\alpha$-K-semistability of the triple adapting Donaldson's methods in \cite{D8}. This fact was first verified for the cscK problem, after the work of S. Zhang, H. Luo and S.K. Donaldson.

\end{document}